\pdfoutput=1
\makeatletter\def\input@path{{FIGURES/}}\makeatother
\documentclass[10pt,a4paper]{article}  
\usepackage[utf8]{inputenc}
\usepackage[T1]{fontenc}
\DeclareFontFamily{U}{mathx}{}
\DeclareFontShape{U}{mathx}{m}{n}{<-> mathx10}{}
\DeclareSymbolFont{mathx}{U}{mathx}{m}{n} 
\DeclareMathAccent{\widehat}{0}{mathx}{"70}
\DeclareMathAccent{\widecheck}{0}{mathx}{"71}  
\usepackage{tikz} 
\usepackage{mathrsfs}
\usepackage{amsbsy}
\usepackage{amsmath, amsfonts, amsthm, amssymb, amscd}
\usepackage{accents, verbatim}
\usepackage{pdfpages}
\usepackage{graphicx}
\usepackage{xcolor}
\usepackage{microtype}
\usepackage[nottoc]{tocbibind} 
\usepackage{zref-perpage} \zmakeperpage{footnote}
\usepackage{float}
\usepackage{xr-hyper}
\usepackage{hyperref}
\usepackage[normalem]{ulem}
\usepackage{array}
\usepackage{a4wide}
\usepackage{titlesec}
\usepackage{hyperref,bookmark,footnotebackref}
\hypersetup{linktoc = all}
\hypersetup{hidelinks}
\hypersetup{bookmarksnumbered}
\usepackage{calligra}
\usepackage[T1]{fontenc} 
\usepackage[mathscr]{euscript}
\usepackage{bbm}
\makeatletter
\let\oldaligned\aligned
\def\@alignedenvir{aligned}
\def\aligned{\ifx\@currenvir\@alignedenvir\expandafter\@firstoftwo\fi\oldaligned\relax}
\makeatother
\numberwithin{figure}{section}
\numberwithin{table}{section}

\usepackage{aliascnt}
\theoremstyle{plain}
\newtheorem{theorem}{Theorem}[section]
\newcommand{\mynewtheorem}[2]{
  \newaliascnt{#1}{theorem}
  \newtheorem{#1}[#1]{#2}
  \aliascntresetthe{#1}
  \expandafter\providecommand\csname #1autorefname\endcsname{#2}
}
\mynewtheorem{definition}{Definition}
\mynewtheorem{proposition}{Proposition}
\mynewtheorem{lemma}{Lemma}
\mynewtheorem{corollary}{Corollary}
\mynewtheorem{remark}{Remark}
\mynewtheorem{claim}{Claim}
\mynewtheorem{conjecture}{Conjecture}
\mynewtheorem{problem}{Problem}
\mynewtheorem{hypothesis}{Hypothesis}
\mynewtheorem{assumption}{Assumption}
\numberwithin{equation}{section}
\DeclareMathAlphabet\mathbfcal{OMS}{cmsy}{b}{n} 

\renewcommand \geq {\geqslant}
\renewcommand \leq {\leqslant}
\newcommand \be 	{\begin{equation}}
\newcommand \ee 	{\end{equation}} 
\newcommand \bel 	{\be \label}
\newcommand \bei 	{\begin{itemize}}
\newcommand \eei 	{\end{itemize}}
\let\oldmarginpar\marginpar
\renewcommand\marginpar[1]{\ifhmode\unskip\fi\- \oldmarginpar[\raggedleft\footnotesize #1]{\raggedright\footnotesize #1}}

\newcommand \bse {\begin{subequations}}
\newcommand \ese {\end{subequations}}

\newcommand \Jperp  {\mathbb J^\perp{}}
\newcommand \Jpar {J^\parallel{}}

\newcommand \Jhatpar { \widehat{J}^\parallel{}}

\newcommand \Jbb {\mathbb J} 

\newcommand \Nbb {\mathbb N}

\newcommand \Xbb {\mathbb X} 
\newcommand \Ybb {\mathbb Y}

\newcommand \Hess {\operatorname{Hess}}

\newcommand \Ebf{E}

\newcommand \del	{\partial}
\newcommand \auth 	{\textsc}   
\newcommand \Mcal 	{\mathcal M}

\newcommand \loc   	{\textnormal{loc}} 
\newcommand \RR 	{\mathbb R}

\newcommand \eps 	{\epsilon}

\newcommand \supp 	{\operatorname{supp}}

\newcommand \Ocal {\mathcal O}

\newcommand \Rcal {\mathcal{R}}

\newcommand \coloneqq {\mathrel{\mathop :}\mathrel{\mkern-1.2mu}=}

\makeatletter
\newcommand*\smallbullet{\mathpalette\smallbullet@{.5}}
\newcommand*\smallbullet@[2]{\mathbin{\vcenter{\hbox{\scalebox{#2}{$\m@th#1\bullet$}}}}}
\makeatother

\newcommand \Mbf{M}

\newcommand \Fcomp {F}
\newcommand \Fvee {F^\vee}

\newcommand \Numb {N}
\newcommand \hNumb {\widehat{N}}
\newcommand \vNumb {{\mathbb N}}

\newcommand \funmu {\underline{\smash{\ensuremath{\mu}}}}

\newcommand{\meff}[1]{\mu_{#1}}

\newcommand{\Neff}[1]{N_{#1}}

\DeclareFontFamily{U}{mathx}{}
\DeclareFontShape{U}{mathx}{m}{n}{<-> mathx10}{}
\DeclareSymbolFont{mathx}{U}{mathx}{m}{n}
\DeclareMathAccent{\widehat}{0}{mathx}{"70}
\DeclareMathAccent{\widecheck}{0}{mathx}{"71}
\newcommand \Rbb {\mathbb R}
\DeclareMathOperator\artanh{artanh}

\newcommand{\Jac}{\mathbf{Jac}}
\newcommand{\yqu}{a}
\newcommand{\entropygenerator}{\vartheta}

\newcommand{\auxeps}{{\delta_1}}
\newcommand{\kernelpair}[1]{\ifstrempty{#1}{(\alpha_\chi,\beta_\chi)}{(\alpha_{\chi_{#1}},\beta_{\chi_{#1}})}}

\titleformat{\subsubsection}[runin]{\normalfont\bfseries}{\thesubsubsection}{.5em}{}[\ifnum\spacefactor<1001.\fi\ \ ]
\titlespacing{\section}{0pt}{2.4ex plus .6ex minus .2ex}{1.4ex plus .2ex}
\titlespacing{\subsection}{0pt}{2ex plus .5ex minus .2ex}{.8ex plus .2ex}
\titlespacing{\subsubsection}{0pt}{1.2ex plus .1ex minus .1ex}{0pt}
\makeatletter
\newcommand{\adjustdisplayspacing}{\abovedisplayskip=6pt plus2pt minus3pt\belowdisplayskip=6pt plus2pt minus3pt\belowdisplayshortskip=4pt plus2pt minus2pt}
\g@addto@macro\normalsize{\adjustdisplayspacing}
\g@addto@macro\small{\adjustdisplayspacing}
\AtBeginDocument{\normalsize}
\makeatother

\begin{document}

\title{\Large\bfseries An infinite family of conservation laws\\ for perfect compressible fluids}

\author{Bruno Le Floch\thanks{Laboratoire de Physique Théorique et Hautes Énergies \& Centre National de la Recherche Scientifique, Sorbonne Universit\'e, 4 Place Jussieu, 75252 Paris, France. Email: {\tt bruno@le-floch.fr}.}
\, and Philippe G. LeFloch\thanks{Laboratoire Jacques-Louis Lions \& Centre National de la Recherche Scientifique, Sorbonne Universit\'e,  4 Place Jussieu, 75252 Paris, France. Email: {\tt contact@philippelefloch.org}. \newline \emph{2020 Mathematics Subject Classification.} 35L65, 76Y05. {\it Keywords and phrases.} Relativistic Euler system; conservation law; spacetime quasi-current; vacuum state; general pressure law. \hfill October 2026.
}}
\date{}
\maketitle

\begin{center}
{\scshape \small Dedicated to Constantine M. Dafermos\\ on the occasion of his 85th birthday}
\vspace{\baselineskip}
\end{center}

\begin{abstract}
We prove the existence of an \emph{infinite family of conservation laws} and, more generally, balance laws, satisfied by plane-symmetric flows of perfect compressible fluids, allowing all components of the velocity vector to be possibly non-vanishing. We cover \emph{general pressure laws} satisfying solely hyperbolicity and asymptotic conditions. We exhibit an \emph{almost quadratic formulation} of the Euler equations, which we refer to as the $\Jbb$-formulation, and introduce a \emph{reduced Euler system} based on an effective mass-energy density and an effective pressure law which account for transverse velocities. We establish the existence of a fundamental kernel which we use to generate all spacetime (quasi-)currents, also called mathematical entropy pairs. We derive explicit formulas for their density and flux components, together with the source-terms of the corresponding balance laws. Furthermore, we establish regularity and asymptotic estimates covering the \emph{whole range} of physical variables, including vacuum. We prove that the kernel is positive on the interior of its support provided the pressure law satisfies a genuine nonlinearity condition, and we also investigate the \emph{future-convexity} of the spacetime currents. Throughout, the fluid momentum parallel to the orbits of plane-symmetry is kept arbitrary. In turn, this family of spacetime quasi-currents provides a key ingredient in our companion study of torus-symmetric Einstein--Euler flows with weak regularity.
\end{abstract} 
 
{\small

\setcounter{secnumdepth}{2} 
\setcounter{tocdepth}{1}
\makeatletter
\patchcmd{\l@section}{1.0em \@plus\p@}{.12em \@plus .08em}{}{\errmessage{TOC section spacing patch failed}}
\makeatother
\let\bfseries\mdseries
\renewcommand{\contentsname}{{\large\fontseries{bx}\selectfont Content}}
\tableofcontents

}

\

\

\section{Introduction}
\label{section-1}

\subsection{Perfect compressible fluids}
\label{section-1-1}

\subsubsection{The importance of conservation laws}

We study here the relativistic Euler system describing the evolution of a compressible perfect fluid on $(1+n)$-dimensional Minkowski spacetime, with arbitrary spatial dimension $n\geq2$. We assume plane symmetry, which restricts the dependence to the time coordinate and one spatial coordinate, while allowing all $n$ possibly non-vanishing velocity components. We fix a dimensionless inverse speed of light $\eps\in(0,1]$ and consider some density-pressure law $p=p(\mu)$, where $\mu\geq0$ is the mass-energy density divided by the squared speed of light. We denote the Minkowski metric by $g$, its covariant derivative by $\nabla$, and the future-directed unit fluid velocity by $u$, with $g(u,u)=-1$. The stress-energy tensor is\footnote{The time coordinate is rescaled by the speed of light; Section~\ref{section:epsilon-geometric-normalization} gives the conversion to physical coordinates and velocities.}
\bel{eq:intro-T}
T_{\alpha\beta}=(\mu+\eps^2p)u_\alpha u_\beta+\eps^2p g_{\alpha\beta}
\ee
and the Euler equations are the divergence laws
\bel{eq:intro-Euler-cov}
\nabla^\alpha T_{\alpha\beta}=0, \qquad \beta=0,\ldots,n,
\ee
with an implicit summation over $\alpha=0,\dots,n$.
We are interested in ``additional'' conservation laws satisfied by all sufficiently regular solutions to \eqref{eq:intro-Euler-cov} and, specifically, we prove the existence of \emph{infinitely many spacetime currents} for the Euler system obtained by fixing the normalized momentum parallel to the plane-symmetry orbits. We cover the full state space \emph{from vacuum to arbitrarily large density}. We also derive the quasi-balance identity satisfied by an arbitrary current, including terms that account for the variation of fluid momentum components that are parallel to the orbits of symmetry.  Further estimates for the constructed families of currents are postponed to the companion article~\cite{TorusT2-3-asymptotic-2026}. 

Recall that conservation laws and, more generally balance laws, are particularly useful to study weak solutions to the Euler equations, for instance to formulate admissibility conditions (i.e. mathematical entropy inequalities), derive invariant regions, as well as derive compensated compactness properties for exact or approximate solutions. Indeed, divergence laws underlie Tartar's compensated compactness method~\cite{Tartar-1979}, applied to compressible Euler flows in plane-symmetry in a seminal work by DiPerna~\cite{DiPerna-1983a,DiPerna-1983b}. Conservation laws for the Euler equations play a key role in Dafermos's a priori regularity and stability theory~\cite{Dafermos-1979-stability,Dafermos-book}. LeFloch~\cite{LeFloch-curved} proved a weak--strong uniqueness theory for relativistic flows. For non-relativistic flows, Abbrescia and Speck~\cite{AbbresciaSpeck-2025} derived localized integral identities.

\subsubsection{Main purpose}

Our objective in the present paper is the \emph{derivation} and \emph{parametrization} of conservation laws satisfied by regular solutions to the Euler equations, whereas, for their application to the study of sequences of exact or regularized solutions, we refer to the companion papers~\cite{TorusT2-1-2026,TorusT2-2-2026}. Our construction combines Lorentz covariance of the density--flux pair with the analysis of a current equation, namely a second-order linear wave equation in (density and velocity) state space which is \emph{singular at vacuum} and enjoys very different asymptotic properties in the small- and large-density regimes. Freezing the (suitably normalized) momentum parallel to the symmetry orbits yields an \emph{effective pressure law}, which also satisfies natural hyperbolicity and asymptotic properties. 

To proceed, we work with the causal fluid momentum $\Jbb$ introduced in~\cite{TorusT2-1-2026}, which ranges over a closed future causal cone and remains meaningful at vacuum. In a \emph{fixed frame}, we denote the time component and the spatial component orthogonal to the symmetry orbits of a current by $(F^{0},F^{1})$, and their values in the \emph{rest frame of the fluid motion} orthogonal to the symmetry orbits by $(\alpha_{F},\beta_{F})$; they are related by a local Lorentz boost by the fluid velocity. We proceed to study current components in the rest frame, expressed in terms of the mass-energy density and the fluid rapidity orthogonal to the symmetry orbits. The equation satisfied by currents is linear hyperbolic and its coefficients only depend on the mass-energy. The equation admits a fundamental kernel, with two components, which is used to represent by convolution every spacetime current whose rest-frame components vanish at vacuum and whose \emph{vacuum profile} (namely the leading term in an expansion near the vacuum) is a prescribed continuous function. Continuous profiles give continuous currents, while suitable growth conditions on the vacuum profile imply finite-energy bounds and some control towards the boundary of state space (as we explain in Section~\ref{section-9-2}, below).

\subsubsection{The notion of spacetime quasi-currents}

In addition to deriving conservation laws, we also derive balance laws satisfied by plane-symmetric flows depending on time $t$ and one spatial variable $x$, with all $n$ velocity components being possibly non-zero. Away from vacuum, a suitably \emph{normalized momentum} parallel to the symmetry orbits is defined and denoted by \(\Jhatpar=(\widehat J_2,\ldots,\widehat J_n)\in\Rbb^{n-1};\) cf.~its definition in~\eqref{eq:def-Jhat}, below. At vacuum, a representative of $\Jhatpar$ must be prescribed separately, since it is not determined by the standard fluid momentum $\Jbb$. We observe that each component $\widehat J_m$ is \emph{advected by the fluid motion} orthogonal to the symmetry orbits. Furthermore, when $\Jhatpar$ is frozen, the remaining Euler equations form a \text{reduced relativistic Euler system} in one spatial dimension which is defined from an \emph{effective density} and an \emph{effective pressure}. More generally, when $\Jhatpar$ is allowed to vary, a current $\Fcomp=(\Fcomp^0,\Fcomp^1)$ of the frozen system satisfies the \emph{balance law} 
\bel{eq:intro-ft-quasi-balance}
\del_t\Fcomp^0+\del_x\Fcomp^1 =\sum_{m=2}^n\Fvee_m\del_x\widehat J_m, 
\ee
and then we call $\Fcomp$ a \emph{spacetime quasi-current}. Its source factors $\Fvee_m$ come from the dependence of the effective pressure and the vacuum profile upon the parallel momentum. The identity~\eqref{eq:intro-ft-quasi-balance} holds for the full plane-symmetric Euler system and reduces to a \emph{conservation law} for solutions with constant normalized parallel momentum.

\subsection{Main contributions}
\label{section-introduction-kernel-constructions}

Our results in the present paper concern the existence and properties of the (weak, strong) fundamental kernels, which then generate via a convolution representation the whole family of spacetime quasi-currents, including the particle-number and stress-energy currents. We construct the density and flux components together, identify the leading profiles of the fundamental kernels at both characteristic boundaries of their support, and establish estimates on currents covering the full density range.  This only relies on hyperbolicity and asymptotic conditions on the pressure law; we cover every polytropic exponent $\gamma>1$ at vacuum and any isothermal behavior in the large density limit.
We also find the condition required on the pressure for the current kernel to be \emph{future-timelike}, as well as the condition required for the currents themselves to be future-convex. 
The companion article~\cite{TorusT2-3-asymptotic-2026} is devoted to quantitative estimates for derivatives in the parallel momentum and to the non-relativistic limit.  

Earlier constructions of currents for Euler flows include the following results.
\bei
\item 
DiPerna~\cite{DiPerna-1983a,DiPerna-1983b} first observed that the current kernel for non-relativistic fluid flows in one spatial variable admits a closed formula provided the fluid is governed by a polytropic equation of state (namely $p(\mu) = \mu^\gamma$ with $\gamma>1$). 

\item Next, current kernels for non-relativistic fluids governed by a class of pressure laws that includes \emph{superpositions} of polytropic laws were constructed by Chen and LeFloch~\cite{CL1-2000,CL2-2003}. Makino and his collaborators~\cite{Hsu-Makino-2001,Makino-2001} studied analytic perturbations of polytropic fluids.
Chen and Schrecker~\cite{CS-2022} used a perturbation argument to extend~\cite{CL1-2000} to relativistic flows, \emph{provided} the fluid velocity is sufficiently small with respect to the light speed. 
LeFloch and Shelukhin~\cite{LeFlochShelukhin-2005} and LeFloch and Yamazaki~\cite{LeFlochYamazaki-2007}, and later on Schrecker and Schulz~\cite{SchreckerSchulz-2020} studied (asymptotically) isothermal laws. 

\eei
\noindent These works motivate the construction of currents throughout the full range of state variables. In this paper, we present a comprehensive theory of (non)-relativistic fluid flows with \emph{arbitrary, unrestricted fluid velocities} and for \emph{general pressure functions} subject only to hyperbolicity and asymptotic conditions. 

\subsection{Organization of this paper}
\label{section-1-5}

\subsubsection{Fluid equations and reference currents}

This paper is organized as follows. In Section~\ref{section-main-results}, we state the assumptions required on the pressure and present the main results on the fundamental kernel for currents, the representation of quasi-currents by their vacuum profiles, and the criteria for future-convexity, as stated in~Theorems~\ref{theo-main-results} and~\ref{theo-intro-current-representation}, and Corollaries~\ref{cor:main-quasi-current-kernel} and~\ref{theo-intro-convexity-classification}, below. We also explain in this section the role of these currents in the companion study of the Einstein-Euler system.

In Section~\ref{section-2}, we introduce the physical variables and their normalization, the fluid momentum variable, and the spacetime currents associated with general flows. The current balance identity (Proposition~\ref{prop:section-3-current-identity}) provides the starting point for the discussion of future-convexity and mathematical entropy inequalities.

In Section~\ref{section-3}, we derive the plane-symmetric equations with all velocity components and identify the advected momenta parallel to the symmetry orbits. We also describe the transformation of current components under Lorentz boosts and obtain the quasi-balance law, with explicit source factors; cf.~Proposition~\ref{prop:section-4-lorentz-quasi-balance}, below.

In Section~\ref{section-4}, we study the particle-number and energy currents that serve as reference currents throughout this paper. Their convexity properties (cf.~Propositions~\ref{lem-convexM} and~\ref{lem-convexN}) lead to the two mathematical formulations of weak solutions proposed in~\cite{TorusT2-1-2026} (in the context of the Einstein-Euler system), based respectively on dissipated energy and particle production.

\subsubsection{The effective Euler system and construction of the kernel}

In Section~\ref{section-5}, we reduce the Euler equations at arbitrary fixed parallel momentum to an effective system with a suitably defined effective density and pressure. We prove that our pressure admissibility assumptions remain valid for the effective system (cf.~Proposition~\ref{prop:fixed-effective-admissibility}) and we then derive the second-order wave equation satisfied by the spacetime currents, in physical and Riemann-invariant variables. The (rather singular) behavior of the coefficients of this equation, both in small- and large-density limits, governs the subsequent analysis of this paper.

In Section~\ref{section-6}, we begin the Fourier analysis of the current equation and establish the existence of the (weak, strong) kernel for any given Fourier parameter (Theorem~\ref{theo:section-7-fourier-branches}). We construct the unique weak branch, as well as the affine family of strong branches selected at vacuum, and we estimate their behavior for all positive density, including oscillatory or exponential behavior at large density.

In Section~\ref{section-fourier-continuation}, we connect these local and asymptotic descriptions and sharpen the estimates uniformly in Fourier frequency. The analysis includes turning points and threshold frequencies, gives the bounds needed for Fourier inversion (cf.~Proposition~\ref{cor:kernel-Sobolev-regularity}), and specifies a strong representative with controlled growth (cf.~Proposition~\ref{prop:section-8-selected-strong-branch}).

In Section~\ref{section-7}, we are in a position to return to the physical state variables by inverse Fourier transformation. We prove finite propagation and determine the leading behavior along the characteristic boundary of the support of the kernel, and we then establish positivity and comparison properties under appropriate pressure conditions; cf.~Propositions~\ref{prop:distributional-leading-trace} and~\ref{theo:section-8-causal-krl-current}. These properties will control both the size and the causal character of currents. 

\subsubsection{Representation of currents and convexity criteria}

We continue in Section~\ref{section-flux-kernels} by constructing the flux kernel associated with each density kernel and reconstructing the resulting spacetime vector. In addition, we also deal with the strong current kernel (cf.~Proposition~\ref{prop:strong-entropy-flux-kernel}) and identify the characteristic profiles of both components, while establishing global bounds and regularity properties at vacuum.

In Section~\ref{section-9}, we prove that weak currents are represented by convolution with prescribed vacuum profiles; cf.~Proposition~\ref{thm:param-weak-entropy-currents}. We explain the covariance of this representation formula, and we identify the vacuum profiles corresponding to the particle-number, energy, and momentum currents. We also give a growth condition ensuring that the resulting currents are globally bounded by the energy, a property that is particularly useful in the applications. 

In Section~\ref{section-10}, we compute the Hessian of a current in the two systems of conservative variables of interest, to find criteria for future-convexity of the current. The vacuum limit gives necessary conditions on the profile, while kernel positivity and the corresponding pressure hypotheses prove sufficiency at positive density; cf.~Proposition~\ref{theo:section-8-pressure-independent-convexity}. Returning to the definition of the effective pressure, we arrive at the relevant criteria associated with quasi-currents at each fixed parallel momentum.

Finally, in Appendix~\ref{appendix:AB-from-ab} we display some calculations in  Riemann-invariant variables. In Appendix~\ref{appendix-olver}, we state and prove the Liouville--Green estimates used for the estimates at fixed Fourier parameter, and in Appendix~\ref{appendix-pressure-endpoints} we derive some consequences of our assumptions on the pressure at vacuum and at infinity.

\section{Main results and applications}
\label{section-main-results}

\subsection{Equations of state}
\label{section-1-3}

\subsubsection{Hyperbolicity and asymptotics conditions}

We state here the main results established in this paper. 
Throughout, the dimensionless inverse speed of light $\eps\in(0,1]$ is fixed. We denote the mass-energy density divided by the squared speed of light by $\mu\geq0$ and the pressure by $p=p(\mu)$.

\bei 

\item {\bf Hyperbolicity.}  We assume that $p\in C^3((0,+\infty))\cap C^1([0,+\infty))$ satisfies
\bel{hyperbolic-eos-1}
p(0)=0, \qquad 0<p'(\mu)<\eps^{-2}\quad(\mu>0), \qquad 0\leq p'(0)<\eps^{-2},
\ee
where the derivative $p'(\mu)$ represents the square of the sound speed.

\item {\bf Strictly subluminal.} We also impose the large-density condition
\bel{hyperbolic-eos-2}
0<\liminf_{\mu\to+\infty}\frac{p(\mu)}{\mu} \leq\sup_{\mu>0}\frac{p(\mu)}{\mu}<\eps^{-2}.
\ee
\eei 

\bei
\item {\bf Vacuum behavior.} There exist $\gamma>1$, $\kappa_0>0$, and $\mu_0>0$ such that\footnote{The remainders \(r_-\) and \(r_+\) may depend on the fixed parameter \(\eps\); we suppress this dependence in their notation.}
\bse\label{equa-cond-small}
\bel{eq:intro-vacuum-B}
\frac{\mu^2p'(\mu)}{(\mu+\eps^2 p(\mu))^2}=\kappa_0\mu^{\gamma-1}(1+r_-(\mu)), \qquad 0<\mu\leq\mu_0,
\ee
where $r_-(\mu)\to0$ as $\mu\to0$ and
\bel{eq:intro-vacuum-r}
\int_0^{\mu_0}\Bigl(|r_-'(\mu)|+|\mu r_-''(\mu)|\Bigr)\,d\mu<+\infty.
\ee
\ese
\item {\bf Large-density behavior.} There exist $k_\infty\in(0,\eps^{-1})$, $\mu_1>0$, and $r_+\in C^2([\mu_1,+\infty))$ such that
\bse
\label{equa-cond-large}
\bel{eq:intro-infty-B}
\frac{\mu^2p'(\mu)}{(\mu+\eps^2 p(\mu))^2}=\frac{k_\infty^2}{(1+\eps^2k_\infty^2)^2}(1+r_+(\mu)), \qquad \mu\geq\mu_1,
\ee
where $r_+(\mu)\to0$ as $\mu\to+\infty$ and
\bel{eq:intro-infty-r}
\int_{\mu_1}^{+\infty}\Bigl(|r_+'(\mu)|+|\mu r_+''(\mu)|\Bigr)\,d\mu<+\infty.
\ee
\ese
\eei
\noindent A key coefficient of the current equation will be
\bel{eq:intro-def-B}
B(\mu)\coloneqq\frac{\eps^2\mu^2p'(\mu)}{(\mu+\eps^2p(\mu))^2}
\ee
which asymptotes to $\eps^2\kappa_0\mu^{\gamma-1}$ and $\eps^2 k_\infty^2/(1+\eps^2k_\infty^2)^2$ at vacuum and large density, respectively, under the asymptotic conditions above.
Proposition~\ref{prop:pressure-endpoint-asymptotics} deduces the asymptotics of the pressure law, namely $p(\mu)=\kappa_0\gamma^{-1}\mu^\gamma(1+o(1))$ at vacuum and $p(\mu)=k_\infty^2\mu(1+o(1))$ at large density.

\begin{definition}
\label{def-pression}
A pressure law $p=p(\mu)$ is called an \emph{admissible equation of state} if it satisfies the hyperbolicity conditions~\eqref{hyperbolic-eos-1}--\eqref{hyperbolic-eos-2}, and the asymptotic conditions~\eqref{equa-cond-small}--\eqref{equa-cond-large}.
\end{definition}

In addition, for some of our results we will also assume one of the following additional conditions on the pressure.
\bei

\item {\bf Genuine nonlinearity conditions.}
The particle-production convexity condition is defined as 
\be
\label{hyperbolic-eos-3}
G_{\rm N}(\mu)\coloneqq \eps^2\bigl((\mu+\eps^2p(\mu))p''(\mu) +2p'(\mu)(1-\eps^2p'(\mu))\bigr)>0 \qquad(\mu>0).
\ee
The mass-energy convexity condition is defined as 
\be
\label{eq:intro-mass-energy-convexity-condition}
G_{\rm E}(\mu)\coloneqq \eps^2\bigl((\mu+\eps^2p(\mu))p''(\mu) +2p'(\mu)\bigl(1-\eps\sqrt{p'(\mu)}\bigr)^2\bigr)\geq0 \qquad(\mu>0).
\ee
These additional signs are used for the properties specified below; admissibility suffices for the kernels.\footnote{Observe that $G_{\rm N}=G_{\rm E}+4c^3(1-c)$, hence \(G_{\rm E}\geq0\) implies \(G_{\rm N}>0\).}

We write $G_{\rm N}[q]$ and $G_{\rm E}[q]$ when a pressure law $q$ replaces $p$; thus $G_{\rm N}=G_{\rm N}[p]$ and $G_{\rm E}=G_{\rm E}[p]$.
\eei

\subsection{Main results}
\label{section-1-2}

\subsubsection{Spacetime current kernels}

We work on Minkowski spacetime of signature $(-,+,\ldots,+)$, in spatial dimension $n\geq2$, and define $h(\mu)\coloneqq\mu+\eps^2p(\mu)$.  Sections~\ref{section-2} to~\ref{section-5} give the fluid variables and symmetry reduction.
The variable $v$ denotes the \emph{fluid rapidity}, defined by $\tanh v=u^1/u^0$ for the future-directed unit fluid velocity $u$. We use $k(\mu)$ for the \emph{acoustic distance} from vacuum, defined by $k'(\mu)=\eps\sqrt{p'(\mu)}/h(\mu)$ and $k(0)=0$. As we shall see, the interval $|v-s|\leq k(\mu)$ is the support region of a kernel centered at $v=s$ for some rapidity~$s$, and $\lambda\coloneqq(3-\gamma)/(2(\gamma-1))$ is the exponent of its leading boundary profile.

We denote the vector kernel by $\chi=\chi(\mu,v;s)$ and its laboratory components as $\chi^0,\chi^1$. Its scalar coefficients in the frame of the fluid motion orthogonal to the symmetry orbits are $\alpha_\chi,\beta_\chi$, so that $\chi=\chi^0e_0+\chi^1e_1=\alpha_\chi\overline u+\beta_\chi\overline u^\perp$. Here $(e_0,e_1)$ is the laboratory orthonormal basis, $\overline u\coloneqq\cosh(v)e_0+\sinh(v)e_1$, and $\overline u^\perp\coloneqq\sinh(v)e_0+\cosh(v)e_1$. Because coefficients of the current equation only depend on~$\mu$, the fundamental kernel only depends on the difference $v-s$, so it suffices to consider the kernel with source at zero rapidity $s=0$ at vacuum. A \emph{weak kernel} has both coefficients zero at vacuum and $\del_\mu\alpha_\chi=\delta_0$ there, with $\delta_0$ the unit Dirac mass at zero rapidity; a \emph{strong kernel} instead has $\alpha_\chi=\delta_0$ at vacuum. A representative of the strong family is selected by Cauchy data at positive density in Section~\ref{section-strong-kernel}, below. We use a hat for the Fourier transform in rapidity, with convention $\widehat f(\xi)=\int_\Rbb e^{-i\xi v}f(v)\,dv$, and $H^\ell=H^\ell(\Rbb)$ for the usual $L^2$ Sobolev space. All vacuum traces are understood in the sense of distributions, and constants may depend on the given pressure function, on $\eps$, and on the stated Sobolev order, but not on density or frequency.

\begin{theorem}[Fundamental kernel of spacetime currents]
\label{theo-main-results}
Let $p=p(\mu)$ be an admissible equation of state in the sense of Definition~\ref{def-pression}. There exists a unique weak current kernel $\chi=\chi(\mu,v)$ satisfying the following properties.
\bei
\item The kernel is defined for every mass-energy density $\mu\in[0,+\infty)$ and every fluid rapidity $v\in\Rbb$, is supported in the acoustic cone $|v|\leq k(\mu)$, and is selected by the vacuum normalization
\bel{eq:intro-kernel-normalization}
\alpha_{\chi}|_{\mu=0}=0, \qquad \del_\mu\alpha_{\chi}|_{\mu=0}=\delta_{v=0}, \qquad \beta_{\chi}|_{\mu=0}=0.
\ee
The scalar coefficients reconstruct the spacetime vector $\chi$ in the fluid frame and solve, at positive density, the first-order system
\be
\label{eq:intro-current-kernel-system}
(\mu+\eps^{2}p)\del_\mu\alpha_{\chi}=\alpha_{\chi}+\del_v\beta_{\chi},\qquad (\mu+\eps^{2}p)\del_\mu\beta_{\chi}=\eps^{2}p'(\beta_{\chi}+\del_v\alpha_{\chi}).
\ee
\item The Fourier transforms of $\alpha_\chi$ and $\beta_\chi$ admit global bounds covering the vacuum, intermediate-density, and large-density regimes, including two critical frequencies at which the limiting large-density Fourier equation changes between asymptotically oscillatory and asymptotically exponential behavior. For $\mu>0$ and any Fourier frequency $\xi\in\Rbb$, 
after setting 
$k_{\min}(\mu)\coloneqq\min\{k(\mu),1\}$, the Fourier transforms satisfy (for a constant $C(\eps)>0$  depending also on the pressure law)
\be
\label{eq:intro-current-Fourier-decay}
|\widehat{\alpha_{\chi}}(\mu,\xi)|+|\widehat{\beta_{\chi}}(\mu,\xi)| \leq \frac{C(\eps)\mu}{\bigl(1+ k_{\min}(\mu) |\xi|\bigr)^{\lambda+1}},
\ee
while, for every regularity exponent $0\leq\ell<\lambda+1/2$, the corresponding Sobolev estimate at each density is
\bel{eq:intro-kernel-Sobolev-bound}
\|\alpha_{\chi}(\mu,\cdot)\|_{H^\ell} +\|\beta_{\chi}(\mu,\cdot)\|_{H^\ell} \leq C_\ell\mu \, k_{\min}(\mu)^{-\ell-1/2}.
\ee
At large density, the weight \(\mu\) improves to \(\mu^{1/(1+\eps^2k_\infty^2)+\delta}\) for every \(\delta>0\), and to \(\sqrt\mu\) on each supercritical frequency range separated from the thresholds; the precise frequency-dependent bounds are given in Proposition~\ref{prop:fixed-current-bounds}, below. The Fourier bound~\eqref{eq:intro-current-Fourier-decay} also gives continuity of both components of $\chi$ in $H^{-1}$ up to vacuum, with value zero.

\item Both kernel components have \emph{explicit leading terms} near the boundaries of their support $v=\pm k(\mu)$: after rescaling the distance to either boundary by a factor $r>0$ and dividing the errors by $r^\lambda$, the \emph{errors tend to zero} in $L^2$ norm on every bounded interval of the rescaled variable as $r\to 0$.

\item Furthermore, if the genuine nonlinearity condition~\eqref{hyperbolic-eos-3} holds, then $\alpha_{\chi}$ and $\eps\sqrt{p'}\,\alpha_{\chi}\pm\beta_{\chi}$ are non-negative measures, and the current vector $\chi$ is future timelike in the interior of the acoustic cone.
\eei
\end{theorem}

In addition, the current equations also admit a \emph{strong spacetime kernel} which instead satisfies  the vacuum condition $\alpha_\chi=\delta_0$ at vacuum. 
To include arbitrary motion parallel to the symmetry orbits, we use the particle density $\Numb$, characterized by $\Numb'/\Numb=1/h$ and $\Numb(\mu)/\mu\to1$ at vacuum, and the normalized parallel momenta defined at positive density to be $\widehat J_j=2h(\mu)u_j/\Numb(\mu)$ for $j=2,\ldots,n$, where $u_j$ are the spatial components of the unit fluid velocity. Their values on the vacuum set are prescribed separately. The effective pressure introduced in the following corollary expresses the same pressure value as a function of a new density variable.

\begin{corollary}[Fundamental kernel of spacetime quasi-currents]
\label{cor:main-quasi-current-kernel}
Let $p$ be an admissible equation of state in the sense of Definition~\ref{def-pression}. For every fixed normalized parallel momentum $\Jhatpar=(\widehat J_2,\ldots,\widehat J_n)$, with $\varpi\coloneqq|\Jhatpar|^2$, consider the effective density and pressure defined by
\be
m=\meff{\varpi}(\mu)\coloneqq\mu+\frac{\varpi\Numb(\mu)^2}{4h(\mu)},\qquad p_\varpi(m)\coloneqq p(\mu).
\ee
Then, this pressure law is also admissible, and all conclusions of Theorem~\ref{theo-main-results} apply with $m$ as the density and $p_\varpi$ as the pressure. The support of the kernel is given by the acoustic distance defined by $k_\varpi'(m)=\eps\sqrt{p_\varpi'(m)}/(m+\eps^2p_\varpi(m))$, and the genuine nonlinearity conditions $G_{\rm N}[p_\varpi]>0$ and $G_{\rm N}[p]>0$ are equivalent. Writing the resulting coefficients as $\alpha_{\chi_{\varpi}},\beta_{\chi_{\varpi}}$, the kernel on the original state space is
\be
\label{eq:main-parallel-momentum-kernel}
\chi(\mu,v;s,\Jhatpar)=\alpha_{\chi_{\varpi}}(m,v-s)\overline u+\beta_{\chi_{\varpi}}(m,v-s)\overline u^\perp.
\ee
The vacuum normalization for $\alpha_{\chi_{\varpi}}$ is taken with respect to the variable $m$; after composition with $m=\meff{\varpi}(\mu)$, its $\mu$-derivative at vacuum is $(1+\varpi/4)\delta_0$. The momenta are regarded as arbitrary fixed parameters in these estimates, whereas the constants may depend on them.
\end{corollary}

\begin{proof}
Proposition~\ref{prop:fixed-effective-admissibility}, below, proves admissibility of each fixed effective pressure, and Proposition~\ref{prop:section-6-effective-pressure}, below, gives the Euler reduction and preservation of genuine nonlinearity. We apply Theorem~\ref{theo-main-results} to this pressure. Since $\meff{\varpi}'(0)=1+\varpi/4$, the chain rule gives the stated vacuum trace in the original density.
\end{proof}

The fundamental kernel centered at rapidity $s$ has \emph{leading profile} $[k(\mu)^2-(v-s)^2]_+^\lambda$ near vacuum, where $[\cdot]_+$ is the positive part.  If $1<\gamma<3$ ($\lambda>0$) this profile vanishes at the support boundary, it is a step function if $\gamma=3$ ($\lambda=0$), and has a square-integrable singularity if $\gamma>3$ ($-1/2<\lambda<0$); see Figure~\ref{fig:kernel-exponent-regimes}.
Figure~\ref{fig:kernel-pressure-transition} shows the weak density kernels for admissible pressure laws with several values of~$\gamma$.
Proposition~\ref{prop:distributional-leading-trace}, below, proves local $L^2$ scaling for all $\gamma>1$ and locally uniform scaling for $1<\gamma<3$. At positive density, the Sobolev bound~\eqref{eq:intro-kernel-Sobolev-bound} gives $L^2(\Rbb)$ regularity for both components of the weak current kernel. Their compact support also gives $L^1(\Rbb)$ regularity and absolute convergence of the convolution against every continuous vacuum profile.

The boundary exponent is $\lambda=1$ when $\gamma=5/3$, so the leading profile vanishes linearly there. A fractional power $0<\lambda<1$ occurs for $5/3<\gamma<3$; for example, $\gamma=7/3$ gives $\lambda=1/4$, whereas $\gamma=7/5$ gives $\lambda=2$. For $1<\gamma<3$, the Sobolev bound and one-dimensional Sobolev embedding give H\"older continuity with every positive exponent smaller than $\min\{\lambda,1\}$. The construction of the kernels includes all finite $\gamma>1$.

\begin{figure}[htbp]
\centering
\includegraphics[width=\linewidth]{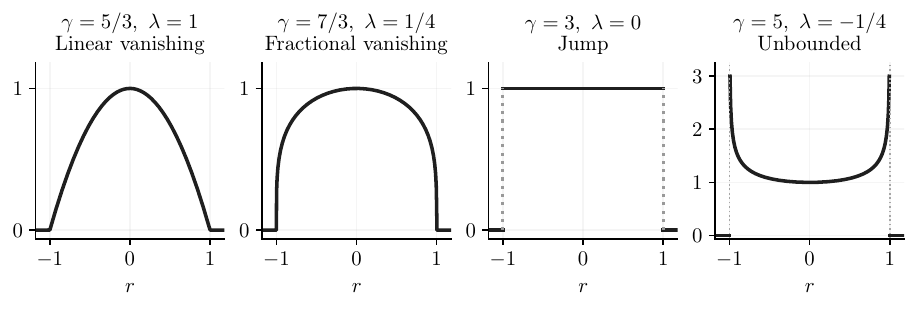}
\caption{Normalized leading boundary profiles $(1-r^2)_+^\lambda$, with $r=(v-s)/k(\mu)$ and value one at $r=0$. The four panels distinguish linear vanishing ($\gamma=5/3$), fractional-power vanishing ($\gamma=7/3$), a jump ($\gamma=3$), and an unbounded profile ($\gamma=5$). These are the leading terms, not the exact global kernels. In the last panel the curve has boundary singularities; for every finite $\gamma>3$, the profile remains in $L^1\cap L^2$.}
\label{fig:kernel-exponent-regimes}
\end{figure}

\begin{figure}[htbp]
\centering
\includegraphics[width=\linewidth]{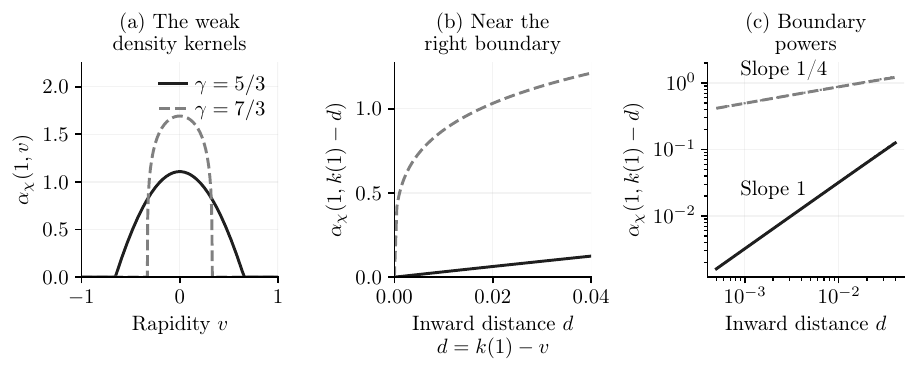}
\caption{The weak density kernel for $p'(\mu)=\mu^{\gamma-1}/(4(1+\mu^{\gamma-1}))$, evaluated at $\mu=1$, and $\eps=1/2$. Solid curves correspond to $\gamma=5/3$ and dashed curves to $\gamma=7/3$. Panel (a) shows the kernels themselves, computed from their Fourier equation. Panel (b) enlarges the right boundary using the inward distance $d=k(1)-v$; the acoustic radius $k(1)$ is that of the corresponding pressure. Panel (c) uses logarithmic axes: the dotted lines are the theoretical leading powers, with coefficients fixed by the vacuum normalization, and the computed curves approach slopes $\lambda=1$ and $\lambda=1/4$.}
\label{fig:kernel-pressure-transition}
\end{figure}

\subsubsection{Representation of currents}

By definition, a vacuum profile is the derivative in the density variable at vacuum of the fluid-frame time component of a spacetime current. For a continuous function $\vartheta$ of the vacuum label $s$, the subscript $\vartheta$ denotes the current and its coefficients obtained by integration against that profile. We write
\be
\label{eq:main-current-vector-convolution}
F_\vartheta\coloneqq\chi*\vartheta=F_\vartheta^0e_0+F_\vartheta^1e_1=\alpha_\vartheta\overline u+\beta_\vartheta\overline u^\perp,\qquad
\ee
in which 
\be
(\chi*\vartheta)(\mu,v)\coloneqq\int_\Rbb\chi(\mu,v;s)\vartheta(s)\,ds.
\ee
The vector integral keeps the state $(\mu,v)$, and hence the fluid frame, fixed. Its two scalar coefficients are ordinary convolutions in $v-s$.

\begin{theorem}[An infinite family of quasi-conservation laws]
\label{theo-intro-current-representation}
Let $p$ be an admissible equation of state in the sense of Definition~\ref{def-pression}, and introduce the notation $\Jhatpar=(\widehat J_2,\ldots,\widehat J_n)$ with $\varpi=|\Jhatpar|^2$ and $m=\meff{\varpi}(\mu)$. 
\bei 

\item For a prescribed profile $\vartheta=\vartheta(s,\Jhatpar)$ continuous in $s$, the formulas
\be
\label{eq:intro-fixed-kernel-representation}
\aligned
\alpha_\vartheta(\mu,v,\Jhatpar) &=\int_\Rbb\alpha_{\chi_{\varpi}}(m,v-s)\vartheta(s,\Jhatpar)\,ds,\qquad
\\
\beta_\vartheta(\mu,v,\Jhatpar)&=\int_\Rbb\beta_{\chi_{\varpi}}(m,v-s)\vartheta(s,\Jhatpar)\,ds,\\
(F_\vartheta^0,F_\vartheta^1)&=\bigl(\cosh v\,\alpha_\vartheta+\sinh v\,\beta_\vartheta,\sinh v\,\alpha_\vartheta+\cosh v\,\beta_\vartheta\bigr)
\endaligned
\ee
represent the (weak) spacetime currents of the reduced Euler system at fixed parallel momentum, characterized by the initial condition that the fluid-frame coefficients vanish at the vacuum and the derivative of the time component, taken in the effective density variable $m$, equals $\vartheta$. 

\item The spacetime current uniquely determines its vacuum profile and, for $\vartheta(\cdot,\Jhatpar)\in H^1_{\loc}(\Rbb)$, is $C^1$ in $(\mu,v)$ for all positive density. For a vacuum profile of class $C^2_{\loc}$ in $s$, it is $C^2$ in these variables and its divergence vanishes along every sufficiently regular solution of the Euler system with $\Jhatpar$ kept constant. 

\item In particular, the vacuum profiles $1$, $\cosh s$, and $\sinh s$ allow one to recover the particle-number, energy, and momentum currents (with the effective normalizations specified in Proposition~\ref{prop:N0-M00-M01-weak-entropies}, below). Whenever the resulting family indexed by~$\Jhatpar$ is $C^1$ in the full state variables, its quasi-balance law is
\be
\label{eq:main-profile-quasi-balance}
\del_tF_\vartheta^0+\del_xF_\vartheta^1
=\sum_{j=2}^n\Fvee_j\del_x\widehat J_j,\qquad
\Fvee_j=\frac1{\cosh v}\del_{\widehat J_j}\beta_\vartheta\big|_{\mu,v}.
\ee
\eei
\end{theorem}

All regularity assertions above keep $\Jhatpar$ fixed: namely for each fixed parallel momentum we relate the regularity of the vacuum profile to that of the current in density and rapidity variables.  This is done by transferring derivatives to that profile in the convolution.
In~\cite{TorusT2-3-asymptotic-2026}, we show that a vacuum profile with $C^1$ dependence on~$\Jhatpar$ with values in $H^1_{\loc}(\Rbb)$ leads to a $C^1$~family of currents.  This involves an analysis of the dependence of the fundamental kernels with respect to~$\varpi$, hence to the effective pressure law~$p_\varpi$.
In the present paper, we show that, conditional on the family of currents being sufficiently regular, the quasi-balance identity applies.

\subsubsection{Convexity criteria}

Convexity depends on the conservative variables in which the current density is expressed. For the reduced system at fixed parallel momentum, we let $E$, $P$, and $r$ denote its energy, momentum, and particle-number densities measured by a fixed inertial observer. In terms of the effective density $m$ and pressure $p_\varpi$, these are 
\be
E=(m+\eps^2p_\varpi)\cosh^2v-\eps^2p_\varpi, 
\quad P=(m+\eps^2p_\varpi)\sinh v\cosh v,
\quad
r=\Neff{\varpi}(m)\cosh v,
\ee
where $\Neff{\varpi}$ is the particle density of the reduced fluid, normalized by $\Neff{\varpi}(m)/m\to1$ at vacuum. The choice $U_{\rm N}=(E,P)$ uses negative particle number as the reference mathematical entropy and is called the \emph{particle-production formulation}; the choice $U_{\rm E}=(r,P)$ uses energy instead and is called the \emph{dissipated-energy formulation}. With the parallel momentum fixed, we write $\vartheta\in C^2_{\loc}(\Rbb)$ for the vacuum profile and define
\be
\mathcal D_{\rm N}\vartheta\coloneqq\vartheta''-\vartheta, \qquad \mathcal D_{\rm E}\vartheta\coloneqq\vartheta''-\tanh(s)\vartheta'.
\ee
The pressure signs $G_{\rm N}>0$ and $G_{\rm E}\geq0$ are defined in  the conditions~\eqref{hyperbolic-eos-3} and~\eqref{eq:intro-mass-energy-convexity-condition}.

\begin{corollary}[Future-convexity criteria for spacetime quasi-currents]
\label{theo-intro-convexity-classification}
Let $p$ be an admissible equation of state in the sense of Definition~\ref{def-pression}, fix $\Jhatpar$, and let $F^{0}_{\vartheta}$ be the (laboratory-frame) density of the weak quasi-current with profile $\vartheta(\cdot,\Jhatpar)\in C^2_{\loc}(\Rbb)$. All Hessians are evaluated at fixed~$\Jhatpar$.
\bei
\item The condition $\mathcal D_{\rm N}\vartheta\geq0$ is necessary for convexity of $F^{0}_{\vartheta}$ in $U_{\rm N}$. Under $G_{\rm N}>0$, it is also sufficient throughout the state space.
\item The condition $\mathcal D_{\rm E}\vartheta\geq0$ is necessary for convexity in $U_{\rm E}$. Under $G_{\rm E}\geq0$, it is sufficient throughout the state space.
\eei
Effective-pressure reduction preserves the sufficient pressure conditions, so the corresponding criteria apply to quasi-currents in the two conservative variables of the reduced Euler system at each fixed normalized momentum parallel to the symmetry orbits.
\end{corollary}

\subsection{Applications}
\label{section-1-4}

\subsubsection{Bounded solutions}
 
In a seminal work~\cite{DiPerna-1983a,DiPerna-1983b}, DiPerna developed the theory of \emph{bounded solutions} with weak regularity for the non-relativistic Euler equations. This theory was extended by Makino, Mizohata, and Ukai~\cite{MakinoMizohataUkai-1992,MakinoMizohataUkai-1994}, Lions, Perthame, and Tadmor~\cite{Lions-2-1994}, Lions, Perthame, and Souganidis~\cite{Lions-1-1996}, and Chen and LeFloch~\cite{CL1-2000,CL2-2003}. For relativistic flows, Hsu, Lin, and Makino~\cite{Hsu-Makino-2001} treated analytic perturbations of polytropic fluids, LeFloch and Yamazaki~\cite{LeFlochYamazaki-2007} studied isothermal fluids, and Chen and Schrecker~\cite{CS-2022} covered more general pressure laws.

\subsubsection{Finite-energy solutions}

LeFloch and Westdickenberg~\cite{LeFlochWestdickenberg-2007} established the theory of \emph{finite-energy weak solutions} for the non-relativistic Euler equations with geometric effects, while Chen and Perepelitsa~\cite{CP-2010} and Germain and LeFloch~\cite{PG-PLF} extended it to the zero-viscosity--capillarity limit. Next, Schrecker and Schulz~\cite{SchreckerSchulz-2020,SchreckerSchulz-2021-corrigendum} treated asymptotically isothermal pressures. Both for bounded solutions and finite-energy solutions, these results are restricted to fluid velocities that are sufficiently small with respect to the light speed.

The current families constructed here provide the balance laws used in our companion study~\cite{TorusT2-1-2026,TorusT2-2-2026,TorusT2-4-2026} of finite-energy (non-)relativistic fluids with general pressure laws satisfying the
hyperbolicity conditions~\eqref{hyperbolic-eos-1}--\eqref{hyperbolic-eos-2}, and the asymptotic conditions~\eqref{equa-cond-small}--\eqref{equa-cond-large}.

\subsubsection{Torus-symmetric Einstein--Euler flows}

Indeed, the construction in the present paper is part of a study of \emph{torus-symmetric flows} satisfying the Einstein-Euler equations. The first article~\cite{TorusT2-1-2026} introduced the weakly regular geometric and fluid formulation, while the second~\cite{TorusT2-2-2026} developed finite-energy Cauchy estimates for both the Einstein--Euler and the Einstein--Navier--Stokes systems. Here, we develop the theory of the spacetime quasi-currents, which provide us the tool to formulate mathematical entropy admissibility conditions and, in turn, apply Tartar's compensated method. The fourth part~\cite{TorusT2-4-2026} will establish further compactness properties and the zero-viscosity limit for the Einstein-Navier-Stokes system.

\section{Spacetime currents}
\label{section-2}

\subsection{Light speed and normalized variables}
\label{section:epsilon-geometric-normalization}

We begin by introducing the physical variables and their normalized counterparts, then presenting standard currents together with the notion of future-convexity. The speed of light in chosen units is $c_{\rm L}\coloneqq\eps^{-1}$. The density-pressure law is denoted by $p=p(\mu)$, where $\mu$ is the mass-energy density divided by $c_{\rm L}^2$. All constants may depend on $\eps$. (The non-relativistic limit of the current kernels is postponed to the companion paper~\cite{TorusT2-3-asymptotic-2026}.)

We work in spatial dimension $n\geq2$ on Minkowski spacetime with metric $g$ of signature $(-,+,\ldots,+)$. At positive density, we define the enthalpy factor by
\be
h(\mu)=\mu+\eps^2p(\mu),
\ee
and write the fluid momentum in terms of the future-directed unit velocity $u$, with $g(u,u)=-1$, as
\bel{eq:def-J}
\Jbb = (J_0,J_1,\ldots,J_n), \qquad J_m \coloneqq \sqrt{2(\mu+\eps^{2}p(\mu))}\,u_m.
\ee
Here $u_m$ denotes the covariant component of $u$ in an inertial orthonormal frame, with $m=0,\ldots,n$. For plane-symmetric flows depending on time $t$ and one spatial coordinate $x$, the transverse components with indices $2,\ldots,n$ are parallel to the symmetry orbits. We will use both terminologies (transverse, parallel) interchangeably. The particle density is normalized at vacuum by
\bel{eq:section-3-particle-density}
\lim_{\mu\to 0}\frac{\Numb(\mu)}{\mu} \, =1, \qquad \Numb(\mu) \coloneqq \mu \exp \Bigl( \int_0^{\mu} \Bigl( \frac{1}{\sigma+\eps^{2}p(\sigma)} - \frac{1}{\sigma} \Bigr) d\sigma \Bigr),
\ee
and determines, together with $h$, the \emph{normalized momentum parallel to the symmetry orbits}, defined by
\bel{eq:def-Jhat}
\Jhatpar \coloneqq (\widehat J_2,\ldots,\widehat J_n) \coloneqq \frac{\sqrt{2h(\mu)}}{\Numb(\mu)}(J_2,\ldots,J_n),
\ee
whose squared norm is
\bel{eq:section-6-varpi-definition}
\varpi=|\Jhatpar|^2 =\sum_{m=2}^n\widehat J_m^2.
\ee 

\begin{definition}
When $\Jhatpar$ is held fixed, the remaining Euler equations describe a one-dimensional fluid after a change of density variable. One calls its density $m=\meff{\varpi}(\mu)$ the \emph{effective density} and its pressure $p_\varpi(m)$ the \emph{effective pressure}; they are defined by
\bel{eq:section-6-effective-pressure}
\meff{\varpi}(\mu) = \mu+ \frac{\varpi\Numb(\mu)^2}{4h(\mu)}, \qquad p_\varpi\bigl(\meff{\varpi}(\mu) \bigr) = p(\mu).
\ee
Thus $p_\varpi$ is the pressure expressed as a function of~$m$.
\end{definition}

The components of $\Jhatpar$ are transported along fluid particle paths, which is the meaning of \emph{advected momentum} below. Their values on the vacuum set are prescribed separately, since the formula~\eqref{eq:def-Jhat} is defined only for $\mu>0$. Sections~\ref{section-3-3} and~\ref{section-5-1} establish the transport equations and this reduction.

We denote the physical time by $t_{\rm phys}$ and put $t\coloneqq t_{\rm phys}/\eps$. The physical Minkowski metric and its expression in the rescaled coordinates are
\bse\label{eq:physical-geometric-kinematics}
\be
g=-\eps^{-2}(dt_{\rm phys})^2+\sum_{j=1}^n dx_j^2 =-dt^2+\sum_{j=1}^n dx_j^2.
\ee
We use the orthonormal frame $e_0\coloneqq\del_t=\eps\del_{t_{\rm phys}}$ and $e_j\coloneqq\del_{x_j}$. A physical spatial velocity $V\in\Rbb^n$ with $|V|<\eps^{-1}$ determines the future unit velocity through
\be
\Gamma\coloneqq(1-\eps^2|V|^2)^{-1/2}, \qquad u^0=\Gamma, \qquad u^j=\eps\Gamma V_j.
\ee
\ese
Thus $g(u,u)=-(u^0)^2+\sum_{1\leq j\leq n}(u^j)^2=-1$, and the stress tensor in this orthonormal frame is
\bse\label{eq:physical-stress-conservation}
\be
T^{\alpha\beta} \coloneqq h u^\alpha u^\beta +\eps^2p g^{\alpha\beta}.
\ee
We define the physical densities and stress components by
\be
E\coloneqq h\Gamma^2-\eps^2p, \qquad M_i\coloneqq h\Gamma^2V_i, \qquad S_{ij}\coloneqq h\Gamma^2V_iV_j+p\delta_{ij}.
\ee
Conservation of the stress tensor is equivalent, in the physical coordinates $(t_{\rm phys},x)$, to the coupled system
\be
\del_{t_{\rm phys}} E+\sum_{i=1}^n\del_{x_i}M_i=0, \qquad \del_{t_{\rm phys}} M_j+\sum_{i=1}^n\del_{x_i}S_{ij}=0, \quad j=1,\ldots,n.
\ee
\ese
Indeed, $T^{00}=E$, $T^{0i}=\eps M_i$, and $T^{ij}=\eps^2S_{ij}$, while $\del_t=\eps\del_{t_{\rm phys}}$. Plane symmetry requires dependence on $(t_{\rm phys},x_1)$ and allows all $n$ components of $V$ to be non-vanishing.

The fluid momentum \(\Jbb=\sqrt{2h}\,u\) from the equation~\eqref{eq:def-J} has covariant components
\be
J_0=-\sqrt{2h}\,\Gamma, \qquad J_j=\eps\sqrt{2h}\,\Gamma V_j, \qquad -\Jbb\cdot\Jbb=2h.
\ee
The \textit{fluid rapidity} $v$ and its physical normalization $w$ are defined by
\be
v\coloneqq\artanh(\eps V_1), \qquad w\coloneqq v/\eps, \qquad V_1=\frac{\tanh(\eps w)}{\eps}.
\ee
Using $v$ keeps the Lorentz boost formulas in their usual form familiar to relativists. In the same convention, a current $(F^0,F^1)$ in the rescaled coordinates has physical density and flux $(F^0,\eps^{-1}F^1)$.

Here \(\Numb(\mu)\) is the particle density, normalized by \(\Numb'=\Numb/h\) and \(\Numb/\mu\to1\) at vacuum. For \(\widehat J_m=2hu_m/\Numb\), the physical normalization is
\bse\label{eq:physical-parallel-transport}
\be
a_m\coloneqq\eps^{-1}\widehat J_m =\frac{2h\Gamma V_m}{N}, \quad m=2,\ldots,n.
\ee
The components $a_m$ and $\widehat J_m$ obey the same physical transport law; for $a_m$, we have
\be
\del_{t_{\rm phys}} a_m+V_1\del_{x_1}a_m=0.
\ee
\ese
The momentum $\Jhatpar$, with $\varpi=|\Jhatpar|^2$ as in the equation~\eqref{eq:section-6-varpi-definition}, is fixed in the reduced kernel theory; its variation and physical normalization are studied in~\cite{TorusT2-3-asymptotic-2026}.

We write $c\coloneqq\eps\sqrt{p'}$ for the sound speed divided by the speed of light and $\omega$ for the oriented two-dimensional volume form. The auxiliary parameter $\auxeps$ used in cutoff and remainder arguments is unrelated to light speed.

\begin{proposition}[Non-relativistic limits of the state variables]
\label{prop:physical-state-nonrelativistic-limits}
Fix the pressure law \(p\) and a physical state \((\mu,V)\in(0,+\infty)\times\Rbb^n\), and evaluate the quantities defined above for sufficiently small \(\eps>0\). As \(\eps\to0\), the following limits hold locally uniformly in \((\mu,V)\). We write \(V_\parallel=(V_2,\ldots,V_n)\) and \(m=\meff{\varpi}(\mu)\).
\bse\label{eq:physical-state-nonrelativistic-limits}
\be
\label{eq:state-limit-scalar}
h\to\mu,\qquad \Numb\to\mu,\qquad \Gamma\to1,\qquad c/\eps=\sqrt{p'(\mu)},\qquad w=v/\eps\to V_1.
\ee
\be
\label{eq:state-limit-momentum}
u^0\to1,\qquad u^j/\eps\to V_j,\qquad J_0\to-\sqrt{2\mu},\qquad J_j/\eps\to\sqrt{2\mu}\,V_j.
\ee
\be
\label{eq:state-limit-physical-stress}
E\to\mu,\qquad M_i\to\mu V_i,\qquad S_{ij}\to\mu V_iV_j+p(\mu)\delta_{ij},\qquad \Numb\Gamma\to\mu.
\ee
\be
\label{eq:state-limit-parallel}
\aligned
a_j=\widehat J_j/\eps&\to2V_j\quad(j=2,\ldots,n),&\quad \varpi/\eps^2&\to4|V_\parallel|^2,\\
(m-\mu)/\eps^2&\to\mu|V_\parallel|^2,& p_\varpi(m)&=p(\mu).
\endaligned
\ee
\be
\label{eq:state-limit-mechanical-energy}
\aligned
\frac{E-\Numb\Gamma}{\eps^2}&\longrightarrow\frac12\mu|V|^2+\mu\int_0^\mu\frac{p(s)}{s^2}\,ds,\\
\frac{M_i-\Numb\Gamma V_i}{\eps^2}&\longrightarrow
\Bigl(\frac12\mu|V|^2+\mu\int_0^\mu\frac{p(s)}{s^2}\,ds+p(\mu)\Bigr)V_i.
\endaligned
\ee
\ese

\end{proposition}
\begin{proof}
The vacuum conditions~\eqref{eq:intro-vacuum-B}--\eqref{eq:intro-vacuum-r} give \(p(s)=\Ocal(s^\gamma)\) with \(\gamma>1\), as established in Proposition~\ref{prop:pressure-endpoint-asymptotics}, below. The particle normalization therefore yields
\be
\label{eq:state-limit-particle-expansion}
\frac1{\eps^2}\log\frac{\Numb(\mu)}{\mu}
=-\int_0^\mu\frac{p(s)}{s(s+\eps^2p(s))}\,ds
\longrightarrow-\int_0^\mu\frac{p(s)}{s^2}\,ds.
\ee
The integrable majorant \(p(s)/s^2\) gives uniform convergence for bounded density intervals. Thus, with \(e(\mu)=\int_0^\mu p(s)s^{-2}\,ds\), we have \(\Numb=\mu(1-\eps^2e+o(\eps^2))\). Together with \(\Gamma=1+\eps^2|V|^2/2+\Ocal(\eps^4)\), this proves the limits for the scalars, momenta and physical stresses. The definitions of the parallel momentum give the identity \(m-\mu=\eps^2h\Gamma^2|V_\parallel|^2\), which proves the parallel and effective-density limits. Finally, \(E=\mu+\eps^2\mu|V|^2+o(\eps^2)\) and \(M_i=(\mu+\eps^2(p+\mu|V|^2)+o(\eps^2))V_i\) give the two limits in~\eqref{eq:state-limit-mechanical-energy}. The expansions are uniform on compact subsets of the stated physical state space.
\end{proof}

\subsection{Fluid states and momentum}
\label{section-2-1}

\subsubsection{Extended state space}

We study currents for relativistic Euler flows on Minkowski spacetime in spatial dimension $n\geq2$, using the fluid momentum~$\Jbb$ and an arbitrary inertial observer. We first treat general flows, before imposing plane symmetry in~\autoref{section-3}, below. We use Minkowski spacetime with metric
\be
(\mathbb R^{1+n},g), \qquad g=-dt^2+ \sum_{m=1}^n dx_m^2.
\ee
In global Cartesian coordinates \((t,x_1,\ldots,x_n)\), the standard orthonormal frame \((e_0,e_1,\ldots,e_n)\) is
\be
e_0 = \del_t, \qquad e_j = \del_{x_j}, \quad j=1,\ldots,n.
\ee
Here \(e_0\) is future-directed timelike and \(e_1,\ldots,e_n\) are spacelike.

In the frame \((e_0,e_1,\ldots,e_n)\), the fluid velocity is $u=u^m e_m$, with contravariant components $u^m$. Its covariant components \(u_m \coloneqq g(u,e_m)\) satisfy the normalization and sign
\be
-u_0^2+ \sum_{m=1}^n u_m^2=-1, \qquad u_0=-u^0< 0.
\ee
We recall the fluid momentum from the equation~\eqref{eq:def-J}, with the rescaling used in~\cite{TorusT2-1-2026}. It satisfies
\bel{eq:J-norm}
-\Jbb\cdot\Jbb = J_0^2-\sum_{m=1}^nJ_m^2 = 2(\mu+\eps^{2}p(\mu)).
\ee
Thus every positive-density state gives a  future-directed timelike momentum vector.

The positive-density formula does not define \(u\) at \(\mu=0\). We use \(\Jbb\) as state variable on the closed future causal cone,
\bel{eq:state-space}
\mathcal A \coloneqq \Bigl\{ \Jbb\in\Rbb^{1+n}: J_0\le0, \quad -\Jbb\cdot\Jbb\ge0 \Bigr\},
\ee
which we call the (extended) \emph{admissible state space}. On the interior of \(\mathcal A\), that is, when \(-\Jbb\cdot\Jbb>0\), the mass-energy density \(\mu\) is uniquely recovered from the equation~\eqref{eq:J-norm}, since the map $\mu\mapsto 2(\mu+\eps^{2}p(\mu))$ is a strictly increasing bijection of $(0,+ \infty)$ under assumption \eqref{hyperbolic-eos-1}. The velocity is then reconstructed from
\bel{eq:u-from-J}
u_m= \frac{J_m}{\sqrt{2(\mu+\eps^{2}p(\mu))}} \quad \text{ on the non-vacuum set } {\mathcal A}^\circ.
\ee
The null boundary has density \(\mu=0\) through the identity~\eqref{eq:J-norm}, while the primitive velocity~$u$ is not defined on this boundary. Indeed, any $\mu\to 0$ limit with bounded~$u$ tends to a vanishing fluid momentum vector $\Jbb=0$, while nonzero null states represent a scaling limit in which $\mu\to 0$ while velocity components become unbounded. Only $\Jbb=0$ is the physical vacuum state with zero stress. At a nonzero null momentum, the continuous extension of the Euler tensor is $T^{\alpha\beta}=J^\alpha J^\beta/2$. Its divergence equation remains nontrivial, even though the primitive velocity cannot be reconstructed. The current profiles used below prescribe the small-density behavior at fixed finite rapidity; this normalization differs from approach to a nonzero null state with diverging rapidity.

For the later one-dimensional reduction, we distinguish the \emph{orthogonal} components from those parallel to the symmetry orbits and set
\bel{eq:J-perp-par}
\Jperp \coloneqq (J_0,J_1), \qquad \Jpar \coloneqq (J_2,\ldots,J_n)\in\Rbb^{n-1}.
\ee
This decomposition gives
\bel{eq:Jperp-est}
-\Jperp\cdot\Jperp\ge -\Jbb\cdot\Jbb\ge0.
\ee

\subsubsection{Barotropic Euler flows}

Away from vacuum, we consider a barotropic fluid described by a density \(\mu>0\), a future-directed unit vector field \(u\), and a pressure law \(p=p(\mu)\). Thus, \(g(u,u) =-1\), and its stress-energy tensor is \(T_{\alpha\beta} \, =(\mu+\eps^{2}p)u_\alpha u_\beta+\eps^{2}p g_{\alpha\beta}\). The Euler equations are the covariant conservation laws\footnote{Here $\nabla$ is the Levi--Civita connection of $g$, equal to ordinary differentiation in Cartesian coordinates.}
\bse\label{eq:covariant-Euler-and-projections}
\bel{eq:section-3-Euler}
\nabla_\alpha T^{\alpha\beta} \, =0.
\ee
By contracting the equation~\eqref{eq:section-3-Euler} with \(u_\beta\) or projecting it orthogonally to \(u\), we obtain respectively the energy and acceleration equations
\bel{eq:section-3-Euler-splitting}
u^\alpha\nabla_\alpha\mu+(\mu+\eps^{2}p)\nabla_\alpha u^\alpha=0,\qquad (\mu+\eps^{2}p)u^\alpha\nabla_\alpha u^\beta+(g^{\alpha\beta}+u^\alpha u^\beta)\nabla_\alpha \eps^{2}p=0,
\ee
\ese
which form a system equivalent to the system~\eqref{eq:section-3-Euler}. 

Recall that Frauendiener~\cite{Frauendiener-2003} and Walton~\cite{Walton-2005} gave a symmetric-hyperbolic formulation of relativistic Euler flow using a timelike vector field $w^\alpha\coloneqq\frac{\Numb(\mu)}{\mu+\eps^{2}p}u^\alpha$ where $\Numb(\mu)$ is the particle number defined in the equation~\eqref{eq:section-3-particle-density}.
 LeFloch and Ukai~\cite{LeFlochUkai-2009} provides a symmetrization (based on the spatial angular direction of~$u$, a rapidity variable, and a function of~$\mu$), that remains meaningful at vacuum. In the present work, the fluid momentum vector~\(\Jbb\) is chosen, both to exhibit a quadratic flux structure and to extend the state space to the closed future causal cone.

\subsubsection{A constitutive coefficient}

The principal part of the Euler system depends on the \textit{pressure ratio}
\bse\label{eq:pressure-ratio-and-vacuum}
\bel{eq:def-q}
q(\mu) \coloneqq \frac{\mu-\eps^{2}p(\mu)}{\mu+\eps^{2}p(\mu)}, \qquad \mu>0,
\ee
which belongs to \((0,1)\) for \(\mu>0\). Expressing $\mu$ in terms of $-\Jbb\cdot\Jbb$ through the identity~\eqref{eq:J-norm}, we sometimes denote the pressure ratio as $q_\Jbb \coloneqq q(\mu)$. Moreover, the limit at vacuum
\bel{eq:q-zero}
q(0) \coloneqq \lim_{\mu\to0}q(\mu) = \frac{1-\eps^{2}p'(0)}{1+\eps^{2}p'(0)}\in(0,1]
\ee
\ese
exists; in particular, \(q(0) =1\) when \(\eps^{2}p'(0) =0\). We set
\be
y \coloneqq -\Jbb\cdot\Jbb=2(\mu+\eps^{2}p(\mu)).
\ee
The definition gives
\bel{eq:yq-prime}
\frac{d}{dy}(yq) = \frac{1-\eps^{2}p'(\mu)}{1+\eps^{2}p'(\mu)},
\ee
which belongs to \((0,1)\) away from vacuum. It also governs the singularities of the current kernels.

\subsection{Particle-number and spacetime currents}
\label{section-2-2}

\subsubsection{Particle-number current} 

The particle density normalized at vacuum in the equation~\eqref{eq:section-3-particle-density} satisfies the linear equation
\bse\label{eq:particle-equation-and-current}
\be
\Numb'(\mu) = \frac{\Numb(\mu)}{\mu+\eps^{2}p(\mu)} , \qquad \mu>0.
\ee
The first equation in \eqref{eq:section-3-Euler-splitting} gives
\bel{eq:section-3-particle-conservation}
\nabla_\alpha(\Numb u^\alpha) = \Numb'(\mu)u^\alpha\nabla_\alpha\mu+ \Numb(\mu)\nabla_\alpha u^\alpha=0.
\ee
Hence the \emph{particle current}
\bel{eq:section-3-particle-current}
\mathcal N^\alpha \coloneqq \Numb u^\alpha
\ee
\ese
is conserved on regular Euler flows.

The asymptotic behavior $p(\mu) = \Ocal(\mu^\gamma)$ as $\mu\to 0$, with $\gamma>1$, implies that
\be
\del_\mu \log\bigl(\Numb(\mu)/\mu \bigr) = \frac{1}{\mu+\eps^{2}p(\mu)} - \frac{1}{\mu} = \Ocal(\mu^{\gamma-2})
\ee
is integrable at vacuum. Hence \(\Numb(\mu)/\mu\) has a finite positive limit, and the normalization in the equation~\eqref{eq:section-3-particle-density} sets that limit to unity. Any other positive normalization merely rescales the particle current.\!\footnote{The normalization~\eqref{eq:section-3-particle-density} identifies it with the zero Fourier mode of the current kernels.} For the exact isothermal law $\eps^{2}p(\mu) =c_{\rm iso}^2\mu$, with \(0<c_{\rm iso}<1\), instead, \(\Numb(\mu) =C\mu^{1/(1+c_{\rm iso}^2)}\), with arbitrary \(C>0\); the vacuum normalization above is then unavailable. \autoref{section-3} gives the particle current's vacuum extension and its expression in reduced fluid variables.

\subsubsection{General notion of current}

We first consider a general first-order system on Minkowski spacetime, with components indexed by \(A\):
\bse\label{eq:general-system-covariant-and-frame}
\bel{equakdjkdk3}
\del_\alpha f^{\alpha A}(x,w) = S^A(x,w).
\ee
Here \(w\) ranges in an open finite-dimensional state domain. The fluxes \(f^{\alpha A}\) and sources   \(S^A\)  may depend smoothly on the spacetime point \(x\) and on \(w\), but not on derivatives of \(w\). The index \(A\) ranges in a fixed finite-dimensional vector space \(E\). In Cartesian coordinates \((t,x^1,\ldots,x^n)\), the balance laws read
\bel{eq:section-3-local-balance}
\del_t f^{0A}(x,w)+ \sum_{i=1}^n\del_{x^i} f^{iA}(x,w) = S^A(x,w).
\ee
\ese
Under a local invertibility condition~\eqref{eq:section-3-normal-Jacobian} below, these equations form an evolution system for~$w$. We define currents in state space and deduce their spacetime conservation or balance laws.

\begin{definition}
\label{def:section-3-current}
A \emph{spacetime current} (or current, in short) is a vector field \(F=F(x,w)\) for which there exists an \(E^*\)-valued \emph{current multiplier} \(V=V(x,w)\) satisfying
\bel{eq:section-3-multiplier}
\frac{\del F^\alpha}{\del w^I} \, =V_A\frac{\del f^{\alpha A}}{\del w^I}, \qquad \alpha=0,\ldots,n.
\ee
These are identities between differentials in the state variable, with the spacetime point \(x\) held fixed.
\end{definition}

\subsubsection{Derivation of the current balance law}

The multiplier relation~\eqref{eq:section-3-multiplier} expresses the divergence of a current without derivatives of the state. The vector~$F$ has density~$F^0$ and $n$ spatial flux components; its divergence and squared norm~\(g(F,F)\) are Lorentz scalars, while the density measured by a future unit observer~\(n\) is \(-g(F,n)\). The plane-symmetric reduction is given in~\autoref{section-3-4}.

\begin{proposition}[Current balance law]
\label{prop:section-3-current-identity}
Let \(F\) and \(V\) satisfy the equation~\eqref{eq:section-3-multiplier}. Every regular solution of   \(\nabla_\alpha f^{\alpha A}(x,w) =S^A(x,w)\)  satisfies
\bse\label{eq:general-current-balance-and-source}
\bel{eq:section-3-current-identity}
\nabla_\alpha F^\alpha= \Sigma_F,
\ee
where, in an inertial frame, one has
\bel{eq:section-3-current-source}
\Sigma_F\coloneqq V_A S^A+ \del_\alpha^xF^\alpha-V_A\del_\alpha^xf^{\alpha A}.
\ee
\ese
Here \(\del_\alpha^x\) denotes differentiation in \(x\) with \(w\) fixed. In particular, the two base-derivative terms vanish when  both the current \(F\) and the fluxes \(f\) have no explicit \(x\)-dependence.
\end{proposition}

\begin{proof}
The chain rule and the multiplier identity~\eqref{eq:section-3-multiplier} give \(\del_\alpha F^\alpha=V_A(\del_\alpha f^{\alpha A}-\del_\alpha^xf^{\alpha A})+ \del_\alpha^xF^\alpha\). Substitution of the system~\eqref{eq:section-3-local-balance}, with vanishing connection coefficients in this inertial frame, gives the identities~\eqref{eq:section-3-current-identity}--\eqref{eq:section-3-current-source}.
\end{proof}

\subsubsection{Stress-energy and particle currents}

We write \(\Mcal\coloneqq\Rbb^{1+n}\) for Minkowski spacetime. We fix an interval \(I\) and a spacelike foliation \((\Sigma_t)_{t\in I}\) with future unit normal~\(n\). For every vector field \(X\) on \(\Mcal\), we define the stress-energy current \(F_X\) by
\bse\label{eq:stress-current-and-balance}
\bel{eq:section-3-stress-current}
F_X^\alpha \coloneqq -T^\alpha{}_\beta X^\beta.
\ee
The Euler equations imply the standard balance laws
\bel{eq:section-3-stress-current-identity}
\nabla_\alpha F_X^\alpha=-T^\alpha{}_\beta\nabla_\alpha X^\beta.
\ee
\ese
Indeed, the Leibniz rule gives \(\nabla_\alpha F_X^\alpha=- (\nabla_\alpha T^\alpha{}_\beta)X^\beta-T^\alpha{}_\beta\nabla_\alpha X^\beta\), and the first term vanishes by the equation~\eqref{eq:section-3-Euler}. Thus \(F_X\) is conserved when \(X\) is Killing, without any invariance condition on the fluid. Otherwise the source describes exchange with the deformation of the observer field. In particular, if \(X\) is future-directed, then \(-n_\alpha F_X^\alpha=T(n,X)\) is the corresponding energy density on \(\Sigma_t\).

The observer-independent particle current~\(\mathcal N\), conserved by the equation~\eqref{eq:section-3-particle-conservation}, describes material transport, whereas \(F_X\) describes energy and momentum transport. The convexity below selects a reference entropy inequality for weak solutions.

\subsection{Future-convexity and mathematical entropy inequalities}
\label{section-2-3}

\subsubsection{Invertibility condition}

For the system~\eqref{eq:section-3-local-balance}, we choose a future-directed unit timelike vector~\(n\), for instance normal to a spacelike foliation, and set \(U^A(x,w)\coloneqq-n_\alpha f^{\alpha A}(x,w)\). At each spacetime point, an adapted inertial frame gives \(n=e_0\) and \(U^A=f^{0A}\). The system determines the evolution of~$w$ relative to this normal when the Jacobian
\bel{eq:section-3-normal-Jacobian}
\Jac_I^A(x,w) \coloneqq \frac{\del U^A}{\del w^I}(x,w)
\ee
is invertible on the state domain under consideration. By definition, \(U\) is our set of \emph{conservative variables}; this is the only point at which the chosen foliation enters the current construction.

\subsubsection{The symmetrizer criterion}

At a fixed spacetime point, we use the conservative variables \(U\) of the equation~\eqref{eq:section-3-normal-Jacobian} and write \(f^i=f^i(U)\), \(i=1,\ldots,n\), for the tangential fluxes. Then the equation~\eqref{eq:section-3-multiplier} is equivalent to
\bel{eq:section-3-entropy-flux-relation}
\frac{\del F^{i}}{\del U^B} \, = \frac{\del F^{0}}{\del U^A}\,A^{iA}{}_B, \qquad A^{iA}{}_B \coloneqq \frac{\del f^{iA}}{\del U^B}.
\ee
The Hessian of the current density \(F^{0}\) must therefore symmetrize every tangential flux Jacobian.

\begin{lemma}[Local criterion for a current density]
\label{lem:section-3-symmetrizer}
Let \(U\) range over a simply connected open set of conservative variables on which \(w\mapsto U\) is one-to-one with invertible Jacobian. A \(C^2\) function \(F^{0}= F^{0}(U)\) is the density of a current if and only if, for every \(i=1,\ldots,n\),
\bel{eq:section-3-symmetrizer}
\frac{\del^2F^{0}}{\del U^A\del U^C}\,A^{iA}{}_B = \frac{\del^2F^{0}}{\del U^A\del U^B}\,A^{iA}{}_C.
\ee
Equivalently, the matrices \(D_U^2F^{0}\,A^i\) are symmetric. When this condition holds, the tangential fluxes \(F^{i}\) are determined by the equation~\eqref{eq:section-3-entropy-flux-relation}, up to additive constants.
\end{lemma}

When the Hessian satisfies \(D_U^2F^{0}>0\), it symmetrizes the system and yields a symmetric-hyperbolic formulation~\cite{FriedrichsLax-1971,Godunov-1961}.

\begin{proof}
If \(F^{i}\) exists, the equality of its mixed derivatives in the equation~\eqref{eq:section-3-entropy-flux-relation} gives the equation~\eqref{eq:section-3-symmetrizer}, since \(A^{iA}{}_B= \del f^{iA}/\del U^B\) already has commuting mixed derivatives. Conversely, the equation~\eqref{eq:section-3-symmetrizer} says that the one-form \((\del F^{0}/\del U^A)A^{iA}{}_B\,dU^B\) is closed. The Poincar\'e lemma gives \(F^{i}\), and the multiplier is \(V_A= \del F^{0}/\del U^A\).
\end{proof}

\subsubsection{Convexity relative to a foliation}

We return to the foliation \((\Sigma_t)_{t\in I}\) and its future unit normal~\(n\). We use its conservative variables~\(U\) from the equation~\eqref{eq:section-3-normal-Jacobian}, assuming locally that \(w\mapsto U(w)\) is one-to-one.

\begin{definition}
\label{def:section-3-future-convex}
A current \(F\) is called \emph{future-convex}, relative to the chosen foliation, if its normal density \(-n_\alpha F^\alpha\), regarded as a function of the conservative variable \(U\), is convex on every convex subset of the conservative state domain. For a density of class \(C^2\), this is equivalent to the positive semi-definiteness of its Hessian. The current is called \emph{strictly future-convex} when this Hessian is positive definite.
\end{definition}

This definition is invariant under changes of orthonormal frame tangent to the leaves. Its dependence on the future unit normal enters through both the density \(-g(F,n)\) and the conservative variables \(-n_\alpha f^{\alpha A}\). Convexity for one normal therefore does not, by this definition alone, assert convexity for every future timelike normal. The future-convexity of the two reference currents is established in \autoref{section-4-2} and~\ref{section-4-3}. For the effective Euler systems in one spatial variable, the local preservation of convexity under a change of normal in the plane orthogonal to the symmetry orbits follows from the subluminal acoustic speeds and the Hessian formula~\eqref{eq:convexity-changed-normal-hessian}, below. The corresponding criteria for the full family of quasi-currents are established in Sections~\ref{section-10-1} and~\ref{section-10-2}.

\subsubsection{Weak solutions}

For a regular solution, every current satisfies the identity in \autoref{prop:section-3-current-identity}. For a weak solution satisfying the Euler equations in the sense of distributions, we impose a mathematical entropy inequality as an additional admissibility condition, since the classical chain rule need not hold. If $F=F(x,w)$ and $\Sigma_F= \Sigma_F(x,w)$ are locally integrable, the \textit{mathematical entropy inequality}
\bse\label{eq:entropy-inequality-and-test-form}
\be
\nabla_\alpha F^\alpha\leq\Sigma_F
\ee
is required for one future-convex current~$F$ called the \emph{reference current}, which, by definition, means that
\be
\label{eq:current-distributional-inequality}
\int_{\Mcal}\bigl(F^\alpha(x,w)\nabla_\alpha\varphi+ \Sigma_F(x,w)\varphi \bigr)\,d\operatorname{Vol}_g\geq0\qquad\text{for every }0\leq\varphi\in C_c^\infty(\Mcal).
\ee
\ese
This formulation depends on the choice of the spacetime current~$F$, and not on a particular observer's decomposition into density and flux. Convexity however is evaluated in the conservative variables selected above. Section~\ref{section-4} distinguishes the two reference choices; the construction of the kernels  gives further currents and their convexity criteria.

\section{Plane-symmetric quasi-currents}
\label{section-3}

\subsection{Symmetry reduction} 
\label{section-3-1}

\subsubsection{Reduced Euler system}

Plane-symmetric flows allow all velocity components to be non-vanishing, requiring simply that they depend only on two coordinates. We derive their reduced equations and quasi-balance law; the kernels of Sections~\ref{section-6} to~\ref{section-flux-kernels}, below, will construct the currents, while a control of sources in the resulting balance laws also requires control of derivatives of the advected momentum, which depends on the application. For a pair $\Xbb=(X_0,X_1)$, we define the quadratic expressions entering the system by
\bel{eq:E0E1}
\Ebf_0(\Xbb) \coloneqq \frac12(X_0^2+X_1^2), \qquad \Ebf_1(\Xbb) \coloneqq X_0X_1.
\ee
For the momentum parallel to the symmetry orbits~$\Jpar$, we set
\bel{eq:E0-complex}
\Ebf_0(\Jpar) \coloneqq \frac12\sum_{m=2}^n J_m^2.
\ee
For pairs in the plane orthogonal to the symmetry orbits, we use the additive shorthand \(\Ebf_m(\Xbb,\Ybb) \coloneqq \Ebf_m(\Xbb)+ \Ebf_m(\Ybb)\), \(m=0,1\); the same convention is used for \(\Ebf_0\) when an argument is a vector of momentum components parallel to the symmetry orbits. In particular, we can write
\bel{eq:E0J}
\Ebf_0(\Jbb) = \Ebf_0(\Jperp)+ \Ebf_0(\Jpar) = \frac12 \Bigl( J_0^2+ \sum_{m=1}^n J_m^2 \Bigr).
\ee
We also use the Minkowski inner product on the first two components, together with the associated null form, i.e.
\bel{eq:dot-wedge}
-\Xbb\cdot\Ybb \coloneqq X_0Y_0-X_1Y_1, \qquad \Xbb\wedge\Ybb \coloneqq X_0Y_1-X_1Y_0.
\ee
The components of the stress-energy tensor relevant to our formulation are then
\bel{eq:M-components}
\Mbf_{00}(\Jbb) \coloneqq \Ebf_0(\Jbb) -\frac{q_\Jbb}{2}\,\Jbb\cdot\Jbb, \qquad \Mbf_{01}(\Jbb) \coloneqq \Ebf_1(\Jperp),\qquad \Mbf_{11}(\Jbb) \coloneqq \Ebf_0(\Jperp) -\Ebf_0(\Jpar)+ \frac{q_\Jbb}{2}\,\Jbb\cdot\Jbb.
\ee
Here \(\Mbf_{00} \, =2T_{00}\), \(\Mbf_{01} \, =2T_{01}\), and \(\Mbf_{11} \, =2T_{11}\), where \(T\) is the perfect-fluid stress-energy tensor and the indices here denote its covariant orthonormal components; the present normalization removes irrelevant factors of two from the reduced balance laws. The matter energy density is \(\Mbf_{00}\). Since \(q_\Jbb\in(0,1]\) and \(-\Jbb\cdot\Jbb\ge0\), we have
\bel{eq:M00-positive}
\Mbf_{00}(\Jbb)\ge \Ebf_0(\Jbb)\ge0.
\ee
Hence \(\Mbf_{00}\) controls all $n+1$ components of the fluid momentum.

We impose \emph{invariance under translations tangent to the symmetry orbits}, equivalently,
\be
e_m(\Jbb) =0, \qquad m=2,\ldots,n.
\ee
This reduction of the Euler equations in the variable~\(\Jbb\) arose in the torus-symmetric spacetime formulation of~\cite{TorusT2-1-2026}. In the Minkowski spacetime under consideration, it takes the reduced form
\bel{eq:euler-principal}
\begin{aligned}
e_0\Bigl(
\Ebf_0(\Jbb) -\frac{q_\Jbb}{2}\,\Jbb\cdot\Jbb
\Bigr)
-
e_1\bigl(\Ebf_1(\Jperp) \bigr)
&=0,
\\
e_0\bigl(\Ebf_1(\Jperp) \bigr)
-
e_1\Bigl(
\Ebf_0(\Jperp) -\Ebf_0(\Jpar)+ \frac{q_\Jbb}{2}\,\Jbb\cdot\Jbb
\Bigr)
&=0,
\\
e_0(J_0J_m) -e_1(J_1J_m)
&=0,
\qquad m=2,\ldots,n.
\end{aligned}
\ee
Equivalently, the system may be written in quasilinear form as
\bel{eq:quasilinear-J}
A_0(\Jbb)\,e_0(\Jbb) -A_1(\Jbb)\,e_1(\Jbb) =0,
\ee
with the following geometric \(2+(n-1)\) block representation. To exhibit the coupling with the momentum parallel to the symmetry orbits,  we put \(E_\perp \coloneqq \operatorname{span}(e_0,e_1)\) and \(E_\parallel \coloneqq \operatorname{span}(e_2,\ldots,e_n)\), and denote by \((e^0,\ldots,e^n)\) the dual coframe. We identify \(\Jpar\) with the vector \(\mathbf J_\parallel \coloneqq \sum_{m=2}^nJ_me_m\in E_\parallel\) and write \(\mathbf J_\parallel^\flat \coloneqq \sum_{m=2}^nJ_me^m\). With \(\yqu  \coloneqq (yq)'\in(0,1]\), the matrices are endomorphisms of \(E_\perp\oplus E_\parallel\), and each tensor product below is a rank-one map:
\be
\label{eq:A0A1}
\begin{alignedat}{2}
B_0&\coloneqq \begin{pmatrix}(1+\yqu)J_0&(1-\yqu)J_1\\ J_1&J_0\end{pmatrix}, \qquad &B_1&\coloneqq \begin{pmatrix}J_1&J_0\\ (1-\yqu)J_0&(1+\yqu)J_1\end{pmatrix},\\
A_0(\Jbb)&\coloneqq \begin{pmatrix}B_0&(1-\yqu)e_0\otimes\mathbf J_\parallel^\flat\\ \mathbf J_\parallel\otimes e^0&J_0\operatorname{Id}_{E_\parallel}\end{pmatrix}, \qquad &A_1(\Jbb)&\coloneqq \begin{pmatrix}B_1&(\yqu-1)e_1\otimes\mathbf J_\parallel^\flat\\ \mathbf J_\parallel\otimes e^1&J_1\operatorname{Id}_{E_\parallel}\end{pmatrix}.
\end{alignedat}
\ee
Here $y=-\Jbb\cdot\Jbb$ as in the equation~\eqref{eq:yq-prime}, and primes denote $y$-derivatives. Freezing the normalized parallel momentum gives the effective Euler system in one spatial variable of Section~\ref{section-5}, below; Sections~\ref{section-6} to~\ref{section-flux-kernels} construct its kernels.

\subsubsection{Hyperbolicity and vacuum degeneracy} 

We now describe the admissible states and characteristic fields in the \(\Jbb\)-formulation introduced in~\cite{TorusT2-1-2026} for the Einstein--Euler system. The following theorem gives hyperbolicity away from vacuum and the degeneracy of the acoustic speeds at vacuum.

\begin{theorem}[Structural properties of relativistic fluid flows in the $\Jbb$ variable]
\label{prop:structure-euler}
Let $p\in C^2((0,+\infty))\cap C^1([0,+\infty))$ satisfy $p(0)=0$, $0<\eps^2p'(\mu)<1$ for $\mu>0$, and $0\leq\eps^2p'(0)<1$. Then the Euler equations \eqref{eq:euler-principal}, written in the variable \(\Jbb\), enjoy the following properties.
\bei

\item \emph{Admissible state space.} The set \(\mathcal A\) defined in the definition~\eqref{eq:state-space} is a closed convex cone containing both vacuum and non-vacuum states.

\item \emph{Recovery of the primitive variables.} On the timelike region \(\{-\Jbb\cdot\Jbb>0\}\), the primitive variables \((\mu,u)\) are recovered uniquely from the equations~\eqref{eq:J-norm} and \eqref{eq:u-from-J}. At vacuum \(\mu=0\), \(u\) is undefined.

\item \emph{Quadratic principal structure.} The principal fluxes of the system are quadratic in the components of \(\Jbb\), up to the scalar constitutive coefficient \(q_\Jbb\).

\item \emph{Hyperbolicity away from vacuum.} On the interior of the admissible cone \(\{-\Jbb\cdot\Jbb>0\}\), the principal symbol of the system~\eqref{eq:quasilinear-J} has $n+1$ real speeds: one transport speed of multiplicity $n-1$ and two acoustic speeds.

\item \emph{Characteristic structure.} Away from vacuum, the transport field of multiplicity $n-1$ is linearly degenerate. The two acoustic fields are genuinely nonlinear in the sense of Lax at every state at which \( G_{\rm N}\neq0\), with the notation in~\eqref{hyperbolic-eos-3}.

\item \emph{Degeneracy at vacuum.} At every nonzero null state, the two acoustic speeds coalesce with the transport speed, irrespective of whether \(p'(0) =0\) or \(p'(0)>0\). The interior characteristic splitting therefore degenerates at vacuum. At the zero state the matrices \(A_0,A_1\) both vanish.
\eei
\end{theorem}

\begin{proof}
\bse
\noindent{\bf 1.} The first assertion follows since \(\mathcal A\) is the closed future causal cone in the orthonormal components \(\Jbb\); in particular, it is convex. The fact that \(\mu\mapsto2(\mu+\eps^{2}p(\mu))\) is strictly increasing and bounded below and above by $2\mu,4\mu$ proves that it is a bijection of $(0,+\infty)$, from which the recovery statement immediately follows. The Euler equations~\eqref{eq:euler-principal} manifestly exhibit the quadratic principal structure. To show hyperbolicity away from vacuum we work in the original fluid variables and set (for $\mu>0$)
\be
h = \mu+\eps^{2}p(\mu), \qquad y = -\Jbb\cdot\Jbb=2h, \qquad \sigma_\parallel \coloneqq \frac{|\Jpar|^2}{y}, \qquad D \coloneqq 1+(1-\eps^{2}p'(\mu))\sigma_\parallel.
\ee
Substitution in the principal symbol associated with the system~\eqref{eq:quasilinear-J} gives the transport speed \(\lambda_0 \coloneqq -J_1/J_0\), with multiplicity \(n-1\), and the two acoustic speeds
\bel{eq:section-2-characteristic-speeds}
\lambda_\pm= \frac{\lambda_0\pm c_{\rm ac}}{1\pm\lambda_0 c_{\rm ac}}, \qquad c_{\rm ac}^2 \coloneqq \frac{\eps^{2}p'(\mu)}{D}\in(0,1).
\ee
Hence all characteristic speeds are real and satisfy \(-1<\lambda_-<\lambda_0<\lambda_+<1\).

\noindent{\bf 2.} A basis of the transport eigenspace is
\bel{eq:transport-eigenvectors}
r_j\coloneqq \Bigl( \frac{J_0J_j}{y+|\Jpar|^2},\frac{J_1J_j}{y+|\Jpar|^2},(\delta_{mj})_{m=2}^n \Bigr) , \qquad j=2,\ldots,n.
\ee
Indeed, a direct substitution gives \((\lambda_0A_0+A_1)r_j=0\). Their components parallel to the symmetry orbits are independent, and \(D\lambda_0[r_j]=0\), which proves linear degeneracy.

\ese

\bse
\noindent{\bf 3.} It remains to determine the variation of the acoustic speeds along their characteristic directions. We use the coordinates \((\mu,v,J_2,\ldots,J_n)\), where \(v = \artanh(-J_1/J_0)\), and introduce the vector fields
\be
\mathcal R_\pm \coloneqq \frac{h(1+\sigma_\parallel)}{D} \Bigl( \del_\mu+ \frac{1-\eps^{2}p'(\mu)}{2h}\sum_{m=2}^nJ_m\del_{J_m} \Bigr) \pm c_{\rm ac}\del_v.
\ee
A substitution in the principal symbol shows that \(\mathcal R_\pm\) span the two acoustic characteristic directions. We write \(\mathcal V\coloneqq\del_\mu+(1-\eps^{2}p') (2h)^{-1}\sum_{m=2}^nJ_m\del_{J_m}\) for the density direction in \(\mathcal R_\pm\). The definitions give \(\mathcal V\sigma_\parallel=-2\eps^{2}p'\sigma_\parallel/h\) and
\be
\label{eq:acoustic-speed-density-direction}
\frac{h(1+\sigma_\parallel)}D\mathcal V(c_{\rm ac}^2)+2c_{\rm ac}^2(1-c_{\rm ac}^2)=\frac{(1+\sigma_\parallel)^2}{D^3}G_{\rm N}(\mu).
\ee
Since \(\lambda_\pm= \tanh(v\pm\artanh c_{\rm ac})\),  differentiation along \(\mathcal R_\pm\), using the equation~\eqref{eq:acoustic-speed-density-direction}, gives
\bel{eq:acoustic-speed-directional-derivative}
\mathcal R_\pm\lambda_\pm = \pm \frac{\operatorname{sech}^2(v\pm\artanh c_{\rm ac})} {2c_{\rm ac}(1-c_{\rm ac}^2)} \frac{(1+\sigma_\parallel)^2}{D^3}\, G_{\rm N}(\mu).
\ee
The scalar factor multiplying \(\pm G_{\rm N}(\mu)\) is strictly positive. Thus the two acoustic fields are genuinely nonlinear exactly where \( G_{\rm N}\neq0\), with opposite orientations.

\ese
\noindent{\bf 4.} For the vacuum limit, we evaluate the characteristic polynomial in the momentum variable~\(\Jbb\). We put \(a\coloneqq(yq)'=(1-\eps^{2}p')/(1+\eps^{2}p')\in(0,1]\). Expanding the matrices in the equation~\eqref{eq:A0A1} gives
\bel{eq:general-characteristic-polynomial}
\det(\lambda A_0+A_1) =(\lambda J_0+J_1)^{n-1} \bigl(2a(\lambda J_0+J_1)^2- (1-a)y(1-\lambda^2) \bigr).
\ee
The last factor is responsible for the acoustic wave speeds. At a nonzero null state, \(y=0\), \(J_0\ne0\), and \(a=(1-\eps^{2}p'(0))/(1+\eps^{2}p'(0))>0\) so the last factor reduces to $2a(\lambda J_0+J_1)^2$ and all $n+1$ roots equal \(-J_1/J_0\). The leading coefficient remains nonzero there, since \(\det A_0=2aJ_0^{n+1}\). This proves coalescence of the interior speeds. The matrices $A_0,A_1$ both vanish at the zero state $\Jbb=0$ by the equation~\eqref{eq:A0A1} and boundedness of~$a$.
\end{proof}

\subsection{The isothermal model}
\label{section-3-2} 

\subsubsection{The asymptotic reference system}

The exactly isothermal law
\bel{equa-isothermal}
\eps^2p(\mu)=c_{\rm iso}^2\mu, \qquad 0<c_{\rm iso}<1,
\ee
serves as the reference equation at large density, although it does not belong to the polytropic-vacuum class of~\autoref{def-pression}. In this isothermal model, $c_{\rm iso}$ is the sound speed divided by the speed of light. By Proposition~\ref{prop:pressure-endpoint-asymptotics}, below, every admissible equation of state satisfies \(\eps^{2}p(\mu)/\mu\to \eps^{2}k_\infty^{2}\) and \(\eps^{2}p'(\mu)\to \eps^{2}k_\infty^{2}\) as \(\mu\to+ \infty\). Consequently, its pressure ratio converges to
\be
\label{eq:isothermal-asymptotic-pressure-ratio}
q(\mu) = \frac{\mu-\eps^{2}p(\mu)}{\mu+\eps^{2}p(\mu)} \to \frac{1-\eps^{2}k_\infty^{2}}{1+\eps^{2}k_\infty^{2}}.
\ee
The isothermal system with $c_{\rm iso}=\eps k_\infty$ gives the limiting principal coefficients as $\mu\to+\infty$. The isothermal law is the unique barotropic law with constant $q$, and its Euler system is exactly quadratic in $\Jbb$ (as noted in~\cite{TorusT2-1-2026}). Recall that LeFloch and Shelukhin~\cite{LeFlochShelukhin-2005} studied the non-relativistic isothermal current equation and its symmetries; LeFloch and Yamazaki~\cite{LeFlochYamazaki-2007} treated the relativistic problem, including vacuum and velocities approaching the speed of light.

\subsubsection{Exact quadratic structure}

The following corollary applies, in particular, to the radiation law  \(\eps^{2}p(\mu)=\mu/n\), which is also trace-free, since \(g^{\alpha\beta}T_{\alpha\beta}=-\mu+n\eps^{2}p\) in dimension \(1+n\).

\begin{corollary}[Structural properties of isothermal fluid flows]
\label{prop:isothermal-structure}
For the isothermal equation of state \(\eps^{2}p(\mu) =c_{\rm iso}^2\mu\) with $c_{\rm iso} \in (0,1)$, the Euler system \eqref{eq:euler-principal}, written in the variable \(\Jbb\), enjoys the following properties.
\bei

\item The relation between the density and the fluid momentum is explicit:
\bel{eq:J-norm-isothermal}
-\Jbb\cdot\Jbb=2(1+c_{\rm iso}^2)\mu.
\ee

\item The pressure ratio is constant:
\bel{eq:q-isothermal}
q= \frac{1-c_{\rm iso}^2}{1+c_{\rm iso}^2}, \qquad {q}\in(0,1).
\ee

\item The principal fluxes are exactly quadratic in \(\Jbb\).

\item The matter energy density is the quadratic form
\bel{eq:M00-isothermal}
\Mbf_{00}(\Jbb) = \frac{1+q}{2}J_0^2 + \frac{1-q}{2}\sum_{m=1}^nJ_m^2,
\ee
and is therefore strictly convex and positive-definite on \(\Rbb^{1+n}\).

\item Away from vacuum, the system is hyperbolic, with a linearly degenerate transport field of multiplicity $n-1$ and two genuinely nonlinear acoustic fields.

\item At a nonzero null state the acoustic speeds and the transport speed coincide, despite the positive squared sound speed \(\eps^{2}p'(0) =c_{\rm iso}^2>0\).

\eei
\end{corollary}

\begin{proof} The identity \eqref{eq:J-norm-isothermal} follows from \(-\Jbb\cdot\Jbb=2(\mu+\eps^{2}p(\mu))\), while the definition of the pressure ratio gives \(q=(1-c_{\rm iso}^2)/(1+c_{\rm iso}^2)\). Hence \((yq)'=q\), and all the principal fluxes are quadratic in \(\Jbb\). Moreover, \(\Mbf_{00} \, = \Ebf_0+ \tfrac q2(-\Jbb\cdot\Jbb)\), which gives the equation~\eqref{eq:M00-isothermal}; its two coefficients are positive since \(0<q<1\). \autoref{prop:structure-euler} gives the characteristic properties away from vacuum as well as at a nonzero null state.
\end{proof}

\subsection{Advected parallel momentum}
\label{section-3-3}

\subsubsection{Particle number normalization}

We return to the reduced system \eqref{eq:euler-principal} with a general equation of state, with \(e_0= \del_t\), \(e_1= \del_x\), and \(\Jbb=(\Jperp,\Jpar)\in\mathcal A\). With the particle density \(\Numb\) defined by the equation~\eqref{eq:section-3-particle-density}, we introduce the \emph{particle factor}
\bel{eq:section-4-particle-factor}
\hNumb(\mu) \coloneqq \frac{\Numb(\mu)}{\sqrt{2(\mu+\eps^{2}p(\mu))}}.
\ee
Since \(\Jbb=\sqrt{2(\mu+\eps^{2}p(\mu))}\,u\), the $(1+n)$-dimensional current is \(\vNumb^\alpha= \hNumb_\Jbb J^\alpha\); in particular, its components in the plane orthogonal to the symmetry orbits are
\bel{eq:section-4-reduced-particle-current}
\vNumb^0=-\hNumb_\Jbb J_0, \qquad \vNumb^1= \hNumb_\Jbb J_1, \qquad \del_t\vNumb^0+ \del_x\vNumb^1=0.
\ee
Here \(\hNumb_\Jbb \coloneqq \hNumb(\mu(\Jbb))\); raising the index of~\(J_\alpha\) changes the time-component sign. Plane symmetry reduces the conservation law~\eqref{eq:section-3-particle-conservation} to~\eqref{eq:section-4-reduced-particle-current}.

\subsubsection{Advected momentum parallel to the symmetry orbits} 

The normalized momentum in the equation~\eqref{eq:def-Jhat} satisfies \(\Jhatpar=\hNumb_\Jbb^{-1}\Jpar\), with the particle factor from the equation~\eqref{eq:section-4-particle-factor}. We call it the \emph{advected momentum parallel to the symmetry orbits} in view of the transport equation~\eqref{eq:Jhat-transport-Minkowski} below. The variable \(\Jhatpar\) is not determined by \(\Jbb\) on the vacuum boundary, where \(\hNumb_\Jbb=0\). Every statement below involving \(\Jhatpar\) is therefore understood either on a non-vacuum region or under the explicit hypothesis that a bounded representative of \(\Jhatpar\), inherited for instance from a regular approximation, has been prescribed up to vacuum. On each set $|\Jhatpar|^2\leq R$, because the normalization~\eqref{eq:section-3-particle-density} gives $\Numb(\mu)\leq\mu$, we have $|\Jpar|^2=\Numb(\mu)^2|\Jhatpar|^2/(2(\mu+\eps^{2}p(\mu)))\leq R\mu/2$. Hence every finite momentum limit as $\mu\to0$ has zero component parallel to the symmetry orbits, uniformly in rapidity in the direction orthogonal to the symmetry orbits; it may still be a nonzero null momentum orthogonal to the symmetry orbits. The $n-1$ equations for momentum parallel to the symmetry orbits in the equation~\eqref{eq:euler-principal} can then be written as \(\del_t(\vNumb^0\widehat J_m)+ \del_x(\vNumb^1\widehat J_m) =0\). Combining them with the conservation law~\eqref{eq:section-4-reduced-particle-current} yields
\bel{eq:Jhat-transport-Minkowski}
\del_t\widehat J_m-\frac{J_1}{J_0}\del_x\widehat J_m=0, \qquad m=2,\ldots,n.
\ee
Thus, away from vacuum, the normalized momentum parallel to the symmetry orbits is advected with velocity in the direction orthogonal to the symmetry orbits \(-J_1/J_0\). The particle and $n-1$ parallel-momentum conservation laws yield the following infinite family of currents.

\begin{proposition}[Currents defined from the advected momentum]
\label{prop:section-4-transported-currents}
For plane-symmetric flows and every \(H\in C^1(\Rbb^{n-1})\), every regular non-vacuum solution of the system~\eqref{eq:euler-principal}, or every regular solution for which both the particle current $\vNumb$ and $\Jhatpar$ have $C^1$ extensions to vacuum, satisfies
\bel{eq-H-Jhat-balance}
\del_t\bigl(H(\Jhatpar)\vNumb^0 \bigr)+ \del_x\bigl(H(\Jhatpar)\vNumb^1 \bigr) =0.
\ee

\end{proposition}

\begin{proof} We sum over \(a=0,1\) and use the components of \(\Jhatpar\) directly. The product and chain rules give
\be
\label{eq:transport-current-chain-vacuum}
\del_a\bigl(H(\Jhatpar)\vNumb^a \bigr) =H(\Jhatpar)\del_a\vNumb^a+ \sum_{m=2}^n\del_{\widehat J_m}H(\Jhatpar)\bigl(\del_a(\widehat J_m\vNumb^a) -\widehat J_m\del_a\vNumb^a \bigr) =0.
\ee
The last equality uses the particle conservation law and the $n-1$ conservation laws for momentum parallel to the symmetry orbits. Under the stated $C^1$ extension of the particle current, its divergence extends continuously from the positive-density region and is zero in the interior of the vacuum set, hence everywhere. The parallel-momentum identities likewise extend by continuity. The calculation therefore applies on the vacuum set without division by particle density.
\end{proof}

\subsubsection{Flows with constant advected momentum parallel to the symmetry orbits}

By the transport equation~\eqref{eq:Jhat-transport-Minkowski}, a spatially constant initial value of~$\Jhatpar$ remains constant along the flow. In this case the Euler system reduces to two equations for $\Jperp=(J_0,J_1)$.

\subsection{Quasi-currents and their shear-induced sources}
\label{section-3-4}

\subsubsection{The quasi-balance law}

We freeze $\Jhatpar$ to obtain a current of the two acoustic equations from the kernels constructed below. For a family of such currents that is $C^1$ in all state variables, the chain rule and transport equation give the following quasi-balance law when $\Jhatpar$ is allowed to vary.

\begin{definition} 
\label{def-quasi-cur-prelim}
For plane-symmetric fluid flows, a \(C^1\) vector field in the plane orthogonal to the symmetry orbits \(F= \Fcomp^0e_0+ \Fcomp^1e_1\), whose orthonormal components depend on \(\Jbb\) in the interior of the admissible cone, is called a \emph{quasi-current} for the reduced Euler system provided every regular non-vacuum solution to the system~\eqref{eq:euler-principal} satisfies, in the standard inertial frame,
\bel{eq:section-4-quasi-balance}
\del_t\Fcomp^0(\Jbb)+ \del_x\Fcomp^1(\Jbb) = \sum_{m=2}^n\Fvee_m(\Jbb)\del_x\widehat J_m,
\ee
in which the right-hand side is called the \emph{shear-induced source.}  The continuous functions \(\Fvee_m\), \(m=2,\ldots,n\), are called the \emph{shear-induced factors} associated with \(F\). These coefficients multiply derivatives of the advected momentum parallel to the symmetry orbits, which represent the shear of the reduced flow.
\end{definition} 
 
\subsubsection{Rapidity variable}

For every non-vacuum state, the rapidity \(v\in\Rbb\) in the direction orthogonal to the symmetry orbits satisfies
\bel{equa-rapi}
\tanh v = - \frac{J_1}{J_0} ,
\ee
and we introduce the orthonormal frame adapted to the fluid motion orthogonal to the symmetry orbits
\bel{eq:section-4-fluid-lorentz-frame}
\overline u = \cosh(v)e_0+ \sinh(v)e_1, \qquad \overline u^\perp = \sinh(v)e_0+ \cosh(v)e_1.
\ee
Thus \(\overline u\) is the projection of the fluid velocity~$u$ onto the \((e_0,e_1)\)-plane, normalized to unit norm, and \(\overline u^\perp\)  is the positively oriented unit spacelike vector orthogonal to it in that plane. They form an orthonormal basis, namely they obey $g(\overline u,\overline u)=-1$, $g(\overline u,\overline u^\perp)=0$, and $g(\overline u^\perp,\overline u^\perp)=1$. We shall find the following identities useful,
\be
J_0 \pm J_1 = - (-\Jperp\cdot\Jperp)^{1/2} e^{\mp v}.
\ee

\subsubsection{Lorentzian character and scalar components}

A density and its spatial fluxes are the components of one spacetime vector. For the two-component formulas in this paragraph, we consider a current taking values in a chosen space--time plane with oriented orthonormal frame \((e_0,e_1)\), so that \(F=F^0e_0+F^1e_1\). For a general current in dimension \(1+n\), the same formulas describe its projection onto this plane; its components parallel to the symmetry orbits remain unchanged by a boost in the plane. A boost of rapidity \(\theta\) is defined by \(e_0^\theta\coloneqq \cosh(\theta)e_0+ \sinh(\theta)e_1\) and \(e_1^\theta\coloneqq \sinh(\theta)e_0+ \cosh(\theta)e_1\). The components of the same current in the boosted frame are
\bel{eq:section-3-lorentz-current}
\begin{pmatrix}F^{\theta 0}\\F^{\theta 1}\end{pmatrix} \, = \begin{pmatrix}\cosh(\theta)&-\sinh(\theta)\\-\sinh(\theta)&\cosh(\theta)\end{pmatrix}\begin{pmatrix}F^0\\F^1\end{pmatrix}.
\ee
On the other hand, \(\nabla_aF^a\) is a scalar and is unaffected by the boost.

Based on the projected fluid velocity~$\overline u$ and its orthogonal~$\overline u^\perp$  given in the equation~\eqref{eq:section-4-fluid-lorentz-frame}, we associate with each current in the plane the two \emph{rest-frame scalar components}
\bel{eq:section-3-rest-frame-scalars}
\alpha_{F} \coloneqq -g(F,\overline u), \qquad \beta_{F} \coloneqq g(F,\overline u^\perp).
\ee
They are independent of the orthonormal frame used to calculate them, and they reconstruct the full current by
\bel{eq:section-3-current-rest-frame-decomposition}
F=F^0e_0+F^1e_1=\alpha_F\overline u+\beta_F\overline u^\perp.
\ee 

\begin{lemma}[Current components under Lorentz boosts]
\label{prop:section-3-current-reconstruction}
The rapidity $v=\artanh(u^1/u^0)\in\RR$ gives
\bse\label{eq:section-3-rest-frame-and-fixed-frame-components}
\bel{eq:section-3-rest-frame-components}
\alpha_{F}=\cosh(v)F^0-\sinh(v)F^1,\qquad \beta_{F}=-\sinh(v)F^0+\cosh(v)F^1,
\ee
\bel{eq:section-3-fixed-frame-components}
F^0=\cosh(v)\alpha_{F}+\sinh(v)\beta_{F},\qquad F^1=\sinh(v)\alpha_{F}+\cosh(v)\beta_{F}.
\ee
\ese
\end{lemma}

\begin{proof} The equation~\eqref{eq:section-3-rest-frame-components} follows by contracting \(F\) with \(\overline u\) and \(\overline u^\perp\), while the equation~\eqref{eq:section-3-fixed-frame-components} is the decomposition~\eqref{eq:section-3-current-rest-frame-decomposition}.
\end{proof}

\subsubsection{Frame-independent form of the quasi-balance law}

We now express the quasi-balance law of~\autoref{def-quasi-cur-prelim} without choosing an inertial frame. The source is written in terms of the scalar components of the current in the fluid frame.

\begin{proposition}[Lorentz-covariant quasi-balance law]
\label{prop:section-4-lorentz-quasi-balance}
For plane-symmetric fluid flows, let \(F\) be a quasi-current, with
\bel{eq:section-4-rest-frame-current}
\alpha_{F} = \cosh(v)\Fcomp^0-\sinh(v)\Fcomp^1, \qquad \beta_{F} = -\sinh(v)\Fcomp^0+ \cosh(v)\Fcomp^1
\ee
as its rest-frame scalar components. On every regular solution of the reduced Euler system, one has
\bel{eq:section-4-intrinsic-quasi-balance}
\nabla_aF^a= \sum_{m=2}^n\mathcal R_m\,(\overline u^\perp)^a\nabla_a\widehat J_m, \qquad \mathcal R_m \coloneqq \frac{\del\beta_{F}}{\del\widehat J_m}\Big|_{\mu,v} \, = \cosh(v)\Fvee_m.
\ee
{The parameter derivative keeps the original density and rapidity in the direction orthogonal to the symmetry orbits fixed.} Both sides of the balance law~\eqref{eq:section-4-intrinsic-quasi-balance} are Lorentz scalars. In particular, the quasi-current obeys the vectorial transformation rule \eqref{eq:section-3-lorentz-current}; only its divergence law differs from that of an ordinary current.
\end{proposition}

\begin{proof} The transport equations in \eqref{eq:Jhat-transport-Minkowski} and the definition of \(v\) give
\bel{eq:section-4-fluid-frame-transport}
\overline u^a\nabla_a\widehat J_m=0, \qquad (\overline u^\perp)^a\nabla_a\widehat J_m=(\cosh v)^{-1}\del_x\widehat J_m.
\ee
Moreover, since \(J_1/J_0=-\tanh v\), the second identity in \eqref{eq:section-4-rest-frame-current} gives  \(\beta_{F}=\cosh(v)(\Fcomp^1+(J_1/J_0)\Fcomp^0)\).  Differentiating at a fixed  \((\mu,v)\) gives \(\mathcal R_m=\cosh(v)(\del_{\widehat J_m}\Fcomp^1-\tanh(v)\del_{\widehat J_m}\Fcomp^0)\). The chain rule in the quasi-balance law identifies the expression in parentheses with \(\Fvee_m\), and hence \(\mathcal R_m=\cosh(v)\Fvee_m\). Substitution into the balance law~\eqref{eq:section-4-quasi-balance} proves the balance law~\eqref{eq:section-4-intrinsic-quasi-balance}.
\end{proof}

When using the quasi-balance law in applications, it is useful to distinguish conservation, convexity at fixed advected momentum, and integrability of the source, as follows.
\bei

\item \emph{Conservation.} If shear-induced factors vanish, one obtains a conserved current of the full $(n+1)$-component reduced system. If a family of solutions features constant~$\Jhatpar$, one has conservation on that family, without necessarily defining a conserved current of the full system.

\item \emph{Parametrization and convexity.} Future-convexity at fixed~$\Jhatpar$, as in Definition~\ref{def:section-3-future-convex} is the relevant notion when using these currents as mathematical entropies. Sections~\ref{section-9} and~\ref{section-10}, below, describe growth and convexity through the vacuum profile.

\item \emph{Shear-induced sources.} The products in the law~\eqref{eq:section-4-intrinsic-quasi-balance} belong to $L^1$ if $\Fvee_m\in L^\infty$ and $\del_x\widehat J_m\in L^1$.
Under this condition they are suitable sources in compensated compactness arguments.
\eei

\noindent The quasi-balance law~\eqref{eq:section-4-intrinsic-quasi-balance} is proved here for regular flows. Its use for weakly regular flows requires a justified chain rule and source terms defined as locally integrable functions or Radon measures.

\subsection{Fluid variables and shear-induced factors}
\label{section-3-5}

\subsubsection{Density and rapidity variables}

We now use the rapidity \(v\) defined in the equation~\eqref{equa-rapi} and set
\bel{eq:section-4-acoustic-state-reconstruction}
y=-\Jbb\cdot\Jbb=2\bigl(\mu+\eps^{2}p(\mu)\bigr)>0, \qquad \Sigma\coloneqq(y,v).
\ee
Together with the advected momentum parallel to the symmetry orbits \(\Jhatpar\), the pair \(\Sigma\) determines the fluid state throughout the positive-density interior. Since \(\mu\mapsto 2(\mu+\eps^{2}p(\mu))\) is strictly increasing, we may write \(\mu=\funmu(y)\) for its inverse. The reconstruction formulas are
\bel{eq:section-4-state-reconstruction}
\Jpar=\hNumb(\funmu(y))\Jhatpar, \qquad J_0=-\bigl(y+|\Jpar|^2\bigr)^{1/2}\cosh v, \qquad J_1=\bigl(y+|\Jpar|^2\bigr)^{1/2}\sinh v.
\ee
In particular, \(J_1/J_0=-\tanh v\), and \(\Sigma\) ranges over \((0,+\infty)\times\Rbb\).

These coordinates describe the positive-density interior, but they need not remain bounded along a sequence approaching a nonzero null state. If \(y\to0\) while \(\Jhatpar\) and \(v\) remain bounded, then \(\Jbb\to0\). A nonzero null limit with bounded \(\Jhatpar\) therefore requires \(|v|\to+\infty\); an unbounded \(\Jhatpar\) gives another way for the coordinates to leave every compact set. We formulate all boundary statements directly in the original variable \(\Jbb\).

In these variables, the reduced Euler system~\eqref{eq:euler-principal} takes the form
\bel{Sigmat}
\del_t\Sigma=A(\Sigma,\Jhatpar)\del_x\Sigma, \qquad \del_t\Jhatpar=-\tanh(v)\del_x\Jhatpar.
\ee
We define the matrix and flux variables by
\bel{eq:section-4-acoustic-matrix}
U\coloneqq(\Mbf_{00},J_0J_1), \qquad G\coloneqq(J_0J_1,\Mbf_{11}), \qquad A(\Sigma,\Jhatpar)\coloneqq(D_\Sigma U)^{-1}D_\Sigma G.
\ee
Here \(U\) and \(G\) are evaluated using the equation~\eqref{eq:section-4-state-reconstruction}. For a current \(F=(\Fcomp^0,\Fcomp^1)\), its multiplier \(W=(W_0,W_1)\), at fixed \(\Jhatpar\), is determined by
\bel{eq:section-4-acoustic-multiplier}
D_\Sigma\Fcomp^0=W\cdot D_\Sigma U, \qquad D_\Sigma\Fcomp^1=-W\cdot D_\Sigma G.
\ee
The first relation determines \(W\) on the positive-density state space once \(D_\Sigma U\) is shown to be invertible; the second is the compatibility relation for the current flux. We write the particle factor in the variable~$y$ as
\be
\underline{\Numb}(y)\coloneqq\hNumb(\funmu(y)), \qquad y>0.
\ee

\begin{lemma}[Block diagonalization]
\label{lem:section-4-acoustic-block-diagonalization}
With $c^2= \eps^{2}p'(\funmu(y))$,
for every \(m=2,\ldots,n\), differentiation at fixed \(\Sigma\) gives
\bel{eq:section-4-flux-parameter-derivatives}
\frac{\del U}{\del\widehat J_m} = 2 \underline{\Numb}(y)^2\widehat J_m\cosh^2v\,(1,-\tanh v), \qquad \frac{\del G}{\del\widehat J_m} = -\tanh v\,\frac{\del U}{\del\widehat J_m}.
\ee
Moreover, throughout the positive-density state space, \(D_\Sigma U\) is invertible and
\bel{eq:acoustic-coordinate-determinant}
-\det D_\Sigma U = \frac{(y+ \underline{\Numb}^2|\Jhatpar|^2)\cosh^2v}{1+c^2} \Bigl( 1-c^2\tanh^2v +\frac{1}{y} \underline{\Numb}(y)^2|\Jhatpar|^2 \, (1-c^2) \Bigr) > 0.
\ee
\end{lemma}

These coordinates and differentiations require \(\mu>0\). At vacuum, \(v\) and \(\Jhatpar\) may be undefined; boundary extensions are understood in~\(\Jbb\).

\begin{proof}
\bse
We set $R\coloneqq y+ \underline{\Numb}^2|\Jhatpar|^2$. The reconstruction formula~\eqref{eq:section-4-state-reconstruction} yields
\bel{eq:section-4-J0-square-reconstruction}
J_0^2=R\cosh^2v, \qquad \frac{J_1}{J_0}=-\tanh v.
\ee
At fixed \(\Sigma\), the momentum parallel to the symmetry orbits satisfies \(\Jpar= \underline{\Numb}(y) \Jhatpar\). Differentiating \(\Mbf_{00}\), \(J_0J_1\), and \(\Mbf_{11}\) with respect to \(\widehat J_m\) therefore gives
\bel{eq:section-4-component-parameter-derivatives}
\frac{\del\Mbf_{00}}{\del\widehat J_m} = 2 \underline{\Numb}(y)^2\widehat J_m\cosh^2v, \qquad \frac{\del(J_0J_1)}{\del\widehat J_m} = -\tanh v\,\frac{\del\Mbf_{00}}{\del\widehat J_m}, \qquad \frac{\del\Mbf_{11}}{\del\widehat J_m} = \tanh^2v\,\frac{\del\Mbf_{00}}{\del\widehat J_m}.
\ee
These identities prove the equations~\eqref{eq:section-4-flux-parameter-derivatives}. To calculate the determinant, we use \(\Numb'/\Numb=1/(\mu+\eps^{2}p)\) and \(y=2(\mu+\eps^{2}p)\), which imply
\bel{eq:acoustic-coordinate-derivatives}
\frac{d(\underline{\Numb}^2)}{dy} = \frac{\underline{\Numb}^2(1-c^2)}{y(1+c^2)}, \qquad \frac{d\eps^{2}p}{dy} = \frac{c^2}{2(1+c^2)}, \qquad U= \bigl(R\cosh^2v-2\eps^{2}p,\,-R\sinh v\cosh v\bigr).
\ee
Differentiation gives the determinant identity
\be
-\det D_\Sigma U = R\Bigl( \frac{dR}{dy}\cosh^2v -2\frac{d\eps^{2}p}{dy}\cosh(2v) \Bigr), \qquad \frac{dR}{dy} = 1+\frac{\underline{\Numb}(y)^2|\Jhatpar|^2(1-c^2)}{y(1+c^2)}.
\ee
Substituting the equation~\eqref{eq:acoustic-coordinate-derivatives} gives the equation~\eqref{eq:acoustic-coordinate-determinant}. Its right-hand side is positive since \(R>0\), \(y>0\), and \(0<c^2<1\). Finally, we substitute the transport law for \(\Jhatpar\) into the first two equations of~\eqref{eq:euler-principal}. The identities in~\eqref{eq:section-4-flux-parameter-derivatives} cancel the derivatives of \(\Jhatpar\); multiplication by \((D_\Sigma U)^{-1}\) then gives the first equation of~\eqref{Sigmat}. The remaining reduced equations give its second equation.
\ese
\end{proof}

\begin{lemma}[Expressions for the shear-induced factors]
\label{lem:Fvee-expr-prelim}
With the notation in~\autoref{def-quasi-cur-prelim}, the shear-induced factors are
\bel{Fvee-expr}
\Fvee_m(\Jbb) =  \frac{\del}{\del\widehat J_m} \bigl(\Fcomp^1-\tanh(v)\Fcomp^0\bigr) \Big|_{\Sigma}, \qquad m=2,\ldots,n.
\ee
Thus the derivative is taken with \(\Sigma=(y,v)\) fixed. With the null variables $L_\pm\coloneqq J_0\pm J_1$ and the multiplier $W=(W_0,W_1)$ from the equation~\eqref{eq:section-4-acoustic-multiplier}, one also has
\bel{eq:section-4-null-multiplier}
W_0\pm W_1 = \frac{1}{L_\pm}  \frac{d(\Fcomp^0\mp\Fcomp^1)}{dL_\pm} \Big|_{L_+L_-,\,\Jpar}.
\ee
The subscripts on \(W\) denote its two components, not Lorentzian index lowering.
\end{lemma}

\begin{proof}
\bse
We write \(\Fcomp^a=\Fcomp^a(\Sigma,\Jhatpar)\), substitute the system~\eqref{Sigmat} into the balance law~\eqref{eq:section-4-quasi-balance}, and apply the compatibility relation~\eqref{eq:section-4-acoustic-multiplier}. The terms involving \(\del_x\Sigma\) cancel, and the coefficient of \(\del_x\widehat J_m\) is the right-hand side of~\eqref{Fvee-expr}. For the second identity, we hold \(-\Jperp\cdot\Jperp\) and \(\Jpar\) fixed. We use the identity
\be
y=-\Jperp\cdot\Jperp-|\Jpar|^2,
\ee
which shows that \(y\), \(\mu\), and \(\Jhatpar\) also remain fixed. At a timelike state, the constraint \(L_+L_-=-\Jperp\cdot\Jperp\) specifies a positive product and determines either null variable from the other. Using \(U_0=J_0^2-2\eps^{2}p\), \(U_1=G_0=J_0J_1\), and \(G_1=J_1^2+2\eps^{2}p\), direct differentiation gives
\bel{eq:null-flux-derivatives}
\frac{d(U_0\pm G_0)}{dL_\pm}=L_\pm, \qquad \frac{d(U_1\pm G_1)}{dL_\pm}=\pm L_\pm.
\ee
By the equation~\eqref{eq:section-4-acoustic-multiplier}, we have $ d(\Fcomp^0\mp\Fcomp^1)=W\cdot d(U\pm G). $ The preceding derivatives show that the right-hand side is \(L_\pm(W_0\pm W_1)\). Since \(L_\pm\ne0\) at every timelike state, division by \(L_\pm\) proves the equation~\eqref{eq:section-4-null-multiplier}.
\ese
\end{proof}

\subsubsection{Quasi-currents defined from advected variables}

We take \((F^{0}_{0},F^{1}_{0})=(F^{0}_{0},F^{1}_{0})(\Sigma,\Jhatpar)\)  to be a smooth family of density--flux pairs. For each fixed \(\Jhatpar\), we assume that it satisfies
\bel{eq:section-4-acoustic-current-compatibility}
D_\Sigma F^{1}_{0}=-D_\Sigma F^{0}_{0}\,A,
\ee
where \(A\) is defined in the equation~\eqref{eq:section-4-acoustic-matrix}; the derivatives here keep \(\Jhatpar\) fixed. Given a smooth scalar function \(H\) of the advected variables, we define
\bel{eq:section-4-transported-parameter-current}
\Fcomp^0\coloneqq H(\Jhatpar)F^{0}_{0}(\Sigma,\Jhatpar), \qquad \Fcomp^1\coloneqq H(\Jhatpar)F^{1}_{0}(\Sigma,\Jhatpar).
\ee

\begin{lemma}[Quasi-current defined from an advected variable]
\label{prop:section-4-transported-parameter-current}
For plane-symmetric flows, the pair in~\eqref{eq:section-4-transported-parameter-current} is a quasi-current, and its shear-induced factors are
\bel{eq:section-4-transported-parameter-factors}
\Fvee_m = (\del_{\widehat J_m}H) \bigl(F^{1}_{0}-\tanh(v)F^{0}_{0}\bigr) + H\bigl( \del_{\widehat J_m}F^{1}_{0} -\tanh(v)\del_{\widehat J_m}F^{0}_{0} \bigr), \qquad m=2,\ldots,n.
\ee
All derivatives with respect to \(\widehat J_m\) keep \(\Sigma\) fixed. The second term vanishes when \((F^{0}_{0},F^{1}_{0})\) is independent of \(\Jhatpar\); in that case, its compatibility relation is required for each parameter value under consideration.
\end{lemma}

\begin{proof}
The compatibility relation~\eqref{eq:section-4-acoustic-current-compatibility} cancels the terms involving \(\del_x\Sigma\). By the equation~\eqref{Sigmat}, the remaining terms in \(\del_tF^{0}_{0}+\del_xF^{1}_{0}\) are
\be
\sum_{m=2}^n \bigl( -\tanh(v)\del_{\widehat J_m}F^{0}_{0} +\del_{\widehat J_m}F^{1}_{0} \bigr)\del_x\widehat J_m.
\ee
We apply the product rule to \(H(\Jhatpar)F^{0}_{0}\) and \(H(\Jhatpar)F^{1}_{0}\), and use the transport equation \(\del_tH=-\tanh(v)\del_xH\). The resulting coefficient of \(\del_x\widehat J_m\) is~\eqref{eq:section-4-transported-parameter-factors}.
\end{proof}

\section{Reference currents and mathematical entropy inequalities}
\label{section-4}

\subsection{Particle current}
\label{section-4-1}

\subsubsection{Behavior of the particle current at vacuum}

Here, we compare the properties of the two reference currents: mass-energy and negative particle number. We are going to prove that each is future-convex in its own conservative variables; consequently, it can be used to formulate a mathematical entropy inequality for weakly regular flows, with the remaining conservation laws imposed in the sense of distributions; the two choices select different admissible shock waves, but are fully equivalent in the class of regular solutions. 

We use the particle density and current defined in the equations~\eqref{eq:section-3-particle-density} and~\eqref{eq:section-3-particle-current}, extending \(\Numb\) continuously by \(\Numb(0) =0,\) while writing \(\vNumb\) for this same current when it is expressed through \(\Jbb\). With the particle factor already introduced in the equation~\eqref{eq:section-4-particle-factor}, we have
\bel{eq:section-5-particle-current-components}
\vNumb^\alpha= \hNumb_\Jbb J^\alpha, \qquad (\vNumb^0,\ldots,\vNumb^n) = \hNumb_\Jbb(-J_0,J_1,\ldots,J_n),
\ee
where \(\hNumb_\Jbb = \hNumb(\mu(\Jbb))\). This expression extends continuously to the null vacuum boundary, although the primitive velocity does not. For the admissible polytropic-vacuum laws, this extension follows directly from \(\Numb(\mu)/\mu\to1\) and \(\eps^{2}p(\mu)/\mu\to0\):
\be
\label{eq:particle-factor-vacuum-extension}
\hNumb(\mu) = \sqrt{\mu/2}\,(1+o(1)), \qquad \vNumb^\alpha= \sqrt{\mu/2}\,(1+o(1))J^\alpha.
\ee
Thus the particle current tends to zero whenever \(\Jbb\) approaches any finite null state. The same conclusion holds for \(H(\Jhatpar)\vNumb\) if \(H\) is continuous and \(\Jhatpar\) remains in a fixed compact set: its limit is zero independently of the representative chosen on the vacuum set. This is an extension of the current, not of the primitive velocity or of \(\Jhatpar\). For an isothermal law \(\eps^2p=c_{\rm iso}^2\mu\), with \(0<c_{\rm iso}<1\), the particle factor is instead proportional to \(\mu^{(1-c_{\rm iso}^2)/(2(1+c_{\rm iso}^2))}\), whose exponent is positive, so the particle current again vanishes at finite null states.

The stress-energy current has a different boundary value. Since \(T^{\alpha\beta} \, = \tfrac12J^\alpha J^\beta+\eps^{2}p\,g^{\alpha\beta}\), it tends to \(\tfrac12J^\alpha J^\beta\), which need not vanish at a nonzero null state. Consequently, domination by the reference energy alone neither forces a current to vanish at vacuum nor proves its continuity there. For each admissible pressure, the sharper estimate~\eqref{eq:fixed-Hone-current-bound}, below,  gives that continuity for kernel currents with \(H^1\) vacuum profiles.

\subsubsection{Growth estimates of the particle current}

Differentiating the definition~\eqref{eq:section-4-particle-factor} and using the pressure ratio \(q(\mu) = \frac{\mu-\eps^{2}p(\mu)}{\mu+\eps^{2}p(\mu)}\), we find
\bel{eq-hN-derivative}
d(\log \hNumb) = \frac{1}{2(\mu+\eps^{2}p(\mu))}\,d\bigl((\mu+\eps^{2}p(\mu))\,q(\mu) \bigr) = \frac{1-\eps^{2}p'(\mu)}{2(\mu+\eps^{2}p(\mu))}\,d\mu.
\ee
Hence \(\hNumb\) is increasing. We also have the following growth bounds.

\begin{lemma}[Growth of the particle factor and current]
\label{lem-hN}
There exist constants \(C>1\) and \(\theta\in(1/2,1)\), depending on the equation of state and on $\eps$, such that
\bel{eq:section-5-particle-growth}
\aligned
\hNumb'(\mu)&>0,
&\qquad \lim_{\mu\to0}\hNumb(\mu)&=0,
\\
C^{-1}\mu^{1/2}&\leq\hNumb(\mu)\leq C\mu^{1/2}\quad(0<\mu\leq1),
&\qquad \hNumb(\mu)&\leq C\mu^{\theta-1/2}\quad(\mu\geq1).
\endaligned
\ee
Consequently, on the causal cone one has
\bel{eq:section-5-current-subquadratic}
|\vNumb|\le {C|J_0|^{2\theta}}, \qquad \frac{|\vNumb|}{\Mbf_{00}(\Jbb)} \to 0 \quad\text{as }|J_0| \to + \infty.
\ee

\end{lemma}

\begin{proof}
The identity \eqref{eq-hN-derivative} gives the monotonicity. At vacuum,  the assumption~\eqref{equa-cond-small} yields \(\Numb(\mu)\simeq\mu\) and \(\hNumb(\mu)\simeq\mu^{1/2}\). At large density, the condition~\eqref{hyperbolic-eos-2} gives \(\Numb(\mu)\le C\mu^\theta\) and \(\hNumb(\mu)\le C\mu^{\theta-1/2}\) for some \(\theta\in(1/2,1)\). For \(0<\mu\leq1\), we have \(\mu^{1/2}\leq\mu^{\theta-1/2}\). Thus \(\hNumb(\mu)\leq C\mu^{\theta-1/2}\) holds for every \(\mu>0\). Since \(|\Jbb|_{\rm Eucl}\leq\sqrt2|J_0|\), \(\mu\le J_0^2/2\), and \(\Mbf_{00}(\Jbb)\ge J_0^2/2\),  we obtain
\be
\label{eq:particle-current-homogeneous-growth}
|\vNumb|\leq C\mu^{\theta-1/2}|J_0|\leq C|J_0|^{2\theta},\qquad \frac{|\vNumb|}{\Mbf_{00}(\Jbb)}\leq C|J_0|^{2\theta-2}\quad(J_0\ne0).
\ee
Since \(2\theta-2<0\), the limit follows. The current is zero on the null boundary, and continuity gives the homogeneous bound at the origin.
\end{proof}

\subsubsection{The isothermal example}

For \(\eps^{2}p(\mu) =c_{\rm iso}^2\mu\), the identities in~\autoref{section-2} give \(\hNumb(\mu)\simeq\mu^{q/2}\) and \(\mu\simeq-\Jbb\cdot\Jbb\), with the constant pressure ratio~\eqref{eq:q-isothermal}. Hence we obtain
\bel{eq-interesting-entropy}
\vNumb^\alpha=C_{\rm iso}(-\Jbb\cdot\Jbb)^{q/2}J^\alpha, \qquad C_{\rm iso}>0.
\ee

\subsection{Dissipated-energy formulation}
\label{section-4-2}

\subsubsection{Choice of conservative variables}

With our normalization, we have
\bel{eq:section-5-fluid-stress-components}
T^{00} = \frac12\Mbf_{00}, \qquad T^{0m} = -\frac12J_0J_m, \qquad m=1,\ldots,n,
\ee
and recall that \(\vNumb^0=-\hNumb_\Jbb J_0\geq0\). For the dissipated-energy formulation proposed in~\cite{TorusT2-1-2026}, the conservative variables and the reference current are
\bel{eq:section-5-mass-energy-variables}
Z_{\rm E} \coloneqq (\vNumb^0,(T^{0m})_{m=1}^n), \qquad F_{\rm E}^* \coloneqq (\Mbf_{00},-\Mbf_{01}).
\ee
We choose the future unit normal \(e_0\) to the foliation and set \(\mathcal E_\beta \coloneqq \nabla_\alpha T^\alpha{}_\beta\), \(F_{\rm E}^\alpha \coloneqq -T^\alpha{}_\beta e_0^\beta\), and \(\Sigma_{\rm E} \coloneqq -T^\alpha{}_\beta\nabla_\alpha e_0^\beta\). The signs follow the convention~\eqref{eq:euler-principal}; the two components are density and flux.

\subsubsection{Convexity of the mass-energy component}

We write \(\rho\coloneqq\vNumb^0\), \(Z_m\coloneqq T^{0m}\) for \(m=1,\ldots,n\), and \(\vec Z\coloneqq(Z_1,\ldots,Z_n)\), so that \(Z_{\rm E}=(\rho,\vec Z)\).

\begin{proposition}[Convexity of the mass-energy]
\label{lem-convexM}
Under the admissibility conditions of Definition~\ref{def-pression}, every state \(\Jbb\in\mathcal A\) is uniquely determined by \(Z_{\rm E}\). Moreover, the map
\bel{eq:section-5-mass-energy-convex-map}
Z_{\rm E}\longmapsto\Mbf_{00}(\Jbb)
\ee
is non-negative and convex on \(\{\rho\geq0,\ \vec Z\in\Rbb^n\}\), with positive-definite Hessian when \(\rho>0\), and it extends continuously to the entire null vacuum boundary.
\end{proposition}

\begin{proof}
\bse
{\bf 1.} 
We use \(Z_{\rm E}=(\rho,\vec Z)\), \(y=-\Jbb\cdot\Jbb\), and \(\underline{\Numb}(y)=\hNumb(\funmu(y))\), where \(\funmu\) is the inverse of \(\mu\mapsto2(\mu+\eps^{2}p(\mu))\). For \(\rho>0\), the definitions give
\bel{eq:section-5-energy-reconstruction}
J_0=- \underline{\Numb}(y)^{-1}\rho, \qquad J_m=2 \underline{\Numb}(y)\rho^{-1}Z_m, \qquad m=1,\ldots,n.
\ee
Consequently, the physical value of \(y\) is characterized by
\bel{eq:section-5-energy-scalar-reconstruction}
\phi_{\rho}(y) \coloneqq \underline{\Numb}(y)^{-2} \rho^2 \bigl(\underline{\Numb}(y)^{-2} \rho^2 -y \bigr), \qquad \phi_{\rho}(y)=4|\vec Z|^2.
\ee
The identity \eqref{eq-hN-derivative}, rewritten in the variable \(y\), yields \(\underline{\Numb}'(y)>0\) and
\bel{eq:section-5-energy-reconstruction-monotonicity}
\del_y\phi_{\rho}(y) =-4\frac{\underline{\Numb}'(y)}{\underline{\Numb}(y)}\phi_{\rho}(y) - \underline{\Numb}(y)^{-2} \rho^2  \Bigl( 1+ \frac{d(yq)}{dy} \Bigr) < 0
\ee
on the interval on which \(\phi_{\rho}>0\). The vacuum behavior in  the assumption~\eqref{equa-cond-small} gives \(\phi_{\rho}(y)\to+ \infty\) as \(y\to0\), while \(\phi_{\rho}\) vanishes at the unique point determined by \(y \underline{\Numb}(y)^2=  \rho^2 \). Thus the equation~\eqref{eq:section-5-energy-scalar-reconstruction} has one and only one solution when \(\rho>0\). If \(\rho=0\) and \(\vec Z=0\), the unique state is \(\Jbb=0\). If \(\rho=0\) and \(\vec Z\neq0\), the unique state lies on the null vacuum boundary and is given by
\bel{eq:section-5-energy-null-reconstruction}
J_0=-\sqrt{2|\vec Z|}, \qquad J_m= \sqrt{\frac{2}{|\vec Z|}}Z_m, \qquad m=1,\ldots,n.
\ee
This null state satisfies \(T^{0m} \, =Z_m\), completing recovery on the closed cone.
\ese

\vskip.2cm

{\bf 2.} 
\bse
For fixed \(y>0\),  we write \(p(y) \coloneqq p(\funmu(y))\) and define
\bel{eq:section-5-energy-support-function}
Q_y(\rho,\vec Z) \coloneqq \sqrt{y \underline{\Numb}(y)^{-2} \rho^2 +4|\vec Z|^2}-2\eps^{2}p(y).
\ee
Each \(Q_y\) is convex in \((\rho,\vec Z)\), being a Euclidean norm of a linear map minus a constant. We now prove that
\bel{eq:section-5-energy-supremum}
\Mbf_{00}(\Jbb) = \sup_{y>0}Q_y(\rho,\vec Z).
\ee
At the physical value \(y=-\Jbb\cdot\Jbb\), the right-hand side in the equation~\eqref{eq:section-5-energy-support-function} equals \(J_0^2-2\eps^{2}p= \Mbf_{00}\). Furthermore, the identity \(\del_y(y \underline{\Numb}(y)^{-2}) =4 \underline{\Numb}(y)^{-2}\del_y\eps^{2}p(y)\) gives
\bel{eq:section-5-energy-support-derivative}
\del_yQ_y=2\del_y\eps^{2}p(y) \Biggl( \frac{\underline{\Numb}(y)^{-2} \rho^2 }{\sqrt{y\underline{\Numb}(y)^{-2} \rho^2 +4|\vec Z|^2}}-1 \Biggr).
\ee
The sign of the derivative in the equation~\eqref{eq:section-5-energy-support-derivative} is the sign of \(\phi_{\rho}(y) -4|\vec Z|^2\). For \(\rho>0\), the physical value of \(y\) is therefore the unique maximizer of \(Q_y\), while the large-density assumption guarantees \(Q_y\to-\infty\) as \(y\to+ \infty\). If \(\rho=0\), we have \(Q_y(0,\vec Z) =2|\vec Z|-2\eps^{2}p(y)\), and its supremum as \(y\to0\) is \(2|\vec Z|\), which equals \(\Mbf_{00}\) for the null state \eqref{eq:section-5-energy-null-reconstruction}. Formula \eqref{eq:section-5-energy-supremum} follows on the entire closed cone. Since a supremum of convex functions is convex, the formula also proves convexity; we verify continuity at the boundary below.

For \(0\leq\rho\leq R\) and \(|\vec Z|\leq R\), monotonicity of \(\underline{\Numb}\) and the positive high-density pressure ratio give \(Q_y\leq C R\sqrt y+2R-cy\) for all sufficiently large \(y\), with \(c>0\). This tends to \(-\infty\) uniformly on that state set. Since \(\lim_{y\to0}Q_y\geq0\), the supremum can therefore be restricted to \(0<y\leq Y_R\). Moreover, \(y\underline{\Numb}(y)^{-2}\to4\) at zero by the particle normalization, so \(L_R \coloneqq \sup_{0<y\leq Y_R}\sqrt y/\underline{\Numb}(y)<+\infty\). The reverse triangle inequality for the norm defining \(Q_y\) gives
\be
\label{eq:mass-energy-boundary-continuity}
|Q_y(\rho,\vec Z) -Q_y(\widetilde\rho,\widetilde{\vec Z})| \leq L_R|\rho-\widetilde\rho|+2|\vec Z-\widetilde{\vec Z}|, \qquad 0<y\leq Y_R.
\ee
Taking suprema proves local Lipschitz continuity of the mass-energy density on the closed half-space, including \(\rho=0\), with boundary value \(2|\vec Z|\).
\ese

\bse
{\bf 3.} 
To calculate its Hessian in the interior, we put \(h= \mu+\eps^{2}p\), \(c^2=\eps^{2}p'(\mu)\), \(w\coloneqq(u^1,\ldots,u^n)\), \(\Gamma=(1+|w|^2)^{1/2}\), and \(E=T^{00} \, = \Mbf_{00}/2\). Thus \(\rho= \Numb\Gamma\) and \(\vec Z=h\Gamma w\). The spatial block of the Jacobian of \((\mu,w)\mapsto(\rho,\vec Z)\) is \(h(\Gamma\operatorname{Id}+ {\Gamma}^{-1}w\otimes w)\). Its tangential eigenvalue is \(h\Gamma\), with multiplicity \(n-1\), and its radial eigenvalue is \(h(1+2|w|^2)/\Gamma\). The Schur complement therefore gives the positive determinant \(\Numb h^{n-1}{\Gamma}^{n-1}[1+(1-c^2)|w|^2]\). Direct differentiation gives
\be
\label{eq:mass-energy-gradient-interior}
D_{\rho}E= \frac{h}{\Numb\Gamma}, \qquad D_{\vec Z}E= \frac{w}{\Gamma}.
\ee
For a primitive variation \((\dot\mu,\dot w)\),  we denote by \(\delta Z_{\rm E}\)  the corresponding conservative variation. Differentiating the gradient identity~\eqref{eq:mass-energy-gradient-interior} and pairing with \(\delta Z_{\rm E}\), we obtain
\be
\label{eq:mass-energy-hessian-interior}
D^2_{Z_{\rm E}}E[\delta Z_{\rm E},\delta Z_{\rm E}]= \frac{c^2}{h}\dot\mu^2+ \frac{2c^2}{{\Gamma}^2}\dot\mu\,w\cdot\dot w+h|\dot w|^2-\frac{h}{{\Gamma}^2}(w\cdot\dot w)^2.
\ee
If \(w\neq0\), we decompose \(\dot w= \dot w_\perp+ \lambda w\), with \(w\cdot\dot w_\perp=0\). Completing the square yields
\be
\label{eq:mass-energy-hessian-squares}
D^2_{Z_{\rm E}}E[\delta Z_{\rm E},\delta Z_{\rm E}]=h|\dot w_\perp|^2+ \frac{h}{{\Gamma}^2|w|^2} \Bigl( w\cdot\dot w+ \frac{c^2|w|^2}{h}\dot\mu \Bigr) ^2 + \frac{c^2}{h{\Gamma}^2}\bigl(1+(1-c^2)|w|^2 \bigr)\dot\mu^2.
\ee
At \(w=0\), the Hessian formula~\eqref{eq:mass-energy-hessian-interior} reduces to \(c^2h^{-1}\dot\mu^2+h|\dot w|^2\). In either case the quadratic form vanishes only for the zero variation. Since \(\Mbf_{00} \, =2E\), this proves positive definiteness. The boundary value \(2|\vec Z|\) is not strictly convex along radial segments; the strict conclusion is restricted to positive density.
\ese
\end{proof}

\subsubsection{Convex currents from momentum parallel to the symmetry orbits}

The advected-momentum currents of  \autoref{prop:section-4-transported-currents}  give further convex densities in all $n+1$ conservative variables. Indeed, with \(Z_\parallel\coloneqq(Z_2,\ldots,Z_n)\), the reconstruction formula gives \(\Jhatpar=2Z_\parallel/\rho\). For \(H\in C^2(\Rbb^{n-1})\), the time component of \(H(\Jhatpar)\vNumb\) is therefore \(f_H(\rho,\vec Z) \coloneqq \rho H(2Z_\parallel/\rho)\). Writing \(z\coloneqq2Z_\parallel/\rho\), we calculate
\be
\label{eq:transported-current-full-convexity}
D^2 f_H[(\dot\rho,\dot{\vec Z}), (\dot\rho,\dot{\vec Z})] = \rho^{-1}D^2H(z)[2\dot Z_\parallel-z\dot\rho,2\dot Z_\parallel-z\dot\rho].
\ee
Thus \(f_H\) is convex on \(\rho>0\) if and only if \(H\) is convex: necessity follows at fixed positive \(\rho\), and sufficiency from the Hessian formula. This convexity is not strict, since the density is independent of \(Z_1\). Its extension to vacuum depends on the growth of \(H\). For example, \(H(z) = \sqrt{1+|z|^2}\) gives
\be
\label{eq:transverse-norm-current}
F_H^0= \sqrt{ \rho^2 +4|Z_\parallel|^2}, \qquad F_H^1= \frac{Z_1}{E+\eps^{2}p(\mu)}F_H^0, \qquad E=T^{00}.
\ee
This current extends continuously to the closed state cone. At a nonzero null state, its components are \(2|Z_\parallel|\) and \(2|Z_\parallel|Z_1/|\vec Z|\), respectively; at the zero state both vanish. To check continuity at the origin and its finite-energy growth, we use \(\Numb\leq\mu\), \(E\geq\mu(u^0)^2\), and \(|\vec Z|\leq E\). They give \(\rho\leq E\), \(0\leq F_H^0\leq\sqrt5 E\), and \(|F_H^1|\leq F_H^0\). The density is convex on the closed half-space, being a Euclidean norm, and its current is conserved on regular non-vacuum flows. These facts do not impose a genuine nonlinearity condition on the pressure.

\subsubsection{\texorpdfstring{The selected mass-energy inequality}{The selected mass-energy inequality}}

The identity in \eqref{eq:section-3-stress-current-identity} reads \(\nabla_\alpha F_{\rm E}^\alpha= \Sigma_{\rm E}\). The \textit{dissipated-energy formulation} proposed in~\cite{TorusT2-1-2026} imposes the particle-current equation and the $n$ spatial stress equations as distributional equalities, while imposing
\bel{eq:section-5-mass-energy-selection}
\nabla_\alpha\vNumb^\alpha=0, \qquad \mathcal E\big|_{T\Sigma_t} \, =0, \qquad \nabla_\alpha F_{\rm E}^\alpha\leq\Sigma_{\rm E}.
\ee
The middle equality means that the covector $\mathcal E$ vanishes on every vector tangent to the chosen spacelike leaf; it is therefore independent of the tangential frame used to express it.

\subsection{Particle-production formulation}
\label{section-4-3}

\subsubsection{Convexity of the negative particle-number current}

For the particle-production formulation proposed in~\cite{TorusT2-1-2026}, the conservative variables and the reference current are
\bel{eq:section-5-particle-number-variables}
Z_{\rm N} \coloneqq (T^{00},(T^{0m})_{m=1}^n), \qquad F_{\rm N}^* \coloneqq -\vNumb=(-\vNumb^0,-\vNumb^1).
\ee
The first component is again the density and the second one is the spatial flux. Here the convex density is the negative particle-number density expressed in the stress-energy variables \(Z_{\rm N}\). We write \(E=T^{00}\) and \(P\coloneqq(T^{01},\ldots,T^{0n})\), so that \(Z_{\rm N}=(E,P)\).

\begin{proposition}[Convexity of the negative particle number]
\label{lem-convexN}
Under the admissibility conditions of Definition~\ref{def-pression}, every state \(\Jbb\in\mathcal A\) is uniquely determined by \(Z_{\rm N}\), whose range is the closed cone \(\{(E,P):E\geq|P|\}\). The map
\bel{eq:section-5-particle-convex-map}
Z_{\rm N}\longmapsto-\vNumb^0
\ee
is non-positive and convex on this closed cone, has a positive-definite Hessian in the interior \(E>|P|\), and vanishes continuously on its null boundary.
\end{proposition}

\begin{proof}
\bse
\noindent{\bf 1.} In the variables \(Z_{\rm N}=(E,P)\), the density \(\mu\) is the unique root in \([0,E]\) of
\bel{eq:section-5-particle-reconstruction}
(E+\eps^{2}p(\mu))(E-\mu) -|P|^2=0.
\ee
For \(E>0\), the derivative of the left-hand side is \(\eps^{2}p'(\mu)(E-\mu) - (E+\eps^{2}p(\mu))< 0\). When \(E>|P|\), the root satisfies \(\mu>0\), and the future velocity is recovered from \((\mu+\eps^{2}p)(u^0)^2=E+\eps^{2}p\) and \((\mu+\eps^{2}p)u^0u^m=T^{0m}\). Thus \(-\vNumb^0=-\Numb(\mu)u^0\) is well-defined in the interior.

We put \(h = \mu+\eps^{2}p(\mu)\), \(c^2 = \eps^{2}p'(\mu)\), \(w = (u^1,\ldots,u^n)\), and \(\Gamma = u^0=(1+|w|^2)^{1/2}\). If \((\dot\mu,\dot w)\) is a variation of the primitive state and \(\zeta\) is the corresponding variation of \(Z_{\rm N}\), then
\bel{eq:section-5-particle-differential}
\dot E= \bigl(1+(1+c^2)|w|^2 \bigr)\dot\mu+2h\,w\cdot\dot w, \qquad \dot P=(1+c^2)\Gamma w\,\dot\mu+h\bigl(\Gamma\operatorname{Id}+ {\Gamma}^{-1}w\otimes w \bigr)\dot w.
\ee
The spatial block is the one used in the proof of Proposition~\ref{lem-convexM}; its Schur complement now gives the determinant \(h^n{\Gamma}^{n-2}[1+(1-c^2)|w|^2]>0\), so this differential is invertible in the interior of the future cone. Since \(\Numb'= \Numb/h\), differentiation gives
\bel{eq:section-5-particle-gradient}
D_E(-\Numb\Gamma) =-\frac{\Numb\Gamma}{h}, \qquad D_P(-\Numb\Gamma) = \frac{\Numb}{h}w.
\ee
\ese

\bse
\noindent{\bf 2.} We differentiate this gradient and pair it with the differential~\eqref{eq:section-5-particle-differential}. Only \(\Numb'= \Numb/h\) and \(h'=1+c^2\) enter the calculation, which gives
\bel{eq:section-5-particle-hessian}
D^2_{Z_{\rm N}}(-\vNumb^0)[\zeta,\zeta]= \frac{\Numb c^2\Gamma}{h^2}\dot\mu^2+ \frac{2\Numb c^2}{h\Gamma}\dot\mu\,w\cdot\dot w+ \Numb\Gamma|\dot w|^2-\frac{\Numb}{\Gamma}(w\cdot\dot w)^2.
\ee
When \(w\neq0\), we decompose \(\dot w= \dot w_\perp+ \lambda w\), with \(w\cdot\dot w_\perp=0\). Completing the square in the equation~\eqref{eq:section-5-particle-hessian}, we obtain
\bel{eq:section-5-particle-hessian-squares}
D^2_{Z_{\rm N}}(-\vNumb^0)[\zeta,\zeta]= \Numb\Gamma|\dot w_\perp|^2+ \frac{\Numb}{\Gamma|w|^2} \Bigl( w\cdot\dot w+ \frac{c^2|w|^2}{h}\dot\mu \Bigr) ^2 + \frac{\Numb c^2}{\Gamma h^2}\bigl(1+(1-c^2)|w|^2 \bigr)\dot\mu^2\geq0.
\ee
At \(w=0\), the continuous form of the same identity is \(D^2_{Z_{\rm N}}(-\vNumb^0)[\zeta,\zeta]= \Numb|\dot w|^2+ \Numb c^2h^{-2}\dot\mu^2\geq0\). In either case equality requires \(\dot\mu=0\) and \(\dot w=0\), hence \(\zeta=0\); the Hessian is positive definite in the interior. On the null boundary \(E=|P|\), the unique root of the equation~\eqref{eq:section-5-particle-reconstruction} is \(\mu=0\). At the origin the state is \(\Jbb=0\), while for \(E=|P|>0\) it is the unique null state
\bel{eq:section-5-particle-null-reconstruction}
J_0=-\sqrt{2E}, \qquad J_m= \sqrt{\frac{2}{E}}P_m, \qquad m=1,\ldots,n.
\ee
Moreover, \(\Mbf_{00} \, =2E\) and the bound in the equation~\eqref{eq:M00-positive} imply \(|J_0|\leq2\sqrt E\); hence \(|\vNumb^0|= \hNumb_\Jbb|J_0|\to0\) whenever an interior state approaches this boundary. The convex interior function therefore extends continuously by zero to the boundary of the convex cone \(E\geq|P|\), and the extension remains convex. Its sign follows from \(J_0\leq0\).

\noindent{\bf 3.} For a global bound on this extension, we put \(r_* \coloneqq \sup_{\mu>0}\eps^{2}p(\mu)/\mu<1\). The normalization at vacuum and \((\log(\Numb/\mu))'=-\eps^{2}p/(\mu h)\leq0\) give \(\Numb\leq\mu\). The equation~\eqref{eq:section-5-particle-reconstruction} then implies
\be
\label{eq:particle-boundary-gap}
E^2-|P|^2= \mu E-\eps^{2}p(E-\mu)\geq(1-r_*)\mu E, \qquad (\vNumb^0)^2= \frac{\Numb^2(E+\eps^{2}p)}h\leq2\mu E.
\ee
Consequently, \(0\leq\vNumb^0\leq\sqrt{2/(1-r_*)}\sqrt{E^2-|P|^2}\) on the closed cone. This global bound gives vanishing at every finite null state. Strict convexity holds only in the interior; the boundary value is zero.
\ese
\end{proof}

\subsubsection{\texorpdfstring{The selected particle-number inequality}{The selected particle-number inequality}}

The particle-production formulation imposes all \(n+1\) stress-energy equations and selects negative particle number as the convex current:
\bel{eq:section-5-particle-number-selection}
\mathcal E_\beta=0\quad(\beta=0,\ldots,n), \qquad \nabla_\alpha(-\vNumb^\alpha)\leq0.
\ee
For regular flows, the selected inequalities in \eqref{eq:section-5-mass-energy-selection} and \eqref{eq:section-5-particle-number-selection} are equalities. There are consequently two weak formulations: the dissipated-energy formulation and the particle-production formulation. Each applies both to the subclass \(\Jpar\equiv0\) and to the full system with advected momentum parallel to the symmetry orbits; these are two regimes of the same equations, not two additional admissibility choices. The role of entropy dissipation in admissibility was emphasized by Dafermos~\cite{Dafermos-1973-entropy-rate} through his entropy-rate criterion.

\section{The effective Euler system}
\label{section-5}

\subsection{Effective density and pressure}
\label{section-5-1}

\subsubsection{Frozen momentum parallel to the symmetry orbits}

We hold the normalized momentum parallel to the symmetry orbits fixed and reduce the remaining equations to an effective Euler system in one spatial variable. The reduction identifies the effective density, pressure, and particle normalization and preserves the pressure conditions that are relevant for hyperbolicity, convexity, etc.

In this section we freeze \(\Jhatpar\). Its squared Euclidean norm \(\varpi=|\Jhatpar|^2\), defined in the equation~\eqref{eq:section-6-varpi-definition}, is invariant under rotations of the frame parallel to the orbits. The momentum parallel to the symmetry orbits enters the two Euler equations for the momentum orthogonal to the symmetry orbits only through~\(\varpi\). We recall \(h(\mu)=\mu+\eps^2p(\mu)\), write \(u_\parallel \coloneqq (u^2,\ldots,u^n)\), and define the unit velocity in the direction orthogonal to the symmetry orbits by
\bse\label{eq:effective-velocity-and-rapidity}
\bel{eq:section-6-quotient-velocity}
\overline u^a = \bigl(1+|u_\parallel|^2 \bigr)^{-1/2}u^a, \qquad a=0,1.
\ee
Thus, \(\overline u\) is future-directed and unit timelike for the two-dimensional Minkowski metric, and the rapidity $v$ is given by
\bel{eq:section-6-effective-velocity-rapidity}
\tanh v = \frac{u^1}{u^0} = \frac{\overline u^1}{\overline u^0} \in(-1,1).
\ee
\ese

\subsubsection{The effective Euler system}

When the momentum parallel to the symmetry orbits vanishes, $\varpi=0$, the definitions~\eqref{eq:section-6-effective-pressure} give $\meff{0}(\mu) = \mu$ and $\eps^{2}p_{0} \, =\eps^{2}p$. For $\varpi>0$, the effective pressure changes with the prescribed momentum, and we keep this dependence explicit in the construction.
Figure~\ref{fig:effective-pressure-reduction} illustrates this change of density for $p(0)=0$, $p'(\mu)=\sqrt\mu/(4(1+\sqrt\mu))$, $\eps=1/2$, and $\varpi=0,4,16$. The particle density has the vacuum normalization~\eqref{eq:section-3-particle-density}. Increasing $\varpi$ increases $\meff{\varpi}(\mu)$ at fixed $\mu$; expressing the same pressure value in this new variable therefore lowers $p_\varpi(m)$ at fixed $m$.

\begin{figure}
\centering
\includegraphics[width=\linewidth]{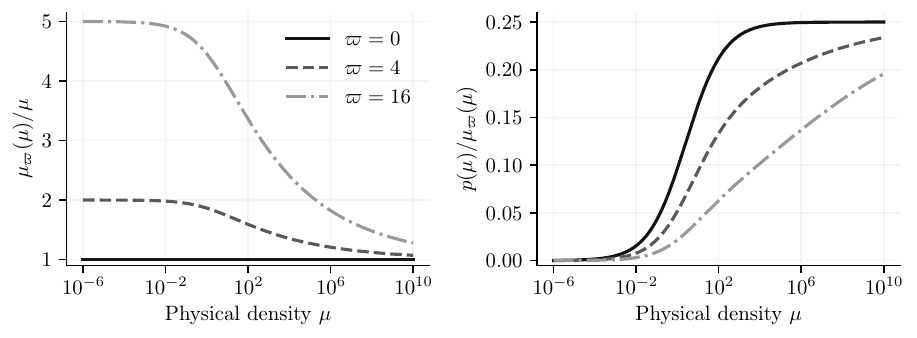}
\caption{Numerical illustration of the effective density and pressure for $p(0)=0$, $p'(\mu)=\sqrt\mu/(4(1+\sqrt\mu))$, and $\eps=1/2$. The line styles distinguish the fixed values $\varpi=|\Jhatpar|^2=0,4,16$ in both panels.  Left: the ratio $\meff{\varpi}(\mu)/\mu$ as a function of the physical density $\mu$.  Right: the ratio $p(\mu)/\meff{\varpi}(\mu)$ on the same physical-density range. The horizontal axes are logarithmic, with $10^{-6}\leq\mu\leq10^{10}$. The pressure value is preserved, $p_\varpi(\meff{\varpi}(\mu))=p(\mu)$; only its density variable changes.}
\label{fig:effective-pressure-reduction}
\end{figure}

\begin{proposition}[Effective equation of state] 
\label{prop:section-6-effective-pressure}
Fix \(\Jhatpar\in\Rbb^{n-1}\) and \(\varpi=|\Jhatpar|^2\), and recall the effective density and pressure from the equation~\eqref{eq:section-6-effective-pressure}. Define the auxiliary coefficients by
\bel{eq:section-6-effective-auxiliary-coefficients}
\sigma_\varpi\coloneqq\frac{\varpi\Numb(\mu)^2}{4h(\mu)^2}, \qquad D_\varpi \coloneqq \frac{d\meff{\varpi}}{d\mu} \, =1+\sigma_\varpi(1-\eps^{2}p'(\mu)) >0.
\ee

 \bei 

\item  \emph{Energy--momentum tensor in effective variables.} The map \(\mu\mapsto\meff{\varpi}(\mu)\) is a strictly increasing bijection of \((0,+ \infty)\) onto itself. Moreover, the first two equations of the  reduced Euler system with \(\Jhatpar\) fixed are the one-dimensional Euler equations for the state variables \((\meff{\varpi},v)\), the velocity \(\overline u\), and the pressure law \(\eps^{2}p_{\varpi}\). Stress-tensor components of the original and effective variables in directions $0,1$ coincide, namely
\bel{eq:section-6-effective-stress}
T^{ab} = T_\varpi^{ab} := \bigl(\meff{\varpi}+ \eps^{2}p_{\varpi} \bigr)\overline u^a\overline u^b+ \eps^{2}p_{\varpi}g^{ab}, \qquad a,b=0,1.
\ee

\item \emph{Particle current in effective variables.}   Let \(\Neff{\varpi}\)  denote the particle density associated with \(\eps^{2}p_{\varpi}\), canonically normalized at the effective vacuum as in the equation~\eqref{eq:section-3-particle-density}.  Then one has
\bel{eq:section-6-effective-particle-density}
\Neff{\varpi}\bigl(\meff{\varpi}(\mu) \bigr) = \Bigl(1+\frac{\varpi}{4}\Bigr)^{1/2} (1+\sigma_\varpi)^{1/2} \Numb(\mu), \qquad \Neff{\varpi}\overline u^a= \Bigl(1+\frac{\varpi}{4}\Bigr)^{1/2} \Numb u^a, \qquad a=0,1.
\ee
For the vacuum normalization below, set \(D_{\varpi,0}\coloneqq\lim_{\mu\to0}D_\varpi(\mu)=1+\varpi/4\).

\item \emph{Hyperbolicity and genuine nonlinearity.} With   the coefficients in~\eqref{eq:section-6-effective-auxiliary-coefficients},  one has
\bse\label{eq:effective-hyperbolicity-and-nonlinearity}
\bel{eq:section-6-effective-sound-speed}
\eps^{2}p_{\varpi}'\bigl(\meff{\varpi}(\mu) \bigr) = D_\varpi^{-1}\eps^{2}p'(\mu)\in(0,1).
\ee
The genuine-nonlinearity coefficient satisfies
\bel{eq:section-6-effective-genuine-nonlinearity}
G_{\rm N}[p_\varpi]\bigl(\meff{\varpi}(\mu)\bigr)
=D_\varpi^{-3}(1+\sigma_\varpi)^2G_{\rm N}[p](\mu).
\ee
\ese
Thus the effective law is hyperbolic and subluminal, and its genuine-nonlinearity coefficient has the same sign as the corresponding coefficient of the original law.
\eei
\end{proposition}

\begin{proof}
\bse
By the definitions~\eqref{eq:def-J} and \eqref{eq:def-Jhat}, we have \(\Jpar= \hNumb(\mu)\Jhatpar= \Numb(\mu)(2h(\mu))^{-1/2}\Jhatpar\), and therefore
\bel{eq:section-6-transverse-velocity}
|u_\parallel|^2= \frac{\varpi\Numb(\mu)^2}{4h(\mu)^2} \, =\sigma_\varpi.
\ee
The definition of \(\overline u\) and the identity   \(h(1+\sigma_\varpi) = \meff{\varpi}+ \eps^{2}p_{\varpi}\)  give the equation~\eqref{eq:section-6-effective-stress}. Thus the first two reduced conservation laws are those of the effective one-dimensional fluid. Since \(\Numb'/\Numb=h^{-1}\), differentiation of the equation~\eqref{eq:section-6-effective-pressure} gives the formula for \(D_\varpi\) in the equation~\eqref{eq:section-6-effective-auxiliary-coefficients}. Hence \(D_\varpi\geq1\), with equality only when \(\varpi=0\), and \(\meff{\varpi}\) is strictly increasing. The endpoint limits follow from \(\Numb(\mu) =\Ocal(\mu)\) at vacuum and \(\Numb(\mu)^2/h(\mu) =o(\mu)\) at large density, which are consequences of the assumptions in~\autoref{section-1-3}. Thus the map is onto. Differentiating \(\eps^{2}p_{\varpi}\bigl(\meff{\varpi}(\mu) \bigr) =\eps^{2}p(\mu)\) proves the equation~\eqref{eq:section-6-effective-sound-speed}.

For the particle density,
the identity   \(\sigma_\varpi'=-2 \eps^{2}p'(\mu) \sigma_\varpi/h\)  gives
\bel{eq:section-6-effective-particle-calculation}
\frac{d}{d\mu}\log\bigl(\Numb\sqrt{1+\sigma_\varpi} \bigr) = \frac{D_\varpi}{h(1+\sigma_\varpi)} \, = \frac{D_\varpi}{\meff{\varpi}+ \eps^{2}p_{\varpi}}.
\ee
After division by \(d\meff{\varpi}/d\mu=D_\varpi\), this is the defining relation for \(\Neff{\varpi}\). Since \(\Numb(\mu)/\mu\to1\), \(h(\mu)/\mu\to1\), and   \(\sigma_\varpi(\mu)\to\varpi/4\), we have \(\meff{\varpi}(\mu)/\mu\to{1+\varpi/4}\); the factor \( \bigl(1+\frac{\varpi}{4}\bigr)^{1/2} \) in the equation~\eqref{eq:section-6-effective-particle-density} is the constant which gives \(\Neff{\varpi}(m)/m\to1\) as \(m\to0\). The second identity in \eqref{eq:section-6-effective-particle-density} follows from   \(u^a= \sqrt{1+\sigma_\varpi}\,\overline u^a\). Differentiating first with respect to \(\mu\) and then with respect to \(\meff{\varpi}\) gives
\bel{eq:section-6-effective-genuine-nonlinearity-calculation}
\aligned
& D_\varpi^3 \Bigl( \eps^{2}p_{\varpi}''+2\frac{(1-\eps^{2}p_{\varpi}')\eps^{2}p_{\varpi}'}{\meff{\varpi}+ \eps^{2}p_{\varpi}} \Bigr) \bigl(\meff{\varpi}(\mu) \bigr)\\
&=(1+\sigma_\varpi)\eps^{2}p''+ \frac{2\eps^{4}(p')^{2}\sigma_\varpi(1-\eps^{2}p'(\mu))}{h}+ \frac{2\eps^{2}p'D_\varpi(D_\varpi- \eps^{2}p'(\mu))}{h(1+\sigma_\varpi)}
\\
&=(1+\sigma_\varpi) \Bigl( \eps^{2}p''+2\frac{\eps^{2}p'(1-\eps^{2}p')}{h} \Bigr),
\endaligned
\ee
which is the identity~\eqref{eq:section-6-effective-genuine-nonlinearity}.
\ese
\end{proof}

The effective reduction preserves the pressure value and the sign in the identity~\eqref{eq:section-6-effective-genuine-nonlinearity}. Mass-energy convexity requires a stronger condition. The relation to genuine nonlinearity is given by
\be
\label{eq:effective-convexity-genuine-nonlinearity-relation}
h\eps^{2}p''+2\eps^{2}p'\bigl(1-\eps\sqrt{p'}\bigr)^2=G_{\rm N}-4(\eps^{2}p')^{3/2}\bigl(1-\eps\sqrt{p'}\bigr).
\ee
Since $0<\eps^{2}p'<1$, non-negativity of the left-hand side implies $G_{\rm N}>0$. It requires a positive lower bound for $G_{\rm N}$ in terms of the sound speed, whereas genuine nonlinearity requires only $G_{\rm N}>0$. Proposition~\ref{theo:section-8-pressure-independent-convexity}, below, uses the stronger condition for mass-energy convexity; kernel positivity needs only genuine nonlinearity.

\begin{lemma}[Convexity condition for the effective pressure]
\label{lem:effective-mass-energy-convexity} 
If the pressure satisfies the \emph{mass-energy convexity condition} $G_{\rm E}\geq 0$~\eqref{eq:intro-mass-energy-convexity-condition}, then every effective pressure in the equation~\eqref{eq:section-6-effective-pressure} satisfies the same condition for every $\varpi\geq0$.
\end{lemma}

\begin{proof}
\bse
The derivatives below are evaluated at \(\mu\) or at \(m=\meff{\varpi}(\mu)\), respectively. We set \(c=\eps\sqrt{p'(\mu)}\), \(\sigma=\sigma_\varpi(\mu)\), \(D=D_\varpi(\mu)=1+\sigma(1-c^2)\), and \(d\coloneqq\sqrt D\). Differentiating \(\eps^{2}p_{\varpi}'=\eps^{2}p'/D\) and using \(\sigma'=-2\eps^{2}p'\sigma/h\) gives
\be
\begin{aligned}
& D^3\bigl((m+\eps^{2}p_{\varpi})\eps^{2}p_{\varpi}''
+2\eps^{2}p_{\varpi}'(1-\eps\sqrt{p_{\varpi}'})^2\bigr)
\\
& = (1+\sigma)^2h\eps^{2}p''+2(1+\sigma)\eps^{4}(p')^{2}\sigma(1-\eps^{2}p') +2\eps^{2}p'D(d-c)^2.
\end{aligned}
\ee
The hypothesis bounds the right-hand side from below by \(2\eps^{2}p'\Delta\), where
\be
\Delta=D(d-c)^2+(1+\sigma)c^2\sigma(1-c^2) -(1+\sigma)^2(1-c)^2.
\ee
Since \(1+\sigma=(d^2-c^2)/(1-c^2)\), direct factorization yields
\be
\Delta=\frac{c(d-1)(d-c)^2((2+c)d+c)}{(1+c)^2} +\frac{c^2(d-c)(d+c)(d-1)(d+1)}{1-c^2}\geq0.
\ee
Here \(0<c<1\) and \(d\geq1\), so every factor on the right is non-negative.
\ese
\end{proof}

\subsection{Admissibility of the effective pressure}

The algebraic reduction gives a pressure with the required hyperbolicity. We verify its asymptotics at $\mu\to 0,+\infty$ for each fixed parallel momentum, so that we can construct the kernels below for this pressure. 
\begin{proposition}[Admissibility of a frozen effective pressure]
\label{prop:fixed-effective-admissibility}
Let $p$ be an admissible equation of state in the sense of Definition~\ref{def-pression} and fix $\varpi\geq0$. Then $p_\varpi$, as a function of the effective density $m$, is admissible. Its vacuum exponent is $\gamma$, its vacuum coefficient is $\kappa_0(1+\varpi/4)^{-\gamma}$, and its large-density sound-speed limit is $k_\infty$. The constants in this assertion may depend on the fixed $\varpi$ and $\eps$.
\end{proposition}
\begin{proof}
{\bf 1.} \emph{Derivative bounds.} We write $A\coloneqq p/\mu$, $c^2=\eps^2p'$, $\sigma=\sigma_\varpi$, and $D=D_\varpi$ in this proof, with all derivatives taken in $\mu$. The proof of Lemma~\ref{lem:epsilon-pressure-admissibility}, below, gives $A',\mu A''\in L^1$ at both endpoints, with bounded $A$ and $\mu A'$. The local coefficient $B_{\rm ref}\coloneqq B/\eps^2=\mu^2p'/(\mu+\eps^2p)^2$ satisfies $B_{\rm ref}',\mu B_{\rm ref}''\in L^1$ there, by its prescribed remainders. The identity $c^2=\eps^2B_{\rm ref}(1+\eps^2A)^2$ then gives $(c^2)',\mu (c^2)''\in L^1$. For the product terms, the estimate~\eqref{eq:weighted-remainder-first-derivative}, below, bounds $\mu B_{\rm ref}'$ and $\mu A'$; for instance, $\mu|B_{\rm ref}'A'|\leq\sup|\mu A'|\,|B_{\rm ref}'|$ and $\mu(A')^2\leq\sup|\mu A'|\,|A'|$.

The identity $\sigma'=-2c^2\sigma/h$ gives $\sigma'=\Ocal(\mu^{\gamma-2})$ at vacuum. At infinity it gives $\sigma\leq C\mu^{-\theta}$ for every $0<\theta<2\eps^2k_\infty^2/(1+\eps^2k_\infty^2)$, hence $|\sigma'|\leq C\sigma/\mu$. Differentiating once gives
\be
\mu|\sigma''|\leq C\bigl(\sigma|(c^2)'|+c^2|\sigma'|+c^2\sigma/\mu\bigr).
\ee
The right-hand side is integrable at both endpoints: its last term is $\Ocal(\mu^{\gamma-2})$ at vacuum and $\Ocal(\mu^{-1-\theta})$ at infinity. Thus $\sigma',\mu\sigma''\in L^1$, and $\mu\sigma'$ is bounded. A smooth function of the bounded quantities $A,c^2,\sigma$, with denominators separated from zero, has the same two derivative-integrability properties. Its second derivative contains products controlled by the same weighted derivative estimate.

{\bf 2.} \emph{Endpoint coefficients.} We set $\zeta\coloneqq m/\mu=1+\sigma(1+\eps^2A)$ and compare the pressure coefficients. Direct substitution gives
\bse\label{eq:effective-coefficient-and-normalized-remainders}
\be
\label{eq:fixed-effective-coefficient-ratio}
B_{{\rm ref},\varpi}(m)\coloneqq\frac{m^2p_\varpi'(m)}{(m+\eps^2p_\varpi(m))^2}=B_{\rm ref}(\mu)F(\mu),\qquad F(\mu)\coloneqq\frac{\zeta^2(1+\eps^2A)^2}{D(\zeta+\eps^2A)^2}.
\ee
The functions $F,\zeta,D$ have the derivative bounds just proved, and $1\leq D,\zeta\leq C_\varpi$. At vacuum, $\zeta,D\to D_{\varpi,0}=1+\varpi/4$, so $F\to D_{\varpi,0}^{-1}$. At infinity, $\sigma\to0$, so $\zeta,D,F\to1$. The two normalized remainders are therefore
\be
\frac{B_{{\rm ref},\varpi}(m)}{\kappa_0D_{\varpi,0}^{-\gamma}m^{\gamma-1}}=(1+r_-(\mu))D_{\varpi,0}^\gamma F(\mu)\zeta^{1-\gamma},\qquad \frac{B_{{\rm ref},\varpi}(m)}{k_\infty^2/(1+\eps^2k_\infty^2)^2}=(1+r_+(\mu))F(\mu).
\ee
\ese
The normalized factors tend to one and have integrable first and weighted second derivatives in $\mu$. Products with $r_\pm$ preserve these properties, since $\mu r_\pm'$ is bounded. Passing to $m$ uses $dm=D\,d\mu$, $m\asymp_\varpi\mu$, and
\be
\frac{d^2g}{dm^2}=\frac{g''}{D^2}-\frac{g'D'}{D^3}.
\ee
The product term is integrable with weight $m$ because $\mu D'$ is bounded and $g'\in L^1$. This proves the endpoint remainder conditions in the effective density. The reduction already gives $p_\varpi(0)=0$, $0<p_\varpi'<\eps^{-2}$, and $C^3$ regularity at positive density. Since $p_\varpi'(m)=p'(\mu)/D\to0$ at vacuum, $p_\varpi$ is also $C^1$ up to zero. Moreover, $m\geq\mu$ gives $p_\varpi(m)/m\leq p(\mu)/\mu$, so the uniform strictly subluminal upper bound is preserved. At infinity, $m/\mu\to1$ and $p(\mu)/\mu\to k_\infty^2$ give $p_\varpi(m)/m\to k_\infty^2>0$, which verifies~\eqref{hyperbolic-eos-2}. The coefficient limits yield the three parameters in the statement.
\end{proof}

\subsection{Quasi-currents and source factors}

\subsubsection{Quasi-balance law and source factors}

The kernels of Sections~\ref{section-6} to~\ref{section-flux-kernels}, below, give these density--flux pairs through the Lorentz transformation~\eqref{eq:section-3-rest-frame-and-fixed-frame-components}.

\begin{proposition}[Quasi-current defined from an effective density-flux pair]
\label{prop:section-6-effective-entropy-pair}
For every \(\varpi\ge0\), let \((F^{0}_{\varpi},F^{1}_{\varpi})\) be a density-flux pair for the effective Euler system in one spatial variable with pressure \(\eps^{2}p_{\varpi}\), depending \(C^1\)-smoothly on \(\varpi\). Define, on the non-vacuum state space,
\bel{eq:section-6-effective-quasi-current}
F^{0}(\Jbb) \coloneqq F^{0}_{\varpi}\bigl(\meff{\varpi}(\mu),v \bigr), \qquad F^{1}(\Jbb) \coloneqq F^{1}_{\varpi}\bigl(\meff{\varpi}(\mu),v \bigr), \qquad \varpi=|\Jhatpar|^2.
\ee
Then \(F=(F^0, F^1)\) is a quasi-current. Its shear-induced factors are
\bel{eq:section-6-effective-transport-factors}
\Fvee_m=2\widehat J_m\frac{\del}{\del\varpi}\Bigl[ F^{1}_{\varpi}\bigl(\meff{\varpi}(\mu),v \bigr) + \frac{J_1}{J_0}F^{0}_{\varpi}\bigl(\meff{\varpi}(\mu),v \bigr)\Bigr]_{\mu,v\ {\rm fixed}}, \qquad m = 2, \ldots, n.
\ee
The derivative includes the change of effective density: with derivatives of \(F^{0}_{\varpi}(m,v)\) and \(F^{1}_{\varpi}(m,v)\) taken at fixed \((m,v)\) instead, this becomes
\be
\Fvee_m=2\widehat J_m\Bigl[
\del_\varpi F^{1}_{\varpi}+ \frac{J_1}{J_0}\del_\varpi F^{0}_{\varpi} + \frac{\Numb(\mu)^2}{4(\mu+\eps^{2}p(\mu))} \Bigl( \del_mF^{1}_{\varpi}+ \frac{J_1}{J_0}\del_mF^{0}_{\varpi} \Bigr) \Bigr] , \qquad m= \meff{\varpi}(\mu).
\ee  
\end{proposition}

\begin{proof}
\bse
The transport equations \eqref{eq:Jhat-transport-Minkowski} give
\bel{eq:section-6-varpi-transport}
\del_t\varpi-\frac{J_1}{J_0}\del_x\varpi=0, \qquad \del_x\varpi=2\sum_{m=2}^n\widehat J_m\del_x\widehat J_m.
\ee
Keeping \(\Sigma\) fixed is equivalent to keeping the original \((\mu,v)\) fixed. Lemma~\ref{lem:section-4-acoustic-block-diagonalization} cancels the parameter derivatives in the Euler system in these variables. The frozen compatibility condition then cancels the derivatives of the current. Thus, writing \(\Fcomp^a= \Fcomp^a(\mu,v,\varpi)\), the chain rule gives
\be
\del_t\Fcomp^0+ \del_x\Fcomp^1= \Bigl( \del_\varpi\Fcomp^1+ \frac{J_1}{J_0}\del_\varpi\Fcomp^0 \Bigr) _{\mu,v\ {\rm fixed}}\del_x\varpi.
\ee
Substitution of the second identity in \eqref{eq:section-6-varpi-transport} yields the equation~\eqref{eq:section-6-effective-transport-factors}.
\ese
\end{proof}

\subsubsection{Reference currents}

The effective energy, momentum orthogonal to the symmetry orbits, and particle currents give conserved currents of the original reduced system even when the momentum parallel to the symmetry orbits varies. To verify this directly, we write \(m= \meff{\varpi}(\mu)\), \(h_\varpi\coloneqq m+ \eps^{2}p_{\varpi}(m)\),  and use the rapidity \(v\) defined in the equation~\eqref{equa-rapi}. Their density--flux pairs are respectively
\bse\label{eq:reference-pairs-and-source-cancellation}
\be
\label{eq:effective-reference-pairs}
\begin{aligned}
(F^{0}_{\rm E},F^{1}_{\rm E})&=(h_\varpi\cosh^2v-\eps^{2}p,\ h_\varpi\sinh v\cosh v),\\
(F^{0}_{\rm P},F^{1}_{\rm P})&=(h_\varpi\sinh v\cosh v,\ h_\varpi\sinh^2v+\eps^{2}p),\\
(F^{0}_{\rm N},F^{1}_{\rm N})&=\Neff{\varpi}(m)(\cosh v,\sinh v),
\end{aligned}
\ee
where \(p_\varpi(m)=p(\mu)\). Since \(J_1/J_0=-\tanh v\), the combinations entering the equation~\eqref{eq:section-6-effective-transport-factors} satisfy
\be
\label{eq:effective-reference-transport-cancellation}
\aligned
F^{1}_{\rm E}-\tanh(v)F^{0}_{\rm E}&=\tanh(v)\eps^{2}p(\mu),\\
F^{1}_{\rm P}-\tanh(v)F^{0}_{\rm P}&=\eps^{2}p(\mu),\\
F^{1}_{\rm N}-\tanh(v)F^{0}_{\rm N}&=0.
\endaligned
\ee
\ese
Their derivatives with respect to \(\varpi\), at fixed original density and velocity, vanish. Thus the shear-induced factors of all three pairs vanish. The first two pairs are the original stress-energy rows. The third is \( \bigl(1+\frac{\varpi}{4}\bigr)^{1/2} (\vNumb^0,\vNumb^1)\), by the equation~\eqref{eq:section-6-effective-particle-density}; multiplication by \(D_{\varpi,0}^{-1/2}\) recovers the original particle current. All three resulting currents are conserved since \(\varpi\) is advected, as also follows from  \autoref{prop:section-4-transported-currents}. In particular, the effective vacuum normalization must be included when comparing the particle currents of the original and reduced fluids.

\subsection{Current equations in physical variables}
\label{section-entropy-density-rapidity}

\subsubsection{The effective Euler system} 

The effective reduction described so far in this section leads to an effective Euler system in one spatial variable, for which we now derive the current equation in the rest frame and describe its singularity at vacuum. This prepares the Fourier construction of Section~\ref{section-6}, below.

We consider the \((1+1)\)-dimensional relativistic Euler system with an admissible pressure~\(p(\mu)\), as in~\autoref{def-pression}. In terms of the mass-energy density \(\mu \geq 0\) and scalar velocity \(\underline\upsilon\coloneqq u^1/u^0\in(-1,1)\), the equations read
\bel{eq:rel-euler-mu-u}
\aligned
\del_t  \Bigl(  (\mu+\eps^{2}p(\mu)) \frac{\underline\upsilon^2}{1-\underline\upsilon^2} + \mu  \Bigr) 
+ \del_x  \Bigl(  (\mu+\eps^{2}p(\mu)) \frac{\underline\upsilon}{1-\underline\upsilon^2}  \Bigr)  
 & =0,
\\
\del_t  \Bigl(  (\mu+\eps^{2}p(\mu)) \frac{\underline\upsilon}{1-\underline\upsilon^2}  \Bigr) 
+ \del_x  \Bigl(  (\mu+\eps^{2}p(\mu)) \frac{\underline\upsilon^2}{1-\underline\upsilon^2} + \eps^{2}p(\mu)  \Bigr) 
 & =0.
\endaligned
\ee
We denote by $ c(\mu) = \eps\sqrt{p'(\mu)}\in(0,1) $ the sound speed, while the characteristic speeds are
\bse\label{eq:acoustic-speeds-and-subluminality}
\bel{eq:lambdapm}
\lambda_\pm(\mu,\underline\upsilon) = \frac{\underline\upsilon\pm c(\mu)}{1\pm \underline\upsilon\,c(\mu)}.
\ee
These speeds satisfy
\be
-1<\lambda_- (\mu,\underline\upsilon)<\lambda_+(\mu,\underline\upsilon)<1 \qquad \text{for all }\mu>0,\ |\underline\upsilon|<1.
\ee
\ese
We construct the fundamental density and flux kernels for this system; Sections~\ref{section-9} and~\ref{section-10}, below, give the convolution representation and convexity criteria.

\subsubsection{Weak and strong currents} 

A \textit{current density} \(F^{0}= F^{0}(\mu,\underline\upsilon)\) and its \textit{current flux} \(F^{1}= F^{1}(\mu,\underline\upsilon)\) form a pair such that every sufficiently regular solution of the system~\eqref{eq:rel-euler-mu-u} satisfies
\be
\del_t F^{0}(\mu,\underline\upsilon)+ \del_x F^{1}(\mu,\underline\upsilon) =0.
\ee
Differentiating along solutions of the system~\eqref{eq:rel-euler-mu-u} and eliminating \(F^{1}\) gives a scalar second-order linear equation for \(F^{0}\), with pressure-dependent coefficients.

At the \emph{vacuum boundary}, normalization distinguishes two classes of currents.

\bei 

\item A \textit{weak current density} is a current density \(F^{0}\) that vanishes at vacuum, namely
\be
F^{0}(0,\underline\upsilon) =0, \qquad \underline\upsilon \in (-1,1).
\ee
We also normalize $F^{1}(0,\underline\upsilon)=0$, removing the irrelevant constant current $(0,C)$. The vanishing of both components at vacuum is Lorentz invariant. The kernel below parametrizes weak currents by the vacuum derivative of their rest-frame density.

\item On the other hand, a \textit{strong current density} has a  nonzero vacuum value. This is standard terminology for singular hyperbolic equations, namely for the equation \eqref{eq:section-6-normalized-entropy-equation} derived below.

\eei

\subsubsection{The relevance of Lorentz covariance}

It is convenient to replace the velocity \(\underline\upsilon\) by the rapidity variable
\be
v = \artanh(\underline\upsilon) = \frac12\log\frac{1+\underline\upsilon}{1-\underline\upsilon}, \qquad \underline\upsilon= \tanh(v).
\ee
With \(h= \mu+\eps^{2}p(\mu)\), the Euler equations in \((\mu,v)\) become
\bel{eq:rel-euler-mu-v}
\del_t\bigl(\mu+h\sinh^2v \bigr)+ \del_x\bigl(h\sinh v\cosh v \bigr)=0,\qquad \del_t\bigl(h\sinh v\cosh v \bigr)+ \del_x\bigl(\eps^{2}p+h\sinh^2v \bigr)=0.
\ee
Here we used \(\underline\upsilon/(1-\underline\upsilon^2) = \sinh v\cosh v\) and \(\underline\upsilon^2/(1-\underline\upsilon^2) = \sinh^2v\). The rapidity variable also describes changes of inertial frame: \emph{Lorentz boosts act by translation in \(v\)}. The corresponding transformation of the current gives the scalar current equation below.

The \emph{sound-speed rapidity} is defined by
\bel{eq:def-kappa}
\kappa(\mu) \coloneqq \artanh(c(\mu)).
\ee
The  density potential, or acoustic distance, is defined by
\bel{eq:standard-acoustic-distance}
k(\mu) = \int_0^\mu \frac{c(r)}{r+\eps^{2}p(r)}\,dr.
\ee
{Thus \(k'(\mu)=c(\mu)/h(\mu)\) and \(k(0)=0\).} By Lemma~\ref{prop:section-3-current-reconstruction}, a boost of rapidity~\(\theta\) replaces the fluid rapidity by \(v-\theta\) and transforms the density--flux pair according to
\bel{eq:section-6-lorentz-current-transformation}
\begin{pmatrix}F^{\theta0}\\F^{\theta1}\end{pmatrix}(\mu,v-\theta) = \begin{pmatrix}\cosh(\theta)&-\sinh(\theta)\\-\sinh(\theta)&\cosh(\theta)\end{pmatrix}  \begin{pmatrix}F^{0}\\F^{1}\end{pmatrix}(\mu,v).
\ee
In particular, the boost \(\theta=v\) takes the fluid state to rest. We denote the two corresponding scalar components by \((\alpha_{F},\beta_{F})=(F^{v0},F^{v1})(\mu,0)\), in agreement with the equation~\eqref{eq:section-3-rest-frame-scalars}, so that
\bel{eq:section-6-rest-frame-entropy-current}
\begin{pmatrix}\alpha_{F}\\\beta_{F}\end{pmatrix}
= \begin{pmatrix}\cosh(v)&-\sinh(v)\\-\sinh(v)&\cosh(v)\end{pmatrix}\begin{pmatrix}F^{0}\\F^{1}\end{pmatrix}.
\ee
The spatial rest-frame component therefore satisfies
\bel{eq:section-6-rest-frame-flux-identity}
F^{1}-\tanh(v)F^{0}=(\cosh(v))^{-1}\beta_{F}.
\ee
The rest-frame density $\alpha_{F}$ satisfies a scalar second-order equation, while the associated first-order system reconstructs $\beta_{F}$. For the kernel current in~\autoref{section-11}, these components are denoted by $(\alpha_{\chi},\beta_{\chi})$, and the original time component is $\cosh(v)\alpha_{\chi}+\sinh(v)\beta_{\chi}$.

\subsubsection{Formulation in the rest frame}

In the fixed frame, the equation for the time component has coefficients depending on \(v\), because a boost mixes the two current components; its derivation is given in Appendix~\ref{appendix:AB-from-ab}, below. Using the Lorentz transformation above, we obtain a scalar equation for the density measured in the rest frame.

\begin{proposition}[Equations for the rest-frame current]
\label{prop:section-6-normalized-entropy-equation}
Let \((F^{0},F^{1})\) be a smooth density-flux pair for the Euler system~\eqref{eq:rel-euler-mu-u}, and let \((\alpha_{F},\beta_{F})\) be its rest-frame components defined by the equation~\eqref{eq:section-6-rest-frame-entropy-current}. Then one has
\bel{eq:section-6-normalized-first-order-system}
h(\mu)\del_\mu\alpha_{F}-\alpha_{F}= \del_v\beta_{F},\qquad h(\mu)\del_\mu\beta_{F}-c(\mu)^2\beta_{F}=c(\mu)^2\del_v\alpha_{F}.
\ee
Consequently, the rest-frame current density satisfies the scalar equation
\bel{eq:section-6-normalized-entropy-equation}
\del_\mu^2\alpha_{F}-\frac{c(\mu)^2}{h(\mu)^2}\bigl(\del_v^2\alpha_{F}-\alpha_{F}\bigr)=0.
\ee
All its coefficients depend only on \(\mu\).
\end{proposition}

\begin{proof}
The Riemann invariants are \(w_\pm\coloneqq v\pm{k}(\mu)\), where \(k\) is defined in the equation~\eqref{eq:standard-acoustic-distance}. The compatibility relations in the Riemann invariants are given in the equation~\eqref{eq:entropy-compat-riemann}, below. Differentiating the rest-frame transformation \eqref{eq:section-6-rest-frame-entropy-current} with respect to \(w_+\) and \(w_-\), and using the equation~\eqref{eq:lambdapm-riemann}, below, gives
\bel{eq:section-6-rest-frame-riemann-system}
\del_{w_+}\beta_{F}+ \frac12\alpha_{F}=c \, \Bigl( \del_{w_+}\alpha_{F}+ \frac12\beta_{F} \Bigr),\qquad \del_{w_-}\beta_{F}+ \frac12\alpha_{F}=-c\, \Bigl(\del_{w_-}\alpha_{F}+ \frac12\beta_{F} \Bigr).
\ee
Since   \(\del_{w_+}+ \del_{w_-}= \del_v\), \(\del_{w_+}-\del_{w_-}=(h/c)\del_\mu\), and \(h'=1+c^2\), the sum and difference of these identities give the system~\eqref{eq:section-6-normalized-first-order-system}. Differentiating its first equation with respect to \(\mu\), and using the second equation to replace \(\del_\mu\beta_{F}\), yields the equation~\eqref{eq:section-6-normalized-entropy-equation}.
\end{proof}

\subsubsection{Weak density kernel}

The vector kernel $\chi$ is determined by its two fluid-frame coefficients through $\chi=\chi^0e_0+\chi^1e_1=\alpha_\chi\overline u+\beta_\chi\overline u^\perp$. We first construct $\alpha_\chi=\alpha_\chi(\mu,v;s)$ as the solution of the singular boundary problem
\bel{eq:kernel-physical}
\aligned
&\del_\mu^2\alpha_{\chi}-\frac{c(\mu)^2}{h(\mu)^2}\bigl(\del_v^2\alpha_{\chi}-\alpha_{\chi}\bigr)=0,\qquad\mu>0,\\
&\alpha_{\chi}|_{\mu=0}=0,\qquad\del_\mu\alpha_{\chi}|_{\mu=0}=\delta_{v=s}.
\endaligned
\ee
Here \(s\in\Rbb\) labels the base vacuum state. The condition \(\alpha_{\chi}|_{\mu=0}=0\) is the weak normalization, while \(\del_\mu\alpha_{\chi}\) prescribes the point-supported vacuum profile centered at~\(s\). We solve this singular boundary problem on $\mu>0$ by Fourier transformation. Since Lorentz boosts translate rapidity, the kernel at any vacuum state follows from the kernel at rest:
\bel{eq:kernel-shift}
\alpha_{\chi}(\mu,v;s) = \alpha_{\chi}(\mu,v-s;0) =: \alpha_{\chi}(\mu,v-s).
\ee

We therefore construct the kernel at \(s=0\) in Section~\ref{section-6}, with the corresponding Fourier normalization.

\subsection{Formulation in Riemann invariants}
\label{section-5-4}

\subsubsection{Riemann invariants}

We express the equations in \emph{Riemann invariant} variables, defined as
\bse\label{eq:riemann-change-and-inverse}
\bel{eq:def-riemann}
w_+ = v+ k(\mu), \qquad w_- = v-k(\mu)
\ee
and, conversely,
\bel{eq:inverse-riemann}
v= \frac{w_++w_-}{2}, \qquad {k}(\mu) = \frac{w_+-w_-}{2}.
\ee
\ese
Since \(k\) is strictly increasing, the density \(\mu\) is recovered from \(w_+-w_-\). The vacuum corresponds to the diagonal \(w_+=w_-\). The characteristic speeds then become
\bel{eq:lambdapm-riemann}
\lambda_+(w_+,w_-) = \tanh(v+ \kappa), \qquad \lambda_- (w_+,w_-) = \tanh(v-\kappa).
\ee
If   \(F^{0}= F^{0}(w_+,w_-)\)  is a current density and   \(F^{1}= F^{1}(w_+,w_-)\)  the auxiliary flux, the compatibility relations take the form
\bel{eq:entropy-compat-riemann}
\del_{w_+} F^{1}= \lambda_+ \,\del_{w_+}F^{0}, \qquad \del_{w_-} F^{1}= \lambda_-\,\del_{w_-}F^{0}.
\ee
We define the coefficients of the second-order equation by
\bse
\label{eq:entropy-eq-riemann}
\bel{eq:ab-def}
a(w_+,w_-) \coloneqq \frac{\del_{w_-}\lambda_+}{\lambda_+-\lambda_-}, \qquad b(w_+,w_-) \coloneqq -\frac{\del_{w_+}\lambda_-}{\lambda_+-\lambda_-}.
\ee
Eliminating \(F^{1}\) from the compatibility relations then gives
\be
L_{w_+w_-}(F^{0}) \coloneqq F^{0}_{w_+w_-}+a(w_+,w_-)F^{0}_{w_+}+b(w_+,w_-)F^{0}_{w_-}=0.
\ee
We define the derivative of sound-speed rapidity with respect to acoustic distance by
\bel{eq:kappa-k}
\kappa_k\coloneqq\frac{d\kappa/d\mu}{dk/d\mu} = \frac{\mu+\eps^{2}p(\mu)}{c(\mu)(1-c(\mu)^2)}\,c'(\mu).
\ee
A direct computation then gives the coefficients
\bel{eq:ab-explicit}
a(w_+,w_-) = \frac{1-\kappa_k}{2}\, \frac{\cosh(v-\kappa)} {\cosh(v+ \kappa)\,\sinh(2\kappa)}, \qquad b(w_+,w_-) = -\frac{1-\kappa_k}{2}\, \frac{\cosh(v+ \kappa)} {\cosh(v-\kappa)\,\sinh(2\kappa)}.
\ee

\ese

\subsubsection{The Riemann function}

An alternative to using fundamental kernels is the Riemann function, defined by data specified along constant $w_-$ and constant $w_+$ characteristic boundaries of its support.
The diagonal   \(w_+=w_-\)  is the vacuum boundary, where the coefficients of the current equation may be singular. Away from this diagonal, the adjoint operator is
\bel{eq:def-Lstar}
L_{w_+w_-}^*(R) \coloneqq R_{w_+w_-}- (a \, R)_{w_+}- (b \, R)_{w_-}.
\ee
For   \(w_{+,0}>w_{-,0}\), we consider a compact rectangle with   \(w_+\geq w_{+,0}\) and \(w_-\leq w_{-,0}\). The Riemann function   \(\Rcal(w_+,w_-;w_{+,0},w_{-,0})\)  satisfies
\bse\label{eq:riemann-adjoint-boundary-problem}
\bel{eq:riemann-function}
L_{w_+w_-}^*\Rcal(w_+,w_-;w_{+,0},w_{-,0}) =0
\ee
and the characteristic boundary conditions
\bel{eq:riemann-goursat}
\aligned
\del_{w_-}\Rcal &= a \, \Rcal
\qquad \text{on } w_+=w_{+,0},\qquad \del_{w_+}\Rcal = b \, \Rcal
\qquad \text{on } w_-=w_{-,0},\\
\Rcal(w_{+,0},w_{-,0};w_{+,0},w_{-,0})&=1.
\endaligned
\ee
\ese
Indeed, integration of the equation and its boundary conditions gives the Volterra equation
\be
\Rcal(w_+,w_-) =1+ \int_{w_{+,0}}^{w_+} b(r,w_-)\Rcal(r,w_-)\,dr + \int_{w_{-,0}}^{w_-} a(w_+,r)\Rcal(w_+,r)\,dr.
\ee
Successive substitution converges on every such rectangle: if \(L\) is the sum of its side lengths and \(C\) bounds \(|a|+|b|\), the terms of order \(j\) are bounded by \((2CL)^j/j!\). To obtain this bound, we group the iterated integrals by the numbers of integrations in each variable and use \(\binom{j}{k}\leq2^j\). The same estimate applied to a difference proves uniqueness, and differentiation, using the regular coefficients away from vacuum, gives the equations~\eqref{eq:riemann-function}--\eqref{eq:riemann-goursat}.

The auxiliary Riemann kernel obtained by extension by zero is
\bse\label{eq:riemann-fundamental-solution}
\bel{eq:kernel-riemann}
\mathcal G^{\rm R}_{w_{+,0},w_{-,0}}(w_+,w_-) = \mathbf 1_{\{w_+\geq w_{+,0},\ w_-\leq w_{-,0}\}}\Rcal(w_+,w_-;w_{+,0},w_{-,0}).
\ee
The boundary conditions cancel the distributions on the two open characteristic sides, while the corner term is   \(-\Rcal(w_{+,0},w_{-,0};w_{+,0},w_{-,0})\delta_{w_+=w_{+,0}}\otimes\delta_{w_-=w_{-,0}}\). Consequently, in a neighborhood of the rectangle excluding its other two sides, we have
\bel{eq:entropy-krl-eq}
L_{w_+w_-}^*\mathcal G^{\rm R}_{w_{+,0},w_{-,0}} =-\delta_{w_+=w_{+,0}}\otimes\delta_{w_-=w_{-,0}}.
\ee
\ese
This regular adjoint fundamental solution differs from the weak kernel: sending its base point to \(w_{+,0}=w_{-,0}=s\) requires singular-coefficient estimates beyond the regular Goursat construction. We therefore use the singular boundary problem~\eqref{eq:kernel-physical}.

\subsection{Vacuum and large density}
\label{section-5-5}

\subsubsection{Behavior of the pressure law}

The kernel equation has distinct regimes at vacuum and large density. The coefficient associated with the pressure law is
\bse\label{eq:current-coefficient-and-endpoints}
\be
B(\mu) = \frac{\mu^2\eps^{2}p'(\mu)}{\bigl(\mu+\eps^{2}p(\mu) \bigr)^2}.
\ee
By~\autoref{def-pression} and Lemma~\ref{lem:epsilon-pressure-admissibility}, below, there exist \(\kappa_0>0\) and \(\gamma>1\), together with $k_\infty\in(0,\eps^{-1})$ such that
\be
B(\mu) \sim\eps^{2}\kappa_0\mu^{\gamma-1} \quad \text{ as }\mu\to 0, \qquad B(\mu) \to \frac{\eps^{2}k_\infty^{2}}{(1+\eps^{2}k_\infty^{2})^2} \quad \text{ as }\mu \to + \infty.
\ee
\ese
At large density, the current equation is asymptotically isothermal, and the weak vacuum branch determines both asymptotic coefficients by continuation of its value and density derivative, without an additional condition at infinity.

\subsubsection{Formal asymptotics near the vacuum}

The pressure behavior near \(\mu=0\) determines the singular vacuum boundary. We have \(p'(0)=0\), whereas the case \(p'(0)>0\) would have a different singular boundary: the coefficient \(\eps^{2}p'/(\mu+\eps^{2}p)^2\) is then of order \(\mu^{-2}\), and the integral defining \(k\) diverges at zero. In particular, the vacuum cannot then be identified with the finite diagonal \(w_+=w_-\) used here.

Under the present assumptions, we have $p(\mu)\sim(\kappa_0/\gamma)\mu^\gamma$, where $\kappa_0$ is the vacuum coefficient in the equation~\eqref{eq:intro-vacuum-B}. We recall the boundary exponent of Section~\ref{section-main-results},
\bel{eq:lambda-def}
\lambda= \frac{3-\gamma}{2(\gamma-1)}.
\ee
The formulas~\eqref{eq:ab-explicit} therefore give, locally uniformly in rapidity,
\bel{eq:ab-singular}
a(w_+,w_-) = \frac{\lambda+o(1)}{w_+-w_-}, \qquad b(w_+,w_-) =-\frac{\lambda+o(1)}{w_+-w_-}.
\ee
Thus, near vacuum, the current equation is a perturbation of the Poisson--Darboux equation, namely
\bel{eq:entropy-model}
F^{0}_{w_+w_-}+ \frac{\lambda}{w_+-w_-}(F^{0}_{w_+}-F^{0}_{w_-}) = \text{lower-order terms}.
\ee

The derivative bounds for the pressure remainders give \(c\sim\eps\sqrt{\kappa_0}\,\mu^{(\gamma-1)/2}\), \(k\sim2c/(\gamma-1)\), and \(\kappa_k\to(\gamma-1)/2\); substitution in~\eqref{eq:ab-explicit} gives~\eqref{eq:ab-singular}. We treat the exact current equation as a singular perturbation of the Poisson--Darboux model~\eqref{eq:entropy-model}, also arising for non-relativistic polytropic fluids.

\section{Weak and strong Fourier branches}
\label{section-6}

\subsection{Global Fourier branches}
\label{section-6-1}

\subsubsection{Fourier transform of the weak kernel}

We now use the \emph{Fourier transform in the rapidity variable} to turn the current equation into an ordinary differential equation in density. After stating the global results, we construct the branches selected at vacuum, estimate them up to a fixed positive density, and construct bases describing their behavior at large density.

The current equation~\eqref{eq:section-6-normalized-entropy-equation} has coefficients depending only on~\(\mu\), and its natural Fourier variable is dual to the rapidity~\(v\). We consider the kernel based at the vacuum rest state and write \(\alpha_{\chi}(\mu,v)\coloneqq\alpha_{\chi}(\mu,v,0).\) We define its Fourier transform by
\bel{eq:def-chi-tilde}
\widehat{\alpha_{\chi}}(\mu,\xi)\coloneqq\int_{\Rbb}e^{-i\xi v}\alpha_{\chi}(\mu,v)\,dv, \qquad \xi\in\Rbb.
\ee
Fourier transformation of the boundary problem~\eqref{eq:kernel-physical}, with $s=0$, gives the initial conditions
\bel{eq:chi-hat-vacuum}
\widehat{\alpha_{\chi}}(0,\xi)=0, \qquad \del_\mu\widehat{\alpha_{\chi}}(0,\xi)=1.
\ee 
By introducing the ``Fourier speed coefficient'' 
\bel{eq:def-alpha0-fourier}
\alpha_0(\mu,\xi)\coloneqq(1+\xi^2)\frac{c(\mu)^2}{h(\mu)^2}=(1+\xi^2)\frac{\eps^{2}p'(\mu)}{(\mu+\eps^{2}p(\mu))^2},
\ee
the current equation in Fourier variable becomes
\bel{eq:chi-tilde-rho}
\del_\mu^2\widehat{\alpha_{\chi}}+\alpha_0(\mu,\xi)\widehat{\alpha_{\chi}}=0.
\ee 

It is convenient to introduce the \textit{logarithmic mass-energy density}
\bel{eq:def-z-logmu}
z\coloneqq\log\mu,
\ee
so that \(\mu\to0\) corresponds to \(z\to-\infty\), while \(\mu\to+\infty\) corresponds to \(z\to+\infty\). Writing \(\widehat{\alpha_{\chi}}(\mu,\xi)=\widehat{\alpha_{\chi}}(e^z,\xi)\), the equation~\eqref{eq:chi-tilde-rho} becomes
\bel{eq:chi-tilde-z}
\del_z^2\widehat{\alpha_{\chi}}-\del_z\widehat{\alpha_{\chi}}+(1+\xi^2)B(e^z)\widehat{\alpha_{\chi}}=0,
\ee
where the Fourier coefficient $B$ is recalled from~\eqref{eq:intro-def-B}:
\bel{eq:def-B-sec9}
B(\mu)=\frac{\mu^2\eps^{2}p'(\mu)}{(\mu+\eps^{2}p(\mu))^2}.
\ee
For the analysis in this section, we introduce the function~\(X\) by
\bel{eq:def-X-canonical}
\widehat{\alpha_{\chi}}(e^z,\xi)=e^{z/2}X(z,\xi).
\ee
We define the `Fourier speed potential' by
\bel{eq:X-canonical-sec9}
V(z,\xi)\coloneqq(1+\xi^2)B(e^z) -\frac14,
\ee
so that  a substitution into the equation~\eqref{eq:chi-tilde-z} gives us 
\bel{eq:X-canonical}
\del_z^2X+V(z,\xi)X=0.
\ee
Equivalently, in terms of the original pressure coefficients, $V$ is given by 
\bel{eq:V-canonical-expanded}
V(z,\xi)=\frac{e^{2z}c(e^z)^2}{h(e^z)^2}(1+\xi^2) -\frac14, \qquad h(\mu)=\mu+\eps^{2}p(\mu), \qquad c(\mu)=\eps\sqrt{p'(\mu)}.
\ee 
Hence, the Fourier analysis is governed by the single coefficient~\(B\), on which we made asymptotic assumptions in the introduction.

\subsubsection{Branches selected at vacuum}

The vacuum normalizations determine a weak branch and an affine family of strong branches of the equation~\eqref{eq:X-canonical} on the full logarithmic-density line. Theorem~\ref{theo:section-7-fourier-branches}, below, describes these branches and their large-density decompositions, while Proposition~\ref{cor:section-7-fourier-kernel-bounds}, below, gives the global bounds needed for Fourier inversion. Their proofs combine the construction near vacuum, estimates uniform in frequency up to a fixed positive density, and the Liouville--Green estimates, which we postponed in~Appendix~\ref{appendix-olver}.

\begin{theorem}[Weak and strong  global Fourier branches]
\label{theo:section-7-fourier-branches}
Let \(p\) be an admissible equation of state in the sense of~\autoref{def-pression}, and let \(V\) be defined by the equation~\eqref{eq:X-canonical-sec9}. The following properties hold for every \(\xi\in\Rbb\).
\bei
\item There exists a unique \emph{global weak Fourier branch} \(X^{\rm weak}(\cdot,\xi)\) solving the equation~\eqref{eq:X-canonical} and satisfying the normalization
\bel{eq:section-7-weak-vacuum-normalization}
e^{-z/2}X^{\rm weak}(z,\xi)\to1\qquad\text{as }z\to-\infty.
\ee

\item For every compact interval \(I\subset\Rbb\), there exists a constant \(C_I>0\) (independent of~\(\xi\)) such that
\bel{eq:section-7-uniform-canonical-bound}
\sup_{z\in I}\Bigl(|X^{\rm weak}(z,\xi)|+|\del_zX^{\rm weak}(z,\xi)|\Bigr)\le C_I(1+|\xi|)\qquad\text{for all }\xi\in\Rbb.
\ee

\item The strong vacuum condition
\bel{eq:section-7-strong-vacuum-normalization}
e^{z/2}X(z,\xi)\to1\qquad\text{as }z\to-\infty
\ee
defines a non-empty affine family \(\mathscr S^{\rm strong}_\xi\) of \emph{global strong Fourier branches}. If \(X_1,X_2\in\mathscr S^{\rm strong}_\xi\), then \(X_1-X_2\) is a multiple of \(X^{\rm weak}\). 

\item Every representative of \(\mathscr S^{\rm strong}_\xi\), together with \(X^{\rm weak}\), is a fundamental system of solutions for~\eqref{eq:X-canonical}.

\item Consider \(\xi_\infty\coloneqq(1-\eps^{2}k_\infty^{2})/(2\eps k_\infty)\). 

\bei 

\item When \(|\xi|>\xi_\infty\) (\emph{supercritical frequencies}), both branches have a unique decomposition in the \emph{oscillatory basis} constructed in~\autoref{prop:section-7-oscillatory-large-density-branches}, below.

\item When \(|\xi|<\xi_\infty\) (\emph{subcritical frequencies}), they have a unique decomposition in the \emph{exponential basis} constructed in~\autoref{prop:section-7-exponential-large-density-branches}, below. 

\item At the two  critical frequencies \(\xi= \pm \xi_\infty\), the global branches remain well defined, although these two asymptotic bases at large~$\mu$ do not apply.
\eei
\eei
\noindent Moreover, the weak branch depends continuously on~\(\xi\), including at the threshold frequencies, and its connection coefficients\footnote{For a fixed frequency, a branch is written as \(X=a(\xi)X_1+b(\xi)X_2\) in the oscillatory or exponential basis specified above. The coefficients \(a(\xi)\) and \(b(\xi)\) are independent of logarithmic density and are given by the Wronskian ratios~\eqref{eq:fourier-wronskian-coefficients}, below.} are continuous on any compact frequency set contained in either open regime. The equation~\eqref{eq:strong-Cauchy-selection}, below, selects a strong representative with continuous real-frequency data, and its connection coefficients are continuous on the same compact frequency sets.
\end{theorem}

Figure~\ref{fig:weak-fourier-kernel} illustrates the weak Fourier branch for the pressure with $\gamma=5/3$ used in Figure~\ref{fig:kernel-pressure-transition}.  Its value at zero frequency is $\Numb(\mu)$; the figure divides by this value to compare oscillations at different densities. We use the same Fourier solutions when reconstructing the physical kernels and their boundary profiles in Figure~\ref{fig:kernel-pressure-transition}. 
The branches of Theorem~\ref{theo:section-7-fourier-branches} determine Fourier kernels after returning to the density variable $\mu$. The following proposition  gives bounds on compact density intervals and over the full density range, which will justify inverse Fourier transformation.

\begin{figure}[htbp]
\centering
\includegraphics[width=\linewidth]{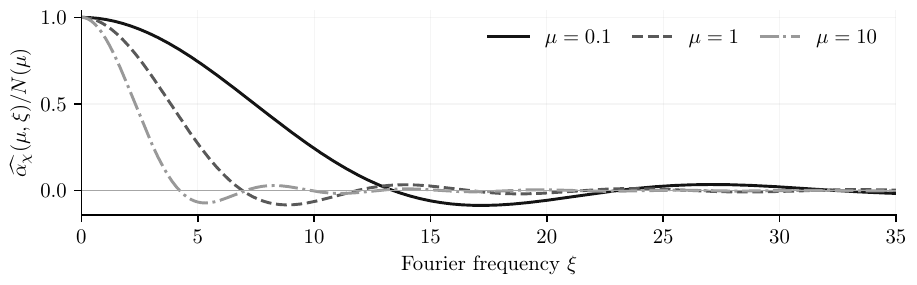}
\caption{The  normalized Fourier transform $\widehat{\alpha_\chi}(\mu,\xi)/N(\mu)$ of the weak density kernel for $p(0)=0$, $p'(\mu)=\mu^{2/3}/(4(1+\mu^{2/3}))$, $\eps=1/2$, and densities $\mu= 0.1,1,10$. The curves are computed from the Fourier equation with vacuum data $\widehat{\alpha_\chi}(0,\xi)=0$ and $\del_\mu\widehat{\alpha_\chi}(0,\xi)=1$. Only non-negative frequencies are shown, since the transform is real and even.  Division by $N(\mu)=\widehat{\alpha_\chi}(\mu,0)$ gives every curve value one at zero frequency, making their different oscillation scales visible.}
\label{fig:weak-fourier-kernel}
\end{figure}

\begin{proposition}[Global bounds for the Fourier kernels]
\label{cor:section-7-fourier-kernel-bounds}
Let \(X^{\rm weak}\) and \(\mathscr S^{\rm strong}_\xi\) be the branches in~\autoref{theo:section-7-fourier-branches}. For the weak branch,  write \(\widehat{\alpha_{\chi}^{\rm weak}}(\mu,\xi)=e^{z/2}X^{\rm weak}(z,\xi)\) where \(z=\log\mu\). \bei 

\item Then \(\widehat{\alpha_{\chi}^{\rm weak}}\) depends continuously on~\(\xi\), including at the threshold frequencies, and satisfies \(\widehat{\alpha_{\chi}^{\rm weak}}(\mu,\xi)=\mu(1+o(1))\) as \(\mu\to0\). For every compact interval \(K\subset(0,+\infty)\), there exists a constant \(C_K>0\) such that
\bel{eq:section-7-uniform-fourier-bound}
\sup_{\mu\in K}\bigl(|\widehat{\alpha_{\chi}^{\rm weak}}(\mu,\xi)|+|\mu\del_\mu\widehat{\alpha_{\chi}^{\rm weak}}(\mu,\xi)|\bigr)\le C_K(1+|\xi|)\qquad\text{for all }\xi\in\Rbb.
\ee
\item In addition,  there exists a constant \(C(\eps)>0\), independent of density and frequency, such that
\bel{eq:Fourier-global-density-polynomial}
|\widehat{\alpha_{\chi}^{\rm weak}}(\mu,\xi)|+\frac{\mu}{\sqrt{1+\xi^2}}|\del_\mu\widehat{\alpha_{\chi}^{\rm weak}}(\mu,\xi)|\le C(\eps) \mu,
\qquad \mu>0,\quad \xi\in\Rbb.
\ee
\item For the strong representative selected by the equation~\eqref{eq:strong-Cauchy-selection}, below, the associated Fourier transform \(e^{z/2}X(z,\xi)\) has \emph{at most polynomial growth} in~\(\xi\), uniformly in density  with weight \(1+\mu\); cf.~the estimate~\eqref{eq:strong-global-density-bound}, below. 

\item Finally, for every \(\auxeps>0\), there are constants \(C_\auxeps>0\) and a neighborhood of the threshold frequencies such that the Fourier kernels \(\widehat{\alpha_{\chi}}=e^{z/2}X\) associated with the weak branch and with the strong representative (as selected by~\eqref{eq:strong-Cauchy-selection}, below) satisfy
\be
|\widehat{\alpha_{\chi}}(\mu,\xi)|\le C_\auxeps\mu^{1/2+\auxeps}, \qquad \mu\ge1.
\ee
\eei
\noindent  Once established below, these polynomial bounds\footnote{Proposition~\ref{cor:kernel-Sobolev-regularity}, below, gives sharper weighted frequency decay.} 
permit to apply the inverse Fourier transform in the space of tempered distributions throughout the whole density range. \end{proposition}

Sections~\ref{section-6-2} to~\ref{section-6-4} construct the vacuum branches and prove bounds uniform in frequency from vacuum to positive density, together with the large-density estimates.  Section~\ref{section-global-fourier-continuation}, below, completes the proof of Theorem~\ref{theo:section-7-fourier-branches} and the polynomial weak-branch bounds of Proposition~\ref{cor:section-7-fourier-kernel-bounds}. Sections~\ref{section-7-5} and~\ref{section-strong-kernel}, below, then prove sharper weak-branch decay and the polynomial frequency bound~\eqref{eq:strong-global-density-bound} for the selected strong branch, respectively. Proposition~\ref{prop:section-8-selected-strong-branch}, below, gives the strong distributional vacuum value; Section~\ref{section-7} will reconstruct the weak kernel in physical variables.

\subsection{The Volterra construction}
\label{section-6-2}

\subsubsection{Vacuum weak and strong kernels}

At this stage, we first solve the vacuum problem at \emph{each fixed frequency}. Next, in Section~\ref{section-6-3} we will make our bounds uniform in frequency. For the equation~\eqref{eq:X-canonical}, with a pressure satisfying Definition~\ref{def-pression}, we introduce the \emph{vacuum scale}
\bel{eq:def-k-vacuum-sec9}
\zeta \coloneqq \sqrt{(1+\xi^2)B(\mu)} =\sqrt{1+\xi^2}\,\frac{\mu c(\mu)}{\mu+\eps^{2}p(\mu)}.
\ee
Thus the Fourier parameter enters through $ (1+\xi^2)B(\mu)=\zeta^2. $ Since \(B(\mu)\to0\) as \(\mu\to0\),
\be
V(z,\xi)\to-\frac14 \qquad\text{as }z\to-\infty.
\ee
Thus the equation is of \emph{exponential type near vacuum}. We set
\bel{eq:def-Qminus-sec9}
Q_- (z,\xi) \coloneqq \frac14- (1+ \xi^2)B(e^z), \qquad \omega_- (z,\xi) \coloneqq \sqrt{Q_- (z,\xi)},
\ee
whenever \(Q_- (z,\xi)>0\). Then we have \(X''-\omega_-^2X=0.\)

\begin{proposition}[Weak and strong kernels defined from the vacuum]
\label{thm:olver-vacuum} 
Suppose that \(B(\mu) =\Ocal(\mu^{\gamma-1})\) as \(\mu\to 0\), with \(\gamma>1\). 

\bei 

\item 
\bse For each fixed real frequency \(\xi\), there is a {unique weak solution \(y=y_{\rm weak}(\mu,\xi)\) of}
\be
y''+(1+ \xi^2)\frac{B(\mu)}{\mu^2}y=0, \qquad y_{\rm weak}(\mu,\xi)/\mu \to 1.
\ee
Writing \(H_\xi(\mu) \coloneqq (1+\xi^2) \int_0^\mu B(r)\,dr/r\), one has
\bel{eq:vacuum-olver-error-sec9}
\bigl|\frac{y_{\rm weak}(\mu,\xi)}\mu-1\bigr| +|y_{\rm weak}'(\mu,\xi) -1| \leq 2\bigl(e^{H_\xi(\mu)}-1 \bigr).
\ee
\ese

\item 
\bse
There also exists a solution \(y_{\rm strong}\) with vacuum value one. Define the error scale by
\be
 R_\gamma(\mu) \coloneqq
 \begin{cases}
 \mu^{\gamma-1},&1<\gamma<2,\\
 \mu(1+|\log\mu|),&\gamma=2,\\
 \mu,&\gamma>2.
 \end{cases}
\ee
The strong branch satisfies \(|y_{\rm strong}-1|+|\mu y_{\rm strong}'|={\Ocal}_\xi(R_\gamma(\mu))\). 
\item The corresponding canonical solutions satisfy
\bel{eq:section-7-local-weak-strong-normalizations}
X^{\rm weak} \, = \mu^{-1/2}y_{\rm weak}\sim\mu^{1/2}, \qquad X^{\rm strong} \, = \mu^{-1/2}y_{\rm strong}\sim\mu^{-1/2}.
\ee 
The canonical solutions with the strong normalization in the equation~\eqref{eq:section-7-strong-vacuum-normalization} are exactly
\bel{eq:section-7-local-strong-affine-class}
X^{\rm strong}+a(\xi)X^{\rm weak}, \qquad a(\xi)\in\Rbb.
\ee
\ese
\eei
\noindent The error constants are locally uniform in frequency (while global bounds follow below). 
\end{proposition}

\begin{proof}
\bse
{\bf 1.} \emph{The weak normalization.} We construct the weak solution from the integral equation
\bel{eq:vacuum-plus-branch-sec9}
y_{\rm weak}(\mu) = \mu- (1 + \xi^2) \int_0^\mu(\mu-r)\frac{B(r)}{r^2}y_{\rm weak}(r)\,dr.
\ee
For \(W(\mu) =y_{\rm weak}(\mu)/\mu\), the absolute value of the integral operator is bounded by the positive Volterra operator with density    \((1 + \xi^2) B(r)/r\). Its iterates are bounded by \(H_\xi(\mu)^j/j!\). The Neumann series therefore converges uniformly on every sufficiently small closed density interval, gives a unique continuous \(W\) with \(W(0) =1\), and yields
\be
|W(\mu)|\leq e^{H_\xi(\mu)}, \qquad |W(\mu) -1|\leq e^{H_\xi(\mu)}-1.
\ee
Differentiating the integral equation is legitimate since    \(B(r)/r\)  is integrable. It gives
\be
y_{\rm weak}'(\mu) =1- (1 + \xi^2) \int_0^\mu\frac{B(r)}r W(r)\,dr,
\ee
and hence the estimate~\eqref{eq:vacuum-olver-error-sec9}. Conversely, any solution with \(y/\mu\to1\) has integrable second derivative near zero. Its first derivative thus has a limit, necessarily one, and integration twice gives the same integral equation. This proves uniqueness under this vacuum normalization.

\ese

{\bf 2.} \emph{The strong normalization.} We choose \(a_\xi>0\) so small that \(y_{\rm weak}>0\) on \((0,a_\xi]\), and define
\bel{eq:vacuum-minus-branch-sec9}
y_{\rm strong}(\mu) \coloneqq y_{\rm weak}(\mu)\int_\mu^{a_\xi}\frac{dr}{y_{\rm weak}(r)^2}.
\ee
Direct differentiation proves the differential equation and the Wronskian identity \(W(y_{\rm weak},y_{\rm strong}) =-1\). Since \(y_{\rm weak}(r) =r(1+{\Ocal}_\xi(r^{\gamma-1}))\), the integrand is \(r^{-2}(1+{\Ocal}_\xi(r^{\gamma-1}))\). Integrating this expression gives \(\mu\int_\mu^{a_\xi}y_{\rm weak}(r)^{-2}\,dr=1+{\Ocal}_\xi(R_\gamma(\mu))\). Using \(y_{\rm weak}'=1+{\Ocal}_\xi(\mu^{\gamma-1})\) in the derivative of the equation~\eqref{eq:vacuum-minus-branch-sec9} gives the same bound for \(\mu y_{\rm strong}'\). The non-zero Wronskian proves that these solutions form a basis, whose vacuum values give the affine class. On each compact frequency set, the estimates permit a choice of \(a_\xi\) uniform in \(\xi\).

\end{proof}

The vacuum normalization selects a unique weak solution, whereas a strong solution can be changed by adding a multiple of that weak solution.  At \(\gamma=2\), the estimates just proved give \(y_{\rm weak}/\mu-1={\Ocal}_\xi(\mu)\) and \(y_{\rm strong}-1={\Ocal}_\xi(\mu(1+|\log\mu|))\); this is why we estimate the two remainders separately.\footnote{The logarithm comes from integrating a term of order \(r^{-1}\) in the reduction-of-order formula~\eqref{eq:vacuum-minus-branch-sec9}, above.} The integral construction above already gives both branches near vacuum.
In Appendix~\ref{appendix-pressure-endpoints}, below, we work out detailed coefficient estimates that can be used to apply the Liouville--Green method near vacuum.

\subsubsection{An example: polytropic fluids}

If \(r_-=0\), the canonical equation is the explicit polytropic model
\bel{eq:section-7-polytropic-model}
\del_z^2X+ \Bigl( (1+ \xi^2)\eps^{2}\kappa_0 e^{(\gamma-1) z}-\frac14 \Bigr) X=0.
\ee
We introduce the variables
\bel{eq:section-7-bessel-variable}
y \coloneqq \frac{2\eps\sqrt{\kappa_0(1+\xi^2)}}{\gamma-1}e^{(\gamma-1) z/2}, \qquad \nu \coloneqq \frac1{\gamma-1},
\ee
in which the equation becomes \(y^2X_{yy}+yX_y+(y^2-\nu^2)X=0\). The weak solution is therefore
\bel{eq:section-7-polytropic-weak-branch}
X_{\rm pol}^{\rm weak}(z,\xi) \coloneqq \Gamma(1+ \nu) \Bigl( \frac{\gamma-1}{\eps\sqrt{\kappa_0(1+\xi^2)}} \Bigr) ^\nu J_\nu(y).
\ee
Indeed, the expansion of \(J_\nu\) at the origin gives \(e^{-z/2}X_{\rm pol}^{\rm weak}(z,\xi)\to1\). In this reference model, the scale \(\zeta=\sqrt{(1+\xi^2)B(\mu)}\) introduced in the definition~\eqref{eq:def-k-vacuum-sec9} satisfies \(y=2\zeta/(\gamma-1)\), including at zero frequency. We use this scale to connect the vacuum estimates to the global energy estimates, with constants independent of \(\xi\).

\subsection{From vacuum to positive density}
\label{section-6-3}

\subsubsection{The uniform transition scale}

Our aim in this section is the compact-density bound~\eqref{eq:section-7-uniform-canonical-bound} in Theorem~\ref{theo:section-7-fourier-branches}. Relying on the assumed normalization at the vacuum, we now bound the weak Fourier branch at any fixed positive density, with constants \emph{independent of frequency}. We introduce a \emph{transition scale} \(\mu_\xi\), with \((1+\xi^2)B(\mu_\xi)\simeq1\), which separates the interval where the Volterra argument is applied from the interval where a \emph{positive energy} for the ordinary differential equation suffices for the bounds of interest. 
We prove that the intervening interval has uniformly bounded length in logarithmic density, and we then propagate the polynomial frequency bounds over compact density intervals. 

\begin{lemma}[Logarithmic control of the vacuum coefficient]
\label{lem:section-7-logarithmic-vacuum-control}
There exist a {vacuum threshold} \(\mu_{\rm vac}\in(0,\mu_0]\)\footnote{The subscript \({\rm vac}\) refers to vacuum: this threshold may be smaller than \(\mu_0\) and is chosen independently of the Fourier frequency \(\xi\).} and positive constants \(c_{\rm vac},C_{\rm vac}\) such that
\bel{eq:section-7-logarithmic-vacuum-control}
c_{\rm vac}\mu^{\gamma-1}\le B(\mu)\le C_{\rm vac}\mu^{\gamma-1}, \qquad \frac{\gamma-1}{2}\le\frac{\mu B_\mu(\mu)}{B(\mu)}\le\frac{3(\gamma-1)}2 \qquad (0<\mu\le\mu_{\rm vac}).
\ee

\end{lemma}

\begin{proof}
The identity \(B(\mu) = \eps^{2}\kappa_0\mu^{\gamma-1}(1+r_- (\mu))\) and the proof of   Lemma~\ref*{prop:olver-vacuum-pressure} show that \(r_- (\mu)\to0\) and \(\mu r_-'(\mu)\to0\) as \(\mu\to 0\). Hence we obtain
\bel{eq:section-7-logarithmic-derivative-B}
\frac{\mu B_\mu(\mu)}{B(\mu)} \, = (\gamma-1)+ \frac{\mu r_-'(\mu)}{1+r_- (\mu)}.
\ee
After decreasing \(\mu_{\rm vac}\), the second term has modulus at most \((\gamma-1)/2\); the two conclusions follow at once.

\end{proof}

\begin{lemma}[Frequency-dependent transition scale]
\label{lem:section-7-transition-density}
Choose \(0<\beta_\gamma\leq\min\{1/16,(\gamma-1)/8\}\). The constants below may depend on this choice. The lower level \(\beta_\gamma\) makes the Volterra integral small; the upper level \(1/2\) ensures positivity of the Fourier energy. There exist \(\Xi\ge1\) and constants \(c,C>0\) such that, for \(|\xi|\ge\Xi\), there is a unique \emph{transition scale} \(\mu_\xi\in(0,\mu_{\rm vac})\) satisfying
\bse
\bel{eq:section-7-transition-density}
(1+ \xi^2)B(\mu_\xi) = \beta_\gamma.
\ee
Moreover, one has
\bel{eq:section-7-transition-density-comparison}
c(1+ \xi^2)^{-1/(\gamma-1)}\le\mu_\xi\le C(1+ \xi^2)^{-1/(\gamma-1)}.
\ee
The upper endpoint \(\mu_\xi^+\) of the transition interval is chosen so that \(V(\log\mu_\xi^+,\xi)=1/4\). If \(\mu_\xi^+ \) is defined by \((1+ \xi^2)B(\mu_\xi^+) =1/2\), then
\bel{eq:section-7-transition-layer-width}
0\le\log\frac{\mu_\xi^+}{\mu_\xi}\le\frac2{\gamma-1}\log\frac1{2\beta_\gamma}.
\ee
\ese

\end{lemma}

\begin{proof}
Lemma~\ref{lem:section-7-logarithmic-vacuum-control} implies that \(B\) is strictly increasing on \((0,\mu_{\rm vac}]\). We choose \(\Xi\geq1\) such that \((1+ \Xi^2)B(\mu_{\rm vac})>1/2\). 
Both transition densities then exist uniquely in \((0,\mu_{\rm vac})\). Their comparison follows from the first estimate in \eqref{eq:section-7-logarithmic-vacuum-control}. The lower bound for the logarithmic derivative, integrated from \(\mu_\xi\) to \(\mu_\xi^+ \), yields the equation~\eqref{eq:section-7-transition-layer-width}.
\end{proof}

\begin{lemma}[Continuous dependence on the Fourier parameter]
\label{lem:ode-parameter-dependence}
Let \(I\subset\Rbb\) be compact and let \(K\subset\Rbb\) be a compact frequency set. If the  initial data of a solution of the equation~\eqref{eq:X-canonical} are continuous in \(\xi\in K\) at one point of \(I\), then the solution and its first \(z\)-derivative depend continuously on \((z,\xi)\in I\times K\).
\end{lemma}

\begin{proof}
We write the equation~\eqref{eq:X-canonical} as a first-order system. Its coefficient matrix is continuous and bounded on \(I\times K\); the conclusion follows from the integral formulation and Gronwall's lemma.
\end{proof}

\subsubsection{Uniform Fourier bounds at a fixed small density}

\begin{lemma}[Uniform Fourier bounds at the fixed density \(\mu_{\rm vac}\)]
\label{prop:section-7-uniform-vacuum-to-interior}
Let \(z_{\rm vac} \coloneqq \log\mu_{\rm vac}\), fixed independently of \(\xi\). There exists a constant \(C_{\rm vac}>0\) such that the weak branch satisfies
\bel{eq:section-7-uniform-vacuum-to-interior}
|X^{\rm weak}(z_{\rm vac},\xi)|+|\del_zX^{\rm weak}(z_{\rm vac},\xi)|\le C_{\rm vac}(1+|\xi|) \qquad\text{for all }\xi\in\Rbb.
\ee

\end{lemma}

\begin{proof}
\bse
\noindent{\bf 1.} We first consider \(|\xi|\ge\Xi\) and use the transition scale of Lemma~\ref{lem:section-7-transition-density}. We define
\bel{eq:section-7-rescaled-weak-branch}
Y_\xi(r) \coloneqq \mu_\xi^{-1}\widehat{\alpha_{\chi}^{\rm weak}}(\mu_\xi r,\xi), \qquad 0\le r\le1.
\ee
The vacuum normalization gives \(Y_\xi(0) =0\) and \(\del_r Y_\xi(0) =1\). Moreover, \(Y_\xi\) solves
\bel{eq:section-7-rescaled-weak-equation}
\del_r^2Y_\xi+a_\xi(r)Y_\xi=0, \qquad a_\xi(r) \coloneqq (1+ \xi^2)\frac{B(\mu_\xi r)}{r^2}.
\ee
The lower logarithmic bound in \eqref{eq:section-7-logarithmic-vacuum-control} gives \(B(\mu_\xi r)/B(\mu_\xi)\le r^{(\gamma-1)/2}\). Hence we obtain
\bel{eq:section-7-rescaled-coefficient-bound}
0\le a_\xi(r)\le\beta_\gamma r^{(\gamma-1)/2-2} \qquad (0<r\le1).
\ee
\ese
\bse
The Volterra formula is
\bel{eq:section-7-rescaled-volterra}
Y_\xi(r) = r-\int_0^r(r-t)a_\xi(t)Y_\xi(t)\,dt.
\ee
If \(M_\xi \coloneqq \sup_{0<r\le1}r^{-1}|Y_\xi(r)|\), then \eqref{eq:section-7-rescaled-coefficient-bound} and \eqref{eq:section-7-rescaled-volterra} imply
\bel{eq:section-7-rescaled-Volterra-bound}
M_\xi\le1+ \frac{2\beta_\gamma}{\gamma-1}M_\xi\le1+ \frac14M_\xi.
\ee
Thus \(M_\xi\le2\). Differentiating the integral equation~\eqref{eq:section-7-rescaled-volterra} gives \(\sup_{[0,1]}|\del_r Y_\xi|\le2\). Since \(\widehat{\alpha_{\chi}}=e^{z/2}X\), these estimates yield
\bel{eq:section-7-transition-Cauchy-bound}
|X^{\rm weak}(\log\mu_\xi,\xi)|+|\del_zX^{\rm weak}(\log\mu_\xi,\xi)|\le C\mu_\xi^{1/2}.
\ee
On the transition interval \([\log\mu_\xi,\log\mu_\xi^+]\), the potential \(V\) and its \(z\)-derivative are bounded independently of \(\xi\), by Lemma~\ref{lem:section-7-logarithmic-vacuum-control} and \eqref{eq:section-7-transition-layer-width}. The first-order system associated with \eqref{eq:X-canonical} and Gronwall's lemma therefore give the same estimate as \eqref{eq:section-7-transition-Cauchy-bound} at \(z= \log\mu_\xi^+ \).

\ese

\bse
\noindent{\bf 2.} For \(\log\mu_\xi^+ \le z\le z_{\rm vac}\), we have \(V(z,\xi)\ge1/4\).  We introduce
\bel{eq:section-7-vacuum-transfer-energy}
E(z,\xi) \coloneqq |\del_zX^{\rm weak}(z,\xi)|^2+V(z,\xi)|X^{\rm weak}(z,\xi)|^2.
\ee
The equation gives \(\del_zE=(\del_zV)|X^{\rm weak}|^2\). We use the sign of \(\del_zV\), which is positive by the equation~\eqref{eq:section-7-logarithmic-vacuum-control}. Differentiating the weighted energy gives
\be
\label{eq:section-7-vacuum-transfer-weighted-monotonicity}
\del_z\bigl(\frac{E}{V}\bigr)=-\frac{\del_zV}{V^2}|\del_zX^{\rm weak}|^2\leq0.
\ee
{At \(z=\log\mu_\xi^+\), we have \(V=1/4\) and \(E/V\leq C\mu_\xi\) by the transition estimate. Therefore, throughout \([\log\mu_\xi^+,z_{\rm vac}]\), we obtain}
\be
\label{eq:section-7-vacuum-transfer-pointwise-bound}
|X^{\rm weak}(z,\xi)|\leq C\sqrt{\mu_\xi},\qquad |\del_zX^{\rm weak}(z,\xi)|\leq C\sqrt{\mu_\xi V(z,\xi)}.
\ee
{In particular, at the fixed density \(\mu_{\rm vac}\), we have}
\bel{eq:section-7-vacuum-transfer-energy-growth}
E(z_{\rm vac},\xi)\leq C\mu_\xi V(z_{\rm vac},\xi)\leq C\mu_\xi(1+\xi^2).
\ee
{Here \(V(z_{\rm vac},\xi)\leq B(\mu_{\rm vac})(1+\xi^2)\) and \(\mu_\xi\leq\mu_{\rm vac}\). Thus~the estimate~\eqref{eq:section-7-vacuum-transfer-pointwise-bound} proves the equation~\eqref{eq:section-7-uniform-vacuum-to-interior} for \(|\xi|\geq\Xi\).} 

\noindent{{\bf 3.}} For \(|\xi|\le\Xi\), we choose \(\mu_*\in(0,\mu_{\rm vac})\) so small that the rescaled Volterra argument above, with \(\mu_\xi\) replaced by \(\mu_*\), has coefficient at most \((\gamma-1)/8\). It gives initial data bounded uniformly on the compact frequency interval at \(\mu=\mu_*\). The first-order formulation on \([\log\mu_*,z_{\rm vac}]\) then proves the same conclusion.
\ese
\end{proof}

\subsubsection{Estimates on compact density intervals}

\begin{corollary}[Uniform vacuum limits]
\label{cor:uniform-vacuum-traces}
For the weak Fourier kernel, one has
\bel{eq:uniform-vacuum-traces}
|\widehat{\alpha_{\chi}^{\rm weak}}(\mu,\xi)|\leq C\mu, \qquad |\del_\mu\widehat{\alpha_{\chi}^{\rm weak}}(\mu,\xi)|\leq C(1+|\xi|), \qquad 0<\mu\leq\mu_{\rm vac},\quad \xi\in\Rbb.
\ee
{Moreover, one has} \(\widehat{\alpha_{\chi}^{\rm weak}}(\mu,\cdot)/\mu\to1\) and \(\del_\mu\widehat{\alpha_{\chi}^{\rm weak}}(\mu,\cdot)\to1\)  in the space of tempered distributions \(\mathcal S'(\Rbb)\) as \(\mu\to 0\). 
One has
\be
\label{eq:joint-vacuum-frequency-trace}
\bigl|\frac{\widehat{\alpha_{\chi}^{\rm weak}}(\mu,\xi)}{\mu}-1\bigr| + \bigl|\del_\mu\widehat{\alpha_{\chi}^{\rm weak}}(\mu,\xi) -1\bigr| \leq C (1 + \xi^2) k(\mu)^2, \qquad (1 + \xi^2)^{1/2} k(\mu)\leq1,\quad 0<\mu\leq\mu_{\rm vac}.
\ee
Thus the two  limits converge uniformly on frequency ranges for which \((1 + \xi^2)^{1/2} k(\mu)\to0\).
\end{corollary}

\begin{proof} 
For $|\xi|\geq\Xi$, the proof of Lemma~\ref{prop:section-7-uniform-vacuum-to-interior} gives $|\widehat{\alpha_{\chi}}|\leq C\mu$ and $|\del_\mu\widehat{\alpha_{\chi}}|\leq C$ on $0<\mu\leq\mu_\xi$ by its rescaled Volterra estimate. On $[\mu_\xi,\mu_\xi^+]$, it gives $|X|+|X_z|\leq C\sqrt{\mu_\xi}$; on $[\mu_\xi^+,\mu_{\rm vac}]$, estimate~\eqref{eq:section-7-vacuum-transfer-pointwise-bound} gives $|X|\leq C\sqrt{\mu_\xi}$ and $|X_z|\leq C (1 + \xi^2)^{1/2} \sqrt{\mu_\xi}$. The identities $\widehat{\alpha_{\chi}}=\sqrt\mu X$ and $\del_\mu\widehat{\alpha_{\chi}}=\mu^{-1/2}(X_z+X/2)$, with $\mu_\xi\leq\mu$ in the latter two intervals, prove the equation~\eqref{eq:uniform-vacuum-traces}. For bounded frequencies, the vacuum Volterra estimate~\eqref{eq:vacuum-olver-error-sec9} gives these bounds near zero, and ODE propagation covers the remaining compact interval.

For the joint estimate, the quantity $H_\xi$ in the equation~\eqref{eq:vacuum-olver-error-sec9} satisfies $H_\xi(\mu)\leq C (1 + \xi^2) k(\mu)^2$, since $B(\mu)\simeq\mu^{\gamma-1}$ and $k(\mu)^2\simeq\mu^{\gamma-1}$. When $(1 + \xi^2)^{1/2}  k(\mu)\leq1$, the same estimate therefore gives the equation~\eqref{eq:joint-vacuum-frequency-trace}, using $e^{H_\xi}-1\leq C (1 + \xi^2) k(\mu)^2$.
Moreover, at each fixed frequency, the Volterra estimate gives convergence of $\widehat{\alpha_{\chi}}/\mu$ and $\del_\mu\widehat{\alpha_{\chi}}$ to one. The polynomial majorants~\eqref{eq:uniform-vacuum-traces}, independent of $\mu$, permit dominated convergence against Schwartz functions and prove the convergence in $\mathcal S'(\Rbb)$.
\end{proof}

On a compact positive-density interval, the energy controls \(|X_z|\) and \((1 + \xi^2)^{1/2} |X|\); the initial bound below uses no frequency weight and therefore introduces a frequency factor. The weighted energy~\eqref{eq:section-7-vacuum-transfer-weighted-monotonicity} avoids that loss near vacuum. These bounds suffice for Fourier inversion in \(\mathcal S'(\Rbb)\); Section~\ref{section-7-5}, below, gives \emph{sharper decay} through Bessel and Liouville--Green estimates.

\begin{lemma}[Extension of polynomial frequency bounds to compact density intervals]
\label{lem:uniform-propagation-weak-branch}
Let \(I\subset\Rbb\) be compact and let \(z_0\in I\). Assume that
\bel{eq:section-7-polynomial-Cauchy-data}
|X(z_0,\xi)|+|\del_zX(z_0,\xi)|\le C_0(1+|\xi|)^{N_0} \qquad\text{for all }\xi\in\Rbb.
\ee
Then there exists a constant \(C_I>0\) such that
\bel{eq:section-7-polynomial-compact-propagation}
\sup_{z\in I}\Bigl(|X(z,\xi)|+|\del_zX(z,\xi)|\Bigr)\le C_I(1+|\xi|)^{N_0+1}\qquad\text{for all }\xi\in\Rbb.
\ee

\end{lemma}

\begin{proof}
\bse
We define \(\beta_I\coloneqq\min_{z\in I}B(e^z)>0\). For all sufficiently large \(\xi\), \(V(z,\xi)= (1 + \xi^2) B(e^z) -1/4\geq c_I (1 + \xi^2)\) on \(I\). For \(E\coloneqq|X_z|^2+V|X|^2\), the equation gives \(E_z=V_z|X|^2\). Therefore, we have
\be
|E_z|\leq \frac{|V_z|}{V}E\leq C_I E,
\ee
since \(|V_z|/V\leq C_I\) uniformly for large \(\xi\), by the positive lower bound for \(B\) and boundedness of \(B_z\) on \(I\). Gronwall's inequality, applied in either direction from \(z_0\), yields \(E(z,\xi)\leq C_I E(z_0,\xi)\). The  prescribed initial-data bound then gives \(E(z_0,\xi)^{1/2}\leq C_I (1 + \xi^2)^{(N_0+1)/2}\),  and we obtain
\be
|X_z(z,\xi)|\leq C_I (1 + \xi^2)^{(N_0+1)/2}, \qquad |X(z,\xi)|\leq C_I (1 + \xi^2)^{N_0/2}, \qquad z\in I.
\ee
For bounded \(\xi\), the first-order formulation of the equation~\eqref{eq:X-canonical} has uniformly bounded coefficients on \(I\), so Gronwall gives the same bound. This proves the estimate~\eqref{eq:section-7-polynomial-compact-propagation}.
\ese
\end{proof}

\begin{corollary}[Uniform bounds for the weak Fourier kernel]
\label{cor:section-7-uniform-fourier-kernel}
For every compact interval \(I\subset\Rbb\), the weak branch satisfies the estimate~\eqref{eq:section-7-uniform-canonical-bound}. Consequently, for every compact interval \(K\subset(0,+ \infty)\), the Fourier kernel satisfies the estimate~\eqref{eq:section-7-uniform-fourier-bound}.
\end{corollary}

\begin{proof}
We use the energy calculation in Lemma~\ref{lem:uniform-propagation-weak-branch} on the convex hull of \(I\cup\{z_{\rm vac}\}\). For large \(|\xi|\), estimate~\eqref{eq:section-7-vacuum-transfer-energy-growth} gives \(E(z_{\rm vac},\xi)\leq C(1+\xi^2)\), and the coefficient \(V\) is bounded above and below by positive multiples of \(1+\xi^2\) on this fixed interval. Gronwall's inequality therefore yields \(|X|\leq C_I\) and \(|\del_zX|\leq C_I(1+|\xi|)\), which prove the estimate~\eqref{eq:section-7-uniform-canonical-bound}. Bounded frequencies follow by continuous dependence. The identities \(\widehat{\alpha_{\chi}}=e^{z/2}X\) and \(\mu\del_\mu\widehat{\alpha_{\chi}}=e^{z/2}(\del_zX+X/2)\) give the estimate~\eqref{eq:section-7-uniform-fourier-bound}. This establishes the compact-density assertion of Theorem~\ref{theo:section-7-fourier-branches}; its vacuum-to-large-density continuation and its asymptotic decompositions will be completed in Section~\ref{section-fourier-continuation}, below.
\end{proof}

\begin{lemma}[Fourier multiplier bounds]
\label{lem:section-7-Fourier-multiplier-bounds}
For every compact interval \(K\Subset(0,+ \infty)\) and every integer \(j\geq0\), there exist constants  \(C_{K,j}>0\) and  \(N_{K,j}\geq0\) such that
\bel{eq:section-7-Fourier-multiplier-bounds}
\sup_{\mu\in K} \Bigl( |\del_\xi^j\widehat{\alpha_{\chi}}(\mu,\xi)|+|\mu\del_\mu\del_\xi^j\widehat{\alpha_{\chi}}(\mu,\xi)| \Bigr) \leq C_{K,j}(1+|\xi|)^{N_{K,j}}
\ee
for all \(\xi\in\Rbb\). Consequently, for every fixed positive density, \(\xi\mapsto\widehat{\alpha_{\chi}}(\mu,\xi)\) is a multiplier of \(\mathcal S'(\Rbb_\xi)\), locally uniformly in \(\mu\).
\end{lemma}

\begin{proof}
\bse
We first extend the frequency to a complex variable \(z\). The vacuum Volterra series is holomorphic in \(z\), uniformly on compact subsets, since its coefficient is the polynomial \(1+z^2\), and ODE continuation preserves this holomorphic dependence. The estimates below are polynomial on the disks \(|z-\xi|\leq1/4\), uniformly for real \(\xi\). For large \(|\xi|\), we keep the real transition densities \(\mu_\xi\) and \(\mu_\xi^+ \). On such a disk, \(|1+z^2|\leq2(1+ \xi^2)\). Thus the rescaled Volterra operator has norm at most \(1/2\), by the choice of \(\beta_\gamma\), and the  transition estimate for the initial data remains \(C\sqrt{\mu_\xi}\). Propagation across the transition interval has uniformly bounded coefficients and length, so the same estimate holds at \(\mu_\xi^+ \). To avoid confusing the complex frequency with logarithmic density, we write \(t\coloneqq \log\mu\), set \(V_\xi(t) \coloneqq(1+ \xi^2)B(e^t) -1/4\), and put
\be
E=|X_t(t,z)|^2+V_\xi(t)|X(t,z)|^2, \qquad \Delta(t,z) \coloneqq(z^2-\xi^2)B(e^t).
\ee
On the remaining vacuum interval, \(V_\xi\geq1/4\) and \(V_\xi'\geq0\). Direct differentiation yields
\be
\del_t(E/V_\xi) \leq\frac{|\Delta|}{\sqrt{V_\xi}}\frac{E}{V_\xi} \leq C\sqrt{B(e^t)}\frac{E}{V_\xi}.
\ee
Here \(|z^2-\xi^2|\leq C(1+|\xi|)\) and \(V_\xi\geq(1+ \xi^2)B/2\). The integral of \(\sqrt{B(e^t)}\) from \(-\infty\) to \(\log\mu_{\rm vac}\) is finite. Hence \(E/V_\xi\leq C\mu_\xi\), uniformly on these complex disks. On a fixed compact positive-density interval, the same energy calculation gives
\be
|E'|\leq \Bigl( \frac{|V_\xi'|}{V_\xi} + \frac{|\Delta|}{\sqrt{V_\xi}} \Bigr) E\leq C_I E
\ee
for sufficiently large \(|\xi|\). It gives a polynomial bound on \(X\) and \(X_t\) throughout the interval. The bounded real-frequency range is handled on one compact complex set by the vacuum Volterra series and ordinary  ODE continuation.  Cauchy's estimates on the disks \(|z-\xi|\leq1/4\) then give the bounds for every frequency derivative \(\del_\xi^jX\) and \(\del_\xi^jX_t\). Returning to \(\widehat{\alpha_{\chi}}=e^{t/2}X\) proves the estimate~\eqref{eq:section-7-Fourier-multiplier-bounds}.
\ese
\end{proof}

\subsection{Large-density branches}
\label{section-6-4}

\subsubsection{Oscillatory large-density branches}

Thanks to the results just proven, the vacuum branch is now constructed and bounded on compact positive-density intervals. We  now build bases for its large-density behavior in the two open frequency regimes.

 Under the assumption~\eqref{eq:intro-infty-B}--\eqref{eq:intro-infty-r}, we have
\be
B(\mu) \to B_\infty \coloneqq \frac{\eps^{2}k_\infty^{2}}{(1+\eps^{2}k_\infty^{2})^2} \qquad\text{as }\mu\to+ \infty.
\ee
The potential therefore satisfies
\be
V(z,\xi) \to V_\infty(\xi) \coloneqq B_\infty(1+ \xi^2) -\frac14,
\ee
so the frequencies of the oscillatory limiting equation are characterized by
\be
V_\infty(\xi)>0 \qquad\Longleftrightarrow\qquad |\xi|>\frac{1-\eps^{2}k_\infty^{2}}{2\eps k_\infty}.
\ee
We also define
\bel{eq:def-Qplus-sec9}
Q_+(z,\xi) \coloneqq (1+ \xi^2)B(e^z) -\frac14, \qquad \omega_+(z,\xi) \coloneqq \sqrt{Q_+(z,\xi)},
\ee
and the associated  error-control function and its integral are
\bel{eq:def-Eplus-Q-sec9}
E_{\rm LG,+}(z,\xi) \coloneqq \frac{4Q_+(z,\xi)\,Q_{+,zz}(z,\xi) -5Q_{+,z}(z,\xi)^2} {16\,Q_+(z,\xi)^{5/2}}, \qquad \mathcal V_+(z,\xi) \coloneqq \int_z^{+ \infty} | E_{\rm LG,+}(r,\xi)|\,dr.
\ee

\begin{lemma}[Large-density condition on \(B\)]
\label{lem:large-density-B-condition}
 Under the assumption~\eqref{eq:intro-infty-B}--\eqref{eq:intro-infty-r},  one has
\be
\int_{\mu_1}^{+ \infty} \Bigl( |\mu^2B_{\mu\mu}(\mu)+ \mu B_\mu(\mu)| + |\mu B_\mu(\mu)|^2 \Bigr) \frac{d\mu}{\mu} <+ \infty.
\ee

\end{lemma}

\begin{proof}
We write \(B(\mu) =B_\infty(1+r_+(\mu))\) and obtain
\be
\mu B_\mu=B_\infty\,\mu r_+'(\mu), \qquad \mu^2B_{\mu\mu}+ \mu B_\mu = B_\infty\bigl(\mu^2r_+''(\mu)+ \mu r_+'(\mu) \bigr).
\ee
The weighted derivative estimate~\eqref{eq:weighted-remainder-first-derivative}, below, applied to \(r_+\), gives \(\mu r_+'\to0\) and boundedness. Thus \(\int_{\mu_1}^{+\infty}\mu|r_+'|^2\,d\mu\leq\|\mu r_+'\|_{L^\infty}\|r_+'\|_{L^1}<+\infty\); the other term follows directly from the condition~\eqref{eq:intro-infty-r}.
\end{proof}

\begin{proposition}[Oscillatory large-density branches]
\label{prop:section-7-oscillatory-large-density-branches}
Let \(\xi\in\Rbb\) be such that \(V_\infty(\xi)>0\). Then there exists a \emph{density threshold} \(z_+(\xi)\in\Rbb\) such that
\bse
\be
Q_+(z,\xi)>0 \qquad\text{for all }z\ge z_+(\xi),
\ee
and \(\mathcal V_+(z_+(\xi),\xi)<+ \infty\). Moreover, the equation~\eqref{eq:X-canonical} admits on \([z_+(\xi),+ \infty)\) two oscillatory solutions
\be
X_\infty^\pm(\cdot,\xi) = \omega_+(z,\xi)^{-1/2} \exp\!\Bigl(\pm i\int_{z_+(\xi)}^{z}\omega_+(r,\xi)\,dr\Bigr) \bigl(1+ \auxeps_+^\pm(z,\xi) \bigr),
\ee
whose relative errors satisfy
\be
|\auxeps_+^\pm(z,\xi)| \le C\Bigl(\exp\!\bigl(C\mathcal V_+(z,\xi) \bigr) -1\Bigr),
\ee
\ese
where \(C>0\) is universal.
\end{proposition}

\begin{proof}
The limit \(Q_+(z,\xi)\to V_\infty(\xi)>0\) gives the sign, and Lemma~\ref{lem:large-density-B-condition} gives \(\mathcal V_+<+\infty\). Proposition~\ref{theo:appendix-olver-oscillatory}  in the Appendix  gives the basis.
\end{proof}

\subsubsection{Exponential and threshold regimes}

The critical frequency separating exponential and oscillatory behavior was introduced in Theorem~\ref{theo:section-7-fourier-branches}; we recall its value
\bel{eq:section-7-critical-frequency}
\xi_\infty = \frac{1-\eps^{2}k_\infty^{2}}{2\eps k_\infty}.
\ee
Since \(0<\eps k_\infty<1\), we have \(V_\infty(\xi)< 0\) if and only if \(|\xi|<\xi_\infty\). In this frequency range, the large-density equation is of exponential type.

Figure~\ref{fig:fourier-regimes} illustrates the potential and its turning points for the pressure already used in Figure~\ref{fig:effective-pressure-reduction}. The level $1/4$ in $(1+\xi^2)B(e^z)$ determines the sign of $V$; the critical frequency $\xi_\infty$ instead describes its large-density limit. For this example, $B$ is increasing, so a frequency above $\xi_\infty$ encounters one turning point.   The continuation constructed in Section~\ref{section-fourier-continuation}, below, remains valid through every finite zero of \(V\), without requiring these zeros to be simple or isolated.

\begin{figure}[htbp]
\centering
\includegraphics[width=\linewidth]{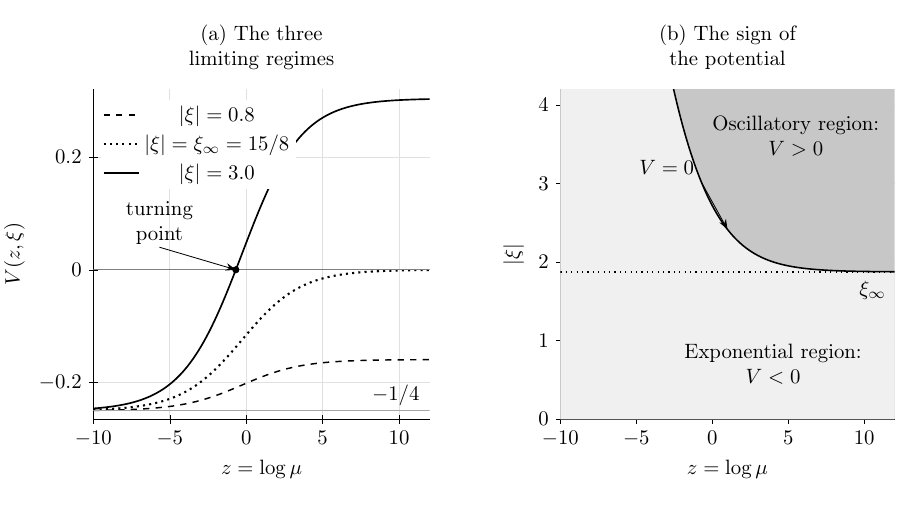}
\caption{Fourier regimes for $p'(\mu)=\sqrt\mu/(4(1+\sqrt\mu))$, $p(0)=0$, and $\eps=1/2$, for which $k_\infty=1/2$ and $\xi_\infty=15/8$. Left: the potential $V=(1+\xi^2)B(e^z)-1/4$ for a subcritical, a critical, and a supercritical frequency. The critical curve tends to zero at large density. Right: the curve $V=0$ separates the regions of exponential and oscillatory type. The horizontal dotted line is the limiting critical frequency; it is not a line of finite-density turning points.}
\label{fig:fourier-regimes}
\end{figure}

\begin{proposition}[Exponential large-density branches]
\label{prop:section-7-exponential-large-density-branches}
Let \(\xi\in\Rbb\) be such that \(|\xi|<\xi_\infty\). 
\bei 

\item Then there exists a \emph{density threshold} \(z_+(\xi)\in\Rbb\) such that
\bse
\bel{eq:section-7-Q-plus-exponential}
Q_+^{\rm exp}(z,\xi) \coloneqq \frac14- (1+ \xi^2)B(e^z)>0 \qquad\text{for all }z\ge z_+(\xi).
\ee

\item 
Define
\bel{eq:section-7-exponential-olver-variation}
\begin{aligned}
E_{\rm LG,+}^{\rm exp}(z,\xi) &\coloneqq \frac{4Q_+^{\rm exp}(z,\xi)\,\del_z^2Q_+^{\rm exp}(z,\xi) -5\bigl(\del_zQ_+^{\rm exp}(z,\xi) \bigr)^2}{16\bigl(Q_+^{\rm exp}(z,\xi) \bigr)^{5/2}},\\
\mathcal V_{+,\rm exp}(z,\xi) &\coloneqq \int_z^{+ \infty}| E_{\rm LG,+}^{\rm exp}(r,\xi)|\,dr.
\end{aligned}
\ee
The equation~\eqref{eq:X-canonical} admits on \([z_+(\xi),+ \infty)\) two linearly independent solutions
\bel{eq:section-7-exponential-large-density-basis}
X_{\infty,\rm exp}^{\pm}(z,\xi) = \bigl(Q_+^{\rm exp}(z,\xi) \bigr)^{-1/4}\exp\!\Bigl(\pm\int_{z_+(\xi)}^z\sqrt{Q_+^{\rm exp}(r,\xi)}\,dr\Bigr)\bigl(1+ \auxeps_{+,\rm exp}^{\pm}(z,\xi) \bigr),
\ee
and the relative error of the decreasing solution satisfies
\bel{eq:section-7-exponential-large-density-error}
|\auxeps_{+,\rm exp}^{-}(z,\xi)|\le C\Bigl(\exp\!\bigl(C\mathcal V_{+,\rm exp}(z,\xi) \bigr) -1\Bigr),
\ee
\ese
where \(C>0\) is universal, and \(\mathcal V_{+,\rm exp}\) is the Liouville--Green variation associated with \(Q_+^{\rm exp}\). 

\item In particular, \(\mathcal V_{+,\rm exp}(z_+(\xi),\xi)<+ \infty\). The growing solution also has relative error tending to zero. Writing    \(\tau= \int_{z_+(\xi)}^z\sqrt{Q_+^{\rm exp}(r,\xi)}\,dr\)  and \(v(\tau) = \mathcal V_{+,\rm exp}(z(\tau),\xi)\), one may choose it so that, for sufficiently large \(T\),
\be
|\auxeps_{+,\rm exp}^{+}(z(\tau),\xi)| \leq C \Bigl( v(\tau)+e^{-2(\tau-T)} + \int_T^\tau e^{-2(\tau-r)}v(r)\,dr \Bigr) , \qquad\tau\geq T.
\ee
This distinction between the two errors is necessary even though their relative limits are both zero.
\eei
\end{proposition}

\begin{proof}
The limit \(Q_+^{\rm exp}(z,\xi)\to-V_\infty(\xi)>0\) gives the sign, and Lemma~\ref{lem:large-density-B-condition} makes the Liouville--Green variation finite. Proposition~\ref{theo:appendix-olver-exponential} in the appendix gives the basis.
\end{proof}

\subsubsection{A bound valid at every frequency}

\begin{proposition}[Large-density bound across the frequency thresholds]
\label{prop:global-density-Fourier-bound}
Let \(z_0\in\Rbb\), and suppose that
\be
b_0 \coloneqq \inf_{z\geq z_0}B(e^z)>0, \qquad L_0 \coloneqq \int_{z_0}^{+ \infty}\frac{|\del_z B(e^z)|}{B(e^z)}\,dz<+ \infty.
\ee
For a solution of the equation~\eqref{eq:X-canonical}, define 
\be
H(z,\xi) \coloneqq |\del_zX(z,\xi)|^2+( (1 + \xi^2) B(e^z)+1/4)|X(z,\xi)|^2.
\ee
Then, for all real \(\xi\) and all \(z\geq z_0\), one has
\bel{eq:global-density-Fourier-energy}
H(z,\xi)^{1/2} \leq H(z_0,\xi)^{1/2} \exp \Bigl( \frac{L_0}{2}+ \frac{z-z_0}{4\sqrt{(1 + \xi^2) b_0+1/4}} \Bigr).
\ee
In particular, with \(\mu_0=e^{z_0}\) and \(\widehat{\alpha_{\chi}}(\mu,\xi) = \sqrt\mu X(\log\mu,\xi)\), one has
\bel{eq:global-density-Fourier-kernel}
\aligned
 &|\widehat{\alpha_{\chi}}(\mu,\xi)|+\frac{\mu}{(1 + \xi^2)^{1/2}}|\del_\mu\widehat{\alpha_{\chi}}(\mu,\xi)|
 \\
 &\leq C_{b_0}e^{L_0/2}\frac{\sqrt\mu}{(1 + \xi^2)^{1/2}}\,
 H(z_0,\xi)^{1/2}
  \Bigl( \frac{\mu}{\mu_0} \Bigr) ^{1/(4\sqrt{(1 + \xi^2) b_0+1/4})},
 \qquad \mu\geq\mu_0.
\endaligned
\ee
These estimates include the threshold frequencies, with constants depending only on the quantities specified above, and not on the density or the frequency. Since the exponent of \(\mu/\mu_0\) is less than \(1/2\), the right-hand side is bounded by \(C_{b_0}e^{L_0/2}\mu\mu_0^{-1/2}(1 + \xi^2)^{-1/2}H(z_0,\xi)^{1/2}\).
\end{proposition}

For the pressure laws considered here, these hypotheses hold once the reference density is sufficiently large. Indeed, writing \(B_\infty=\eps^{2}k_\infty^{2}/(1+\eps^{2}k_\infty^{2})^2\), we choose \(\mu_0\) so that \(|r_+(\mu)|\leq1/2\) for \(\mu\geq\mu_0\). Then we have
\be
B(e^z)\geq B_\infty/2, \qquad \int_{\log\mu_0}^{+ \infty}\frac{|\del_zB(e^z)|}{B(e^z)}\,dz = \int_{\mu_0}^{+ \infty}\frac{|r_+'(\mu)|}{1+r_+(\mu)}\,d\mu \leq2\int_{\mu_0}^{+ \infty}|r_+'(\mu)|\,d\mu<+ \infty.
\ee

\begin{proof}
\bse
We write \(B_z\coloneqq\del_zB(e^z)\). Since the equation is \(X''+((1 + \xi^2) B-1/4)X=0\), differentiation gives, also for complex-valued solutions,
\be
\del_zH= \operatorname{Re}(X\overline{X'})+  (1 + \xi^2) B_z|X|^2.
\ee
The inequalities \(2\sqrt{(1 + \xi^2) B+1/4}|X||X'|\leq H\) and \((1 + \xi^2) B|X|^2\leq H\) imply
\be
\del_zH\leq \Bigl( \frac{1}{2\sqrt{(1 + \xi^2) B+1/4}}+ \frac{|B_z|}{B} \Bigr) H \leq \Bigl( \frac{1}{2\sqrt{(1 + \xi^2) b_0+1/4}}+ \frac{|B_z|}{B} \Bigr) H.
\ee
Integration and taking the square root prove the equation~\eqref{eq:global-density-Fourier-energy}. Finally, \(\mu\del_\mu\widehat{\alpha_{\chi}}= \sqrt\mu(X'+X/2)\), \(|X'|\leq H^{1/2}\), and \(|X|\leq H^{1/2}/((1 + \xi^2)^{1/2}\sqrt{b_0})\) give \(|\widehat{\alpha_{\chi}}|+\mu(1 + \xi^2)^{-1/2}|\del_\mu\widehat{\alpha_{\chi}}|\leq C_{b_0}\sqrt\mu\,(1 + \xi^2)^{-1/2}H^{1/2}\). Together with the energy estimate~\eqref{eq:global-density-Fourier-energy}, this proves the equation~\eqref{eq:global-density-Fourier-kernel}.
\ese
\end{proof}

\subsubsection{Sublinear density growth near the thresholds}

The asymptotic bases just constructed describe frequencies away from the thresholds. To estimate solutions near the thresholds as well, we now prove a bound valid for either vacuum branch, and more generally for every solution whose initial data are bounded on the indicated compact frequency set.

\begin{lemma}[A bound in a neighborhood of the thresholds]
\label{lem:threshold-sublinear-bound}
Let \(0<\auxeps<1/2\). There is \(d_\auxeps>0\) such that, for solutions \(X\) of the canonical equation with \(\bigl||\xi|-\xi_\infty \bigr|\leq d_\auxeps\), one has
\be
\label{eq:threshold-sublinear-bound}
|f(\mu,\xi)|+|\mu\del_\mu f(\mu,\xi)| \leq C_\auxeps\bigl(|X(0,\xi)|+|\del_zX(0,\xi)| \bigr)\mu^{1/2+ \auxeps}, \qquad \mu\geq1,
\ee
where \(f(\mu,\xi) = \sqrt \mu X(\log \mu,\xi)\). The constant is \emph{independent} of density and of the frequency \emph{in this neighborhood}.
\end{lemma}

In particular, at either threshold the original Fourier branch grows at most as \(\mu^{1/2+ \auxeps}\) for every \(\auxeps>0\), with an \(\auxeps\)-dependent constant. If \(\omega\in C_c^\infty(\Rbb)\) is supported in the above neighborhood and the  initial data there are bounded by \(M\), Plancherel's identity gives, for every real \( \ell\),
\be
\label{eq:threshold-localized-inverse}
\bigl\|\mathcal F_\xi^{-1}(\omega f(\mu,\xi))\bigr\|_{H^{\ell}} + \bigl\|\mathcal F_\xi^{-1}(\omega \mu\del_\mu f(\mu,\xi))\bigr\|_{H^{\ell}} \leq C_{\auxeps, \ell,\omega}M \mu^{1/2+ \auxeps}.
\ee
The weak branch and the strong representative selected in the equation~\eqref{eq:strong-Cauchy-selection}, below, have such bounded  initial data on compact frequency sets. Thus the near-threshold part of Fourier inversion has sublinear large-density growth. This estimate does not select an asymptotic basis at the threshold itself.

\begin{proof}
\bse
The limiting potential \(V_\infty(\xi) =(1+ \xi^2)B_\infty-1/4\) vanishes at both thresholds. We choose \(d_\auxeps\) so that \(|V_\infty|\leq\auxeps^2/2\) on this neighborhood. Uniform convergence of \(B(e^z)\) then gives \(z_\auxeps\geq0\), independent of~\(\xi\), such that \(|V(z,\xi)|\leq\auxeps^2\) for \(z\geq z_\auxeps\). For the positive quadratic energy
\be
\label{eq:threshold-positive-energy}
E_\auxeps=|\del_zX|^2+ \auxeps^2|X|^2,
\ee
we calculate
\be
\label{eq:threshold-positive-energy-growth}
\del_zE_\auxeps =2(\auxeps^2-V)\operatorname{Re}(X\overline{\del_zX}) \leq2\auxeps E_\auxeps, \qquad E_\auxeps(z)\leq E_\auxeps(z_\auxeps)e^{2\auxeps(z-z_\auxeps)}.
\ee
The coefficients are uniformly bounded on \([0,z_\auxeps]\), so finite-interval  ODE estimates control \(E_\auxeps(z_\auxeps)\) by the squared initial norm in the estimate~\eqref{eq:threshold-sublinear-bound}. We use \(|X|\leq\auxeps^{-1}E_\auxeps^{1/2}\), \(|\del_zX|\leq E_\auxeps^{1/2}\), and \(\mu\del_\mu f= \sqrt \mu(\del_zX+X/2)\). These identities prove the bound for \(\mu\geq e^{z_\auxeps}\); enlarging the constant covers \(1\leq \mu\leq e^{z_\auxeps}\).
\ese
\end{proof}

\subsubsection{Growth rates in the isothermal density regime}

The limiting coefficient determines the large-density growth, while the vacuum exponent determines the frequency weights. The next lemma separates these two roles. For the pressure coefficient~\eqref{eq:current-coefficient-and-endpoints}, the large-density assumptions~\eqref{eq:intro-infty-B}--\eqref{eq:intro-infty-r} imply the coefficient conditions below, by Lemma~\ref{lem:large-density-B-condition}.

\begin{lemma}[Growth rates at large density]
\label{lem:isothermal-growth-frequency-regimes}
Assume that \(B(e^z)>0\) is of class \(C^1\), that \(B(e^z)\to B_\infty\in(0,1/4)\), and that \(\int_{z_0}^{+\infty}|\del_z B(e^z)|\,dz<+\infty\), where \(\mu_0=e^{z_0}\geq1\). Fix \(d_0\in\Rbb\) and \(A>0\). For every real \(\xi\), let \(X(\cdot,\xi)\) solve~\eqref{eq:X-canonical} and satisfy
\be
|X(z_0,\xi)|+(1+\xi^2)^{-1/2}|X_z(z_0,\xi)|\leq A(1+\xi^2)^{d_0/2}.
\ee
Set \(f(\mu,\xi)\coloneqq \sqrt\mu X(\log\mu,\xi)\). For every \(\delta>0\), there is a constant \(C_\delta\), independent of \(\mu\) and \(\xi\), such that
\bse
\be
\label{eq:isothermal-growth-all-frequencies}
\aligned
|f(\mu,\xi)|+\frac{\mu|f_\mu(\mu,\xi)|}{\sqrt{1+\xi^2}}
&\leq C_\delta A\sqrt\mu(1+\xi^2)^{d_0/2}
\Bigl(\frac\mu{\mu_0}\Bigr)^{\sqrt{(1/4-(1+\xi^2)B_\infty)_+}+\delta},\qquad \mu\geq\mu_0.
\endaligned
\ee
For every \(\eta>0\), the supercritical range \(|\xi|\geq\xi_\infty+\eta\), where \(\xi_\infty=\sqrt{1/(4B_\infty)-1}\), satisfies the stronger bound
\be
\label{eq:isothermal-growth-oscillatory}
|f(\mu,\xi)|+\frac{\mu|f_\mu(\mu,\xi)|}{\sqrt{1+\xi^2}}
\leq C_\eta A\sqrt\mu(1+\xi^2)^{d_0/2},\qquad \mu\geq\mu_0.
\ee
\ese
The constants depend on the coefficient bounds and variation, on \(\mu_0\), and on \(\delta\) or \(\eta\). In the fluid problem,
\be
B_\infty=\frac{\eps^2k_\infty^2}{(1+\eps^2k_\infty^2)^2},\qquad
\frac12+\sqrt{\frac14-B_\infty}=\frac1{1+\eps^2k_\infty^2}.
\ee
\end{lemma}
\begin{proof}
{\bf 1.} \emph{Frequencies in the oscillatory regime.} We first take \(|\xi|\geq\xi_\infty+\eta\). Since \(B\to B_\infty\), there are \(Z\geq z_0\) and \(c_\eta>0\) such that \(V(z,\xi)\geq c_\eta(1+\xi^2)\) for \(z\geq Z\) throughout this frequency range. The positive energy
\be
E=|X_z|^2+V|X|^2,\qquad E_z=(1+\xi^2)\del_zB(e^z)|X|^2
\ee
therefore satisfies \(|E_z|\leq C_\eta|\del_zB(e^z)|E\). Its integrated bound, together with the weighted Cauchy estimate~\eqref{eq:connection-frequency-uniform}, below, on \([z_0,Z]\), gives
\be
|X(z,\xi)|+(1+\xi^2)^{-1/2}|X_z(z,\xi)|\leq C_\eta A(1+\xi^2)^{d_0/2},\qquad z\geq z_0.
\ee
The identities \(f=\sqrt\mu X\) and \(\mu f_\mu=\sqrt\mu(X_z+X/2)\) prove~\eqref{eq:isothermal-growth-oscillatory}.

{\bf 2.} \emph{A bounded interval of frequencies.} It remains to prove~\eqref{eq:isothermal-growth-all-frequencies} on a bounded interval \(|\xi|\leq R\), since the preceding argument covers all sufficiently large frequencies. In this proof we write \(\kappa_\xi\coloneqq \sqrt{(-V_\infty(\xi))_+}+\delta\). We choose \(Z\) independently of \(\xi\) in this interval so that \(|V-V_\infty|\leq\delta^2/2\) for \(z\geq Z\). Since \(V_\infty+\kappa_\xi^2\geq\delta^2\), the coefficient \(V+2\kappa_\xi^2\) is at least \(\kappa_\xi^2\). The energy
\be
E_\delta\coloneqq|X_z|^2+(V+2\kappa_\xi^2)|X|^2
\ee
satisfies
\be
(E_\delta)_z=4\kappa_\xi^2\operatorname{Re}(X\overline{X_z})+(1+\xi^2)\del_zB(e^z)|X|^2
\leq\bigl(2\kappa_\xi+C_{R,\delta}|\del_zB(e^z)|\bigr)E_\delta.
\ee
Here we used \(E_\delta\geq |X_z|^2+\kappa_\xi^2|X|^2\). Integration and the compact-interval Cauchy bound give \(|X|+|X_z|\leq C_\delta A(1+\xi^2)^{d_0/2}e^{\kappa_\xi(z-z_0)}\). This proves~\eqref{eq:isothermal-growth-all-frequencies} on the bounded frequency interval. Together with the first step, it proves the estimate for all real frequencies.
\end{proof}

\section{Global Fourier estimates}
\label{section-fourier-continuation}

\subsection{Continuation and matching}
\label{section-global-fourier-continuation}

\subsubsection{Continuation of the vacuum branches}
At this stage, the weak and strong branches have been constructed near vacuum, and the compact-density bound~\eqref{eq:section-7-uniform-canonical-bound} of Theorem~\ref{theo:section-7-fourier-branches} has been proved. We have also constructed the large-density bases for subcritical and supercritical frequencies, together with estimates near the critical frequencies. In this section, we complete the theorem by continuing the vacuum branches through all finite densities, expressing them in these bases, and proving the stated frequency continuity. The polynomial bounds of Proposition~\ref{cor:section-7-fourier-kernel-bounds} will be completed for the weak branch here and for the selected strong branch in Section~\ref{section-strong-kernel}, below.

By definition, a \emph{turning point} is a point \(z_t=z_t(\xi)\) such that
\be
V(z_t,\xi) =0, \qquad\text{equivalently}\qquad (1+ \xi^2)B(e^{z_t}) = \frac14.
\ee
At a simple zero, the sign changes from exponential to oscillatory type, or conversely; at a multiple zero, the sign need not change. The standard WKB expansions are not uniform there, since the local frequency vanishes and the approximation error terms contain inverse powers of \(Q_\pm=|V|\). Near a simple turning point, the local model is of Airy type. Turning points do not obstruct the global  ODE continuation of either vacuum branch; the energy estimate \eqref{eq:global-density-Fourier-energy} remains valid without a simplicity assumption on the zeros.

\subsubsection{Connection coefficients}

Standard ODE theory extends the weak vacuum branch and each strong representative uniquely to \emph{global solutions} of the equation~\eqref{eq:X-canonical}; the Wronskian formulas below give their asymptotic coefficients. To make the \emph{connection through threshold frequencies} explicit, we choose one finite reference point \(z_*\). We denote by \(C(z,\xi)\) and \(S(z,\xi)\) the solutions with  initial data \((1,0)\) and \((0,1)\) at \(z_*\), respectively, and put
\be
\label{eq:Fourier-Cauchy-connection-matrix}
 \mathcal U(z,z_*,\xi) \coloneqq
 \begin{pmatrix}C(z,\xi)&S(z,\xi)\\ C_z(z,\xi)&S_z(z,\xi)\end{pmatrix},
 \qquad
 \begin{pmatrix}X(z,\xi)\\X_z(z,\xi)\end{pmatrix}
 = \mathcal U(z,z_*,\xi)
 \begin{pmatrix}X(z_*,\xi)\\X_z(z_*,\xi)\end{pmatrix}.
\ee
The fundamental matrix of the first-order equation is defined by 
\be
\label{eq:Fourier-Cauchy-matrix-equation}
 \del_z\mathcal U=
 \begin{pmatrix}0&1\\ \tfrac14- (1+ \xi^2)B(e^z)&0\end{pmatrix}\mathcal U,
 \qquad \mathcal U(z_*,z_*,\xi) = \mathrm{Id},
 \qquad \det\mathcal U=1.
\ee
Its trace is zero, which proves the determinant identity. Picard iteration of this matrix equation on any finite interval converges locally uniformly for complex \(\xi\): on a compact frequency set its \(j\)-th term is bounded by \(M^j|z-z_*|^j/j!\). Thus \(C,S\), and their \(z\)-derivatives are \emph{entire even functions} of \(\xi\). In particular, this basis is \emph{non-singular} at both threshold frequencies and at every turning point.

\subsubsection{Matching on finite density intervals}

The  fundamental matrix has a uniform bound in the frequency-weighted norm. For a compact interval \(I=[a,b]\) of logarithmic density, we put 
\(D_\xi\coloneqq\operatorname{diag}((1 + \xi^2)^{1/2},1)\) and \(\mathcal U_\xi\coloneqq D_\xi\mathcal U D_\xi^{-1}\).  We assume that \(b_0\leq B(e^z)\leq b_1\) on \(I\), and  put \(L_I\coloneqq\int_I|\del_z B(e^z)|/B(e^z)\,dz\). The energy calculation above, with an absolute value on \(\del_zH\), gives
\be
\label{eq:connection-two-sided-energy}
|\del_z\log H| \leq \frac{1}{2\sqrt{(1 + \xi^2) b_0+1/4}}+ \frac{|\del_zB(e^z)|}{B(e^z)}
\ee
for every nonzero solution. Since \(H\) is uniformly equivalent to \((1 + \xi^2) |X|^2+|X_z|^2\), integration in either direction proves
\be
\label{eq:connection-frequency-uniform}
\|\mathcal U_\xi(z,s,\xi)\|+ \|\mathcal U_\xi(z,s,\xi)^{-1}\| \leq C_I, \qquad z,s\in I,\quad\xi\in\Rbb.
\ee
The Euclidean operator norm is used throughout. The constant depends only on \(b_0,b_1,L_I\), and \(|I|\); matching value and derivative at finite density therefore preserves the frequency weight.

\subsubsection{Continuous selection of the strong branch}

\bei
\item {\emph{Definition and non-vanishing denominator.}} The fundamental basis just constructed also allows us to choose a strong branch \emph{continuously for all real frequencies.} We write \(a(\xi) \coloneqq X^{\rm weak}(z_*,\xi)\), \(b(\xi) \coloneqq X_z^{\rm weak}(z_*,\xi)\), and define
\be
\label{eq:strong-Cauchy-selection}
X^{\rm strong}(z,\xi) \coloneqq \frac{b(\xi)C(z,\xi) -a(\xi)S(z,\xi)}{a(\xi)^2+b(\xi)^2}, \qquad W(X^{\rm weak},X^{\rm strong}) =-1.
\ee
The weak branch is real for real \(\xi\), and its  initial-data vector cannot vanish, by uniqueness and its nonzero vacuum normalization. The denominator is therefore strictly positive, and at \(z_*\) the Wronskian equals \(-1\), hence has this value at every density. 

\item {\emph{Vacuum normalization.}} A strong representative  provided by  \autoref{thm:olver-vacuum} has the same Wronskian with the weak branch, as follows from their differentiated vacuum asymptotics. Their difference consequently has zero Wronskian with the weak branch and is a multiple of it. Thus the selection~\eqref{eq:strong-Cauchy-selection} has the required strong vacuum normalization. 

\item {\emph{Continuity and frequency bounds.}} Its  initial data are continuous in real frequency, including both thresholds; on each compact frequency interval their denominator has a positive minimum. This choice depends on \(z_*\), as expected for the affine strong family. The quantitative bound \eqref{eq:strong-compact-frequency-bound}, below, proved from the weak Bessel estimates in~\autoref{section-7-5}, also gives polynomial growth at large real frequency. It involves no division by the value of the weak branch alone and so remains valid at its zeros.

The matrix matches both Cauchy components at the reference point. For real frequencies and \(z\geq z_*\), the energy estimate~\eqref{eq:global-density-Fourier-energy} controls the square root of the energy \(H\) of the continued solution, including at the thresholds. Fourier inversion uses this continued weak branch; the large-density bases describe it in their two open regimes.
\eei

Figure~\ref{fig:fourier-matching} summarizes this construction. The vacuum branches are continued by their value and derivative, and the large-density bases then describe the continued solutions in the two open frequency regimes. At a critical frequency, the bound from Lemma~\ref{lem:threshold-sublinear-bound} replaces those asymptotic descriptions.

\begin{figure}[!htb]
\centering
\begin{tabular}{c|c|c|c}
  \hline
  Near vacuum & \multicolumn{3}{c}{\rule[-3pt]{0pt}{14pt}All frequencies: $X^{\rm weak}\sim\mu^{1/2}$, $X^{\rm strong}\sim\mu^{-1/2}$} \\
  \hline
  Finite densities & \multicolumn{3}{c}{\rule[-12pt]{0pt}{30pt}Continuation through $\begin{pmatrix}X(z)\\X_z(z)\end{pmatrix}=\mathcal U(z,z_*,\xi)\begin{pmatrix}X(z_*)\\X_z(z_*)\end{pmatrix}$}
  \\
  \hline
  & \rule[12pt]{0pt}{0pt} subcritical: $|\xi|<\xi_\infty$ & near-critical: $|\xi|\simeq\xi_\infty$ & supercritical: $|\xi|>\xi_\infty$
  \\
  Large densities & Exponential basis & Threshold bound & Oscillatory basis
  \\
  & \rule[-7pt]{0pt}{0pt} $X^+_{\infty,\rm exp},X^-_{\infty,\rm exp}$ & $|\sqrt{\mu} X|\leq C_{\delta_1}\mu^{1/2+\delta_1}$ & $X^+_\infty,X^-_\infty$
  \\
  \hline
\end{tabular}
\caption{The same vacuum-normalized branches are continued through all finite densities and described at large density according to frequency. The critical bound holds for $\mu\geq1$ and $0<\delta_1<1/2$, near either threshold.}
\label{fig:fourier-matching}
\end{figure}

\subsubsection{Wronskian formulas} On \(I=[z_0,z_1]\subset\Rbb\),  we choose \(X_1(\cdot,\xi)\) and \(X_2(\cdot,\xi)\)  as independent solutions of the equation~\eqref{eq:X-canonical}. For any other solution \(X\), solving the two  initial-value equations at one point gives
\be
\label{eq:fourier-wronskian-representation}
X(z,\xi) =a(\xi)X_1(z,\xi)+b(\xi)X_2(z,\xi), \qquad z\in I,
\ee
where the coefficients are
\be
\label{eq:fourier-wronskian-coefficients}
a(\xi) = \frac{W(X,X_2)}{W(X_1,X_2)}, \qquad b(\xi) =-\frac{W(X,X_1)}{W(X_1,X_2)}.
\ee
The Wronskians are constant in \(z\), since the equation has no first-derivative term. Independence makes the denominator nonzero, and uniqueness extends the identity throughout~\(I\).

\subsubsection{\texorpdfstring{Continuous dependence of the asymptotic coefficients}{Continuous dependence of the asymptotic coefficients}}

We verify the continuity of these coefficients within each open large-density regime. We fix a compact frequency interval \(\mathcal I\) contained either in \(\{V_\infty>0\}\) or in \(\{V_\infty<0\}\). Since \(V(z,\xi)\to V_\infty(\xi)\) uniformly on \(\mathcal I\), we can choose \(z_+\) independently of \(\xi\in\mathcal I\) such that
\bel{eq:connection-positive-potential}
Q_\xi(z)\coloneqq|V(z,\xi)|\geq q_*>0,\qquad z\geq z_+,\quad\xi\in\mathcal I.
\ee
We use the same origin \(z_+\) for every phase \(\tau_\xi(z)\coloneqq\int_{z_+}^z\sqrt{Q_\xi(r)}\,dr\). For a positive \(C^2\) coefficient \(Q\) on an interval, we define the \emph{error-control function} by
\bel{eq:appendix-olver-error-control-density}
\mathcal E_Q(z) \coloneqq \frac{4Q(z)\del_z^2Q(z) -5\bigl(\del_zQ(z) \bigr)^2}{16Q(z)^{5/2}}.
\ee
With \(Q=Q_\xi\), the large-density conditions~\eqref{eq:intro-infty-B}--\eqref{eq:intro-infty-r} give
\be
\label{eq:connection-uniform-LG-tail}
|\mathcal E_{Q_\xi}(z)|\leq C_{\mathcal I}\bigl(|B_{zz}(e^z)|+|B_z(e^z)|^2\bigr), \qquad \sup_{\xi\in\mathcal I}\int_Z^{+\infty}|\mathcal E_{Q_\xi}(z)|\,dz\to0\quad(Z\to+\infty).
\ee
Here $B_z(e^z)=\del_z(B(e^z))$ and $B_{zz}(e^z)=\del_z^2(B(e^z))$. Lemma~\ref{lem:large-density-B-condition} makes the majorant integrable and  gives the uniform derivative tails in the parameter argument in Appendix~\ref{appendix-olver}, with parameter $\xi\in\mathcal I$. This gives continuous bases and first logarithmic-density derivatives, with uniform remainder decay.

In the oscillatory regime, we take the conjugate solutions \(X_\infty^+,X_\infty^-\) constructed from these terminal equations. Their Wronskian is \(-2i\), by Proposition~\ref{theo:appendix-olver-oscillatory}  in the Appendix. In the exponential regime, the terminal equation gives the decreasing solution \(D_\xi=Q_\xi^{-1/4}e^{-\tau_\xi}(1+\auxeps_\xi^-)\). The uniform remainder bound permits \(Z\geq z_+\) independent of \(\xi\in\mathcal I\) such that \(D_\xi\ne0\) on \([Z,+\infty)\). We define the growing solution by
\be
\label{eq:connection-continuous-growing-solution}
G_\xi(z)\coloneqq2D_\xi(z)\int_Z^zD_\xi(r)^{-2}\,dr, \qquad W(G_\xi,D_\xi)=-2.
\ee
Its Cauchy data at $Z$ are $(0,2/D_\xi(Z))$. Proposition~\ref{theo:appendix-olver-exponential} in the appendix, and the parameter argument in Appendix~\ref{appendix-olver} give the uniform differentiated asymptotic formula $G_\xi=Q_\xi^{-1/4}e^{\tau_\xi}(1+o(1))$ and continuity through finite densities.

For either basis, the Wronskian denominator is a fixed nonzero constant. For each global branch continuous in frequency, evaluating the ratios following~\eqref{eq:fourier-wronskian-representation} at a fixed finite density gives coefficients continuous and bounded on \(\mathcal I\), with matching Cauchy data. The argument uses the positive lower bound for \(Q_\xi\) on a compact frequency interval separated from the thresholds. At a threshold, \(V_\infty(\xi)=0\), so this lower bound is lost and the preceding asymptotic estimates are no longer uniform. The same global branches nevertheless remain continuous in frequency: the Cauchy matrix~\eqref{eq:Fourier-Cauchy-connection-matrix} continues their value and derivative through every finite density, including at the thresholds. Their large-density growth is then controlled by~\eqref{eq:threshold-sublinear-bound}, uniformly in a sufficiently small frequency neighborhood of either threshold. Thus we change the asymptotic description and the growth estimate near the thresholds, while continuing the same weak and strong branches throughout.

\subsubsection{Proof of the Fourier statements}

\begin{proof}[Proof of Theorem~\ref{theo:section-7-fourier-branches}]
\noindent{{\bf 1.}} For each fixed real frequency, Proposition~\ref{thm:olver-vacuum} constructs the unique solution $y_{\rm weak}$ with $y_{\rm weak}(\mu,\xi)/\mu\to1$. Under the transformation~\eqref{eq:def-X-canonical}, this is the weak normalization~\eqref{eq:section-7-weak-vacuum-normalization}. The coefficient $V$ is continuous on every finite logarithmic-density interval, so the associated first-order system extends the solution uniquely through every finite density, including all zeros of $V$. Any other global solution with the same vacuum normalization agrees with it near vacuum, by Proposition~\ref{thm:olver-vacuum}, and hence everywhere by uniqueness for this first-order system. This proves the first assertion, while Corollary~\ref{cor:section-7-uniform-fourier-kernel} proves the bound~\eqref{eq:section-7-uniform-canonical-bound} in the second assertion.

\noindent{{\bf 2.}} Proposition~\ref{thm:olver-vacuum} also  gives a strong solution with $y_{\rm strong}(\mu,\xi)\to1$ and $W(y_{\rm weak},y_{\rm strong})=-1$. The canonical transformation preserves this Wronskian, since $\del_z=\mu\del_\mu$ and each solution is multiplied by $\mu^{-1/2}$. The transformed pair is therefore independent and extends globally by the same first-order system. Every solution is a linear combination of this pair; its strong normalization requires the coefficient of the strong vector to be one, and leaves the weak coefficient arbitrary. Consequently, the normalized strong solutions form the affine class in the third assertion, and every member forms a fundamental system with $X^{\rm weak}$.

\noindent{{\bf 3.}} The vacuum Volterra series and its first density derivative converge locally uniformly on compact frequency sets at a fixed sufficiently small positive density. Their Cauchy data are continuous in frequency, and Lemma~\ref{lem:ode-parameter-dependence} preserves this continuity at every finite logarithmic density. In particular, the weak branch is continuous across both threshold frequencies. For the selected strong branch, we put $a=X^{\rm weak}(z_*,\xi)$ and $b=\del_zX^{\rm weak}(z_*,\xi)$. The pair $(a,b)$ cannot vanish, since the weak branch is nonzero. The Cauchy vector $(b,-a)/(a^2+b^2)$ in~\eqref{eq:strong-Cauchy-selection} is therefore continuous for every real frequency and has Wronskian $-1$ with the weak vector. Its solution differs from the strong vector just constructed by a multiple of the weak vector, so it has the required strong vacuum normalization as well as continuous real-frequency data.

\noindent{{\bf 4.}} Finally, $B(e^z)\to B_\infty=\eps^2k_\infty^2/(1+\eps^2k_\infty^2)^2$ gives $V_\infty(\xi)=B_\infty(1+\xi^2)-1/4$, whose zeros are $\pm\xi_\infty$. Propositions~\ref{prop:section-7-oscillatory-large-density-branches} and~\ref{prop:section-7-exponential-large-density-branches} provide fundamental pairs in the two open regimes. The Wronskian representation~\eqref{eq:fourier-wronskian-representation} therefore gives a unique decomposition of either global branch in the corresponding pair. On each compact frequency set within an open regime, the construction preceding this proof chooses a common phase origin and yields bases and first derivatives continuous in frequency. Their Wronskians are the fixed nonzero constants $-2i$ and $-2$, respectively. Evaluation of the Wronskian ratios at a fixed finite density then proves continuity of the connection coefficients for the weak branch and the selected strong branch. At the thresholds, the Cauchy matrix~\eqref{eq:Fourier-Cauchy-connection-matrix} remains nonsingular, which proves the remaining continuation assertion without requiring an asymptotic basis there.
\end{proof}

For the weak-branch assertion of Proposition~\ref{cor:section-7-fourier-kernel-bounds} and its global estimate~\eqref{eq:Fourier-global-density-polynomial}, Corollary~\ref{cor:uniform-vacuum-traces} gives \(|\widehat{\alpha_{\chi}}|\leq C\mu\) and \(|\del_\mu\widehat{\alpha_{\chi}}|\leq C (1 + \xi^2)^{1/2}\) on \(0<\mu\leq\mu_{\rm vac}\). On each fixed positive-density interval, the weighted energy argument in Corollary~\ref{cor:section-7-uniform-fourier-kernel} gives \(|X|\leq C\) and \(|X_z|\leq C (1 + \xi^2)^{1/2}\), uniformly in frequency. We choose \(\mu_0\) so that the hypotheses of~\autoref{prop:global-density-Fourier-bound} hold on \([\mu_0,+\infty)\). At \(z_0=\log\mu_0\), the preceding compact-density bounds imply \(H(z_0,\xi)^{1/2}\leq C (1 + \xi^2)^{1/2}\). Estimate~\eqref{eq:global-density-Fourier-kernel} therefore proves the same linear density bound on this half-line. The positive lower density on the intervening compact interval gives the same bound, uniformly for all \(\mu>0\) and real frequencies. The strong-branch assertion of this proposition is completed in Section~\ref{section-strong-kernel}, below.

\subsection{The weak branch}
\label{section-7-5}

\subsubsection{The leading vacuum equation}

Theorem~\ref{theo:section-7-fourier-branches} and Proposition~\ref{thm:olver-vacuum} are now proved, together with the weak-branch bounds of Proposition~\ref{cor:section-7-fourier-kernel-bounds}. We now prove \emph{sharper estimates} for the weak branch, describing its decay as the Fourier frequency increases and its dependence on density, including near vacuum and at large density. We first compare the equation near vacuum with a Bessel equation and estimate their difference by a weighted Volterra integral equation; we then extend the resulting bounds to all positive densities and deduce Sobolev regularity. In Section~\ref{section-strong-kernel}, below, we will complete the strong-branch bounds of Proposition~\ref{cor:section-7-fourier-kernel-bounds} and establish the corresponding distributional vacuum trace. We write $\mu(y)\coloneqq k^{-1}(y)$ for the inverse acoustic map and put \(y=k(\mu)\) and \(a_k(y)\coloneqq k''(\mu)/k'(\mu)^2\).
We introduce the \emph{weighted Fourier solution}
\be
\aligned
&\Psi(y,\xi) \coloneqq \sqrt{k'(\mu(y))}\,\widehat{\alpha_{\chi}}(\mu(y),\xi),\\
&\nu=\lambda+\tfrac12>0,
\endaligned
\ee
for which the Fourier equation takes the form
\be
\Psi_{yy}+ \Bigl( (1 + \xi^2) -\frac{\nu^2-1/4}{y^2} \Bigr) \Psi = \mathcal R(y)\Psi, \qquad \mathcal R(y) \coloneqq \tfrac12a_k'(y)+ \tfrac14a_k(y)^2 -\frac{\lambda(\lambda+1)}{y^2}.
\ee
Indeed, \(\widehat{\alpha_{\chi,yy}}+a_k\widehat{\alpha_{\chi,y}}+ (1 + \xi^2) \widehat{\alpha_{\chi}}=0\), and multiplication by \(\sqrt{k'}\) removes the first derivative. For the exact power profile \(k(\mu) =C\mu^{(\gamma-1)/2}\), we have \(a_k=-2\lambda/y\), so \(\mathcal R=0\). With $J_\nu$ and $Y_\nu$ denoting the Bessel functions of the first and second kind, the independent solutions are $\sqrt y J_\nu(y (1+\xi^2)^{1/2})$ and $\sqrt y Y_\nu(y (1+\xi^2)^{1/2})$. We use this singular Bessel operator to estimate the perturbation at vacuum. The non-relativistic Bessel construction was developed in~\cite[Section~2]{CL2-2003}; here the Fourier parameter \(\xi\) includes the relativistic zero-order term.

\begin{lemma}[Weighted integrability of the Bessel perturbation]
\label{lem:Bessel-weighted-remainder}
Assume the vacuum expansion~\eqref{eq:intro-vacuum-B}, with \(\gamma>1\), and the weighted derivative bounds~\eqref{eq:intro-vacuum-r} on \(r_-\). Then the following integral is finite for every sufficiently small \(Y>0\) and tends to zero as \(Y\to0\):
\be
\int_0^Y y|\mathcal R(y)|\,dy<+ \infty.
\ee
\end{lemma}

\begin{proof}
\bse
We write \(t\coloneqq\log\mu\), \(q\coloneqq(\gamma-1)/2>0\), and \(b(t) \coloneqq\sqrt{B(e^t)}\). These definitions give
\be
k_t=b, \qquad \frac{b_t}{b} \, =q+e(t), \qquad e(t) \coloneqq\frac{\del_t r_- (e^t)}{2(1+r_- (e^t))}.
\ee
The vacuum conditions~\eqref{eq:intro-vacuum-B}--\eqref{eq:intro-vacuum-r} imply \(e,e_t\in L^1(( -\infty,t_0])\) and \(e(t)\to0\). We put \(L\coloneqq k/b\) and \(w\coloneqq L-1/q\). The integral definition of \(k\) gives \(L\to1/q\), and
\be
w_t=- (q+e)w-e/q.
\ee
On a sufficiently negative half-line, \(q+e\geq q/2\). Variation of constants from \(-\infty\) therefore gives \(\|w\|_{L^1}\leq C\|e\|_{L^1}\), and the equation also gives \(w_t\in L^1\). These identities give the following expression for the singular coefficient:
\be
{y\,a_k(y)} =(q+e-1)L=-2\lambda+d(y), \qquad d\coloneqq(q-1)w+eL.
\ee
Thus \(d,d_t\in L^1(dt)\), \(d\to0\), and \(d\) is bounded. Since \(dy/y=dt/L\), with \(L\) bounded above and below, substitution into the definition of \(\mathcal R\) yields
\be
y^2\mathcal R= \tfrac12 y d_y- (\lambda+ \tfrac12)d+ \tfrac14d^2.
\ee
Consequently \(\int_0^Y y|\mathcal R|\,dy\) is controlled by the tails of \(\int(|d_t|+|d|+d^2)\,dt\), which are finite and tend to zero.
\ese
\end{proof}

\begin{proposition}[One-term Fourier expansion with an integral remainder]
\label{prop:Bessel-one-term-bound}
Fix an admissible equation of state \(p\) in the sense of Definition~\ref{def-pression}. Let \(\nu=\lambda+1/2>0\)
and set
\be
M_\xi(y) \coloneqq \frac{y^{\nu+1/2}}{(1+ y (1+\xi^2)^{1/2})^{\nu+1/2}}, \qquad {\mathcal H_\xi}(Y) \coloneqq \int_0^Y\frac{t|\mathcal R(t)|}{1+ (1+\xi^2)^{1/2} t}\,dt.
\ee
Writing \(q=(\gamma-1)/2\) and \(K\coloneqq\lim_{\mu\to 0}k(\mu)/\mu^q\), the weak vacuum normalization gives \(C_*=2^\nu\Gamma(\nu+1)\sqrt q\,K^{-\nu}\), where \(K=\eps\sqrt{\kappa_0}/q\).  Define
\be
\Psi_0(y,\xi) \coloneqq C_* (1+\xi^2)^{-\nu/2}\sqrt y J_\nu(y (1+\xi^2)^{1/2}).
\ee
The weak solution satisfies the following bound for all \(y>0\), with \(C\) depending only on \(\nu\):
\be
|\Psi(y,\xi) -\Psi_0(y,\xi)| \leq C C_* M_\xi(y)\bigl(e^{C \mathcal H_\xi(y)}-1 \bigr).
\ee
For each fixed \(Y\), one has \(\mathcal H_\xi(Y)\to0\) as \(|\xi|\to\infty\). The derivative satisfies, for all \(y>0\),
\be
\label{eq:Bessel-derivative-remainder}
|\del_y(\Psi-\Psi_0)(y,\xi)|\leq C C_*((1+\xi^2)^{1/2}+y^{-1})M_\xi(y) \bigl(e^{C \mathcal H_\xi(y)}-1 \bigr).
\ee
For a fixed \(Y_0>0\), there is also a constant \(C_0\) such that
\be
\label{eq:Bessel-global-error-growth}
\mathcal H_\xi(y)\leq {\mathcal H_\xi}(Y_0)+ \frac{C_0(1+y)}{(1+\xi^2)^{1/2}}, \qquad y\geq Y_0.
\ee
The constant \(C_0\) may depend on the pressure law and \(\eps\), while the constants in the two Bessel remainder bounds depend only on \(\nu\) after the factor \(C_*\) has been displayed. All are independent of frequency, and \(\mathcal H_\xi(Y_0)\to0\) as \(|\xi|\to\infty\). On an unbounded density range, smallness of the high-frequency error requires control of \(k(\mu)/(1+\xi^2)^{1/2}\), with \(y=k(\mu)\).
\end{proposition}

\begin{proof}
\bse
{\bf 1.} \emph{The Volterra equation and its derivative.} We set \(f_\xi(y) \coloneqq \sqrt y J_\nu(y (1+\xi^2)^{1/2})\) and \(g_\xi(y) \coloneqq \sqrt y Y_\nu(y (1+\xi^2)^{1/2})\). Their Wronskian is \(2/\pi\). The Green kernel for zero incoming weak correction is
\be
K_\xi(y,t) \coloneqq \frac\pi2\bigl(f_\xi(t)g_\xi(y) -g_\xi(t)f_\xi(y) \bigr), \qquad 0<t\leq y.
\ee
The small-argument and large-argument Bessel bounds give
\be
|\Psi_0(y,\xi)|\leq C C_*M_\xi(y), \qquad |K_\xi(y,t)|\frac{M_\xi(t)}{M_\xi(y)} \leq\frac{Ct}{1+  t (1+\xi^2)^{1/2}}.
\ee
For clarity, when \(y (1+\xi^2)^{1/2}\leq1\), these follow from \(J_\nu(x) =\Ocal(x^\nu)\), \(Y_\nu(x) =\Ocal(x^{-\nu})\). When \((1+\xi^2)^{1/2} t\geq1\), both Bessel functions are \(\Ocal(x^{-1/2})\), giving \(C/(1+\xi^2)^{1/2}\). In the remaining case \(t (1+\xi^2)^{1/2}\leq1\leq y (1+\xi^2)^{1/2}\), the mixed small- and large-argument bounds give \(Ct\). These bounds follow from the Bessel asymptotics, recalled for instance in \cite[equations~(2.5)--(2.8)]{CL2-2003}. The Bessel normalization in that reference differs from ours, changing the constants but not these inequalities. Variation of constants gives
\be
\Psi(y,\xi) =\Psi_0(y,\xi)+ \int_0^yK_\xi(y,t)\mathcal R(t)\Psi(t,\xi)\,dt.
\ee
After division by \(M_\xi(y)\), successive approximation is dominated by the iterates of the positive integral  operator with weight \(Ct|\mathcal R(t)|/(1+ t (1+\xi^2)^{1/2})\). Their factorial bounds prove existence with the weak normalization and the remainder estimate on every finite interval, with the same Bessel constants. Uniqueness of the weak Fourier branch identifies this solution with the one already constructed. To differentiate, we use \(K_\xi(y,y) =0\) and
\be
|\del_yK_\xi(y,t)|\frac{M_\xi(t)}{M_\xi(y)} \leq\frac{C( (1+\xi^2)^{1/2} +y^{-1})t}{1+ t (1+\xi^2)^{1/2}}.
\ee
The bound follows from the same three Bessel regimes: differentiating at \(y (1+\xi^2)^{1/2}\leq1\) costs \(C/y\), and at \(y (1+\xi^2)^{1/2}\geq1\) costs \(C(1+\xi^2)^{1/2}\). It provides an integrable majorant at \(t=0\) for each positive \(y\), so differentiation under the Volterra integral is justified. Inserting the bound for \(\Psi/M_\xi\) proves the equation~\eqref{eq:Bessel-derivative-remainder}. Dominated convergence with majorant \(t|\mathcal R(t)|\) proves \(\mathcal H_\xi(Y)\to0\).

\ese

{\bf 2.} \emph{Large acoustic distances.} To control large \(y\), we write \(z= \log\mu\) and \(b(z) = \sqrt{B(e^z)}\). On a sufficiently large half-line, \(b\) is bounded above and below, \(b_z\) is bounded, and \(b_z,b_{zz}\in L^1(dz)\), by~\eqref{eq:intro-infty-B}--\eqref{eq:intro-infty-r} and Lemma~\ref{lem:large-density-B-condition}. Since \(dy=b\,dz\), the identities
\be
{a_k}= \frac{b_z}{b^2}-\frac1b, \qquad {(a_k)_y}= \frac{b_{zz}}{b^3}-\frac{2b_z^2}{b^4}+ \frac{b_z}{b^3}
\ee
show that \(a_k\) is bounded and \((a_k)_y\in L^1(dy)\) there. The definition of \(\mathcal R\) therefore gives \(\int_{Y_0}^y|\mathcal R(t)|\,dt\leq C_0(1+y)\), after including the intervening compact interval. Using \(t/(1+ t (1+\xi^2)^{1/2})\leq(1+\xi^2)^{-1/2}\) proves the equation~\eqref{eq:Bessel-global-error-growth}.

\end{proof}

\noindent\emph{Remark on the construction.} The Volterra equation used in this proof also gives existence of the weak branch under Definition~\ref{def-pression}. We could therefore begin with this Bessel construction and then obtain the strong branch by reduction of order. We have chosen to present the elementary vacuum construction first, since it requires only \(B(\mu)=\Ocal(\mu^{\gamma-1})\), with \(\gamma>1\), and establishes both vacuum normalizations and uniqueness of the weak branch before we derive estimates uniform in frequency.

In the density variable, these estimates have explicit weights. We put
\be
\label{eq:Bessel-density-remainder-weights}
W(\mu,\xi) \coloneqq \frac{k(\mu)^{\nu+1/2}}{\sqrt{k'(\mu)}(1+ k(\mu) (1+\xi^2)^{1/2})^{\nu+1/2}}, \qquad r(\mu,\xi) \coloneqq \widehat{\alpha_{\chi}}(\mu,\xi) -\frac{\Psi_0(k(\mu),\xi)}{\sqrt{k'(\mu)}}.
\ee
The preceding proposition and the chain rule give, for every \(\mu>0\),
\be
\label{eq:Bessel-density-remainder-global}
\begin{aligned}
 |r(\mu,\xi)|&\leq C C_*W(\mu,\xi)\bigl(e^{C\mathcal H_\xi(k(\mu))}-1 \bigr),\\
 |\del_\mu r(\mu,\xi)|&\leq C C_* \Bigl( k'(\mu)((1+\xi^2)^{1/2}+k(\mu)^{-1})+
                  \frac{|k''(\mu)|}{k'(\mu)} \Bigr) 
 W(\mu,\xi)\bigl(e^{C\mathcal H_\xi(k(\mu))}-1 \bigr).
\end{aligned}
\ee
Near vacuum, \(W \simeq \mu(1+  k (1+\xi^2)^{1/2})^{-\nu-1/2}\) and \(\mathcal H_\xi(k)\leq H_0(k)\to0\). At large density, \(k \simeq 1+ \log\mu\), \(k' \simeq \mu^{-1}\), and \eqref{eq:Bessel-global-error-growth} bounds the exponential by a constant times \(\exp(C(1+ \log\mu)/(1+\xi^2)^{1/2})\). Here \(H_0(y) \coloneqq \int_0^yt|\mathcal R(t)|\,dt\). These global growth bounds are distinct from high-frequency asymptotic smallness.

\subsubsection{Simultaneous large-frequency and large-density limit}
For a density interval growing with frequency, we choose \(M(\xi)\geq1\)  such that \(\log M(\xi) =o((1+\xi^2)^{1/2})\), and put
\be
\label{eq:Bessel-remainder-derivative-weight}
A(\mu,\xi) \coloneqq k'(\mu)((1+\xi^2)^{1/2}+k(\mu)^{-1})+ \frac{|k''(\mu)|}{k'(\mu)}.
\ee
Then the estimates above imply
\be
\label{eq:Bessel-growing-density-window}
\sup_{0<\mu\leq M(\xi)} \Bigl( \frac{|r(\mu,\xi)|}{W(\mu,\xi)} + \frac{|\del_\mu r(\mu,\xi)|}{A(\mu,\xi)W(\mu,\xi)} \Bigr) \to 0, \qquad 
\xi \to \pm \infty.
\ee
Indeed, \(\mathcal H_\xi\) is nondecreasing in its argument, \(k(\mu)\leq C(1+ \log\mu)\) for \(\mu\geq1\), and \eqref{eq:Bessel-global-error-growth} bounds its supremum on this density interval by \(\mathcal H_\xi(Y_0)+C(1+ \log M(\xi))/(1+\xi^2)^{1/2}\). The preceding expression tends to zero, and the estimates also cover the portion where \(k(\mu)\leq Y_0\), by monotonicity. Thus \(M(\xi) = \exp((1+\xi^2)^{1/4})\), for example, gives an unbounded density interval including every positive density near vacuum. The denominators are the explicit envelopes of the remainder, not the oscillatory leading Bessel function, which has zeros.

\begin{proposition}[Polynomial Fourier decay of the kernel]
\label{cor:kernel-Sobolev-regularity}
Fix an admissible equation of state \(p\) in the sense of Definition~\ref{def-pression}. Set \(\nu= \lambda+1/2>0\) and \(k_{\min}(\mu) =\min\{k(\mu),1\}\). There is a constant \(C(\eps)\), independent of \(\mu>0\) and \(\xi\in\Rbb\), such that
\be
\label{eq:global-krl-frequency-decay}
|\widehat{\alpha_{\chi}}(\mu,\xi)| + \frac{\mu|\del_\mu\widehat{\alpha_{\chi}}(\mu,\xi)|}{1+ \sqrt{1+ \xi^2}\,k_{\min}(\mu)} \leq \frac{C\mu}{(1+ \sqrt{1+ \xi^2}\,k_{\min}(\mu))^{\nu+1/2}}.
\ee

At large density these bounds can be sharpened. Set \(d=\nu+1/2\) and recall \(B_\infty=\eps^2k_\infty^2/(1+\eps^2k_\infty^2)^2\). For every \(\delta>0\), there is \(C_\delta(\eps)\) such that
\bse
\be
\label{eq:weak-kernel-isothermal-growth}
\aligned
|\widehat{\alpha_\chi}(\mu,\xi)|+\frac{\mu|\del_\mu\widehat{\alpha_\chi}(\mu,\xi)|}{\sqrt{1+\xi^2}}
&\leq C_\delta(\eps)\sqrt\mu(1+\xi^2)^{-d/2}
\mu^{\sqrt{(1/4-(1+\xi^2)B_\infty)_+}+\delta},\qquad \mu\geq1.
\endaligned
\ee
For every \(\eta>0\), on \(|\xi|\geq\xi_\infty+\eta\), one has
\be
\label{eq:weak-kernel-isothermal-oscillatory}
|\widehat{\alpha_\chi}(\mu,\xi)|+\frac{\mu|\del_\mu\widehat{\alpha_\chi}(\mu,\xi)|}{\sqrt{1+\xi^2}}\leq C_\eta(\eps)\sqrt\mu(1+\xi^2)^{-d/2},\qquad \mu\geq1.
\ee
\ese
Thus the vacuum exponent determines the power of frequency, while the isothermal limit determines the large-density growth. The combined frequency scale in~\eqref{eq:global-krl-frequency-decay} is explicitly
\be
\sqrt{1+\xi^2}\,k(\mu)=\eps\sqrt{1+\xi^2}\,\frac{k(\mu)}\eps.
\ee
All constants here are for the fixed light-speed parameter; their stated independence concerns density and frequency.
\end{proposition}

\begin{proof}
{\bf 1.} \emph{Vacuum and bounded densities.} We first consider \(0<\mu\leq\mu_{\rm vac}\). The Bessel estimate and the weighted integrability of \(\mathcal R\) give
\be
|\widehat{\alpha_{\chi}}(\mu,\xi)| \leq\frac{C C_* k(\mu)^{\nu+1/2}}{\sqrt{k'(\mu)}(1+ k(\mu) (1+\xi^2)^{1/2})^{\nu+1/2}}.
\ee
Writing \(q=(\gamma-1)/2\), we have \(k(\mu) \simeq \mu^q\), \(k'(\mu) \simeq \mu^{q-1}\), and \(q\nu=1/2\). The normalization satisfies \(C_*k^{\nu+1/2}/\sqrt{k'}\sim2^\nu\Gamma(\nu+1)\mu\), so the powers of \(K=\eps\sqrt{\kappa_0}/q\) cancel in the leading vacuum coefficient. This proves the equation~\eqref{eq:global-krl-frequency-decay} near vacuum. On every fixed compact positive-density interval, the same Bessel bound gives \(C (1+\xi^2)^{-(\nu+1/2)/2}\).

\bse
{\bf 2.} \emph{Global coarse bounds.} We next extend the frequency decay to arbitrarily large density. We choose \(\mu_0>0\) as in  \autoref{prop:global-density-Fourier-bound} and use a fixed compact neighborhood of \(\mu_0\). On this neighborhood, the Fourier equation gives \(|\del_\mu^2\widehat{\alpha_{\chi}}|\leq C (1+\xi^2) |\widehat{\alpha_{\chi}}|\). Taylor's formula with an increment    \( \delta_\mu=c/(1+\xi^2)^{1/2}\), contained in the neighborhood, yields
\be
|\del_\mu\widehat{\alpha_{\chi}}(\mu_0,\xi)| \leq (2/\delta_\mu+C\delta_\mu (1+\xi^2))\sup_{|\mu-\mu_0|\leq\delta_\mu}|\widehat{\alpha_{\chi}}(\mu,\xi)| \leq C(1+\xi^2)^{1/4-\nu/2}.
\ee
It follows that the energy in that proposition satisfies \(H(\log\mu_0,\xi)^{1/2}\leq C (1+\xi^2)^{1/4-\nu/2}\). We keep the factor \((1+\xi^2)^{-1/2}\) in the estimate \(|X|\leq H^{1/2}/( (1+\xi^2)^{1/2}\sqrt{b_0})\), obtaining
\be
|\widehat{\alpha_{\chi}}(\mu,\xi)| \leq C\sqrt\mu\,(1+\xi^2)^{-\nu/2-1/4} \Bigl( \frac{\mu}{\mu_0} \Bigr) ^{1/(4\sqrt{(1+\xi^2) b_0+1/4})} \leq C\mu (1+\xi^2)^{-\nu/2-1/4}, \qquad \mu\geq\mu_0.
\ee
The same energy estimate gives \(\mu|\del_\mu\widehat{\alpha_{\chi}}|\leq C\mu (1+\xi^2)^{1/4-\nu/2}\) on this half-line. Near vacuum we apply Taylor's formula on an interval of length    \(\delta_\mu=c\mu/(1+  k(\mu) (1+\xi^2)^{1/2})\), with \(c>0\) sufficiently small. Throughout that interval the density is comparable to \(\mu\), the distance is comparable to \(k(\mu)\), and \(|\del_\mu^2\widehat{\alpha_{\chi}}|\leq C (1+\xi^2) k(\mu)^2\mu^{-2}|\widehat{\alpha_{\chi}}|\). The already proved bound for \(\widehat{\alpha_{\chi}}\) therefore gives
\be
\mu|\del_\mu\widehat{\alpha_{\chi}}(\mu,\xi)| \leq C(1+ k(\mu) (1+\xi^2)^{1/2}) \frac{\mu}{(1+ k(\mu) (1+\xi^2)^{1/2})^{\nu+1/2}}.
\ee
The estimate on the intervening compact density interval completes the proof of the estimate~\eqref{eq:global-krl-frequency-decay}.

\ese

{\bf 3.} \emph{Isothermal growth rates.} On any fixed positive-density interval, the Bessel estimate and its differentiated form give
\be
|X^{\rm weak}|+(1+\xi^2)^{-1/2}|X_z^{\rm weak}|\leq C(\eps)(1+\xi^2)^{-d/2}.
\ee
Applying Lemma~\ref{lem:isothermal-growth-frequency-regimes} with \(d_0=-d\) proves~\eqref{eq:weak-kernel-isothermal-growth} and~\eqref{eq:weak-kernel-isothermal-oscillatory}; a fixed compact density interval fills the range down to \(\mu=1\). The loss \(\delta\) allows for an arbitrarily slow approach of \(B\) to its isothermal limit. At zero frequency the weak solution is \(\Numb(\mu)\), with \(\Numb'/\Numb=1/h\), so
\be
\frac{\log\Numb(\mu)}{\log\mu}\longrightarrow\frac1{1+\eps^2k_\infty^2}\qquad(\mu\to+\infty).
\ee
This identifies the largest limiting density exponent in~\eqref{eq:weak-kernel-isothermal-growth}. The hypotheses do not impose a rate on the remainder \(r_+\), so a bound by a constant times exactly this power is not asserted.

\end{proof}

\begin{corollary}[Sobolev regularity of the kernel in physical variables]
\label{cor:kernel-physical-Sobolev}
With the notation of Proposition~\ref{cor:kernel-Sobolev-regularity}, for \(0\leq\ell<\nu\), one has
\be
\label{eq:global-krl-Sobolev}
\|\alpha_{\chi}(\mu,\cdot)\|_{H^{\ell}(\Rbb)} \leq C_{\ell}\mu\,k_{\min}(\mu)^{- \ell-1/2}, \qquad \mu>0.
\ee
In particular, near vacuum this bound is \(C_{\ell}\mu^{1- (\gamma-1)( \ell+1/2)/2}\), while at large density it improves to \(C_{\ell,\delta}(\eps)\mu^{1/(1+\eps^2k_\infty^2)+\delta}\) for every \(\delta>0\), by~\eqref{eq:weak-kernel-isothermal-growth}. Thus the kernel belongs to \(L^2\) for every positive density and every \(\gamma>1\). Its \(L^1\) norm satisfies \(\|\alpha_{\chi}(\mu)\|_{L^1}\leq C\mu\sqrt{k(\mu)/k_{\min}(\mu)}\); near vacuum this is \(C\mu\). Once compact support is proved in Lemma~\ref{prop:kernel-positive-finite-speed-prelim}, below, these bounds will justify convolution with any continuous vacuum profile.
\end{corollary}

\begin{proof}
With \(a=k_{\min}(\mu)\in(0,1]\), splitting the frequency integral at \(|\xi|=a^{-1}\) gives, for \(0\leq  \ell<\nu\),
\be
\int_\Rbb\frac{(1+ \xi^2)^{\ell}}{(1+a\sqrt{1+ \xi^2})^{2\nu+1}}\,d\xi \leq C_{\ell} a^{-2 \ell-1}.
\ee
Plancherel's identity proves \eqref{eq:global-krl-Sobolev}; the same integral with~\eqref{eq:weak-kernel-isothermal-growth} proves the large-density refinement. Lemma~\ref{prop:kernel-positive-finite-speed-prelim}, proved below,  gives the support bound; together with the bound~\eqref{eq:global-krl-Sobolev} at \(\ell=0\), Cauchy--Schwarz then gives \(L^1\) integrability.
\end{proof}

\subsection{The strong branch}
\label{section-strong-kernel} 

\subsubsection{Strong branch}

The strong representative has already been selected by the equation~\eqref{eq:strong-Cauchy-selection}. We now use the sharper weak-branch estimates to bound that representative polynomially in frequency and justify its inverse Fourier transform and vacuum trace.

\begin{proposition}[The selected strong branch]
\label{prop:section-8-selected-strong-branch}
Fix an admissible equation of state \(p\) in the sense of Definition~\ref{def-pression}. Let \(z_*\in\Rbb\) be a reference point and let \(X^{\rm strong}\) be defined by the equation~\eqref{eq:strong-Cauchy-selection}. This function is a representative of the affine strong class in \autoref{theo:section-7-fourier-branches}, depends continuously on the real frequency, and satisfies the compact-density estimate~\eqref{eq:strong-compact-frequency-bound} and the global estimate~\eqref{eq:strong-global-density-bound}, below. Consequently, the inverse Fourier transform \(\alpha_{\chi}^{\rm strong}(\mu,\cdot)\) is a tempered distribution for every \(\mu>0\). With \(d_{\rm F}=\nu+3/2\), the following assertions hold. For every \( \ell>d_{\rm F}+1/2\), it satisfies the global bound and  vacuum limits in the equation~\eqref{eq:strong-vacuum-Sobolev-trace}, below. Moreover, \(\alpha_{\chi}^{\rm strong}\in C^2((0,+ \infty);H^{- \ell}(\Rbb))\) for \(\ell>\nu+3\). The constants may depend on \(\eps\) and \(z_*\). The large-density refinements~\eqref{eq:strong-isothermal-frequency-growth} and~\eqref{eq:strong-isothermal-oscillatory}, below, also hold. Near vacuum the quantitative estimate~\eqref{eq:strong-vacuum-rate-frequency}, below, yields convergence at rate \(R_\gamma(\mu)\) in \(H^{-\ell}\) whenever \(\ell>\nu+4\).
\end{proposition}

The spatial component will be constructed in~\autoref{section-strong-flux-kernel}. To bound the density branch, we first quantify the nonvanishing denominator in the equation~\eqref{eq:strong-Cauchy-selection} by the differentiated Bessel estimate. We set \(d\coloneqq\nu+1/2= \lambda+1\).
At the reference point \(z_*\), the weak  initial data \(a= X^{\rm weak}(z_*,\xi)\), \(b= X_z^{\rm weak}(z_*,\xi)\) satisfy
\be
\label{eq:weak-Cauchy-two-sided-frequency}
c_{\rm ref} (1+\xi^2)^{-d}\leq a(\xi)^2+ (1+\xi^2)^{-1} b(\xi)^2 \leq C_{\rm ref}(1+\xi^2)^{-d}, \qquad \xi\in\RR.
\ee
Here the positive constants depend on the reference density and on the pressure law. To prove the large-frequency bounds, we put \(y_*\coloneqq k(e^{z_*})\). The Bessel asymptotics and \eqref{eq:Bessel-derivative-remainder} give
\be
\label{eq:weak-Cauchy-phase-pair}
 \begin{pmatrix}\Psi(y_*,\xi)\\ (1+\xi^2)^{-1/2}\Psi_y(y_*,\xi)\end{pmatrix}
 =A_* (1+\xi^2)^{-d/2}
 \left[\begin{pmatrix}\cos\omega_*\\-\sin\omega_*\end{pmatrix}+o(1)\right],
 \qquad \omega_*=y_*(1+\xi^2)^{1/2}-\frac{\pi\nu}{2}-\frac\pi4,
\ee
where \(A_*>0\). The error tends to zero as \(|\xi|\to+ \infty\). The change from \((\Psi,(1+\xi^2)^{-1/2} \Psi_y)\) to \((a,(1+\xi^2)^{-1/2} b)\) is triangular: writing \(\ell(z) \coloneqq e^{-z/2}/\sqrt{k'(e^z)}\), it has diagonal entries \(\ell(z_*)\) and \(\ell(z_*)e^{z_*}k'(e^{z_*})\), and lower off-diagonal entry \((1+\xi^2)^{-1/2} \ell_z(z_*)\). Its norm and inverse norm are uniformly bounded. The identity \(\cos^2\omega_*+ \sin^2\omega_*=1\) gives both large-frequency bounds; continuity and the nonvanishing initial vector give them on bounded frequency intervals.

\noindent\emph{Frequency growth for the chosen representative.} The exponent \(d\) in~\eqref{eq:strong-compact-frequency-bound}, below, cannot be lowered uniformly on an interval containing \(z_*\). Indeed, the phase formula~\eqref{eq:weak-Cauchy-phase-pair} gives opposite signs for \(b-a\) at frequencies whose phases are \(j\pi-\pi/4\) and \(j\pi+\pi/4\), for all sufficiently large integers \(j\). By continuity, there is a frequency between them with \(b=a\). The two-sided bound~\eqref{eq:weak-Cauchy-two-sided-frequency} then gives \(|a|\simeq(1+\xi^2)^{-d/2}\), whereas the selection~\eqref{eq:strong-Cauchy-selection} gives \(|X^{\rm strong}(z_*,\xi)|=1/(2|a|)\). Thus the power \((1+\xi^2)^{d/2}\) is attained along an unbounded sequence. This assertion concerns the selected representative; the global Sobolev bounds below are not claimed to have optimal order.

\subsubsection{Selection at a positive density}

The strong initial vector \((b,-a)/(a^2+b^2)\) has Euclidean norm \((a^2+b^2)^{-1/2}\leq C(1+\xi^2)^{d/2}\). Applying the weighted propagation estimate~\eqref{eq:connection-frequency-uniform} on a compact logarithmic-density interval \(I\) containing \(z_*\) gives
\be
\label{eq:strong-compact-frequency-bound}
\sup_{z\in I} \Bigl( |X^{\rm strong}(z,\xi)|+ (1+\xi^2)^{-1/2} |X_z^{\rm strong}(z,\xi)| \Bigr) \leq C_I (1+\xi^2)^{d/2}.
\ee
 The branch and its first two density derivatives therefore have tempered inverse transforms on compact positive-density intervals. Compact support is established below for the weak kernel only.

\subsubsection{Bounds near vacuum}

We put \(f(\mu,\xi)=\sqrt\mu X^{\rm strong}(\log\mu,\xi)\) and \(d_{\rm F}=d+1=\nu+3/2\). For every \(\delta>0\), the global bound is
\bel{eq:strong-global-density-bound}
|f(\mu,\xi)|+\mu|f_\mu(\mu,\xi)|
\leq C_\delta(\eps)(1+\mu)^{1/(1+\eps^2k_\infty^2)+\delta}(1+\xi^2)^{d_{\rm F}/2},\qquad \mu>0.
\ee
To prove it near vacuum, we choose a fixed \(\mu_{\rm vac}<e^{z_*}\) so small that \(B_z\geq0\) for \(z\leq\log\mu_{\rm vac}\), as guaranteed by~\eqref{eq:section-7-logarithmic-vacuum-control}. The canonical energy satisfies
\be
H=|X_z|^2+\bigl((1+\xi^2)B+1/4\bigr)|X|^2,\qquad
H_z=\operatorname{Re}(X\overline{X_z})+(1+\xi^2)B_z|X|^2\geq-H.
\ee
Consequently \(\mu H\) is nondecreasing in logarithmic density on this interval. The compact-density bound~\eqref{eq:strong-compact-frequency-bound} gives \(H(\log\mu_{\rm vac},\xi)\leq C(1+\xi^2)^{d+1}\). Since \(f=\sqrt\mu X\) and \(\mu f_\mu=\sqrt\mu(X_z+X/2)\), it follows that
\be
|f(\mu,\xi)|+\mu|f_\mu(\mu,\xi)|\leq C(\eps)(1+\xi^2)^{(d+1)/2},\qquad0<\mu\leq\mu_{\rm vac}.
\ee
This proves~\eqref{eq:strong-global-density-bound} near vacuum without a frequency-dependent density cutoff.

The vacuum normalization has a quantitative version. With the function \(R_\gamma\) defined in Proposition~\ref{thm:olver-vacuum}, one has
\be
\label{eq:strong-vacuum-rate-frequency}
|f(\mu,\xi)-1|+\mu|f_\mu(\mu,\xi)|
\leq C(\eps)(1+\xi^2)^{(d+3)/2}R_\gamma(\mu),\qquad0<\mu\leq\mu_{\rm vac}.
\ee
Indeed, the equation for \(f\) and the bound just proved give \(|f_{\mu\mu}|\leq C(1+\xi^2)^{(d+3)/2}\mu^{\gamma-3}\). Integration from \(\mu\) to \(\mu_{\rm vac}\), followed by multiplication by \(\mu\), bounds \(\mu|f_\mu|\) by the right-hand side of~\eqref{eq:strong-vacuum-rate-frequency}. The three cases \(\gamma<2\), \(\gamma=2\), and \(\gamma>2\) give precisely the three definitions of \(R_\gamma\). Integrating once more from zero and using \(f(0,\xi)=1\) proves the bound for \(f-1\).

\subsubsection{Bounds at large density}

Applying Lemma~\ref{lem:isothermal-growth-frequency-regimes} to the compact-density data~\eqref{eq:strong-compact-frequency-bound}, with \(d_0=d\), gives, for every \(\delta>0\),
\be
\label{eq:strong-isothermal-frequency-growth}
\aligned
|f(\mu,\xi)|+\frac{\mu|f_\mu(\mu,\xi)|}{\sqrt{1+\xi^2}}
&\leq C_\delta(\eps)\sqrt\mu(1+\xi^2)^{d/2}
\mu^{\sqrt{(1/4-(1+\xi^2)B_\infty)_+}+\delta},\qquad\mu\geq1.
\endaligned
\ee
For every \(\eta>0\), the range \(|\xi|\geq\xi_\infty+\eta\) has the sharper bound
\be
\label{eq:strong-isothermal-oscillatory}
|f(\mu,\xi)|+\frac{\mu|f_\mu(\mu,\xi)|}{\sqrt{1+\xi^2}}
\leq C_\eta(\eps)\sqrt\mu(1+\xi^2)^{d/2},\qquad\mu\geq1.
\ee
The largest exponent in~\eqref{eq:strong-isothermal-frequency-growth} is \(1/(1+\eps^2k_\infty^2)+\delta\). Multiplying its derivative term by \(\sqrt{1+\xi^2}\) proves~\eqref{eq:strong-global-density-bound} at large density; the compact-density estimate covers the remaining interval. Choosing \(\delta<1-1/(1+\eps^2k_\infty^2)\) recovers the former linear growth bound as a consequence. The estimates distinguish exponential growth, oscillatory growth, and the threshold regime, without imposing a rate of convergence on the isothermal pressure remainder.

\subsubsection{The prescribed vacuum value}

The differentiated strong vacuum normalization gives \(f(\mu,\xi)\to1\) and \(\mu f_\mu(\mu,\xi)\to0\) at each real frequency. The polynomial majorant in the estimate~\eqref{eq:strong-global-density-bound} permits dominated convergence against Schwartz functions. Consequently, the selected strong kernel  converges to \(\delta_0\) in \(\mathcal S'\) at vacuum, while its density derivative multiplied by \(\mu\) tends to zero. For the weak kernel, the density tends to zero and its unweighted density derivative tends to \(\delta_0\), as proved in Corollary~\ref{cor:uniform-vacuum-traces}.

More quantitatively, we denote this inverse Fourier transform by \(\alpha_{\chi}^{\rm strong}(\mu,\cdot)\). For every \( \ell>d_{\rm F}+1/2\), the global bound and the  vacuum limits give
\be
\label{eq:strong-vacuum-Sobolev-trace}
 \begin{aligned}
 \|\alpha_{\chi}^{\rm strong}(\mu,\cdot)\|_{H^{- \ell}}+
 \|\mu\del_\mu\alpha_{\chi}^{\rm strong}(\mu,\cdot)\|_{H^{- \ell}}
 &\leq C_{\ell,\delta}(\eps)(1+\mu)^{1/(1+\eps^2k_\infty^2)+\delta},
 \\
 \|\alpha_{\chi}^{\rm strong}(\mu,\cdot) -\delta_0\|_{H^{- \ell}}+
 \|\mu\del_\mu\alpha_{\chi}^{\rm strong}(\mu,\cdot)\|_{H^{- \ell}}
 & \to 0\quad(\mu\to 0).
 \end{aligned}
\ee
Indeed, the squared Fourier integrands are bounded near vacuum by \(C (1+\xi^2)^{d_{\rm F}- \ell}\), which is integrable on \(\RR\). Dominated convergence proves the strong convergence in \(H^{- \ell}\). At positive density, the equation \(f_{\mu\mu} \, =-(1+\xi^2) B(\mu)\mu^{-2}f\) similarly gives \(\alpha_{\chi}^{\rm strong}\in C^2((0,+ \infty);H^{- \ell})\) whenever \(\ell>d+5/2\). On each compact density interval the Fourier derivatives of orders \(j=0,1,2\) are bounded by \(C(1+\xi^2)^{(d+j)/2}\); continuity and differentiation under Fourier inversion follow by the same weighted dominated-convergence argument. For \(\ell>d+7/2\), integrating the square of~\eqref{eq:strong-vacuum-rate-frequency} against \((1+\xi^2)^{-\ell}\) gives the stated \(H^{-\ell}\) convergence rate \(R_\gamma(\mu)\). The current equation and vacuum trace thus hold in these spaces. These estimates prove Proposition~\ref{prop:section-8-selected-strong-branch} and complete the strong-branch assertions of Proposition~\ref{cor:section-7-fourier-kernel-bounds}.

Theorem~\ref{theo:section-7-fourier-branches} is now proved in full, including global continuation, the decompositions in the two large-density bases, and the asserted frequency continuity. The bounds of Proposition~\ref{cor:section-7-fourier-kernel-bounds} have also been established for both vacuum normalizations. Proposition~\ref{cor:kernel-Sobolev-regularity} gives the sharper weighted frequency estimates for the weak branch, while Proposition~\ref{prop:section-8-selected-strong-branch} proves the polynomial bounds and the distributional vacuum trace of the selected strong density branch. Sections~\ref{section-7} and~\ref{section-flux-kernels}, below, reconstruct the density and flux kernels in physical variables; the support, characteristic-boundary behavior and positivity of the weak kernel are established in Section~\ref{section-7}.

\section{Weak spacetime kernel in physical variables}
\label{section-7}  

\subsection{Fourier inversion}
\label{section-7-1}

\subsubsection{Reconstruction of the density kernel}

 We now apply the inverse Fourier transform. The Fourier estimates~\eqref{eq:Bessel-density-remainder-global} and~\eqref{eq:global-krl-frequency-decay} provide us with the key control on the weak density kernel. We now pass to the state variables and investigate the
support of the kernel and its positivity.  
Recall that the transformation~\eqref{eq:def-X-canonical} relates $\widehat{\alpha_{\chi}}$ to $X^{\rm weak}$, which solves the canonical equation~\eqref{eq:X-canonical} with potential~\eqref{eq:X-canonical-sec9}. We denote by $\mathcal S'(\Rbb_v)$ the space of tempered distributions. 

\begin{proposition}[Inverse Fourier reconstruction of the weak kernel]
\label{thm:physical-krl-existence}
Let \(p\) be an admissible equation of state in the sense of Definition~\ref{def-pression}. For every \(\mu>0\),  define
\bel{eq:def-physical-krl-inv-fourier-prelim}
\alpha_{\chi}(\mu,\cdot) \coloneqq \mathcal F_{\xi\to v}^{-1}\bigl(\widehat{\alpha_{\chi}}(\mu,\xi) \bigr) \qquad\text{in }\mathcal S'(\Rbb_v).
\ee
Then the following properties hold.

\bei
\item For every \(\mu>0\), the weak kernel \(\alpha_{\chi}(\mu,\cdot)\) exists as a tempered distribution, and the map $ \mu\longmapsto \alpha_{\chi}(\mu,\cdot) $ is continuous from \((0,+ \infty)\) into \(\mathcal S'(\Rbb_v)\).

\item The distribution \(\alpha_{\chi}\) solves the  weak current equation in \((0,+ \infty)\times\Rbb\), together with the vacuum normalization
\be
\alpha_{\chi}|_{\mu=0} \, =0, \qquad \del_\mu\alpha_{\chi}|_{\mu=0} \, = \delta_{v=0},
\ee
in the sense of distributions.

\item For every compact interval \(I\subset(0,+ \infty)\) and every pair \((a,b)\in\Nbb^2\) with \(a+b\le2\), there exist \(C_{I,a,b},N_{I,a,b}>0\) such that
\bel{eq:section-8-uniform-physical-krl-bound}
\bigl|\bigl\langle \del_\mu^a\del_v^b\alpha_{\chi}(\mu,\cdot),\varphi\bigr\rangle\bigr| \le C_{I,a,b}\|\varphi\|_{\mathcal S,N_{I,a,b}} \qquad \text{for all }\mu\in I,\ \varphi\in\mathcal S(\Rbb),
\ee
where \(\|\cdot\|_{\mathcal S,N}\) denotes the Schwartz seminorm\footnote{It controls finitely many derivatives of the test function, weighted by powers of the variable, with orders bounded by~$N$.} of order \(N\).
\item There exist \(C>0\) and the Schwartz seminorm \(\|\cdot\|_{\mathcal S,N}\) such that, for every \(\mu>0\),
\bel{eq:physical-krl-all-densities}
|\langle\alpha_{\chi}(\mu),\varphi\rangle| + \mu|\langle\del_\mu\alpha_{\chi}(\mu),\varphi\rangle| \leq C\mu\|\varphi\|_{\mathcal S,N}.
\ee
Moreover,  one has
\be
\alpha_{\chi}(\mu)/\mu\to\delta_0, \qquad \quad \del_\mu\alpha_{\chi}(\mu)\to\delta_0 \qquad \text{ in } \mathcal S'
\ee
at vacuum. The estimate covers all frequencies, including the large-density thresholds.
\eei
\end{proposition}

\begin{proof}
Corollary~\ref{cor:section-7-uniform-fourier-kernel} bounds \(\widehat{\alpha_{\chi}}(\mu,\xi)\) polynomially in \(\xi\), locally uniformly in~\(\mu\). This gives \(\widehat{\alpha_{\chi}}(\mu,\cdot)\in\mathcal S'(\Rbb_\xi)\) and defines the inverse transform~\eqref{eq:def-physical-krl-inv-fourier-prelim} in \(\mathcal S'(\Rbb_v)\). Continuity of \(X^{\rm weak}(z,\xi)\) and the polynomial bound on compact \(z\)-intervals give continuity in \(\mu\). The Fourier isomorphism of \(\mathcal S'(\Rbb)\) transfers the weak current equation to physical variables. For the vacuum normalization, Corollary~\ref{cor:uniform-vacuum-traces} bounds both \(\widehat{\alpha_{\chi}}/\mu\) and \(\del_\mu\widehat{\alpha_{\chi}}\) by polynomial majorants independent of \(0<\mu\leq\mu_{\rm vac}\); both tend pointwise to one. Dominated convergence against Fourier transforms of Schwartz test functions gives the two vacuum limits and \(\alpha_{\chi}(\mu)\to0\).

For \(a=0\), the polynomial Fourier bounds on compact \(\mu\)-intervals give~\eqref{eq:section-8-uniform-physical-krl-bound}. A \(v\)-derivative multiplies by \(i\xi\); Corollary~\ref{cor:section-7-uniform-fourier-kernel} controls one \(\mu\)-derivative. The Fourier equation expresses the second as a bounded multiple of \((1+\xi^2)\widehat{\alpha_{\chi}}\) on \(I\), proving the estimate for \(a+b\leq2\).
Estimate~\eqref{eq:Fourier-global-density-polynomial} gives \(|\widehat{\alpha_{\chi}}|+\mu|\del_\mu\widehat{\alpha_{\chi}}|\leq C\mu\sqrt{1+\xi^2}\) for all densities and frequencies. Integrating against \(|\widehat\varphi(\xi)|\) proves the equation~\eqref{eq:physical-krl-all-densities}; the integral \(\int_\Rbb\sqrt{1+\xi^2}\,|\widehat\varphi(\xi)|\,d\xi\) is bounded by a fixed Schwartz seminorm.
\end{proof}

\subsubsection{\texorpdfstring{Rest-frame spatial component}{Rest-frame spatial component}}

We also recover the \emph{rest-frame spatial component} of the weak kernel from the same Fourier branch. For \(\xi\ne0\), we define
\bel{eq:tau-recovery-sec9}
\widehat{\beta_{\chi}}(\mu,\xi)
= \frac{1}{i\xi} \Bigl( h(\mu)\,\del_\mu\widehat{\alpha_{\chi}}(\mu,\xi)-\widehat{\alpha_{\chi}}(\mu,\xi) \Bigr).
\ee
At zero frequency, the numerator and its first derivative in~\(\xi\) vanish. Indeed, \(\widehat{\alpha_{\chi}}(\mu,0)=\Numb(\mu)\), \(h(\mu)\Numb'(\mu)=\Numb(\mu)\), and the weak branch is even in~\(\xi\). Differentiation of the vacuum Volterra equation and continuation of its initial data show that the numerator is \(\Ocal(\xi^2)\), locally uniformly at positive density. By continuity, we therefore define \(\widehat{\beta_{\chi}}(\mu,0)\coloneqq0\), obtaining the unique odd continuous extension. Together with the density kernel already reconstructed, the formula~\eqref{eq:tau-recovery-sec9} determines the two fluid-frame coefficients of the current kernel \(\chi\). Section~\ref{section-11}, below, proves the flux estimates.

The scalar Fourier equation and $h'=1+c^2$ give, for $\xi\ne0$,
\be
h\del_\mu\widehat{\beta_\chi}-c^2\widehat{\beta_\chi}
=i\xi c^2\widehat{\alpha_\chi}.
\ee
The smooth dependence on frequency given by the vacuum Volterra equation and its positive-density continuation also holds for the first two density derivatives of $\widehat{\alpha_\chi}$. Both $h\del_\mu\widehat{\alpha_\chi}-\widehat{\alpha_\chi}$ and its density derivative are even in frequency and vanish at $\xi=0$, so they are $\mathcal O(\xi^2)$ locally uniformly at positive density. Thus $\widehat{\beta_\chi}$ and its density derivative extend continuously with value zero at $\xi=0$, and the preceding identity holds at every real frequency. Together with the defining formula~\eqref{eq:tau-recovery-sec9}, it proves both first-order current equations. The vacuum trace of the flux is established independently by the bounds in Section~\ref{section-11-1}, below.

\subsection{Support and characteristic boundary}
\label{section-7-4}

\subsubsection{The fluid cone from a vacuum point}

We next prove that the reconstructed weak kernel is \emph{supported in the acoustic cone}, which will justify convolution with continuous vacuum profiles. Using the acoustic distance~\eqref{eq:standard-acoustic-distance}, we introduce the characteristic profiles:
\bel{eq:section-8-acoustic-radius}
k(\mu) = \int_0^\mu\frac{c(r)}{r+\eps^{2}p(r)}\,dr, \qquad \theta \coloneqq v-s, \qquad G_b(\mu,\theta) \coloneqq \bigl[k(\mu)^2-\theta^2 \bigr]_+^b.
\ee
At positive density, for \(b>-1\), this last expression denotes the \emph{function} \((k(\mu)^2-\theta^2)^b\) on \(|\theta|<k(\mu)\), \emph{extended by zero} on the exterior. In particular, \(G_0\) is the indicator of this interval. Negative exponents give a locally integrable \emph{boundary singularity}.
The two curves \(\theta= \pm k(\mu)\) are the characteristics issuing from the vacuum point \((0,s)\). After finite propagation, the Fourier asymptotics give the leading profile in~\autoref{section-8-1}. We emphasize that the zero-order term in the scalar equation for \(\alpha_{\chi}\) \emph{prevents a direct sign argument}. We instead use \((\alpha_{\chi},\beta_{\chi})\): its characteristic combinations have non-negative coupling under the genuine nonlinearity condition~\eqref{hyperbolic-eos-3}, extending the \emph{hyperbolic maximum-principle} in~\cite{CL1-2000}.

\subsubsection{Finite propagation through the vacuum singularity} 

We begin with an inclusion concerning the support. Next, under the condition~\eqref{hyperbolic-eos-3}, in \autoref{theo:section-8-causal-krl-current}, below, we will strengthen the following inclusion to equality and proves strict interior positivity of \(\alpha_{\chi}\).

\begin{lemma}[Finite propagation of the physical kernel]
\label{prop:kernel-positive-finite-speed-prelim}
Let \(p\) be an admissible equation of state in the sense of Definition~\ref{def-pression}. Then, for every \(\mu>0\) and \(s\in\Rbb\), one has
\bel{eq:kernel-support-sec10-prelim}
\supp\alpha_{\chi}(\mu,\cdot\,;s)\subseteq\{v\in\Rbb;\ |v-s|\leq k(\mu)\}.
\ee
\end{lemma}

\begin{proof}
\bse
{\bf 1.} \emph{Regularized support cones.} To regularize the vacuum coefficient, for \(0<\auxeps<\mu_{\rm vac}\) we set
\be
 B_\auxeps(\mu) =
 \begin{cases}B(\auxeps)(\mu/\auxeps)^2,&0\leq\mu\leq\auxeps,\\
 B(\mu),&\mu\geq\auxeps,
 \end{cases}
 \qquad k_\auxeps'(\mu) = \sqrt{B_\auxeps(\mu)}/\mu.
\ee
The speed \(k_\auxeps'\) is continuous, piecewise smooth, and constant near zero. For each \(\auxeps\), the regularized wave equation has the usual domain of dependence. With initial velocity a mollifier supported in \(|\theta|\leq\delta\), its solution is supported in \( |\theta|\leq\delta+ \int_0^\mu k_\auxeps'(r)\,dr\). To justify the propagation estimate, we write \(a\coloneqq k_\auxeps'>0\) and  take \(w\)  to satisfy the regularized equation \(w_{\mu\mu}-a^2w_{\theta\theta}+a^2w=0\). The energy density \(e\coloneqq(w_\mu^2+a^2w_\theta^2+w^2)/2\) satisfies
\be
\del_\mu e-\del_\theta(a^2w_\mu w_\theta) =aa'w_\theta^2+(1-a^2)ww_\mu.
\ee
On each fixed finite density interval, the right-hand side is bounded in absolute value by \(C_\auxeps e\). At the right moving boundary \(\theta= \delta+ \int_0^\mu a\), the contribution to the exterior energy derivative is
\be
-a^2w_\mu w_\theta-ae =-\frac a2(w_\mu+aw_\theta)^2-\frac a2w^2\leq0.
\ee
The left boundary gives the same inequality with the opposite sign in the square. Since the initial exterior energy is zero, Gronwall's inequality makes it zero throughout the interval. The calculation applies on both sides of the matching density \(\auxeps\); continuity of \(a\) permits their integration together. This exact support assertion permits a constant depending on \(\auxeps\); removal of the regularization uses the uniform bounds below.
\ese

\vskip.2cm

\noindent{\bf 2.} \emph{Removal of the regularization.} The vacuum bounds give
\be
\int_0^{\mu_{\rm vac}}\frac{|B_\auxeps-B|}{\mu}\,d\mu \leq C\auxeps^{\gamma-1}, \qquad \int_0^{\mu_{\rm vac}}|k_\auxeps'-k'|\,d\mu \leq C\auxeps^{(\gamma-1)/2}.
\ee
Moreover, the logarithmic derivative of \(B_\auxeps\) equals two below \(\auxeps\), and lies between \((\gamma-1)/2\) and \(3(\gamma-1)/2\) above it. It is therefore bounded above and below by positive constants independent of \(\auxeps\), except at the harmless matching point. The proof of Corollary~\ref{cor:uniform-vacuum-traces} applies with these bounds: the rescaled Volterra estimate controls the interval before the transition, and the monotonicity of \(E/V\) controls the interval after it. Neither step requires differentiability at the matching point, since \(B_\auxeps\) is continuous and piecewise smooth. Thus the regularized Fourier kernels obey \( |\widehat{\alpha_{\chi_\auxeps}}|\leq C\mu\) and \( |\del_\mu\widehat{\alpha_{\chi_\auxeps}}|\leq C(1+|\xi|)\) near vacuum, with constants independent of the regularization parameter.

The first integral estimate above and the Volterra equation give convergence to the weak branch on compact frequency intervals. We propagate the polynomial bounds across compact density intervals, uniformly in the regularization parameter, and obtain convergence in the sense of distributions after Fourier inversion. We remove the mollifier and then let \(\auxeps\to0\). The second integral estimate gives convergence of the support cones. Testing against any function supported outside \(|\theta|\leq k(\mu)\) proves the equation~\eqref{eq:kernel-support-sec10-prelim}.
\end{proof}

\subsection{Polytropic leading behavior at the boundary}
\label{section-8-1}

Having located the support, we now use the high-frequency Fourier asymptotics to identify the \emph{leading behavior at the characteristic boundaries}. 
We rescale the distance to either boundary and prove convergence in the sense of distributions and in local Sobolev spaces.
  
\begin{proposition}[Leading behavior at the characteristic boundary]
\label{prop:distributional-leading-trace}
Let \(p\) be an admissible equation of state in the sense of Definition~\ref{def-pression}.
The kernel admits the decomposition
\bel{eq:section-8-one-term-krl-boundary}
\alpha_{\chi}(\mu,v;s) =a_1(\mu)G_\lambda(\mu,v-s)+ \mathcal R_{\alpha_{\chi}}(\mu,v-s), \qquad a_1(\mu) =c_{*,\lambda}k(\mu)^{-\lambda}k'(\mu)^{-1/2}>0,
\ee
where \(c_{*,\lambda}>0\) is fixed by the  vacuum normalization. 
\bei 

\item In either inward coordinate at the two characteristic boundaries, the remainder satisfies
\bel{eq:section-8-krl-boundary-remainder}
r^{-\lambda} \mathcal R_{\alpha_{\chi}}(\mu,\pm(k(\mu) -r\,\cdot)) \to 0 \quad\text{ in }\mathcal D'(\Rbb),\quad r \to 0,
\ee
uniformly on compact positive-density intervals. Equivalently, at the right characteristic boundary, introduce \(q\) by \(f_\mu(q) \coloneqq\alpha_{\chi}(\mu,k(\mu) -q)\), with extension by zero for \(q< 0\). One has
\be
\label{eq:distributional-leading-trace}
r^{-\lambda}f_\mu(r\,\cdot) \to (2k(\mu))^\lambda a_1(\mu)(\cdot)_+^\lambda \quad \text{ in }\mathcal D'(\Rbb), \qquad r \to 0.
\ee

\item The convergence also holds strongly in \(H^{\ell}_{\rm loc}(\Rbb)\) for \(0\leq  \ell<\lambda+1/2\), uniformly on compact positive-density intervals. In particular, it holds in local \(L^2\) for every \(\gamma>1\). When \(1<\gamma<3\), it is locally uniform as a convergence of functions and gives \(\alpha_{\chi}(\mu,\theta)/G_\lambda(\mu,\theta)\to a_1(\mu)\) at either boundary. The left boundary follows by reflection. The coefficient satisfies the identity
\bel{eq:section-8-leading-boundary-transport}
2kk'\del_\mu a_1+ \bigl(kk''+2\lambda k'^2 \bigr)a_1=0.
\ee
Thus the leading term \emph{vanishes} at the boundary for \(1<\gamma<3\), has a \emph{jump discontinuity} for \(\gamma=3\), and is \emph{singular} for \(\gamma>3\).
\eei
\end{proposition}

Some observations are in order.
\bei 

\item 
Observe that the \emph{singular} range $\gamma>3$ is included. Since $\lambda=-1/2+1/(\gamma-1)>-1/2$, the inward boundary power satisfies
\be
\label{eq:characteristic-power-integrability}
\int_0^d q^\lambda\,dq=\frac{d^{\lambda+1}}{\lambda+1}<+\infty,\qquad \int_0^d q^{2\lambda}\,dq=\frac{d^{2\lambda+1}}{2\lambda+1}<+\infty\quad(d>0).
\ee
The global Fourier bounds give $\alpha_{\chi}(\mu,\cdot),\beta_{\chi}(\mu,\cdot)\in L^2(\Rbb)$ at every $\mu>0$. Their compact support then gives $L^1$ integrability. When $\gamma>3$, the nonzero scaling limit~\eqref{eq:distributional-leading-trace} also excludes local essential boundedness at either characteristic boundary: a bounded kernel would give a zero limit after multiplication by $r^{-\lambda}\to0$.  
The integrability constants may depend on $\gamma$.

\item Observe also that the strong \(L^2\) trace gives a quantitative estimate even when the kernel is unbounded at the boundary. For each fixed \(Q>0\), with \(A_\mu=(2k(\mu))^\lambda a_1(\mu)\), we obtain
\be
\label{eq:kernel-boundary-quadratic-mass}
r^{-2\lambda-1}\int_{k(\mu) -rQ}^{k(\mu)}|\alpha_{\chi}(\mu,\theta)|^2\,d\theta \to \frac{A_\mu^2 Q^{2\lambda+1}}{2\lambda+1}, \qquad r \to 0.
\ee
Indeed, the change of variable \(\theta=k(\mu) -rq\) makes the left-hand side equal to the squared \(L^2(0,Q)\) norm of \(r^{-\lambda}f_\mu(rq)\). Proposition~\ref{prop:distributional-leading-trace} above gives norm convergence, uniform on compact positive-density intervals; reflection gives the left-boundary result. Since \(2\lambda+1=2/(\gamma-1)>0\), the integral of the squared kernel over these shrinking intervals tends to zero for every \(\gamma>1\), including \(\gamma>3\), where the leading term is singular.
\eei

\begin{proof}
\bse
\noindent{{\bf 1.}} The standard Bessel asymptotic formula~\cite{Olver} and  \autoref{prop:Bessel-one-term-bound} give, uniformly on compact positive-density intervals,
\be
\widehat{\alpha_{\chi}}(\mu,\xi) = \frac{C_*\sqrt{2/\pi}}{\sqrt{k'(\mu)}}|\xi|^{-\lambda-1} \cos\bigl(k(\mu)|\xi|-\pi(\lambda+1)/2 \bigr)+o(|\xi|^{-\lambda-1}).
\ee
Indeed, replacing \(\sqrt{1+ \xi^2}\) by \(|\xi|\) changes the phase by \(\Ocal(|\xi|^{-1})\) on those intervals. We write
\be
A_\mu\coloneqq \frac{C_*}{\sqrt{2\pi}\,\Gamma(\lambda+1)\sqrt{k'(\mu)}} =(2k(\mu))^\lambda a_1(\mu).
\ee
This also specifies the normalization \(c_{*,\lambda} \, =C_*/(2^\lambda\sqrt{2\pi}\Gamma(\lambda+1))\). With the convention \(\widehat f(\xi) = \int e^{-i\xi q}f(q)\,dq\), reflection and translation yield
\be
\widehat f_\mu(\xi) =A_\mu\Gamma(\lambda+1) \bigl((i\xi)^{-\lambda-1}+e^{-2ik(\mu)\xi}(-i\xi)^{-\lambda-1} \bigr) +o(|\xi|^{-\lambda-1})\qquad (|\xi|\to\infty).
\ee
Here the powers for nonzero real \(\xi\) use arguments \(\pm\pi/2\).
\ese

\vskip.2cm

\bse
\noindent{{\bf 2.}} We avoid interpreting the singular powers at zero by taking an integer \(m>\lambda+1\). For \(g_{\mu,r}(q) \coloneqq r^{m-\lambda}f_\mu^{(m)}(rq)\), we have
\bel{eq:boundary-rescaled-Fourier-bound}
\widehat g_{\mu,r}(\eta) =r^{-\lambda-1}(i\eta)^m\widehat f_\mu(\eta/r), \qquad |\widehat g_{\mu,r}(\eta)|\leq C|\eta|^{m-\lambda-1}.
\ee
The bound follows from \(|\widehat f_\mu(\xi)|\leq C(1+|\xi|)^{-\lambda-1}\), including the bounded frequency interval. The error tends to zero after integration against a Schwartz function by dominated convergence. The term with \(e^{-2ik(\mu)\eta/r}\) tends to zero by the Riemann--Lebesgue lemma, since its remaining factor times that test function is integrable. We obtain
\be
g_{\mu,r} \to A_\mu\,\del_q^m(q_+^\lambda) \quad\text{in }\mathcal S'(\Rbb).
\ee
For the Fourier identity used here, we multiply the elementary Laplace identity
\be
\int_0^{+\infty} q^{\lambda}e^{- (\auxeps+i\eta)q}\,dq = \Gamma(\lambda+1)(\auxeps+i\eta)^{-\lambda-1}
\ee
by \((i\eta)^m\) and let \(\auxeps\to0\). Both \(g_{\mu,r}\) and its limit are supported in \([0,+ \infty)\). On distributions with this support,
 \(m\)-fold integration from zero is continuous locally: convolution with \(q_+^{m-1}/(m-1)!\), tested on a bounded interval, involves only a fixed bounded interval of each factor. Applying this integration gives the equation~\eqref{eq:distributional-leading-trace}, since it inverts \(\del_q^m\) without a polynomial ambiguity on this support class.
\ese

\vskip.2cm

\bse
\noindent{{\bf 3.}} To strengthen this convergence to Sobolev convergence, we use the Fourier bound~\eqref{eq:boundary-rescaled-Fourier-bound}, which gives, for \( \ell<\lambda+1/2\),
\be
\|g_{\mu,r}\|_{H^{\ell-m}}^2 \leq C\int_\Rbb(1+ \eta^2)^{\ell-m}|\eta|^{2m-2(\lambda+1)} \,d\eta<+ \infty,
\ee
uniformly in \(r\). The integrability at zero follows from \(m>\lambda+1\). For a fixed \(Q>0\), we choose a compactly supported smooth cutoff equal to one on \([0,2Q]\), and multiply \(q_+^{m-1}/(m-1)!\) by this cutoff to obtain \(K_m\). Its \(m\) -th distributional derivative is a Dirac mass plus a smooth compactly supported function; hence \(|\widehat K_m(\eta)|\leq C(1+|\eta|)^{-m}\). On \((-Q,Q)\), support in \([0,+ \infty)\) gives
\be
r^{-\lambda}f_\mu(rq) =(K_m*g_{\mu,r})(q).
\ee
Thus the kernels are uniformly bounded in \(H^{\ell}(-Q,Q)\) for every \( \ell<\lambda+1/2\). Compactness on smaller intervals and the unique distributional limit give strong convergence in \(H^{\ell'}_{\rm loc}\) for \(0\leq  \ell'< \ell\). Since \( \ell\) can approach \(\lambda+1/2\), this proves the Sobolev assertion. If \(\lambda>0\), we choose \(1/2< \ell<\lambda+1/2\); one-dimensional Sobolev embedding gives locally uniform convergence, and evaluation at \(q=1\) gives the pointwise boundary ratio. The coefficient bounds, the Bessel remainder, and the positive lower bound for \(k(\mu)\) on compact density intervals make all arguments uniform there. Reflection proves the other boundary assertion. Finally, logarithmic differentiation of \(a_1=c_{*,\lambda}k^{-\lambda}k'^{-1/2}\) gives the identity~\eqref{eq:section-8-leading-boundary-transport}.
\ese
\end{proof}

\subsection{Positivity properties of the weak spacetime kernel}
\label{section-8-3}

We next prove \(|\beta_{\chi}|\leq c \, \alpha_{\chi}\) under the genuine nonlinearity condition~\eqref{hyperbolic-eos-3}. We integrate the two-component system along its characteristics and use the non-negativity of its coupling coefficients. Writing this system for measures allows the argument to include kernels singular at the characteristic boundaries.

\begin{proposition}[Future-timelike character of the weak kernel]
\label{theo:section-8-causal-krl-current}
Let \(p\) be an admissible equation of state in the sense of Definition~\ref{def-pression} and assume in addition the genuine nonlinearity condition \eqref{hyperbolic-eos-3}. Let \((\alpha_{\chi},\beta_{\chi})\) be the rest-frame density-flux kernel determined by the first-order system~\eqref{eq:section-6-normalized-first-order-system} and the vacuum normalization \eqref{eq:kernel-physical}. 
\bei 

\item Then, for every \(\mu>0\), \(\alpha_{\chi}(\mu,\cdot\,;s)\) and \((c\alpha_{\chi}\pm\beta_{\chi})(\mu,\cdot\,;s)\) are \emph{non-negative locally finite measures} in the rapidity variable. Their total variations are uniformly bounded when \(\mu\) ranges in a compact positive-density interval. In the open characteristic cone, they are represented by  continuous functions and one has
\bel{eq:section-8-causal-krl-bound}
\alpha_{\chi}(\mu,v;s)>0, \qquad |\beta_{\chi}(\mu,v;s)|<c(\mu)\alpha_{\chi}(\mu,v;s) \qquad\bigl(\mu>0,\ |v-s|<k(\mu) \bigr).
\ee

\item Consequently, the current vector
\bel{eq:section-8-krl-current-vector}
\chi(\mu,v;s) \coloneqq \alpha_{\chi}(\mu,v;s)\overline u+ \beta_{\chi}(\mu,v;s)\overline u^\perp
\ee
is \emph{future timelike} in that cone, and
\bel{eq:section-8-krl-current-timelike}
g(\chi,\chi) =-\alpha_{\chi}^{2}+(\beta_{\chi})^2\leq- (1-c(\mu)^2)\alpha_{\chi}^{2}< 0.
\ee
\eei
\end{proposition}

In fact, the integral identity \eqref{eq:section-8-causal-krl-mass} yields a sharper global integrability bound under the same sign hypothesis:
\be
\label{eq:positive-krl-global-mass}
\|\alpha_{\chi}(\mu,\cdot)\|_{L^1} \, = \Numb(\mu), \qquad \|\beta_{\chi}(\mu,\cdot)\|_{L^1}\leq c(\mu)\Numb(\mu), \qquad \mu>0.
\ee
The same bounds hold near vacuum for every admissible pressure, even when the sign condition fails at other densities. Indeed, the vacuum conditions~\eqref{eq:intro-vacuum-B}--\eqref{eq:intro-vacuum-r} imply \(\mu \eps^{2}p''/\eps^{2}p'\to\gamma-1\),  by differentiating \(\eps^{2}p'=B(1+\eps^{2}p/\mu)^2\) and using the weighted remainder bounds. Since \(\eps^{2}p/\mu,\eps^{2}p'\to0\), we obtain
\be
\label{eq:local-vacuum-krl-positivity}
\frac{h\eps^{2}p''+2\eps^{2}p'(1-\eps^{2}p')}{\eps^{2}p'} \to \gamma+1>0 \qquad(\mu\to 0).
\ee
We can therefore apply the construction of the non-negative kernels on an interval \(0<\mu\leq\mu_{\rm vac}\). Uniqueness identifies it with the Fourier kernel on this interval, and~\eqref{eq:positive-krl-global-mass} gives \(\|\beta_{\chi}(\mu)\|_{L^1}\leq c(\mu)\Numb(\mu) =\Ocal(\mu^{(\gamma+1)/2})\). This local conclusion does not require a sign condition at larger densities.

\begin{proof}
\bse
\noindent{{\bf 1.}} We write \(h= \mu+\eps^{2}p(\mu)\), let \(\Numb\) be the particle density, and use the distance \(y=k(\mu)\), so that \(\del_y=(h/c)\del_\mu\). We introduce
\bel{eq:section-8-causal-normalized-components}
U_{\rm c} \coloneqq \frac{\alpha_{\chi}}{\Numb}, \qquad W_{\rm c} \coloneqq -\frac{\beta_{\chi}}{c\Numb}, \qquad R_\pm \coloneqq \frac12(U_{\rm c}\pm W_{\rm c}).
\ee
The identities \(\Numb'/\Numb=h^{-1}\) and \eqref{eq:section-6-normalized-first-order-system} give \(\del_y U_{\rm c}=-\del_v W_{\rm c}\) and \(\del_y W_{\rm c}=-\del_v U_{\rm c}-(c^{-1}-c+c_y/c)W_{\rm c}\). Since \(c_y=h\eps^{2}p''/(2c^2)\), we define the \emph{coupling rate} by
\bel{eq:section-8-causal-switching-rate}
\aligned
\Lambda_{\rm c}(\mu) 
& \coloneqq \frac{h(\mu)\eps^{2}p''(\mu)+2(1-c(\mu)^2)c(\mu)^2}{4c(\mu)^3} 
\\
& = \frac{(\mu+\eps^{2}p(\mu))\eps^{2}p''(\mu)+2\bigl(1-\eps^{2}p'(\mu) \bigr)\eps^{2}p'(\mu)} {4\eps^{3}(p'(\mu))^{3/2}} >0.
\endaligned
\ee
The genuine nonlinearity condition~\eqref{hyperbolic-eos-3} gives its positivity. Adding and subtracting the equations for \(U_{\rm c}\) and \(W_{\rm c}\) therefore gives the first-order system
\bel{eq:section-8-causal-coeaie-system}
\aligned (\del_y+ \del_v)R_+&= \Lambda_{\rm c}(R_--R_+), \qquad (\del_y-\del_v)R_-&= \Lambda_{\rm c}(R_+-R_-), \endaligned
\ee
The vacuum normalization gives \(R_+|_{y=0} \, =R_-|_{y=0} \, = \frac12\delta_{v=s}\), by the equation~\eqref{eq:section-3-particle-density}. To accommodate the singular switching rate at the vacuum boundary, we replace \(\Lambda_{\rm c}\) by \(\min(\Lambda_{\rm c},M)\) and the Dirac mass by a non-negative mollifier. Integration along the two characteristic families yields a positive Volterra operator, and Picard iteration preserves the cone \(R_\pm\geq0\).
\ese

\vskip.2cm

\noindent{{\bf 2.}}  Conservation of the total integral of $R_++R_-$ gives a uniform
bound on these positive measures, which we use to pass to the
limit in the two regularizations.  Integrating the two regularized equations in \(v\), using finite propagation, and adding them gives
\bel{eq:section-8-causal-krl-mass}
\frac{d}{dy}\int_\Rbb(R_++R_-)(y,v)\,dv=0, \qquad \int_\Rbb(R_++R_-)(y,v)\,dv=1.
\ee
\bse
Positivity gives a uniform total-variation bound for each regularized measure. Moreover, if \(U_M=R_{+,M}+R_{-,M}\) and \(W_M=R_{+,M}-R_{-,M}\), then
\be
\del_y U_M=-\del_v W_M, \qquad \bigl|\frac{d}{dy}\langle U_M,\varphi\rangle\bigr| \leq\|\varphi'\|_{L^\infty}, \qquad \varphi\in C_c^1(\Rbb).
\ee
This bound is independent of the truncated rate. After removal of the mollifier, the support is contained in \(|v-s|\leq y\); hence we obtain
\be
|\langle U_M(y),\varphi\rangle-\varphi(s)| \leq y\|\varphi'\|_{L^\infty}.
\ee
For the other component, the difference equation is \(\del_yW_M=-\del_vU_M-2\Lambda_MW_M\), where \(\Lambda_M= \min(\Lambda_{\rm c},M)\). On each interval \([\auxeps,Y]\) with \(\auxeps>0\), it gives
\be
\bigl|\frac{d}{dy}\langle W_M,\varphi\rangle\bigr| \leq\|\varphi'\|_{L^\infty} +2\|\Lambda_{\rm c}\|_{L^\infty([\auxeps,Y])}\|\varphi\|_{L^\infty}.
\ee
The uniform total-variation bounds and equicontinuity estimates allow us to apply the Arzel\`a--Ascoli theorem to the integrals
against each smooth compactly supported test function. A diagonal argument then selects a single subsequence along which
both components converge weakly as measures, uniformly on compact positive-density intervals. We first obtain this convergence for
a countable dense collection of test functions; the uniform total-variation bounds extend it to every test function.

Both components converge weakly as measures, uniformly in \(y\in[\auxeps,Y]\); the bound for \(U_M\) extends its convergence to \(y=0\), with trace \(\delta_s\). On every such positive interval the truncated rates eventually equal \(\Lambda_{\rm c}\), so passage to the integrated equations gives the original first-order system and continuous measure-valued traces on every positive slice. Setting \(\alpha_{\chi}= \Numb U\), we recover the scalar current equation. Since \(\Numb/\mu\to1\) and the support shrinks to \(s\), its Fourier transform satisfies \(\widehat{\alpha_{\chi}}/\mu\to e^{-i\xi s}\). The derivative normalization follows from the first-order equations as well. In fact, the integrals of the two regularized components remain equal to $1/2$: their difference solves the scalar equation \(d'(y) =-2\Lambda_M(y)d(y)\) with \(d(0) =0\). After taking the limit, \(\int W(y,v)\,dv=0\), and the shrinking support gives \(|\langle W(y),\varphi\rangle|\leq y\|\varphi'\|_\infty\). In the density variable, \(\del_\mu\alpha_{\chi}=(\alpha_{\chi}+ \del_v\beta_{\chi})/h=(\Numb/h)(U-c\del_vW)\). Hence we obtain
\be
\label{eq:positive-krl-vacuum-derivative}
\bigl|\langle\del_\mu\alpha_{\chi},\varphi\rangle-\varphi(s)\bigr| \leq\bigl|\frac{\Numb}{h}-1\bigr|\|\varphi\|_\infty + \frac{\Numb}{h}(k(\mu)+c(\mu))\|\varphi'\|_\infty.
\ee
Since \(\Numb/h\to1\), \(k(\mu)\to0\), and \(c(\mu)\to0\), the right-hand side tends to zero. Thus the positive construction satisfies both \(\alpha_{\chi}/\mu\to\delta_s\) and \(\del_\mu\alpha_{\chi}\to\delta_s\) in the sense of distributions. For each real frequency \(\xi\), compact support permits Fourier transformation of the scalar equation and of these vacuum limits. The transformed density solves the ODE in \autoref{thm:olver-vacuum}, with normalization \(\widehat{\alpha_{\chi}}/\mu\to e^{-i\xi s}\). Uniqueness in \autoref{thm:olver-vacuum}  gives equality with the weak kernel. The first-order system determines the rapidity derivative of its flux; compact support excludes an additive constant in rapidity and  determines the flux uniquely. This proves the non-negativity of the kernel measures independently of the conormal remainder argument.
\ese

\vskip.2cm

\bse
\noindent{{\bf 3.}} Proposition~\ref{cor:kernel-Sobolev-regularity} shows that \(\alpha_{\chi}\) has an \(L^2\) density on every positive-density slice. The measure inequalities \(c\alpha_{\chi}\pm\beta_{\chi}\geq0\) then imply that \(\beta_{\chi}\) also has an \(L^2\) density, with \(|\beta_{\chi}|\leq c\alpha_{\chi}\). We first justify the continuous representatives used in the pointwise inequalities. We use translation invariance in rapidity to set \(s=0\), fix an interior point \((y,v)\), and choose \(0<\auxeps<(y-|v|)/2\). At time \(\auxeps\), each positive measure \(R_+(\auxeps)\) and \(R_- (\auxeps)\) is supported in \([-\auxeps,\auxeps]\) and has total integral \(1/2\). On \([\auxeps,y]\), the coefficient \(\Lambda_{\rm c}\) is continuous and strictly positive. We choose \(Y>y\) and put \(L_\auxeps\coloneqq\max_{[\auxeps,Y]}\Lambda_{\rm c}\). In the characteristic expansion from a point \(w\) on the slice \(\auxeps\), the contribution with exactly \(j\geq1\) changes of family has a transition density \(\mathcal H_j\). Solving for the last switching time contributes the Jacobian \(1/2\); the remaining \(j-1\) times lie in an ordered simplex. The non-negative attenuation factors are at most one, so we obtain
\be
\label{eq:positive-krl-interior-transition-bound}
0\leq\mathcal H_j(y,\auxeps;v-w) \leq\frac{L_\auxeps}{2}\frac{(L_\auxeps(y-\auxeps))^{j-1}}{(j-1)!}, \qquad \auxeps<y\leq Y,\quad j\geq1.
\ee
{Each density is continuous on \(|v-w|<y-\auxeps\). For \(j=1\), this follows directly from the continuous rate evaluated at the switching time; for \(j\geq2\), dominated convergence applies to the integrals over the remaining times, whose moving boundary has measure zero. The factorial bound makes the sum uniformly convergent. This expansion represents \(R_\pm\): the characteristic integral equations on \([\auxeps,Y]\) have at most one finite-measure solution with the prescribed initial traces, by Gronwall's inequality in total variation. The contribution with no change is supported on \(v=w\pm(y-\auxeps)\), outside a neighborhood of our chosen point for every \(|w|\leq\auxeps\). Integration of the continuous transition densities against \(R_\pm(\auxeps,dw)\) therefore gives continuous representatives of both components in that neighborhood. Varying the point proves interior continuity.}

\vskip.2cm

\noindent{{\bf 4.}} To prove strict positivity, we iterate the characteristic integral equations once and keep only the non-negative term containing one change of family. If a negative characteristic starts at \(z\), changes family at time \(r\), and reaches \((y,v)\) as a positive characteristic, then \(v=z+y+ \auxeps-2r\). The reverse change gives \(v=z+2r-y-\auxeps\). The change from the switching time to the terminal position has Jacobian \(1/2\). Consequently, the absolutely continuous parts containing exactly one change satisfy
\be
\aligned
R_+(y,v)&\geq\frac12e^{-\int_\auxeps^y\Lambda_{\rm c}(t)\,dt}
 \int_{-\auxeps}^{\auxeps}\Lambda_{\rm c} \Bigl( \frac{z+y+ \auxeps-v}{2} \Bigr) 
 \mathbf 1_{(\auxeps,y)} \Bigl( \frac{z+y+ \auxeps-v}{2} \Bigr) R_- (\auxeps,dz),
\\
R_- (y,v)&\geq\frac12e^{-\int_\auxeps^y\Lambda_{\rm c}(t)\,dt}
 \int_{-\auxeps}^{\auxeps}\Lambda_{\rm c} \Bigl( \frac{v-z+y+ \auxeps}{2} \Bigr) 
 \mathbf 1_{(\auxeps,y)} \Bigl( \frac{v-z+y+ \auxeps}{2} \Bigr) R_+(\auxeps,dz).
\endaligned
\ee
Our choice of \(\auxeps\) makes both indicators equal to one for every \(|z|\leq\auxeps\). Since each initial component has mass \(1/2\), these formulas give, separately for each sign,
\be
\label{eq:strict-acoustic-component-lower-bound}
R_\pm(y,v)\geq\frac14\min_{[\auxeps,y]}\Lambda_{\rm c} \exp \Bigl( -\int_\auxeps^y\Lambda_{\rm c}(t)\,dt \Bigr) >0 \quad\text{ near } (y,v).
\ee
Adding these estimates yields
\be
U_{\rm c}(y,v)\geq\frac12\min_{[\auxeps,y]}\Lambda_{\rm c} \exp \Bigl( -\int_\auxeps^y\Lambda_{\rm c}(t)\,dt \Bigr) >0 \quad\text{ near } (y,v).
\ee
Since \(\alpha_{\chi}= \Numb(R_++R_-)\) and \(c\alpha_{\chi}\pm\beta_{\chi}=2c\Numb R_\mp\), the bounds above give strict positivity and the inequality~\eqref{eq:section-8-causal-krl-bound}. On a compact subset of the open cone, the distance from its sides, the minimum of \(\Lambda_{\rm c}\) on the required positive intervals, and the minima of \(c,\Numb\) are uniformly positive. A finite covering therefore bounds \(\alpha_{\chi}\) and \(c\alpha_{\chi}\pm\beta_{\chi}\) uniformly below by a positive constant. The inequalities \(R_\pm\geq0\) give \(U_{\rm c}\geq|W_{\rm c}|\), hence \(\alpha_{\chi}\geq0\) and \(|\beta_{\chi}|\leq c\alpha_{\chi}\). Since \(0<c<1\), strict interior positivity yields the Lorentzian estimate~\eqref{eq:section-8-krl-current-timelike}.
\ese
\end{proof}

\section{Weak and strong density-flux kernels}
\label{section-flux-kernels}
\label{section-11}

\subsection{The weak flux kernel}
\label{section-11-1}

\subsubsection{From the density kernel to the current vector}

By the Fourier formula~\eqref{eq:tau-recovery-sec9}, we already defined the weak flux kernel $\beta_{\chi}$ in the rest-frame. We now study it further and also construct the strong flux kernel. Parameter dependence is postponed to~\cite{TorusT2-3-asymptotic-2026}. Recall that the density and flux satisfy the system~\eqref{eq:section-6-normalized-first-order-system}:
\bel{eq:section-9-normalized-krl-flux-system}
h(\mu)\del_\mu\alpha_{\chi}-\alpha_{\chi}= \del_v\beta_{\chi},\qquad h(\mu)\del_\mu\beta_{\chi}-c(\mu)^2\beta_{\chi}=c(\mu)^2\del_v\alpha_{\chi}.
\ee
For the vector-valued kernel $\chi$, the convention~\eqref{eq:section-3-rest-frame-scalars} gives $\alpha_\chi=-g(\chi,\overline u)$ and $\beta_\chi=g(\chi,\overline u^\perp)$, where $g$ denotes the Minkowski metric.
 The transformation~\eqref{eq:section-9-fixed-frame-current-by-convolution}, below, gives its laboratory components. Eliminating $\alpha_\chi$ from the system~\eqref{eq:section-9-normalized-krl-flux-system} gives the equation for $\beta_\chi$.
\begin{lemma}[Equation for the rest-frame flux component]
\label{prop:section-9-spatial-component-equation}
The weak rest-frame  flux kernel satisfies
\bel{eq:section-9-shifted-flux-equation}
\del_\mu^2\beta_{\chi}-\frac{c(\mu)^2}{h(\mu)^2}\del_v^2\beta_{\chi}-\frac{\eps^{2}p''(\mu)}{\eps^{2}p'(\mu)}\del_\mu\beta_{\chi}+ \frac{c(\mu)^2}{h(\mu)^2}\beta_{\chi}=0.
\ee
The  flux kernel is uniquely determined by the first equation in \eqref{eq:section-9-normalized-krl-flux-system} and compact support in the rapidity variable. Its  vacuum limit is zero. The second-order equation alone does not specify this normalization.
\end{lemma}

\begin{proof}
The second equation in \eqref{eq:section-9-normalized-krl-flux-system} gives \(\del_v\alpha_{\chi}=(h/c^2)\del_\mu\beta_{\chi}-\beta_{\chi}\). We differentiate this identity in \(\mu\) and the first equation in \(v\). Substitution gives
\be
\frac{h^2}{c^2}\del_\mu^2\beta_{\chi} + \Bigl( h \Bigl( \frac{h}{c^2} \Bigr) '-h-\frac{h}{c^2} \Bigr) \del_\mu\beta_{\chi} + \beta_{\chi}-\del_v^2\beta_{\chi}=0.
\ee
After multiplication by \(c^2/h^2\), the coefficient of \(\del_\mu\beta_{\chi}\) is \((h'-c^2-1)/h- (c^2)'/c^2=-\eps^{2}p''/\eps^{2}p'\), proving the equation~\eqref{eq:section-9-shifted-flux-equation}. Two compactly supported  flux kernels associated with the same \(\alpha_{\chi}\) have zero difference derivative in \(v\); their difference is a  constant in rapidity, which compact support forces to vanish. The  zero vacuum limit follows from the first-order construction below. For comparison, the  rapidity-independent function \(h/\Numb\) solves the second-order equation and tends to one at vacuum, so its exclusion is necessary.
\end{proof}
Fourier transformation in rapidity gives \(\widehat{\beta_{\chi}}\) through the first equation of~\eqref{eq:section-9-normalized-krl-flux-system}:
\bel{eq:section-9-first-order-flux-system} 
h(\mu)\,\del_\mu\widehat{\alpha_{\chi}}(\mu,\xi)=\widehat{\alpha_{\chi}}(\mu,\xi)+i\xi\,\widehat{\beta_{\chi}}(\mu,\xi).
\ee 
For $\xi\neq0$, this identity determines $\widehat{\beta_\chi}$ by dividing
$h\partial_\mu\widehat{\alpha_\chi}-\widehat{\alpha_\chi}$ by $i\xi$, as in~\eqref{eq:tau-recovery-sec9}. As proved in Section~\ref{section-7}, the numerator is $\mathcal O(\xi^2)$ near zero, so the quotient extends
continuously to $\xi=0$ with value zero. Thus the flux is determined by the weak density kernel
already constructed.

\subsubsection{Vanishing flux component at vacuum}

Near vacuum, we improve the bound
$\|\beta_\chi(\mu)\|_{H^{-2}}\leq C\mu$
to $\|\beta_\chi(\mu)\|_{H^{-2}}\leq C\mu^\gamma$,
with $\gamma>1$, so the flux component vanishes faster
than linearly in density.
Indeed, let us prove that there exists  a threshold \(\mu_{\rm vac}>0\) such that
\be
\label{eq:flux-vacuum-improved-bound}
\|\beta_{\chi}(\mu)\|_{H^{-2}} + \mu\|\del_\mu\beta_{\chi}(\mu)\|_{H^{-2}} \leq C\mu^\gamma, \qquad 0<\mu\leq\mu_{\rm vac}.
\ee
First, \(c^2/h\leq C\mu^{\gamma-2}\) near vacuum. The bound~\eqref{eq:uniform-vacuum-traces} gives \(\|\del_v\alpha_{\chi}\|_{H^{-2}}\leq C\mu\). Formula~\eqref{eq:tau-recovery-sec9} gives \(\|\beta_{\chi}\|_{H^{-2}}\leq C\mu\) as well: for \(|\xi|\geq1\), its numerator is bounded by \(C\mu(1+|\xi|)\). For \(|\xi|\leq1\), we apply the vacuum Volterra formula to \(\widehat{\alpha_{\chi}}(\mu,\xi)-\Numb(\mu)\), whose equation has forcing \(-\xi^2\eps^{2}p'\widehat{\alpha_{\chi}}/h^2\). After decreasing \(\mu_{\rm vac}\), the integral \(\int_0^\mu r\eps^{2}p'(r)/h(r)^2\,dr\) can be absorbed, giving \(|\widehat{\alpha_{\chi}}-\Numb|\leq C\xi^2\mu^\gamma\) and \(|\del_\mu\widehat{\alpha_{\chi}}-\Numb'|\leq C\xi^2\mu^{\gamma-1}\). Since \(h\Numb'=\Numb\), the numerator is \(\Ocal(\mu^\gamma\xi^2)\), uniformly down to vacuum. The second current equation and  these two bounds therefore give \(\|\del_\mu\beta_{\chi}\|_{H^{-2}}\leq C\mu^{\gamma-1}\). This is integrable at zero.  More explicitly, \(h'=1+\eps^{2}p'\) and \(\Numb'=\Numb/h\) give \(\del_\mu\log(h/\Numb)=\eps^{2}p'/h\). Multiplication of the second current equation by \(\Numb/h^2\) and integration from zero therefore yield
\be
\label{eq:weak-flux-vacuum-integral}
\beta_{\chi}(\mu,\cdot)=\frac{h(\mu)}{\Numb(\mu)} \int_0^\mu\frac{\Numb(r)\eps^{2}p'(r)}{h(r)^2}\,\del_v\alpha_{\chi}(r,\cdot)\,dr \qquad\text{in }H^{-2}(\Rbb).
\ee
{The lower boundary term is zero since \(\Numb/h\to1\) and \(\|\beta_{\chi}(\mu)\|_{H^{-2}}\leq C\mu\). The integrand has \(H^{-2}\) norm at most \(Cr^{\gamma-1}\), and \(h/\Numb\) is bounded near zero. This proves the bound \(C\mu^\gamma\) in the estimate~\eqref{eq:flux-vacuum-improved-bound}; the derivative bound follows from the same current equation.} 

Consequently, the weak current has the following vacuum normalization. The estimate~\eqref{eq:flux-vacuum-improved-bound} gives
$\|\beta_\chi(\mu)/\mu\|_{H^{-2}}\leq C\mu^{\gamma-1}$,
which tends to zero since $\gamma>1$. Combining this with
$\alpha_\chi(\mu)/\mu\to\delta_0$, we obtain
\be
\frac{1}{\mu}\bigl(\alpha_\chi(\mu,\cdot),
\beta_\chi(\mu,\cdot)\bigr)
\longrightarrow(\delta_0,0)
\qquad\text{as }\mu \to 0.
\ee
(Here the kernel is centered at zero rapidity.)

\subsection{The strong flux kernel}
\label{section-strong-flux-kernel}

\subsubsection{The odd flux primitive}

The selected strong density kernel determines a corresponding  \emph{flux component} in the space of tempered distributions, as follows. We set
\be
A^{\rm strong}(\mu,\cdot) \coloneqq h(\mu)\del_\mu\alpha_{\chi}^{\rm strong}(\mu,\cdot) -\alpha_{\chi}^{\rm strong}(\mu,\cdot).
\ee
This distribution is even in rapidity, since the weak branch and the two propagators in the equation~\eqref{eq:strong-Cauchy-selection} are even in frequency, so the selected strong branch and its inverse Fourier transform are even as well. For a Schwartz test function \(\varphi\), we define
\be
\label{eq:strong-flux-test-primitive}
(\mathcal P\varphi)(v) \coloneqq \frac12\int_{-\infty}^v\bigl(\varphi(w) -\varphi(-w) \bigr)\,dw.
\ee
The integrand is odd with zero integral, so \(\mathcal P\varphi\) is an even Schwartz function. On the positive half-line we integrate from \(v\) to \(+\infty\): each weighted supremum is bounded by one of \(\varphi\) with two additional powers. Reflection gives the negative half-line; positive-order derivatives are linear combinations of derivatives of \(\varphi\). Thus \(\mathcal P\) is continuous on Schwartz space, and the unique odd primitive of \(A^{\rm strong}\) is defined by
\be
\label{eq:strong-flux-dual-primitive}
\langle\beta_{\chi}^{{\rm strong}},\varphi\rangle \coloneqq -\langle A^{\rm strong},\mathcal P\varphi\rangle.
\ee
Indeed, \(\mathcal P(\varphi')\) is the even part of \(\varphi\), and hence we obtain
\bel{eq:strong-flux-primitive}
\del_v\beta_{\chi}^{{\rm strong}} \, =A^{\rm strong}.
\ee
Two primitives differ by a constant, and oddness forces that constant to vanish. Away from zero frequency, the same definition reads
\bel{eq:strong-flux-fourier}
\widehat{\beta_{\chi}^{{\rm strong}}}(\mu,\xi) = \frac{h(\mu)\del_\mu\widehat{\alpha_{\chi}^{\rm strong}}(\mu,\xi) -\widehat{\alpha_{\chi}^{\rm strong}}(\mu,\xi)}{i\xi}, \qquad \xi\ne0,
\ee
which, after inverse Fourier transform, can be taken as the definition of the strong flux kernel $\beta_{\chi}^{{\rm strong}}$. Strictly speaking, formula~\eqref{eq:strong-flux-dual-primitive} provides the distributional extension across zero frequency.

\subsubsection{Properties of the strong spacetime kernel}

\begin{proposition}[The strong density-flux current kernel]
\label{prop:strong-entropy-flux-kernel}
At every \(\mu>0\), the strong density and flux satisfy the following system in the sense of tempered distributions:
\bel{eq:strong-krl-current-system}
h\del_\mu\alpha_{\chi}^{\rm strong}-\alpha_{\chi}^{\rm strong}= \del_v\beta_{\chi}^{{\rm strong}},\qquad h\del_\mu\beta_{\chi}^{{\rm strong}}-c^2\beta_{\chi}^{{\rm strong}}=c^2\del_v\alpha_{\chi}^{\rm strong}.
\ee
Recall that $d_{\rm F}=\nu+3/2=\lambda+2$ denotes the exponent controlling polynomial growth in the Fourier frequency
in~\eqref{eq:strong-global-density-bound}.
\bei

\item 
For every \(\ell>d_{\rm F}+5/2\) and every \(\delta>0\), there exists \(C_{\ell,\delta}(\eps)\), independent of density, such that
\be
\label{eq:strong-flux-global-test-bound}
|\langle\beta_{\chi}^{{\rm strong}}(\mu),\varphi\rangle| +|\langle\mu\del_\mu\beta_{\chi}^{{\rm strong}}(\mu),\varphi\rangle| \leq C_{\ell,\delta}(\eps)(1+\mu)^{1/(1+\eps^2k_\infty^2)+\delta}\|\mathcal P\varphi\|_{H^{\ell}}, \qquad \varphi\in\mathcal S(\Rbb).
\ee
\item 
The density component  converges to \(\delta_0\) at vacuum, and the  flux component satisfies
\bel{eq:strong-flux-vacuum-trace}
\beta_{\chi}^{{\rm strong}}(\mu,\cdot) \to -\frac12\operatorname{sgn}(v), \qquad \mu\del_\mu\beta_{\chi}^{{\rm strong}}(\mu,\cdot) \to 0 \quad\text{in }\mathcal S'(\Rbb)\quad(\mu\to 0).
\ee
\eei
\noindent Both components and their first density derivatives have finite distributional order on every compact positive-density interval, uniformly in density.
\end{proposition}

\begin{proof}

{\bf 1.} \emph{The first-order system.} The first current equation holds by definition; to prove the second, we differentiate in \(v\) the difference between the two sides of the second equation and use the scalar current equation for \(\alpha_{\chi}^{\rm strong}\). Since \(h'=1+c^2\), we obtain
\be
\label{eq:strong-flux-equation-cancellation}
\del_v\bigl(h\del_\mu\beta_{\chi}^{{\rm strong}}-c^2\beta_{\chi}^{{\rm strong}}-c^2\del_v\alpha_{\chi}^{\rm strong}\bigr)=h^2\del_\mu^2\alpha_{\chi}^{\rm strong}+c^2\alpha_{\chi}^{\rm strong}-c^2\del_v^2\alpha_{\chi}^{\rm strong}=0.
\ee
The difference is constant in rapidity and odd, hence zero. Differentiation of the primitive is justified by the continuity of \(\mathcal P\) and the density regularity in  \autoref{prop:section-8-selected-strong-branch}.

\vskip.2cm

\bse
\noindent{\bf 2.} \emph{The global growth bound.} For the global bounds, we write \(f= \widehat{\alpha_{\chi}^{\rm strong}}\).  
Since \(h'=1+c^2\) and \(f_{\mu\mu} \, =-(1+\xi^2) B(\mu)\mu^{-2}f\), we have
\be
\label{eq:strong-flux-primitive-density-derivative}
\widehat A^{\rm strong} \, =hf_\mu-f, \qquad \mu\del_\mu\widehat A^{\rm strong} =c^2\mu f_\mu-\frac{h}{\mu} (1+\xi^2) B(\mu)f.
\ee
The coefficients \(h/\mu\), \(c^2\), and \(B\) are bounded throughout the positive-density range. Thus~the estimate~\eqref{eq:strong-global-density-bound} gives
\be
|\widehat A^{\rm strong}|+|\mu\del_\mu\widehat A^{\rm strong}| \leq C_\delta(\eps)(1+\mu)^{1/(1+\eps^2k_\infty^2)+\delta}(1+\xi^2)^{d_{\rm F}/2+1}.
\ee
Integration with the Fourier weight  \((1+\xi^2)^{-\ell}\) yields
\be
\label{eq:strong-flux-primitive-Sobolev-bound}
\|A^{\rm strong}(\mu)\|_{H^{- \ell}}+ \|\mu\del_\mu A^{\rm strong}(\mu)\|_{H^{- \ell}} \leq C_{\ell,\delta}(\eps)(1+\mu)^{1/(1+\eps^2k_\infty^2)+\delta}, \qquad \ell>d_{\rm F}+5/2.
\ee
Duality in the equation~\eqref{eq:strong-flux-dual-primitive} proves the global estimate~\eqref{eq:strong-flux-global-test-bound}.
\ese

\vskip.2cm

\bse
\noindent{\bf 3.} \emph{The vacuum trace.} At vacuum, \(f\to1\), \(\mu f_\mu\to0\), \(h/\mu\to1\), and \(c^2,B\to0\). The two Fourier expressions in the equation~\eqref{eq:strong-flux-primitive-density-derivative} therefore tend to \(-1\) and zero, respectively. Their polynomial majorant is uniform for \(0<\mu\leq1\), so weighted dominated convergence gives
\be
A^{\rm strong}(\mu) \to -\delta_0, \qquad \mu\del_\mu A^{\rm strong}(\mu) \to 0 \quad\text{in }H^{- \ell}, \qquad \ell>d_{\rm F}+5/2.
\ee
The continuous odd-primitive map proves the equation~\eqref{eq:strong-flux-vacuum-trace}. In particular, the odd primitive of \(-\delta_0\) is \(-\tfrac12\operatorname{sgn}(v)\). The sign follows directly from the test-function formula:
\be
\label{eq:strong-flux-vacuum-sign-test}
-\langle-\delta_0,\mathcal P\varphi\rangle=(\mathcal P\varphi)(0)=-\frac12\int_\Rbb\operatorname{sgn}(v)\varphi(v)\,dv.
\ee
On compact positive-density intervals, the preceding bounds and the continuity of \(\mathcal P\) also give a finite distributional order uniform on each such interval.
\ese
\end{proof}

The bound~\eqref{eq:strong-flux-global-test-bound} uses Schwartz test functions: the sign-function vacuum limit has Fourier transform proportional to the \emph{Principal Value} distribution 
 \(\operatorname{pv}(1/\xi)\), so it belongs to no global \(H^{-\ell}(\Rbb)\). The bound controls the flux from vacuum to arbitrarily large density with the same isothermal growth weight as the strong density kernel.

The obstruction persists at every positive density: at zero frequency, the scalar equation and \(h'=1+c^2\) give \(\del_\mu\widehat A^{\rm strong}=(c^2/h)\widehat A^{\rm strong}\). Since \(\Numb'=\Numb/h\), we obtain
\be
\label{eq:strong-flux-zero-frequency}
\del_\mu\Bigl(\frac{\Numb}{h}\widehat A^{\rm strong}(\mu,0)\Bigr)=0, \qquad \widehat A^{\rm strong}(\mu,0)=-\frac{h(\mu)}{\Numb(\mu)}.
\ee
{The constant follows from \(\Numb/h\to1\) and \(\widehat A^{\rm strong}(\mu,0)\to-1\) at vacuum. The weak Volterra series in the equation~\eqref{eq:vacuum-plus-branch-sec9} and its differentiated series converge locally uniformly in frequency, since \(B(r)/r\) is integrable. Thus the weak branch is smooth and even in real frequency; ODE continuation and the positive denominator in the equation~\eqref{eq:strong-Cauchy-selection} give the same properties for the selected strong branch. In particular, \(\widehat A^{\rm strong}(\mu,\xi)-\widehat A^{\rm strong}(\mu,0)={\Ocal}_\mu(\xi^2)\). The Fourier transform of its odd primitive is therefore}
\be
\label{eq:strong-flux-principal-value}
\widehat{\beta_{\chi}^{{\rm strong}}}(\mu,\xi) =-\frac{h(\mu)}{\Numb(\mu)}\operatorname{pv}\frac{1}{i\xi} +\frac{\widehat A^{\rm strong}(\mu,\xi)-\widehat A^{\rm strong}(\mu,0)}{i\xi}.
\ee
{Here \(\operatorname{pv}\) denotes the symmetric principal value, and the second term extends continuously across zero frequency with value zero. Oddness excludes an additional Dirac mass there. The nonzero principal-value term cannot be represented by a locally square-integrable function of frequency, so \(\beta_{\chi}^{{\rm strong}}(\mu)\notin H^{-\ell}(\Rbb)\) for any \(\ell\). This explains the use of the estimate~\eqref{eq:strong-flux-global-test-bound} rather than a global Sobolev bound for the flux.}

\subsection{Behavior at the characteristic boundary}

\label{section-11-3}

\subsubsection{Leading characteristic expansion of the flux component}

We recall the acoustic distance \(k\) defined in the equation~\eqref{eq:standard-acoustic-distance}  and the leading term at the characteristic boundary  in the equation~\eqref{eq:section-8-acoustic-radius}:
\bel{eq:section-9-characteristic-distance}
\theta {=} v-s, \qquad G_b(\mu,\theta) = \bigl[k(\mu)^2-\theta^2 \bigr]_+^b.
\ee
The coefficients of the two leading terms are
\bel{eq:section-9-leading-current-coefficients}
a_1=c_{*,\lambda}k^{-\lambda}k'^{-1/2}>0, \qquad b_1= \frac{hk'}{k}a_1>0,
\ee
and the characteristic behavior in~\autoref{prop:distributional-leading-trace}, together with the first equation in~\eqref{eq:section-9-normalized-krl-flux-system}, gives the expansions
\bel{eq:section-9-one-term-current-boundary}
\alpha_{\chi}(\mu,\theta) =a_1(\mu)G_\lambda(\mu,\theta)+ \mathcal R_{\alpha_{\chi}}(\mu,\theta), \qquad \beta_{\chi}(\mu,\theta) =-\theta b_1(\mu)G_\lambda(\mu,\theta)+ \mathcal R_{\beta_\chi}(\mu,\theta),
\ee
and, whenever \(1<\gamma<3\), the remainders satisfy, on every compact interval \(I\Subset(0,+ \infty)\),
\bel{eq:section-9-one-term-current-remainders}
\frac{\mathcal R_{\alpha_{\chi}}(\mu,\theta)}{G_\lambda(\mu,\theta)} \to 0, \qquad \frac{\mathcal R_{\beta_\chi}(\mu,\theta)}{G_\lambda(\mu,\theta)} \to 0 \qquad\text{as }|\theta|\nearrow k(\mu),
\ee
uniformly for \(\mu\in I\). For any \(\gamma>1\), the corresponding assertion is the strong local \(L^2\) scaling limit
\be
\label{eq:current-remainders-leading-L2}
r^{-\lambda} \mathcal R_{\alpha_{\chi}}(\mu,\pm(k(\mu) -r\,\cdot)) \to 0, \qquad r^{-\lambda} \mathcal R_{\beta_\chi}(\mu,\pm(k(\mu) -r\,\cdot)) \to 0, \qquad r \to 0.
\ee
The convergence is uniform for $\mu\in I$. Proposition~\ref{prop:distributional-leading-trace} gives the density limit, and the proof below gives the flux limit. Both use $\lambda>-1/2$, which makes $q_+^\lambda$ locally square integrable. The next proposition derives the characteristic integral formula and the improved leading behavior of $c\alpha_{\chi}\pm\beta_{\chi}$; this gain concerns the combination, not either remainder separately.

\begin{proof}[Proof of the flux scaling limit]
\bse
To establish the distributional leading trace for both components, 
we define the Fourier amplitude $\mathcal A_\mu\coloneqq C_*\sqrt{2/\pi}/\sqrt{k'(\mu)}$. We first obtain the asymptotic formula
\be
\widehat{\beta_{\chi}}(\mu,\xi) =i c(\mu) \mathcal A_\mu\operatorname{sgn}(\xi)|\xi|^{-\lambda-1} \sin\bigl(k(\mu)|\xi|-\pi(\lambda+1)/2 \bigr)+o(|\xi|^{-\lambda-1}).
\ee
We use the differentiated remainder bound~\eqref{eq:Bessel-derivative-remainder}, rather than differentiating an undifferentiated asymptotic error. Since \(\widehat{\alpha_{\chi}}=\Psi/\sqrt{k'}\), that bound and the Bessel asymptotics give
\be
\label{eq:weak-current-differentiated-Fourier-profile}
\begin{pmatrix}
\widehat{\alpha_{\chi}}(\mu,\xi)\\
\dfrac{\del_\mu\widehat{\alpha_{\chi}}(\mu,\xi)}{k'(\mu) (1+\xi^2)^{1/2}}
\end{pmatrix}
=\mathcal A_\mu(1+\xi^2)^{-(\lambda+1)/2}
\begin{pmatrix}
\cos((1+\xi^2)^{1/2} k(\mu)-\pi(\lambda+1)/2)\\
-\sin((1+\xi^2)^{1/2} k(\mu)-\pi(\lambda+1)/2)
\end{pmatrix}
+o((1+\xi^2)^{-(\lambda+1)/2}).
\ee
{Both remainders are uniform on compact positive-density intervals. The derivative of \((k')^{-1/2}\) contributes \(\Ocal((1+\xi^2)^{-\lambda/2-1})\) to the second entry. Substitution into the equation~\eqref{eq:tau-recovery-sec9}, using \(hk'=c\), gives the announced sine asymptotic; replacing \((1+\xi^2)^{1/2}\) by \(|\xi|\) changes the phase by \(\Ocal(|\xi|^{-1})\).} The same estimates and continuity at \(\xi=0\) give \(|\widehat{\beta_{\chi}}|\leq C(1+|\xi|)^{-\lambda-1}\), hence \(\beta_{\chi}\in L^2\) since \(\lambda>-1/2\). The identity \(\del_\theta\beta_{\chi}=h\del_\mu\alpha_{\chi}-\alpha_{\chi}\) shows that \(\beta_{\chi}\) is constant on each exterior interval \(\theta>k(\mu)\) and \(\theta<-k(\mu)\). Both constants vanish by its \(L^2\) integrability.

The differentiation and integration argument of  \autoref{prop:distributional-leading-trace} applies to this sine asymptotic as well as to the cosine asymptotic of \(\alpha_{\chi}\). In the inward coordinate \(q=k-\theta\), it gives
\be
 r^{-\lambda}
 \begin{pmatrix}\alpha_{\chi}(\mu,k-r\,\cdot)\\ \beta_{\chi}(\mu,k-r\,\cdot)\end{pmatrix}
  \to  (2k)^\lambda
 \begin{pmatrix}a_1\\-ca_1\end{pmatrix}(\cdot)_+^\lambda
 \quad\text{in }\mathcal D'(\Rbb).
\ee
The local Sobolev argument in  \autoref{prop:distributional-leading-trace} applies to both Fourier asymptotic formulas, so this convergence is also strong in local \(L^2\). At the other boundary the sign of the second component is reversed, and \(ca_1=kb_1\).  With \(s_\pm=v\pm k(\mu)\), the identity \(\theta=v-s\) gives
\be
G_\lambda(\mu,v-s) = \bigl((s-s_-)(s_+-s) \bigr)_+^\lambda.
\ee
This proves the flux limit~\eqref{eq:current-remainders-leading-L2} in local $L^2$; Proposition~\ref{prop:distributional-leading-trace} gives the density limit.
\ese
\end{proof}

\begin{proposition}[Leading behavior of the density-flux kernel at the characteristic boundaries]
\label{prop:section-9-conormal-current-remainder}
Suppose that the locally integrable density kernel has the locally uniform leading  term \(\alpha_{\chi}/G_\lambda\to a_1\) on a positive-density characteristic strip. 
\bei 

\item Then the distributional first-order current equations and support in the acoustic cone imply
\bel{eq:section-9-conormal-current-remainders}
\mathcal R_{\alpha_{\chi}}=o(G_\lambda), \qquad \mathcal R_{\beta_\chi}=o(G_\lambda).
\ee
This conclusion requires no estimate of either remainder one order beyond its leading scaling. Uniformity holds for families with uniform leading  terms and uniform coefficient bounds on the strip.

\item Near the right characteristic boundary, the leading terms of $c\alpha_\chi$ and $\beta_\chi$ cancel in their sum, whereas near the left boundary they cancel in their difference.
For every admissible pressure and every $\gamma>1$, one determines the resulting boundary behavior by rescaling the inward distance by $r$ and dividing these combinations by $r^{\lambda+1}$;
the limits below hold in $L^2$. Write \(\mu(y) = k^{-1}(y)\), use \(c(y) =c(\mu(y))\), and set \(\mathcal C_\pm\coloneqq c\alpha_{\chi}\pm\beta_{\chi}\). For every compact interval \(I\Subset(0,+ \infty)\) of acoustic distances and every \(L>0\), one has
\be
\label{eq:rel-characteristic-combination-scaling}
r^{-\lambda-1}\mathcal C_\pm\bigl(\mu(y),\pm(y-rq) \bigr) \to \frac{(2y)^\lambda a_1(\mu(y))}{2(\lambda+1)} \bigl(c_y(y)+1-c(y)^2 \bigr)q_+^{\lambda+1} \quad\text{in }L^2(I\times(-L,L))\quad(r \to 0).
\ee

\item 
The coefficient has the sign of the genuine-nonlinearity expression, since
\be
\label{eq:rel-characteristic-coefficient-sign}
c_y+1-c^2= \frac{h\eps^{2}p''+2\eps^{2}p'(1-\eps^{2}p')}{2\eps^{2}p'}.
\ee
This limit uses only the leading trace of the density kernel, including when that trace is available in the local \(L^2\) sense rather than pointwise.
\eei
\end{proposition}

\begin{proof}
\bse
{\bf 1.} \emph{Integration from a characteristic boundary.} We put \(y=k(\mu)\), \(\theta=v-s\), and \(\mathcal C_+= \beta_{\chi}+c\alpha_{\chi}\). The first-order current equations give
\be
\label{eq:section-9-D-remainder-current-identity}
(\del_y-\del_\theta)\mathcal C_+=c\mathcal C_++(c_y+1-c^2)\alpha_{\chi}.
\ee
In the characteristic coordinates \(a=y+ \theta\), \(b=y-\theta\), this becomes \(2\del_b\mathcal C_+=c\mathcal C_++(c_y+1-c^2)\alpha_{\chi}\). The kernels vanish for \(b< 0\), and the equation holds distributionally across \(b=0\). Multiplication by the corresponding integrating factor therefore gives
\be
\label{eq:section-9-D-joint-remainder-trace}
\mathcal C_+(a,b) = \frac12\int_0^b \exp \Bigl( \frac12\int_t^b c((a+r)/2)\,dr \Bigr) (c_y+1-c^2)((a+t)/2)\alpha_{\chi}(a,t)\,dt,
\ee
where the right-hand side is set equal to zero for \(b< 0\). Indeed, the integral satisfies the same first-order equation as \(\mathcal C_+\). Their difference, multiplied by the integrating factor, has zero distributional derivative in \(b\) and vanishes on the exterior region. Since \(\lambda>-1\) and \(\alpha_{\chi}=\Ocal(G_\lambda)\), this representation gives \(\mathcal C_+=\Ocal(G_{\lambda+1})\). Consequently, \(\beta_{\chi}/G_\lambda\to-ca_1\). Since \(kb_1=ca_1\), this proves the flux trace in the estimate~\eqref{eq:section-9-one-term-current-boundary}. Reflection in \(\theta\) gives the other characteristic boundary. The density assertion follows from the leading-trace hypothesis of the proposition; all steps are uniform under its coefficient and trace bounds.

\ese

\vskip.2cm

\bse
\noindent{\bf 2.} \emph{The limiting profile.} For~\eqref{eq:rel-characteristic-combination-scaling}, we use the equation~\eqref{eq:section-9-D-joint-remainder-trace} together with the density scaling trace in  \autoref{prop:distributional-leading-trace}. On a slightly larger compact interval \(I'\), this trace gives
\be
\label{eq:rel-density-scaling-for-combination}
f_r(y,q) \coloneqq r^{-\lambda}\alpha_{\chi}(\mu(y),y-rq) \to a_1(\mu(y))(2y)^\lambda q_+^\lambda \quad\text{in }L^2(I'\times(-L,L)).
\ee
We substitute \(a=2y-rq\), \(b=rq\), and \(t=r\theta\) in the integral for \(\mathcal C_+\). After division by \(r^{\lambda+1}\), its density factor is \(f_r(y+r(\theta-q)/2,\theta)\). The exponential factor tends uniformly to one, and the other coefficient tends uniformly to \(c_y(y)+1-c(y)^2\), for \(0\leq\theta\leq q\leq L\) and \(y\in I\). Cauchy--Schwarz in \(\theta\), followed by integration in \(y,q\), bounds this integral operator from \(L^2(I'\times(0,L))\) to \(L^2(I\times(0,L))\), independently of small \(r\). It therefore sends the difference in the equation~\eqref{eq:rel-density-scaling-for-combination} to zero. For the limiting profile, continuity of the coefficients and \(\theta^\lambda\in L^2(0,L)\) justify passage to the limit under the same translations. The resulting integral is
\be
\frac12\bigl(c_y+1-c^2 \bigr)a_1(\mu(y))(2y)^\lambda \int_0^q\theta^\lambda\,d\theta = \frac{(2y)^\lambda a_1(\mu(y))}{2(\lambda+1)} \bigl(c_y+1-c^2 \bigr)q^{\lambda+1}.
\ee
Outside the cone both kernels vanish; reflection proves the assertion for \(\mathcal C_-\) at the other boundary. Finally, \(k'=c/h\) gives \(c_y=hc'/c=h\eps^{2}p''/(2\eps^{2}p')\), proving the equation~\eqref{eq:rel-characteristic-coefficient-sign}.

\ese
\end{proof}

\subsubsection{Lorentzian reconstruction}

The density $\alpha_{\chi}$ and flux $\beta_{\chi}$ reconstruct the current of~\autoref{prop:section-3-current-reconstruction} by an inverse Lorentz transformation. For every fixed \(\mu>0\),  the inverse Fourier transform is defined directly in the space of tempered distributions by
\bel{eq:inverse-fourier-tau-sec9}
\beta_{\chi}(\mu,\cdot) \coloneqq \mathcal F_{\xi\to v}^{-1}\bigl(\widehat{\beta_{\chi}}(\mu,\xi)\bigr)\qquad\text{in }\mathcal S'(\Rbb_v).
\ee
For a vacuum profile \(\entropygenerator\in\mathcal S(\Rbb)\), we first define
\bel{eq:section-9-normalized-current-by-convolution}
\alpha_{\entropygenerator}(\mu,v) \coloneqq \int_\Rbb\alpha_{\chi}(\mu,v-s)\entropygenerator(s)\,ds,\qquad \beta_{\entropygenerator}(\mu,v) \coloneqq \int_\Rbb\beta_{\chi}(\mu,v-s)\entropygenerator(s)\,ds.
\ee
The corresponding fixed-frame current is obtained by the inverse Lorentz transformation
\bel{eq:section-9-fixed-frame-current-by-convolution}
\begin{aligned}
F^{0}_{\entropygenerator}(\mu,v)&= \cosh(v)\alpha_{\entropygenerator}(\mu,v)+ \sinh(v)\beta_{\entropygenerator}(\mu,v),\\
F^{1}_{\entropygenerator}(\mu,v)&= \sinh(v)\alpha_{\entropygenerator}(\mu,v)+ \cosh(v)\beta_{\entropygenerator}(\mu,v).
\end{aligned}
\ee
The vacuum normalization determines the flux constant, and the equivalent formula~\eqref{eq:section-8-fixed-frame-current-from-normalized-profile}, below,  gives the current in the formal balance law~\eqref{eq:intro-ft-quasi-balance}; Section~\ref{section-9}, below, extends the construction beyond Schwartz profiles.

\subsection{Global estimates for the weak density-flux kernel} 
\label{section-fixed-current-bounds}

We also have bounds for the spatial component of the spacetime kernel as well. We use $a\coloneqq\lambda+1=1/2+1/(\gamma-1)>1/2$ and $k_{\min}=\min\{k,1\}$.

\begin{proposition}[Global bounds for the weak density-flux kernel]
\label{prop:fixed-current-bounds}
The weak density-flux kernel $\chi$ satisfies
\be
\label{eq:fixed-current-global-decay}
|\widehat{\alpha_{\chi}}(\mu,\xi)|+|\widehat{\beta_{\chi}}(\mu,\xi)|\leq C\mu(1+ (1+\xi^2)^{1/2} k_{\min}(\mu))^{-a},\qquad \mu>0,\quad\xi\in\Rbb.
\ee
For large density, the isothermal coefficient \(B_\infty=\eps^2k_\infty^2/(1+\eps^2k_\infty^2)^2\) gives, for every \(\delta>0\),
\bse
\be
\label{eq:fixed-current-isothermal-growth}
|\widehat{\alpha_\chi}(\mu,\xi)|+|\widehat{\beta_\chi}(\mu,\xi)|
\leq C_\delta(\eps)\sqrt\mu(1+\xi^2)^{-a/2}
\mu^{\sqrt{(1/4-(1+\xi^2)B_\infty)_+}+\delta},\qquad\mu\geq1.
\ee
For every \(\eta>0\), the range \(|\xi|\geq\xi_\infty+\eta\) satisfies the stronger bound
\be
\label{eq:fixed-current-isothermal-oscillatory}
|\widehat{\alpha_\chi}(\mu,\xi)|+|\widehat{\beta_\chi}(\mu,\xi)|
\leq C_\eta(\eps)\sqrt\mu(1+\xi^2)^{-a/2},\qquad\mu\geq1.
\ee
\ese
For $0\leq\ell<a-1/2$, both components belong to $H^\ell$, with the bound in~\eqref{eq:intro-kernel-Sobolev-bound}; at large density this improves to \(C_{\ell,\delta}(\eps)\mu^{1/(1+\eps^2k_\infty^2)+\delta}\). Extended by $\kernelpair{}(0)=0$, the pair belongs to $C^1([0,+\infty);H^{-1}(\Rbb)^2)$, with $\del_\mu\kernelpair{}(0)= (\delta_0,0)$. There are $C>0$ and $\mu_{\rm vac}>0$ such that
\be
\label{eq:fixed-current-vacuum-Hminusone}
\|\mu^{-1}\kernelpair{}(\mu)- (\delta_0,0)\|_{H^{-1}}+\|\del_\mu\kernelpair{}(\mu)- (\delta_0,0)\|_{H^{-1}}\leq C\mu^{(\gamma-1)/4},\qquad0<\mu\leq\mu_{\rm vac}.
\ee
\end{proposition}
\begin{proof}
{\bf 1.} \emph{Global coarse bounds.} Proposition~\ref{cor:kernel-Sobolev-regularity} gives the density bound and its differentiated version. For $|\xi|\geq1$, the recovery formula~\eqref{eq:tau-recovery-sec9} divides the differentiated density bound by $\xi$; the inequality $(1+(1+\xi^2)^{1/2} k_{\min})/(1+\xi^2)^{1/2}\leq2$ gives the same decay for the flux. For $|\xi|\leq1$, we integrate its first-order equation from vacuum. Its integrating factor from density $r$ to density $\mu$ is $G(\mu,r)\coloneqq(h/\Numb)(\mu)/(h/\Numb)(r)\leq2\sqrt{\mu/r}$, since $h\leq2\mu$ and $\Numb'/\Numb=1/h\geq1/(2\mu)$. The bound $|\widehat{\alpha_{\chi}}(r,\xi)|\leq Cr$ therefore gives $|\widehat{\beta_{\chi}}(\mu,\xi)|\leq C|\xi|\mu$. This proves the global Fourier bound at bounded frequencies, including the continuous zero value at $\xi=0$. Fourier integration gives the stated coarse Sobolev bounds.

\vskip.2cm

\noindent{\bf 2.} \emph{The isothermal regime.} For the isothermal refinements, the density and its derivative satisfy~\eqref{eq:weak-kernel-isothermal-growth} and~\eqref{eq:weak-kernel-isothermal-oscillatory}. The flux recovery formula gives the same estimates on every frequency range separated from zero, since \(h/\mu\) is bounded. To cover zero frequency, we use its first-order equation on \(\mu\geq1\). We choose a number \(\sigma_0\) strictly between \(\eps^2k_\infty^2/(1+\eps^2k_\infty^2)\) and \(1/2\). The limit \(\mu c^2/h\to\eps^2k_\infty^2/(1+\eps^2k_\infty^2)\) implies
\be
G(\mu,r)=\exp\Bigl(\int_r^\mu\frac{c(t)^2}{h(t)}\,dt\Bigr)
\leq C(\eps)\Bigl(\frac\mu r\Bigr)^{\sigma_0},\qquad\mu\geq r\geq1.
\ee
The integral formula is
\be
\widehat{\beta_\chi}(\mu,\xi)
=G(\mu,1)\widehat{\beta_\chi}(1,\xi)
+i\xi\int_1^\mu G(\mu,r)\frac{c(r)^2}{h(r)}\widehat{\alpha_\chi}(r,\xi)\,dr.
\ee
For \(|\xi|\leq1\), the density bound has power at least \(1/2+\delta>\sigma_0\), while \(c^2/h\leq C/r\). Integrating the corresponding power of \(r\) proves~\eqref{eq:fixed-current-isothermal-growth}, with a constant uniform in this bounded frequency interval. The initial term has the same bound. The supercritical range in~\eqref{eq:fixed-current-isothermal-oscillatory} is separated from zero and is already covered by recovery. Finally, the largest density exponent is \(1/(1+\eps^2k_\infty^2)+\delta\); Fourier integration yields the refined Sobolev estimate.

\vskip.2cm

\noindent{\bf 3.} \emph{Differentiation at vacuum.} We give the vacuum estimates to justify differentiation at the endpoint. With $k=k(\mu)\leq1$, one has $c^2\asymp k^2\asymp\mu^{\gamma-1}$. On $(1+\xi^2)^{1/2} k\leq1$, the Volterra equation gives $|\widehat{\alpha_{\chi}}/\mu-1|+|\widehat{\alpha_{\chi,\mu}}-1|\leq C (1+\xi^2) k^2$, while the integral flux equation and its derivative give $|\widehat{\beta_{\chi}}|/\mu+|\widehat{\beta_{\chi}}_\mu|\leq C|\xi|k^2$. On $(1+\xi^2)^{1/2} k\geq1$, the differentiated Bessel bound gives $|\widehat{\alpha_{\chi}}/\mu-1|+|\widehat{\alpha_{\chi,\mu}}-1|\leq C(1+(1+\xi^2)^{1/2} k)^b$, where $b\coloneqq\max\{0,1-a\}<1/2$. Recovery of the flux and its first-order equation give $|\widehat{\beta_{\chi}}|/\mu\leq Ck((1+\xi^2)^{1/2} k)^{-a}$ and $|\widehat{\beta_{\chi}}_\mu|\leq Ck((1+\xi^2)^{1/2} k)^{1-a}$. Splitting the squared $H^{-1}$ integral at $(1+\xi^2)^{1/2}=k^{-1}$ bounds the density contribution by $Ck$ and the flux contribution by $Ck^3$. Taking square roots proves~\eqref{eq:fixed-current-vacuum-Hminusone}. At positive density, continuity follows from the same Fourier majorants and the first-order equations. The endpoint limits identify the derivative at zero and prove the $C^1$ assertion.
\end{proof}

\begin{remark} With the notation introduced in the next section, it follows that for any vacuum profile $\vartheta\in H^1(\Rbb)$, duality and Lorentz reconstruction give
\be
\label{eq:fixed-Hone-current-bound}
|F^{0}_{\vartheta}(\mu,v)|+|F^{1}_{\vartheta}(\mu,v)|\leq C\mu\cosh v\,\|\vartheta\|_{H^1}.
\ee
Since $\mu\cosh v\leq\sqrt{\mu E}$ for $E=\mu\cosh^2v+\eps^2p\sinh^2v$, these currents tend to zero along bounded-energy states approaching vacuum. 
For $\vartheta\in H^1_{\loc}(\Rbb)$, a cutoff equal to one on the relevant compact support reduces density differentiation to $H^{-1}$--$H^1$ duality. Rapidity differentiation pairs the $L^2$ kernel with the $L^2$ derivative of the profile. The same Fourier majorants give continuity of these derivatives on compact positive-density state sets. Thus the current is $C^1$ there and its conservation law holds along $C^1$ positive-density Euler flows.
\end{remark}

\section{Representation of spacetime quasi-currents by vacuum profiles}
\label{section-9}
 
\subsection{The role of Lorentz covariance}
\label{section-9-1}

\subsubsection{Applying Lorentz boosts}

Using the kernels of Sections~\ref{section-6} to~\ref{section-flux-kernels}, we now represent spacetime quasi-currents by their vacuum profiles, as stated in Proposition~\ref{thm:param-weak-entropy-currents}, below. We also recover the reference currents and examine the growth of profiles toward the physical vacuum boundary. Throughout the construction, the normalized momentum parallel to the symmetry orbits is fixed. We recall
\be
\varpi=|\Jhatpar|^2,\qquad
m=\meff{\varpi}(\mu)
=\mu+\frac{\varpi\Numb(\mu)^2}{4h(\mu)},\qquad
p_\varpi(m)=p(\mu).
\ee
The effective pressure is admissible by Proposition~\ref{prop:fixed-effective-admissibility}. We write
\be
h_\varpi(m)=m+\eps^2p_\varpi(m),\qquad
c_\varpi(m)=\eps\sqrt{p_\varpi'(m)},\qquad
k_\varpi'(m)=\frac{c_\varpi(m)}{h_\varpi(m)},\qquad
k_\varpi(0)=0.
\ee
The associated particle density $\Neff{\varpi}$ is normalized by
\be
\frac{(\Neff{\varpi})'(m)}{\Neff{\varpi}(m)}
=\frac1{h_\varpi(m)},\qquad
\frac{\Neff{\varpi}(m)}m\longrightarrow1
\quad(m \to 0).
\ee
We also recall $D_{\varpi,0}=1+\varpi/4$, so that
$m/\mu\to D_{\varpi,0}$ at vacuum.

To display the parameter explicitly, we denote the fluid-frame coefficients of the weak kernel for $p_\varpi$ by
$\alpha_\chi(m,v-s;\varpi)$ and $\beta_\chi(m,v-s;\varpi)$.
The effective density $m$ is their density argument; the physical state is recovered by composition with $m=\meff{\varpi}(\mu)$. All estimates and limits in this section keep $\varpi$ fixed, and their constants may depend on this parameter.

Translation invariance and the vacuum normalization give
\bel{eq:section-8-normalized-krl-translation}
\alpha_\chi(m,v;s,\varpi)
=\alpha_\chi(m,v-s;\varpi).
\ee
Under a boost of rapidity $\theta$, the fluid rapidity and vacuum label become
$v^\theta=v-\theta$ and $s^\theta=s-\theta$, while $\mu$, $m$, and $\varpi$ are unchanged. The identity $v^\theta-s^\theta=v-s$ therefore leaves the scalar kernel coefficients unchanged, so boosts act on vacuum profiles by the \emph{translation representation}
\bel{eq:section-8-lorentz-action-on-traces}
(\mathcal T_\theta\entropygenerator)(s;\varpi)
\coloneqq\entropygenerator(s+\theta;\varpi).
\ee
The profile translation~\eqref{eq:section-8-lorentz-action-on-traces} and component transformation~\eqref{eq:section-3-lorentz-current} describe the same spacetime quasi-current.

We denote by $L(v)$ the reconstruction matrix in the equation~\eqref{eq:section-3-rest-frame-and-fixed-frame-components}. Acting on the column of fluid-frame coefficients, the change of variable $s=r+\theta$ gives
\be
\label{eq:trace-boost-current-equivariance}
\begin{aligned}
&L(v-\theta)\int_\Rbb
\binom{\alpha_\chi(m,v-\theta-r;\varpi)}
      {\beta_\chi(m,v-\theta-r;\varpi)}
\entropygenerator(r+\theta;\varpi)\,dr
\\
&=L(-\theta)L(v)\int_\Rbb
\binom{\alpha_\chi(m,v-s;\varpi)}
      {\beta_\chi(m,v-s;\varpi)}
\entropygenerator(s;\varpi)\,ds.
\end{aligned}
\ee
The identity $L(v-\theta)=L(-\theta)L(v)$ confirms the sign of the boost in the reconstructed quasi-current. Thus the profile must be translated together with the frame. Constant profiles are unchanged, whereas
$\mathcal T_\theta\cosh=\cosh\theta\cosh+\sinh\theta\sinh$:
\emph{the same energy current is represented in the new frame by a linear combination of the energy and momentum profiles}. Choosing the profile $\cosh$ again in that frame instead selects the energy current of the new observer.

\subsubsection{\texorpdfstring{Differentiation in a moving frame}{Differentiation in a moving frame}}

For a smooth point-dependent boost $\theta=\theta(x)$, the transformed profile is
$\entropygenerator^\theta(x,s;\varpi)
=\entropygenerator(s+\theta(x);\varpi)$.
Identity~\eqref{eq:trace-boost-current-equivariance} still holds pointwise, including when $\varpi$ varies with the state. Whenever the profiles and the resulting quasi-current family are sufficiently differentiable, we differentiate both the profile and the frame in this identity. To see the cancellation, we write the connection matrix of the oriented two-dimensional orthonormal frame as $\omega_{\rm fr}H$, where
$H\coloneqq\bigl(\begin{smallmatrix}0&1\\1&0\end{smallmatrix}\bigr)$,
and write $Q$ for the column of current components. We have
$Q^\theta=L(-\theta)Q$ and
$\omega_{\rm fr}^\theta=\omega_{\rm fr}+d\theta$, and obtain
\be
\label{eq:trace-boost-covariant-derivative}
dQ^\theta+\omega_{\rm fr}^\theta H Q^\theta
=L(-\theta)(dQ+\omega_{\rm fr}H Q).
\ee
This proves that the covariant divergence is independent of the chosen orthonormal frame.

\subsection{The representation formula}
\label{section-9-2}

\subsubsection{A main conclusion}

For a quasi-current expressed in the effective density and rapidity variables, its \emph{vacuum profile} is the value at $m=0$ of the derivative of its fluid-frame time coefficient with respect to $m$, with $v$ and $\varpi$ held fixed. The derivative with respect to the physical density $\mu$ has an additional factor $D_{\varpi,0}$.

\begin{proposition}[Representation of weak spacetime quasi-currents at fixed parallel momentum]
\label{thm:param-weak-entropy-currents}
Assume that $p=p(\mu)$ is an admissible equation of state in the sense of~\autoref{def-pression}, and fix $\varpi\geq0$. Every continuous profile
$\entropygenerator(\cdot;\varpi)\in\mathcal C^0(\Rbb)$
defines the fluid-frame density coefficient
\bel{eq:entropy-parametrization-sec10}
\alpha_{\entropygenerator}(m,v;\varpi)
=\int_\Rbb\alpha_\chi(m,v-s;\varpi)
\,\entropygenerator(s;\varpi)\,ds.
\ee
This coefficient is the unique distributional solution of the current equation~\eqref{eq:section-6-normalized-entropy-equation}, with $(\mu,p)$ replaced by $(m,p_\varpi)$, which is continuous up to vacuum and satisfies
\bel{eq:section-8-tempered-vacuum-trace}
\alpha_{\entropygenerator}(0,\cdot;\varpi)=0,\qquad
\del_m\alpha_{\entropygenerator}(0,\cdot;\varpi)
=\entropygenerator(\cdot;\varpi),
\ee
where both boundary values are understood in the sense of distributions. The spatial coefficient, defined by convolution in rapidity as
\bel{eq:weak-current-spatial-component-representation}
\beta_{\entropygenerator}(m,v;\varpi)
=\int_\Rbb\beta_\chi(m,v-s;\varpi)
\,\entropygenerator(s;\varpi)\,ds,
\ee
is continuous, vanishes at vacuum, and is the unique spatial coefficient with this vacuum value satisfying the first-order system~\eqref{eq:section-6-normalized-first-order-system}, with $(\mu,p)$ replaced by $(m,p_\varpi)$. Explicitly, one has 
\be
\begin{aligned}
h_\varpi\del_m\alpha_{\entropygenerator}
&=\alpha_{\entropygenerator}+\del_v\beta_{\entropygenerator},
\\
h_\varpi\del_m\beta_{\entropygenerator}
&=c_\varpi^2
  \bigl(\beta_{\entropygenerator}+\del_v\alpha_{\entropygenerator}\bigr).
\end{aligned}
\ee
These identities hold in the sense of distributions for continuous profiles and pointwise whenever the coefficients are $C^1$.
\end{proposition}

Lorentz reconstruction gives also the quasi-current in the physical state variables:
\bel{eq:section-8-fixed-frame-current-from-normalized-profile}
\begin{aligned}
F^0_{\entropygenerator}(\mu,v;\varpi)
&=\cosh(v)\alpha_{\entropygenerator}(m,v;\varpi)
 +\sinh(v)\beta_{\entropygenerator}(m,v;\varpi),
\\
F^1_{\entropygenerator}(\mu,v;\varpi)
&=\sinh(v)\alpha_{\entropygenerator}(m,v;\varpi)
 +\cosh(v)\beta_{\entropygenerator}(m,v;\varpi),
\qquad m=\meff{\varpi}(\mu).
\end{aligned}
\ee
The following observations are in order here.
\bei

\item 
Its time component vanishes at vacuum, and its density derivatives there are
\be
\del_m F^0_{\entropygenerator}\big|_{m=0}
=\cosh(v)\entropygenerator(v;\varpi),\qquad
\del_\mu F^0_{\entropygenerator}\big|_{\mu=0}
=D_{\varpi,0}\cosh(v)\entropygenerator(v;\varpi).
\ee
Here the derivative with respect to $m$ means differentiation after expressing the current in the effective density variable. Thus a prescribed fixed-frame profile $\entropygenerator_{\rm fix}$ relative to $m$ corresponds to
$\entropygenerator=\entropygenerator_{\rm fix}/\cosh$;
if that profile is prescribed relative to $\mu$, the corresponding relation is
$\entropygenerator=\entropygenerator_{\rm fix}/(D_{\varpi,0}\cosh)$.

\item For $\entropygenerator(\cdot;\varpi)\in H^1_{\loc}(\Rbb)$, the coefficients are $C^1$ in $(m,v)$ at positive density by the regularity proved after Proposition~\ref{prop:fixed-current-bounds}, applied to $p_\varpi$. Profiles of class $C^2_{\loc}$ give $C^2$ regularity in these variables through the first-order current equations. These assertions keep $\varpi$ fixed. The $C^1$ dependence on the parallel momentum is established in the companion article~\cite{TorusT2-3-asymptotic-2026}, under locally $C^1$ dependence of the profile with values in $H^1_{\loc}(\Rbb)$.

\item Whenever the resulting family is $C^1$ in all positive-density state variables, its composition with a regular plane-symmetric Euler flow satisfies
\be
\begin{aligned}
\del_tF^0_{\entropygenerator}
+\del_xF^1_{\entropygenerator}
&=\sum_{j=2}^n\Fvee_j\del_x\widehat J_j,
\qquad\quad
\Fvee_j
=\frac{2\widehat J_j}{\cosh v}
\left(
\del_\varpi\beta_{\entropygenerator}
+\frac{\Numb(\mu)^2}{4h(\mu)}
 \del_m\beta_{\entropygenerator}
\right)_{m=\meff{\varpi}(\mu)}.
\end{aligned}
\ee
The derivative $\del_\varpi\beta_{\entropygenerator}$ on the right keeps $(m,v)$ fixed and includes the dependence of the prescribed profile. The second term accounts for the variation of the effective density at fixed $(\mu,v)$.

\item More generally, the profile may depend on the full vector $\Jhatpar$, rather than only on $\varpi=|\Jhatpar|^2$. The same representation applies with
$\entropygenerator(s;\varpi)$ replaced by
$\entropygenerator(s;\Jhatpar)$, while the kernel depends on $\Jhatpar$ through $\varpi$. All arguments below apply at each fixed $\Jhatpar$, and the source factors of a $C^1$ family are then
\be
\Fvee_j
=\frac1{\cosh v}
\del_{\widehat J_j}
\left[
\beta_{\entropygenerator}
\bigl(\meff{|\Jhatpar|^2}(\mu),v;\Jhatpar\bigr)
\right]_{\mu,v\ {\rm fixed}}.
\ee
\eei

\begin{proof}

\noindent{\bf 1.} \emph{Existence and vacuum normalization.}
We fix $\varpi$ throughout the proof. Compact support and local $L^2$, hence $L^1$, continuity of the effective kernel on compact positive-density intervals make the formula~\eqref{eq:entropy-parametrization-sec10} continuous. For every fixed state this is an ordinary integral, even at a singular characteristic boundary when $\gamma>3$. Profiles of arbitrary growth are allowed, since Cauchy--Schwarz gives
\be
\label{eq:continuous-profile-absolute-convolution}
\begin{aligned}
&\int_\Rbb
|\alpha_\chi(m,v-s;\varpi)\entropygenerator(s;\varpi)|\,ds
\\
&\leq\sqrt{2k_\varpi(m)}
\,\|\alpha_\chi(m,\cdot;\varpi)\|_{L^2}
\sup_{|s-v|\leq k_\varpi(m)}
|\entropygenerator(s;\varpi)|
<+\infty.
\end{aligned}
\ee
The same estimate applies to the flux kernel, and the current equation follows by testing the kernel equation and then convolving in rapidity. Near vacuum,
$\|\alpha_\chi(m,\cdot;\varpi)\|_{L^1}\leq Cm$,
its support shrinks to zero, and its total integral is
$\widehat{\alpha_\chi}(m,0;\varpi)=\Neff{\varpi}(m)$.
We fix a compact rapidity interval $I$ and a compact interval $I_*$ whose interior contains $I$. For sufficiently small $m>0$, all intervals
$v+[-k_\varpi(m),k_\varpi(m)]$, $v\in I$, lie in $I_*$. With $\omega_{\entropygenerator,\varpi}$ denoting the modulus of continuity of the fixed profile on $I_*$, we obtain
\be
\left|
\frac{\alpha_{\entropygenerator}(m,v;\varpi)}m
-\entropygenerator(v;\varpi)
\right|
\leq
C\omega_{\entropygenerator,\varpi}(k_\varpi(m))
+\left|\frac{\Neff{\varpi}(m)}m-1\right|
 \|\entropygenerator(\cdot;\varpi)\|_{L^\infty(I_*)}.
\ee
This tends to zero locally uniformly. The distributional normal derivative follows by convolving a compactly supported smooth test function with the locally truncated profile and applying the vacuum normalization of $\alpha_\chi$. Finite propagation ensures that the truncation does not affect the result on the chosen compact interval.

\vskip.2cm

\bse
\noindent{\bf 2.} \emph{Uniqueness of the density coefficient.}
We let $d$ be the difference of two density coefficients with the same vacuum profile, at the fixed parameter $\varpi$. We fix $M>0$ and a compact rapidity interval $K$. The principal part of the effective current equation has propagation speed $k_\varpi'(m)$, so the backward acoustic cone of $\{M\}\times K$ meets the vacuum line in
\be
K_M\coloneqq K+[-k_\varpi(M),k_\varpi(M)].
\ee
We convolve $d$ in rapidity with a compactly supported smooth function and enlarge $K_M$ by its support. The smoothed difference still satisfies the homogeneous equation; on every compact positive-density rectangle,
\[
d_{mm}=(k_\varpi')^2(d_{vv}-d).
\]
The right-hand side is continuous, so the smoothed solution is $C^2$ in density there. Since its distributional normal derivative tends to zero and its vacuum value vanishes, integration in density gives
$d(m,\cdot)/m\to0$ in $\mathcal D'(\Rbb)$.
After convolution with a fixed compactly supported smooth function, this limit and the limit of the normal derivative hold locally uniformly in rapidity, together with every rapidity derivative. Indeed, the translates of the smoothing function and of each of its derivatives, over a compact rapidity interval, form a compact set of test functions; convergence of the distributional traces is uniform on this set. This justifies the integration in density and the following energy calculation.

We continue to denote the smoothed difference by $d$, and put
$w\coloneqq d/m$ and
$q_\varpi\coloneqq B_\varpi/m^2=(k_\varpi')^2$,
where
$B_\varpi(m)=m^2c_\varpi(m)^2/h_\varpi(m)^2$.
The equation for $w$ and its energy identity are
\be
\label{eq:vacuum-uniqueness-weighted-energy}
\begin{aligned}
w_{mm}+\frac2m w_m
&=q_\varpi(w_{vv}-w),
\\
\del_m\left(\frac{w_m^2}{q_\varpi}+w_v^2+w^2\right)
&=2\del_v(w_mw_v)
-\left(\frac4m+\frac{q_\varpi'}{q_\varpi}\right)
 \frac{w_m^2}{q_\varpi}.
\end{aligned}
\ee
Admissibility of the effective pressure gives
$mB_\varpi'/B_\varpi\to\gamma-1>0$ at vacuum, and hence
\[
\frac4m+\frac{q_\varpi'}{q_\varpi}
=\frac{2+mB_\varpi'/B_\varpi}{m}>0
\]
for sufficiently small $m>0$.
To prove that the initial energy tends to zero, we verify the only term requiring a weight by integrating
$(m^2w_m)_m=B_\varpi(w_{vv}-w)$
from zero. The integration constant vanishes because
$m^2w_m=md_m-d\to0$.
Since $w,w_{vv}\to0$ locally uniformly and
$B_\varpi(m)\simeq m^{\gamma-1}$, we obtain
\be
w_m(m,v)
=m^{-2}\int_0^m B_\varpi(r)(w_{vv}-w)(r,v)\,dr
=o\bigl(B_\varpi(m)/m\bigr),\qquad
\frac{w_m^2}{q_\varpi}=o(B_\varpi(m)).
\ee
We integrate~\eqref{eq:vacuum-uniqueness-weighted-energy} on a backward acoustic cone from $m=\delta>0$. Its boundary speed is
$\sqrt{q_\varpi}=k_\varpi'$; the inequality
\[
2|w_mw_v|
\leq\sqrt{q_\varpi}
\left(\frac{w_m^2}{q_\varpi}+w_v^2+w^2\right)
\]
gives non-positive lateral contributions. We write
\[
v_-(m)=\inf K-k_\varpi(M)+k_\varpi(m),\qquad
v_+(m)=\sup K+k_\varpi(M)-k_\varpi(m),
\]
and define the slice energy by
\be
\label{eq:vacuum-uniqueness-cone-energy}
e\coloneqq\frac{w_m^2}{q_\varpi}+w_v^2+w^2,\qquad
\mathcal E_d(m)
\coloneqq\int_{v_-(m)}^{v_+(m)}e(m,v)\,dv.
\ee
Since $v_-'=\sqrt{q_\varpi}$ and
$v_+'=-\sqrt{q_\varpi}$, differentiation includes both lateral contributions and gives
\be
\label{eq:vacuum-uniqueness-cone-dissipation}
\begin{aligned}
\mathcal E_d'(m)
&+\left(\frac4m+\frac{q_\varpi'}{q_\varpi}\right)
 \int_{v_-(m)}^{v_+(m)}\frac{w_m^2}{q_\varpi}\,dv
 +\sqrt{q_\varpi}
 \sum_\pm
 \left[
 \left(\frac{w_m}{\sqrt{q_\varpi}}\mp w_v\right)^2+w^2
 \right]_{v=v_\pm(m)}
=0.
\end{aligned}
\ee
All boundary values are well defined pointwise for the smoothed solution on $m\geq\delta$. The energy is therefore non-increasing on a sufficiently small vacuum interval. Letting $\delta\to0$ proves that it vanishes there. On every subsequent compact positive-density interval, the same identity and Gronwall's inequality preserve this zero energy, since the coefficients are bounded. Thus the smoothed difference vanishes on $\{M\}\times K$. Taking a sequence of smoothing kernels converging to the Dirac measure and varying $K,M$ proves uniqueness throughout the positive-density region. The normal derivative at vacuum then recovers the profile uniquely.
\ese

\vskip.2cm

\noindent{\bf 3.} \emph{Uniqueness of the spatial coefficient.}
Applied to $p_\varpi$, the reconstruction formula~\eqref{eq:tau-recovery-sec9} and the bounds~\eqref{eq:global-krl-frequency-decay} and~\eqref{eq:Bessel-derivative-remainder} give the Fourier majorant
$C_I(1+\xi^2)^{-(\lambda+1)/2}$
on every compact positive-density interval $I$. The Fourier transforms are continuous in density, including their zero-frequency values. Since $\lambda+1>1/2$, dominated convergence and Plancherel's identity give strong $L^2$ continuity of both kernels. Their supports lie in one fixed compact interval on each such density interval, which justifies continuous convolution with the profile. The first-order kernel identities pass to this convolution in the sense of distributions.

Near vacuum, the local positivity argument~\eqref{eq:local-vacuum-krl-positivity}, applied to the effective pressure, gives
\[
\|\beta_\chi(m,\cdot;\varpi)\|_{L^1}
\leq c_\varpi(m)\Neff{\varpi}(m).
\]
Consequently, on each compact rapidity interval,
\[
|\beta_{\entropygenerator}(m,v;\varpi)|
\leq c_\varpi(m)\Neff{\varpi}(m)
\|\entropygenerator(\cdot;\varpi)\|_{L^\infty}
\longrightarrow0,
\]
where the profile norm is taken on a slightly larger interval containing the acoustic supports. If two spatial coefficients correspond to the same density coefficient, their difference $d$ satisfies
\be
\label{eq:spatial-current-normalization-uniqueness}
\del_vd=0,\qquad
h_\varpi\del_md=c_\varpi^2d,\qquad
d(m)=C\frac{h_\varpi(m)}{\Neff{\varpi}(m)}.
\ee
The last equality follows distributionally from
$(\log(h_\varpi/\Neff{\varpi}))'=c_\varpi^2/h_\varpi$
on positive-density intervals. Since
$h_\varpi/\Neff{\varpi}\to1$ at vacuum, the zero trace gives $C=0$ for each fixed $\varpi$.
\end{proof}

Figure~\ref{fig:profile-to-current} summarizes the construction. At fixed normalized parallel momentum, the profile determines the two fluid-frame coefficients by scalar convolution in the effective density variable $m$; the Lorentz matrix then gives the density and flux in the laboratory frame.

\begin{figure}[htbp]
\centering
\includegraphics[width=\linewidth]{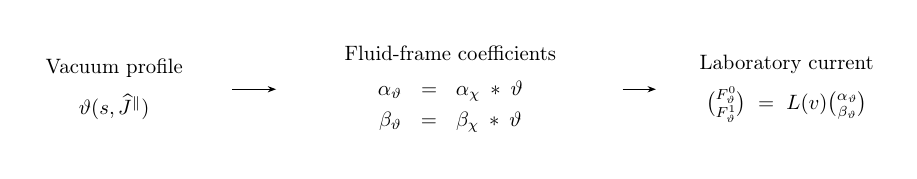}
\caption{Construction at fixed normalized parallel momentum $\varpi=|\Jhatpar|^2$. Convolution of the vacuum profile with the kernel for $p_\varpi$ gives the fluid-frame coefficients at effective density $m=\meff{\varpi}(\mu)$; $L(v)$ gives the laboratory density and flux. The state, parameter, and frame remain fixed during convolution.}
\label{fig:profile-to-current}
\end{figure}

\subsubsection{Quasi-currents defined from non-negative profiles}

Under the condition $G_{\rm N}[p]>0$ in~\eqref{hyperbolic-eos-3}, the effective pressure also satisfies $G_{\rm N}[p_\varpi]>0$, by~\eqref{eq:section-6-effective-genuine-nonlinearity}. The kernel positivity theorem therefore gives
$|\beta_\chi|\leq c_\varpi\alpha_\chi$.
Since $c_\varpi<1$, non-negative profiles produce future-causal quasi-currents at every fixed parallel momentum.

\begin{corollary}[Causal quasi-currents defined from non-negative profiles]
\label{cor:section-8-causal-current-convolution}
Let $p$ be admissible and satisfy~\eqref{hyperbolic-eos-3}, fix $\varpi\geq0$, and let
$\entropygenerator(\cdot;\varpi)$ be continuous and non-negative. Then the fluid-frame coefficients satisfy
\bel{eq:section-8-causal-current-convolution}
\alpha_{\entropygenerator}(m,v;\varpi)\geq0,\qquad
|\beta_{\entropygenerator}(m,v;\varpi)|
\leq c_\varpi(m)\alpha_{\entropygenerator}(m,v;\varpi).
\ee
The associated fixed-frame vector is future causal, and it is future timelike at every positive-density state for which
$[v-k_\varpi(m),v+k_\varpi(m)]$
meets the set where $\entropygenerator(\cdot;\varpi)$ is positive.
\end{corollary}

\begin{proof}
The inequalities follow by convolving the three non-negative measures
$\alpha_\chi$ and $c_\varpi\alpha_\chi\pm\beta_\chi$
against $\entropygenerator$, using
\autoref{theo:section-8-causal-krl-current} for the effective pressure. If the profile is positive at a point of the backward acoustic interval, continuity makes it bounded below by a positive constant on a subinterval of its interior. The positive essential lower bound for $\alpha_\chi$ on compact interior subintervals then gives
$\alpha_{\entropygenerator}>0$. Since $c_\varpi<1$,
\be
-(\alpha_{\entropygenerator})^2
+(\beta_{\entropygenerator})^2
\leq-(1-c_\varpi^2)(\alpha_{\entropygenerator})^2<0.
\ee
The decomposition~\eqref{eq:section-3-current-rest-frame-decomposition} gives the Lorentzian conclusion at each such state.
\end{proof}

\subsubsection{\texorpdfstring{Profile growth toward the causal boundary}{Profile growth toward the causal boundary}}

Continuity up to vacuum in the effective density--rapidity variables means $m \to 0$ with finite rapidity and fixed parallel momentum; equivalently, $\mu \to 0$ under these restrictions. To extend a quasi-current continuously as a function of the physical momentum $\Jbb$, we must also control the growth of its profile as rapidity tends to infinity.

For example, at zero parallel momentum we have $\varpi=0$, $m=\mu$, and $p_0=p$. Take $\vartheta(s;0)=\cosh(2s)$. The support
$|v-s|\leq k(\mu)$, the local positivity~\eqref{eq:local-vacuum-krl-positivity}, and
$\int\alpha_\chi(\mu,s;0)\,ds=\Numb(\mu)$ give
\be
\label{eq:profile-growth-vacuum-momentum}
F^0_\vartheta(\mu,v;0)
\geq
\bigl(\cosh v-c(\mu)|\sinh v|\bigr)
e^{-2k(\mu)}\Numb(\mu)\cosh(2v).
\ee
Along $\mu=e^{-5v/2}$ with $v\to+\infty$, the physical momentum tends to zero, while this lower bound tends to $+\infty$. Thus arbitrary continuous profiles cannot all extend continuously to the closed momentum cone. For a family satisfying the reference bound
$|F^0|+|F^1|\leq C_R(T^{00}+\vNumb^0)$
on $|\Jhatpar|^2\leq R$, both components instead tend to zero when $\Jbb\to0$. The convolution construction covers finite rapidity, while such reference bounds also control sequences of unbounded rapidity.

\subsection{Recovery of the reference currents}
\label{section-9-3}

\subsubsection{Three choices of profile functions}

The vacuum profiles below prescribe
$\del_m\alpha_\vartheta(0,v;\varpi)
=\entropygenerator(v;\varpi)$:
the derivative is taken in the effective density variable, with rapidity and parallel momentum fixed. The unit profile generates the normalized effective particle current. Recovering the physical particle current requires the factor $D_{\varpi,0}^{-1/2}$ from~\eqref{eq:section-6-effective-particle-density}.

\begin{proposition}[Particle-number and stress-energy currents]
\label{prop:N0-M00-M01-weak-entropies}

\

\bei

\item The profiles generating the physical particle-number, energy, and momentum currents are
\bel{eq:section-8-distinguished-normalized-traces}
\entropygenerator_{\rm N}(s;\varpi)=D_{\varpi,0}^{-1/2},
\qquad
\entropygenerator_{T,00}(s;\varpi)=\cosh s,
\qquad
\entropygenerator_{T,01}(s;\varpi)=\sinh s.
\ee
Their fluid-frame coefficients, expressed in the effective density variable, are
\bel{eq:section-8-distinguished-normalized-current-components}
\begin{aligned}
(\alpha_{\rm N},\beta_{\rm N})
&=\bigl(D_{\varpi,0}^{-1/2}\Neff{\varpi}(m),0\bigr),
\\
(\alpha_{T,00},\beta_{T,00})
&=\bigl(m\cosh v,\eps^2p_\varpi(m)\sinh v\bigr),
\\
(\alpha_{T,01},\beta_{T,01})
&=\bigl(m\sinh v,\eps^2p_\varpi(m)\cosh v\bigr).
\end{aligned}
\ee
After composition with $m=\meff{\varpi}(\mu)$ and Lorentz reconstruction~\eqref{eq:section-8-fixed-frame-current-from-normalized-profile}, the time components are respectively
$\vNumb^0$, $T^{00}$, and $T^{01}$.

\item Recall that $\Mbf_{00}=2T^{00}$ and
$\Mbf_{01}=-2T^{01}$ in the orthonormal frame used here. The derivatives at vacuum of the fixed-frame densities
$\vNumb^0$, $\Mbf_{00}$, and $\Mbf_{01}$,
taken with respect to $m$ at fixed $\varpi$ and rapidity $v=s$, are respectively
\bel{eq:section-8-distinguished-fixed-frame-traces}
D_{\varpi,0}^{-1/2}\cosh s,\qquad
2\cosh^2s,\qquad
-2\sinh s\,\cosh s.
\ee
The corresponding derivatives with respect to $\mu$ are obtained by multiplying each expression by $D_{\varpi,0}$.

\eei
\end{proposition}

\begin{proof}
\bse
\noindent{\bf 1.} \emph{Identification of the profiles.}
The identity
$(\Neff{\varpi})'/\Neff{\varpi}=h_\varpi^{-1}$
shows that
$(D_{\varpi,0}^{-1/2}\Neff{\varpi},0)$
satisfies the effective first-order current system. Direct substitution verifies the other two pairs. To identify them with the convolution construction, including the non-tempered profiles $\cosh,\sinh$, we use the compact support of $\alpha_\chi(m,\cdot;\varpi)$. Its hyperbolic cosine moment satisfies
\be
\frac{d^2}{dm^2}
\int_\Rbb\alpha_\chi(m,\theta;\varpi)\cosh\theta\,d\theta=0,
\qquad
\int_\Rbb\alpha_\chi(m,\theta;\varpi)\cosh\theta\,d\theta=m.
\ee
The first identity follows from the effective current equation by integration by parts; the second follows from the vacuum value and derivative. Parity gives the vanishing hyperbolic sine moment of $\alpha_\chi$. Similarly, the identity
$\del_\theta\beta_\chi=h_\varpi\del_m\alpha_\chi-\alpha_\chi$
gives
\be
\int_\Rbb\beta_\chi(m,\theta;\varpi)\sinh\theta\,d\theta
=-\eps^2p_\varpi(m),\qquad
\int_\Rbb\beta_\chi(m,\theta;\varpi)\cosh\theta\,d\theta=0.
\ee
Expanding $\cosh(v-\theta)$ and $\sinh(v-\theta)$ proves the energy and momentum convolution identities. The integral of the density kernel is $\Neff{\varpi}(m)$, so the constant profile $D_{\varpi,0}^{-1/2}$ gives the stated particle coefficient. Since
$\Neff{\varpi}(m)/m\to1$ and
$\eps^2p_\varpi(m)=o(m)$,
the three vacuum profiles are precisely those in~\eqref{eq:section-8-distinguished-normalized-traces}.

\noindent{\bf 2.} \emph{Reconstruction of the physical currents.}
Applying~\eqref{eq:section-8-fixed-frame-current-from-normalized-profile} to the three pairs gives, with $m=\meff{\varpi}(\mu)$,
\bel{eq:section-8-reference-current-reconstruction}
\begin{aligned}
D_{\varpi,0}^{-1/2}\cosh(v)\Neff{\varpi}(m)
&=\vNumb^0,
\\
\cosh(v)\bigl(m\cosh v\bigr)
+\sinh(v)\bigl(\eps^2p_\varpi(m)\sinh v\bigr)
&=T^{00},
\\
\cosh(v)\bigl(m\sinh v\bigr)
+\sinh(v)\bigl(\eps^2p_\varpi(m)\cosh v\bigr)
&=T^{01}.
\end{aligned}
\ee
The particle identity follows from~\eqref{eq:section-6-effective-particle-density}; the stress identities follow from~\eqref{eq:section-6-effective-stress}. Taking the derivatives at $m=0$, and using the factors relating $T^{00},T^{01}$ to $\Mbf_{00},\Mbf_{01}$, yields~\eqref{eq:section-8-distinguished-fixed-frame-traces}. Finally,
$\del_\mu m|_{\mu=0}=D_{\varpi,0}$
gives the conversion to derivatives in the physical density.
\ese
\end{proof}

The reference energy obtained in this way is the physical fluid energy, including the velocity components parallel to the symmetry orbits. Indeed,
$p_\varpi(m)=p(\mu)$ and
$m+\eps^2p(\mu)=h(\mu)(1+\sigma_\varpi)$ give
\be
\label{eq:effective-reference-physical-energy}
m\cosh^2v+\eps^2p(\mu)\sinh^2v
=h(\mu)(1+\sigma_\varpi)\cosh^2v-\eps^2p(\mu)
=T^{00}.
\ee
By contrast, the unit profile gives the physical particle current projected onto the plane orthogonal to the symmetry orbits multiplied by $\sqrt{D_{\varpi,0}}$; the physical particle current itself is generated by $D_{\varpi,0}^{-1/2}$.

When the parallel momentum varies, the dependence of this normalization factor must be included in parameter differentiation. The three physical reference currents nevertheless remain conserved along regular Euler flows. Their composed fluid-frame spatial coefficients are
$0$, $\eps^2p(\mu)\sinh v$, and
$\eps^2p(\mu)\cosh v$, respectively, so their derivatives with respect to the parallel momentum at fixed $(\mu,v)$ vanish. Consequently, their shear-induced source factors are zero.

\section{Convexity criteria}
\label{section-10}

\subsection{Rest-frame Hessians}
\label{section-10-1}
\label{section-10-1-0}

\subsubsection{Convexity relative to selected variables}

We use the representation formulas to compute the Hessian of a spacetime density current, and derive criteria for convexity in the two conservative systems of Section~\ref{section-4}. Each criterion is \emph{pressure-independent at vacuum}, while its propagation to any positive mass-energy density follows from a causal current argument in the particle-number variables, and from a characteristic system argument requiring an additional pressure sign in the mass-energy variables. Convexity concerns the fixed-frame density $F^{0}_{\entropygenerator}$ in the equation~\eqref{eq:section-8-fixed-frame-current-from-normalized-profile}.

\subsubsection{The two systems of principal variables}

The calculations in this section concern the Euler system in one spatial variable, with density \(\mu\), pressure \(\eps^{2}p\), and particle density \(\Numb\) normalized by \(\Numb(\mu)/\mu\to1\) at vacuum. To apply them to a \emph{frozen} momentum parallel to the symmetry orbits, we substitute the effective density \(m\), the pressure \(\eps^{2}p_{\varpi}(m)\), and its vacuum-normalized particle density. In that application, \(T\) and \(\vNumb\) below are the stress and particle current of the reduced system, while the parameter \(\varpi\) is held fixed when taking the Hessians. The two pairs of principal variables are
\bel{eq:section-8-two-principal-variable-pairs}
\aligned
U_{\rm E}& \coloneqq (\vNumb^0,T^{01}) = \bigl(\Numb(\mu)\cosh v,h(\mu)\sinh v\cosh v \bigr),
\\
U_{\rm N}& \coloneqq (T^{00},T^{01}) = \bigl(\mu\cosh^2v+\eps^{2}p(\mu)\sinh^2v,h(\mu)\sinh v\cosh v \bigr).
\endaligned
\ee
The first pair is the reduced counterpart of \(Z_{\rm E}\) in the equation~\eqref{eq:section-5-mass-energy-variables}, while the second one is the reduced counterpart of \(Z_{\rm N}\) in the equation~\eqref{eq:section-5-particle-number-variables}. For either pair, the frozen current is reconstructed from the same vacuum profile, but its Hessian is taken with respect to a \emph{different map from the physical state space}. 

Obviously, the two resulting notions of convexity agree for the \emph{affine currents} belonging to both formulations, but they are otherwise \emph{distinct}. For the effective family, the identity~\eqref{eq:section-6-effective-particle-density} shows that the first component of the reduced mass-energy pair is \( \bigl(1+\frac{\varpi}{4}\bigr)^{1/2} \) times the original particle-current density, the factor being \(\sqrt{D_{\varpi,0}}\); the momentum component is unchanged. At fixed \(\varpi\), this is a constant invertible linear change of conservative variables, so the Hessians in these two mass-energy coordinate systems are related by congruence and have the same sign. The particle-production pair is unchanged since the stress components in the plane orthogonal to the symmetry orbits agree by the equation~\eqref{eq:section-6-effective-stress}.

\subsubsection{\texorpdfstring{Convexity under a change of timelike normal}{Convexity under a change of timelike normal}}

At fixed momentum parallel to the symmetry orbits, the two reduced Euler formulations also permit a direct comparison of normal densities. We write either system as $\del_tU+\del_x f(U)=0$, put $A\coloneqq D_Uf$, and consider a future unit normal $n=(\cosh\theta,\sinh\theta)$ in the plane orthogonal to the symmetry orbits, with $\theta$ held fixed in all state derivatives. The conservative variables and current density selected by this normal are
\be
\label{eq:convexity-changed-normal-variables}
Y\coloneqq\cosh\theta\,U-\sinh\theta\,f(U),\qquad F^{0}_{n}\coloneqq\cosh\theta\,F^{0}-\sinh\theta\,F^{1},\qquad B\coloneqq D_UY=\cosh\theta\,I-\sinh\theta\,A.
\ee
{The acoustic eigenvalues satisfy $|\lambda_\pm|<1$. Consequently, $B$ has positive eigenvalues $\cosh\theta-\sinh\theta\,\lambda_\pm$ and is invertible. The current compatibility relation $D_UF^{1}=(D_UF^{0})A$ gives $D_UF^{0}_{n}=(D_UF^{0})B$. On each local chart of the normal conservative variables, the chain rule therefore yields}
\be
\label{eq:convexity-changed-normal-hessian}
D_YF^{0}_{n}=D_UF^{0},\qquad D_Y^2F^{0}_{n}=(D_U^2F^{0})B^{-1}.
\ee
{The last matrix is symmetric because $(D_U^2F^{0})A$ is symmetric. If $r_\pm$ are the two acoustic eigenvectors in $U$, their mixed Hessian component vanishes, and we obtain}
\be
\label{eq:convexity-changed-normal-acoustic-components}
D_Y^2F^{0}_{n}[Br_\pm,Br_\pm]=(\cosh\theta-\sinh\theta\,\lambda_\pm)D_U^2F^{0}[r_\pm,r_\pm],\qquad D_Y^2F^{0}_{n}[Br_+,Br_-]=0.
\ee
Thus (strict) convexity is \emph{preserved} wherever the normal variables define a state chart. The conclusion concerns variations of the two acoustic variables, with the normalized momentum parallel to the symmetry orbits held fixed and with the normal in the orthogonal plane.

\subsubsection{Rest-frame Hessian representation} 

We first give the calculation in the mass-energy variables. The integrals involving \(\alpha_{\chi,\mu}\) below should be understood as \emph{distributional pairings} (and need not even involve Lebesgue integrable functions). On compact positive-density intervals, the equation~\eqref{eq:global-krl-frequency-decay} gives \(|\widehat{\alpha_{\chi,\mu}}|\leq C(1+|\xi|)^{-\lambda}\), 
with $\lambda >-1/2$, and hence \(\alpha_{\chi,\mu}\in H^{-1}\). Its pairings with a \(C^2\) profile and its first derivative are therefore well defined after inserting a smooth cutoff equal to one near the kernel support, independently of the cutoff chosen. Pairings with \(\alpha_{\chi}\) and the second profile derivative are ordinary integrals, since \(\alpha_{\chi}\in L^1\). With this cutoff, mollification of a \(C^2\) profile converges in \(C^2\) near the support and justifies the calculation made for smooth profiles. This interpretation includes the singular characteristic boundaries for every \(\gamma>1\). 

\vskip.2cm

\noindent{\bf 1.} For brevity, we write \(U \coloneqq U_{\rm E}\), let \((\mu,v) =(\mu(U),v(U))\) denote the inverse change of variables, and set
\bel{eq:eta-krl-U-sec10}
\alpha_{\entropygenerator}(U) = \int_\Rbb \alpha_{\chi}(\mu(U),v(U) -s)\,\entropygenerator(s)\,ds.
\ee
We recall the identities
\be
h(\mu)=\mu+\eps^{2}p(\mu), \qquad c(\mu)^2=\eps^{2}p'(\mu),
\ee
and we denote by \(\Numb= \Numb(\mu)\) the particle density. Then we find
\bse
\be
U_1= \Numb(\mu)\cosh v, \qquad U_2=h(\mu)\sinh v\cosh v.
\ee
The Jacobian matrix of the map \((\mu,v)\mapsto(U_1,U_2)\) is
\bel{eq:Jacobian-U-mu-v-sec10}
D_{(\mu,v)}U
=
\begin{pmatrix}
\frac{\Numb(\mu)}{h(\mu)}\cosh v & \Numb(\mu)\sinh v\\[0.4em]
(1+c(\mu)^2)\sinh v\cosh v & h(\mu)\cosh(2v)
\end{pmatrix}.
\ee
Its determinant is
\bel{eq:det-Jacobian-U-mu-v-sec10}
\det D_{(\mu,v)}U = \Numb(\mu)\cosh v\,\Delta(\mu,v), \qquad \Delta(\mu,v) \coloneqq \cosh^2 v-c(\mu)^2\sinh^2 v.
\ee
\ese
Therefore the first derivatives of the inverse map are
\bel{eq:inverse-Jacobian-U-sec10}
\aligned
\mu_{U_1}
&=
\frac{h(\mu)\cosh(2v)}
{\Numb(\mu)\cosh v\,\Delta(\mu,v)},
&
\mu_{U_2}
&=
-\frac{\sinh v}{\cosh v\,\Delta(\mu,v)},
\\
v_{U_1}
&=
-\frac{(1+c(\mu)^2)\sinh v}
{\Numb(\mu)\,\Delta(\mu,v)},
&
v_{U_2}
&=
\frac{1}{h(\mu)\,\Delta(\mu,v)}.
\endaligned
\ee
The chain rule gives the Hessian of \(\alpha_{\entropygenerator}\):
\bel{eq:eta-second-U-derivatives-prelim-sec10}
\aligned
\del_{U_iU_j}\alpha_{\entropygenerator}
&=
\mu_{U_i}\mu_{U_j}\,\alpha_{\entropygenerator,\mu\mu}
+(\mu_{U_i}v_{U_j}+v_{U_i}\mu_{U_j})\,\alpha_{\entropygenerator,\mu v}
+v_{U_i}v_{U_j}\,\alpha_{\entropygenerator,vv}
+ \mu_{U_iU_j}\,\alpha_{\entropygenerator,\mu}
+v_{U_iU_j}\,\alpha_{\entropygenerator,v}.
\endaligned
\ee

\vskip.2cm

\noindent{\bf 2.}
\bse
Next, the kernel \(\alpha_{\chi}= \alpha_{\chi}(\mu,v-s)\) satisfies
\bel{eq:chi-mumu-elimination-sec10}
\alpha_{\chi,\mu\mu} = \frac{c(\mu)^2}{h(\mu)^2}\alpha_{\chi,vv} -\frac{c(\mu)^2}{h(\mu)^2}\alpha_{\chi}.
\ee
Moreover, since \(\alpha_{\chi}\) depends on \(v\) only through \(v-s\), we have
\be
\del_v\alpha_{\chi}(\mu,v-s) =-\del_s\alpha_{\chi}(\mu,v-s), \qquad \del_{vv}\alpha_{\chi}(\mu,v-s) = \del_{ss}\alpha_{\chi}(\mu,v-s),
\ee
\ese
and therefore, after integration by parts in \(s\),
\bel{eq:v-derivatives-transfer-sec10}
\aligned
\alpha_{\entropygenerator,v}(\mu,v)
&=
\int_\Rbb \alpha_{\chi}(\mu,v-s)\,\entropygenerator'(s)\,ds,
\\
\alpha_{\entropygenerator,vv}(\mu,v)
&=
\int_\Rbb \alpha_{\chi}(\mu,v-s)\,\entropygenerator''(s)\,ds,
\qquad\quad 
\alpha_{\entropygenerator,\mu v}(\mu,v)
=
\bigl\langle\alpha_{\chi,\mu}(\mu,v-\cdot),\entropygenerator' \bigr\rangle.
\endaligned
\ee
These integrations by parts are understood \emph{in the sense of distributions} on the whole rapidity line. A smooth cutoff equal to one near the acoustic interval makes each profile compactly supported without changing the pairing. The distributional derivatives include any singular contribution at the characteristic boundary; only the terms containing derivatives of the cutoff vanish, since they are supported outside that interval. We define the profile integrals \(I_j\) for \(j=0,1,2\) and write \(I_{j,\mu}\coloneqq\del_\mu I_j\) for \(j=0,1\). These integrals and their density derivatives are given by
\bel{eq:def-IJ-kernels-sec10}
\aligned
I_j(\mu,v)& \coloneqq \int_\Rbb \alpha_{\chi}(\mu,v-s)\,\entropygenerator^{(j)}(s)\,ds,
\qquad\quad
I_{j,\mu}(\mu,v) = \bigl\langle\alpha_{\chi,\mu}(\mu,v-\cdot),\entropygenerator^{(j)} \bigr\rangle.
\endaligned
\ee
The coefficient matrices are
\bel{eq:def-ABCD-sec10}
\aligned
A_{ij}& \coloneqq v_{U_i}v_{U_j}+ \frac{c(\mu)^2}{h(\mu)^2}\mu_{U_i}\mu_{U_j},
& B_{ij}& \coloneqq \mu_{U_i}v_{U_j}+v_{U_i}\mu_{U_j},\\
C_{ij}& \coloneqq \mu_{U_iU_j},
& D_{ij}& \coloneqq v_{U_iU_j},
\qquad 
&&E_{ij} \coloneqq -\frac{c(\mu)^2}{h(\mu)^2}\mu_{U_i}\mu_{U_j}.
\endaligned
\ee
Substituting the identities~\eqref{eq:chi-mumu-elimination-sec10} and~\eqref{eq:v-derivatives-transfer-sec10} into the equation~\eqref{eq:eta-second-U-derivatives-prelim-sec10}, we obtain the current Hessian
\bel{eq:eta-second-U-derivatives-final-sec10}
\del_{U_iU_j}\alpha_{\entropygenerator} = A_{ij}\,I_2 + B_{ij}\,I_{1,\mu} + C_{ij}\,I_{0,\mu} + D_{ij}\,I_1 + E_{ij}\,I_0,
\ee

\vskip.2cm

\noindent{\bf 3.}
\bse
The two pairings involving \(\alpha_{\chi,\mu}\) can also be written as ordinary integrals. The identity \(h\alpha_{\chi,\mu}= \alpha_{\chi}+ \del_v\beta_{\chi}\), followed by integration by parts in the vacuum label, gives, for \(j=0,1\),
\be
\label{eq:regu-integrals-J0-J1}
I_{j,\mu}(\mu,v) = \frac1{h(\mu)} \int_\Rbb\bigl(\alpha_{\chi}(\mu,v-s)\entropygenerator^{(j)}(s) + \beta_{\chi}(\mu,v-s)\entropygenerator^{(j+1)}(s) \bigr)\,ds.
\ee
Both kernels are integrable and supported in \([-k(\mu),k(\mu)]\). Thus, with \(K_{\mu,v} \, =[v-k(\mu),v+k(\mu)]\), we have
\be
\label{eq:regu-integrals-J-bound}
\begin{aligned} h(\mu)|I_{j,\mu}(\mu,v)| &\leq\|\alpha_{\chi}(\mu)\|_{L^1}\|\entropygenerator^{(j)}\|_{L^\infty(K_{\mu,v})} + \|\beta_{\chi}(\mu)\|_{L^1}\|\entropygenerator^{(j+1)}\|_{L^\infty(K_{\mu,v})}. \end{aligned}
\ee
This gives a direct interpretation for \(C^2\) profiles without treating the singular kernel at the characteristic boundary as an ordinary function. Under the fluid-cone positivity hypothesis, the equation~\eqref{eq:positive-krl-global-mass} replaces the two kernel norms by \(\Numb(\mu)\) and \(c(\mu)\Numb(\mu)\), respectively, on the entire density range.
\ese
In this notation, the  rest-frame current density and flux satisfy the identities
\bel{eq:section-8-profile-and-flux-identities}
\aligned
\alpha_{\entropygenerator}&=I_0,
&
\del_v\alpha_{\entropygenerator}&=I_1,
&
\del_\mu\alpha_{\entropygenerator}&=I_{0,\mu},
\\
\del_{vv}\alpha_{\entropygenerator}&=I_2,
&
\del_{\mu v}\alpha_{\entropygenerator}&=I_{1,\mu},
&
\del_{\mu\mu}\alpha_{\entropygenerator}&= \frac{c^2}{h^2}(I_2-I_0),
\\
\del_v\beta_{\entropygenerator}&=hI_{0,\mu}-I_0,
&
\del_\mu\beta_{\entropygenerator}&= \frac{c^2}{h}\bigl(\beta_{\entropygenerator}+I_1 \bigr).
\endaligned
\ee
The first two rows follow from the equations~\eqref{eq:chi-mumu-elimination-sec10} and \eqref{eq:v-derivatives-transfer-sec10}; the last row is the first-order system \eqref{eq:section-6-normalized-first-order-system}. Thus the flux reconstruction uses profile derivatives only through order two.

\vskip.2cm

\noindent{\bf 4.}
Using the equation~\eqref{eq:inverse-Jacobian-U-sec10}, we find the explicit leading coefficients \(A_{ij}\) and \(B_{ij}\)
\bse
\bel{eq:Aij-expl-sec10}
\aligned
A_{11}
&=
\frac{(1+c^2)^2\sinh^2 v
+c^2\cosh(2v)^2/\cosh^2 v}
{\Numb(\mu)^2\,\Delta(\mu,v)^2},
&&
A_{22} =
\frac{1+c(\mu)^2\tanh^2 v}
{h(\mu)^2\,\Delta(\mu,v)^2},
\\
A_{12}
&=
-\frac{\sinh v}{\Numb(\mu)h(\mu)\,\Delta(\mu,v)^2}
\Bigl(1+2c(\mu)^2+c(\mu)^2\tanh^2 v\Bigr),
\endaligned
\ee
and the coefficients $B_{ij}$ are
\bel{eq:Bij-expl-sec10}
\aligned
B_{11}
&=
-\frac{2h(\mu)(1+c(\mu)^2)\sinh v\cosh(2v)}
{\Numb(\mu)^2\cosh v\,\Delta(\mu,v)^2},
&& \qquad B_{22}
= -\frac{2\sinh v}{h(\mu)\cosh v\,\Delta(\mu,v)^2},
\\
B_{12}
&=
\frac{\cosh(2v)+(1+c(\mu)^2)\sinh^2 v}
{\Numb(\mu)\cosh v\,\Delta(\mu,v)^2}.
\endaligned
\ee
\ese
The second derivatives of the inverse map determine \(C_{ij}\) and \(D_{ij}\), while \(E_{ij}\) comes from the zero-order term in the current equation. We introduce the Hessians
\be
\mathbb H^{(1)} \coloneqq \Hess_{(\mu,v)}(U_1), \qquad \mathbb H^{(2)} \coloneqq \Hess_{(\mu,v)}(U_2),
\ee
and obtain, by inversion,
\bel{eq:second-derivatives-inverse-map-sec10}
\binom{\mu_{U_iU_j}}{v_{U_iU_j}} = - \bigl(D_{(\mu,v)}U \bigr)^{-1} \binom{\langle \mathbb H^{(1)}X_i,X_j\rangle }{\langle \mathbb H^{(2)}X_i,X_j\rangle }, \qquad X_i \coloneqq \binom{\mu_{U_i}}{v_{U_i}}.
\ee
With the inverse derivatives~\eqref{eq:inverse-Jacobian-U-sec10}, this determines the five matrices~\eqref{eq:def-ABCD-sec10}: curvature \(I_2\), mixed derivative \(I_{1,\mu}\), first derivatives \(I_{0,\mu},I_1\), and zero order \(I_0\).

\vskip.2cm

\noindent{\bf 5.}
The representation~\eqref{eq:eta-second-U-derivatives-final-sec10} depends on pressure through \(c^2/h^2\), the inverse map \(U\mapsto(\mu,v)\), and the kernels. The formula~\eqref{eq:regu-integrals-J0-J1} uses only integrable \(\alpha_{\chi},\beta_{\chi}\) and at most two profile derivatives. Convexity of the rest-frame density in $U_{\rm E}$ means
\bel{eq:convexity-test-krl-sec10}
\xi^\top\Hess_U(\alpha_{\entropygenerator})(U)\,\xi\ge0 \qquad \text{for every non-vacuum state }U\text{ and every }\xi\in\Rbb^2.
\ee
More generally, future-convexity also involves the  rest-frame flux in the equation~\eqref{eq:section-8-fixed-frame-current-from-normalized-profile}.

\subsection{Fixed-frame Hessians}

For smooth \(\entropygenerator\) with locally bounded derivatives, the flux reconstruction of Section~\ref{section-11} and the equation~\eqref{eq:section-8-fixed-frame-current-from-normalized-profile} define the fixed-frame density \(F^{0}_{\entropygenerator}\). Its mass-energy Hessian is
\bel{eq:section-8-p-convex-trace-condition}
\mathscr C_{{\rm E},p}[\entropygenerator](U_{\rm E}) \coloneqq \Hess_{U_{\rm E}}(F^{0}_{\entropygenerator})(U_{\rm E})\text{} \qquad \text{for every non-vacuum state }U_{\rm E}.
\ee
Convexity means $\mathscr C_{{\rm E},p}[\entropygenerator]\geq0$. Under the mass-energy convexity condition~\eqref{eq:intro-mass-energy-convexity-condition}, Proposition~\ref{theo:section-8-pressure-independent-convexity}, below, will reduce this matrix inequality to a pressure-independent profile condition.

For the particle-production formulation, the same fixed-frame current must instead be tested in the stress variables \(U_{\rm N}\) from the equation~\eqref{eq:section-8-two-principal-variable-pairs}. Their Jacobian is
\bel{eq:section-8-particle-number-Jacobian}
D_{(\mu,v)}U_{\rm N} \, = \begin{pmatrix}1+(1+c^2)\sinh^2v&2h\sinh v\cosh v\\(1+c^2)\sinh v\cosh v&h\cosh(2v)\end{pmatrix},
\qquad
\det D_{(\mu,v)}U_{\rm N} \, =h\Delta,
\ee
where \(\Delta\) was introduced in the equation~\eqref{eq:det-Jacobian-U-mu-v-sec10}. Thus the inverse change of variables is regular in the interior of the causal state cone, and its first derivatives are
\bel{eq:section-8-particle-number-inverse-Jacobian}
\aligned
\mu_{U_{{\rm N},1}}&= \frac{\cosh(2v)}{\Delta(\mu,v)},& \mu_{U_{{\rm N},2}}&=-\frac{2\sinh v\cosh v}{\Delta(\mu,v)},\\
v_{U_{{\rm N},1}}&=-\frac{(1+c(\mu)^2)\sinh v\cosh v}{h(\mu)\Delta(\mu,v)},& v_{U_{{\rm N},2}}&= \frac{1+(1+c(\mu)^2)\sinh^2v}{h(\mu)\Delta(\mu,v)}.
\endaligned
\ee

We likewise introduce the particle-number Hessian by
\bel{eq:section-8-N-convexity}
\mathscr C_{{\rm N},p}[\entropygenerator](U_{\rm N}) \coloneqq \Hess_{U_{\rm N}}(F^{0}_{\entropygenerator})(U_{\rm N})\text{} \qquad \text{for every non-vacuum state }U_{\rm N}.
\ee
{Here convexity means $\mathscr C_{{\rm N},p}[\entropygenerator]\geq0$; its propagation is proved below using the positivity of $c\alpha_{\chi}\pm\beta_{\chi}$.}

We arrive at the following decompositions. 

\begin{lemma}[Hessian decompositions in the two conservative systems]
\label{prop:section-8-pressure-dependent-convexity}
\label{prop:section-8-particle-number-convexity}
For \(*={\rm E},{\rm N}\), let \(p=p(\mu)\) be an admissible equation of state in the sense of Definition~\ref{def-pression}. For every \(\entropygenerator\in C^2_{\loc}(\Rbb)\), the two Hessians have the respective representations
\bel{eq:section-8-physical-hessian-trace-decomposition}
\mathscr C_{{\rm E},p}[\entropygenerator]= \mathscr K_{{\rm E},p,2}[\entropygenerator'']+ \mathscr K_{{\rm E},p,1}[\entropygenerator']+ \mathscr K_{{\rm E},p,0}[\entropygenerator],
\ee
\bel{eq:section-8-N-physical-hessian-trace-decomposition}
\mathscr C_{{\rm N},p}[\entropygenerator]= \mathscr K_{{\rm N},p,2}[\entropygenerator'']+ \mathscr K_{{\rm N},p,1}[\entropygenerator']+ \mathscr K_{{\rm N},p,0}[\entropygenerator],
\ee
where \(\mathscr K_{*,p,j}\) are finite sums of convolutions in the rapidity variable, multiplied by state-dependent matrix coefficients. They are determined by \(\alpha_{\chi}\), \(\alpha_{\chi,\mu}\), \(\beta_{\chi}\), and the first two derivatives of the inverse map \(U_{*}\mapsto(\mu,v)\).  On every compact positive-density state set, the operator bounds are finite for any given pressure law.
\end{lemma}

\begin{proof}
\bse
\noindent{{\bf 1.} \emph{Derivatives in density--rapidity variables.}} We differentiate before either conservative change of variables. We abbreviate \(C\coloneqq\cosh v\), \(S\coloneqq\sinh v\), and \(T\coloneqq\beta_{\entropygenerator}\) in this proof. With $a\coloneqq c^2/h$, $b\coloneqq c^2/h^2$, and $d\coloneqq(h\eps^2p''-c^2)/h^2$, the Lorentz reconstruction is \(F^{0}_{\entropygenerator}=CI_0+ST\). The first-order kernel identities give
\be
\label{eq:fixed-frame-density-first-derivatives}
F^{0}_{\mu}=CI_{0,\mu}+aS(T+I_1), \qquad F^{0}_{v}=C(T+I_1)+hSI_{0,\mu}.
\ee
In particular, differentiating the Lorentz factors is essential: the terms \(SI_0\) cancel only after we use \(T_v=hI_{0,\mu}-I_0\). A second differentiation yields
\be
\label{eq:fixed-frame-density-second-derivatives}
\begin{aligned}
 F^{0}_{\mu\mu}&=bC(I_2-I_0)+aSI_{1,\mu}+dS(T+I_1),\\
 F^{0}_{\mu v}&=CI_{1,\mu}+(1+c^2)SI_{0,\mu}+a\bigl(C(T+I_1)+S(I_2-I_0) \bigr),\\
 F^{0}_{vv}&=C(I_2-I_0)+S(T+I_1)+2hCI_{0,\mu}+hSI_{1,\mu}.
\end{aligned}
\ee
We used \((I_{0,\mu})_\mu=b(I_2-I_0)\), \((I_1)_\mu=I_{1,\mu}\), \(T_\mu=a(T+I_1)\), and \(h'=1+c^2\). The coefficient in the first line is \(a'+a^2=(h\eps^{2}p''-c^2)/h^2=d\), where differentiation of the sound speed contributes the term containing \(p''\).
\ese

\vskip.15cm
\noindent{{\bf 2.} \emph{Passage to conservative variables.}} For either \(U=U_{\rm E}\) or \(U=U_{\rm N}\), the conservative Hessian is now given entry by entry by
\be
\label{eq:fixed-frame-conservative-Hessian}
\begin{aligned} (\Hess_UF^{0})_{ij} &= F^{0}_{\mu\mu}\mu_{U_i}\mu_{U_j} + F^{0}_{\mu v}(\mu_{U_i}v_{U_j}+v_{U_i}\mu_{U_j})+ F^{0}_{vv}v_{U_i}v_{U_j} + F^{0}_{\mu}\mu_{U_iU_j}+ F^{0}_{v}\,v_{U_iU_j}. \end{aligned}
\ee
The first inverse derivatives are given by the equation~\eqref{eq:inverse-Jacobian-U-sec10} for \(U_{\rm E}\) and the equation~\eqref{eq:section-8-particle-number-inverse-Jacobian} for \(U_{\rm N}\); their second derivatives follow from the inverse-function formula. Substitution of the equation~\eqref{eq:fixed-frame-density-first-derivatives}--\eqref{eq:fixed-frame-density-second-derivatives} expresses each Hessian entry as a linear combination of \(I_2,I_{1,\mu},I_{0,\mu},I_1,I_0,T\). Grouping these six terms by profile-derivative order gives the representations. The integral formula~\eqref{eq:regu-integrals-J0-J1} for \(I_{0,\mu},I_{1,\mu}\) uses at most two profile derivatives, including at the characteristic boundary.

\vskip.15cm
For the particle-number variables, the inverse-map identities take the following explicit form.
\bse
We write \(W_{\rm N}(U_{\rm N}) \coloneqq (\mu(U_{\rm N}),v(U_{\rm N}))\). For \(\zeta\in\Rbb^2\), we put \(X\coloneqq D_{U_{\rm N}}W_{\rm N}\,\zeta\). The chain rule gives
\bel{eq:section-8-N-hessian-chain-rule}
\Hess_{U_{\rm N}}(F^{0}_{\entropygenerator})[\zeta,\zeta]= \Hess_{(\mu,v)}(F^{0}_{\entropygenerator})[X,X]+D_{(\mu,v)}F^{0}_{\entropygenerator}\cdot D^2_{U_{\rm N}}W_{\rm N}[\zeta,\zeta].
\ee
The first derivatives of \(W_{\rm N}\) are given by the equation~\eqref{eq:section-8-particle-number-inverse-Jacobian}, while the inverse-function formula gives
\bel{eq:section-8-N-second-inverse-derivative}
D^2_{U_{\rm N}}W_{\rm N}[\zeta,\zeta]=- (D_{(\mu,v)}U_{\rm N})^{-1}D^2_{(\mu,v)}U_{\rm N}[D_{U_{\rm N}}W_{\rm N}\,\zeta,D_{U_{\rm N}}W_{\rm N}\,\zeta].
\ee
\ese

\noindent{{\bf 3.} \emph{Uniformity and approximation.}} On a compact positive-density state set, the inverse Jacobian stays nonsingular, its first two derivatives are bounded, and \(a,b,d\) are bounded. The kernel bounds in \autoref{thm:physical-krl-existence} and Proposition~\ref{prop:fixed-current-bounds} therefore bound the operators on that compact state set. These Hessian bounds on compact state sets give no uniform positive convexity constant at vacuum. For a profile of class \(C^2_{\loc}\), we approximate in \(C^2\) on an interval containing the fluid supports over the chosen compact state set. Formula~\eqref{eq:regu-integrals-J0-J1} and the locally uniform \(L^1\) bounds for \(\alpha_{\chi},\beta_{\chi}\) show that all terms in the first and second derivative formulas converge uniformly there. Thus the currents converge in \(C^2\), and the Hessian representation holds for the original profile.
\end{proof}

\subsection{Convexity criteria}
\label{section-10-2}

\subsubsection{Leading-order analysis near vacuum}

We first derive necessary profile conditions from the \emph{vacuum limit}, then prove their sufficiency at positive density under the respective pressure conditions. We put \(r_{\rm E} \coloneqq \vNumb^0\), \(q_{\rm E} \coloneqq T^{01}\), \(E_{\rm N} \coloneqq T^{00}\), and \(P_{\rm N} \coloneqq T^{01}\). As the state tends to vacuum with the ratios below in a compact set, the first nonzero homogeneous terms of the fixed-frame current are
\bel{eq:section-8-vacuum-homogeneous-entropies}
\aligned
F^{0}_{\entropygenerator}&=r_{\rm E}\entropygenerator \Bigl( \operatorname{arsinh} \Bigl( \frac{q_{\rm E}}{r_{\rm E}} \Bigr)  \Bigr) +o(r_{\rm E}),\\
F^{0}_{\entropygenerator}&=E_{\rm N}\sqrt{1- \Bigl( \frac{P_{\rm N}}{E_{\rm N}} \Bigr) ^2}\,\entropygenerator \Bigl( \artanh \Bigl( \frac{P_{\rm N}}{E_{\rm N}} \Bigr)  \Bigr) +o(E_{\rm N}).
\endaligned
\ee
The lines correspond respectively to mass-energy and particle-number variables.

\begin{lemma}[Convexity conditions near the vacuum boundary]
\label{prop:section-8-two-boundary-convexities}
Let \(\entropygenerator\in C^2(\Rbb)\).  The homogeneous vacuum density \emph{in the mass-energy variables} in the equation~\eqref{eq:section-8-vacuum-homogeneous-entropies} is \emph{convex if and only if}
\bel{eq:section-8-E-boundary-trace-condition}
\mathcal D_{\rm E}\entropygenerator(s) = \entropygenerator''(s) -\tanh(s)\entropygenerator'(s)\ge0 \qquad\text{for all }s\in\Rbb.
\ee
The homogeneous vacuum density \emph{in the particle-number variables} is \emph{convex if and only if} 
\bel{eq:section-8-N-boundary-trace-condition}
\mathcal D_{\rm N}\entropygenerator(s) = \entropygenerator''(s) -\entropygenerator(s)\ge0 \qquad\text{for all }s\in\Rbb.
\ee
Consequently, a future-convex fixed-frame current for which the vacuum boundary Hessian is well defined satisfies the first condition in the dissipated-energy formulation and the second condition in the particle-production formulation. Conversely, each inequality makes the corresponding homogeneous boundary current future-convex.
\end{lemma}

\begin{proof}
\bse
The Hessian of a perspective \((r,q)\mapsto r f(q/r)\), with \(r>0\), is non-negative if and only if \(f''\ge0\). For the first line in the equation~\eqref{eq:section-8-vacuum-homogeneous-entropies}, we set \(y\coloneqq q_{\rm E}/r_{\rm E}\), \(s\coloneqq\operatorname{arsinh}y\), and \(f_{\rm E}(y) \coloneqq \entropygenerator(\operatorname{arsinh}y)\). Direct differentiation gives
\bel{eq:section-8-E-boundary-second-derivative}
f_{\rm E}''(y) = \frac{1}{\cosh^2s}\bigl(\entropygenerator''(s) -\tanh(s)\entropygenerator'(s) \bigr).
\ee
For the second line, we set \(x\coloneqq P_{\rm N}/E_{\rm N}\), \(s\coloneqq\artanh x\), and \(f_{\rm N}(x) \coloneqq \sqrt{1-x^2}\entropygenerator(\artanh x)\). Then we find
\bel{eq:section-8-N-boundary-second-derivative}
f_{\rm N}''(x) = \cosh^3s\bigl(\entropygenerator''(s) -\entropygenerator(s) \bigr).
\ee
Since all prefactors in the estimates~\eqref{eq:section-8-E-boundary-second-derivative} and \eqref{eq:section-8-N-boundary-second-derivative} are positive, the two assertions follow.
\ese
\end{proof}

\subsubsection{Criteria for global convexity}

Having identified the profile conditions required by convexity from the vacuum calculation, we now prove their necessity for the reconstructed currents and their sufficiency throughout the density range, using kernel positivity and the pressure hypotheses of the following proposition.
 
\begin{proposition}[Pressure-independent conditions on vacuum profiles]
\label{theo:section-8-pressure-independent-convexity}
Let \(p=p(\mu)\) be an admissible equation of state in the sense of~\autoref{def-pression}.
Then the fixed-frame current defined from \(\entropygenerator\) is \(C^2\) on every compact subset of the non-vacuum state cone and has the following properties.
\bei

\item If the time component of the fixed-frame current is \emph{convex with respect to the particle-number} variables~$U_{\rm N}$, then one has
\bel{eq:section-8-global-N-trace-condition}
\mathcal D_{\rm N}\entropygenerator= \entropygenerator''-\entropygenerator\geq0\quad\text{on }\Rbb.
\ee
Conversely, this condition is \emph{sufficient when} the measures \(c\alpha_{\chi}+\beta_{\chi}\) and \(c\alpha_{\chi}-\beta_{\chi}\) are non-negative. By \autoref{theo:section-8-causal-krl-current}, this positivity follows from the genuine nonlinearity condition~\eqref{hyperbolic-eos-3}.

\item If the time component of the fixed-frame current is \emph{convex with respect to the mass-energy} variables~$U_{\rm E}$, then one has
\bel{eq:section-8-global-E-trace-condition}
\mathcal D_{\rm E}\entropygenerator= \entropygenerator''-\tanh(s)\entropygenerator'\geq0\quad\text{on }\Rbb.
\ee
Conversely, the profile condition~\eqref{eq:section-8-global-E-trace-condition} \emph{implies convexity} throughout the positive-density cone under the mass-energy convexity condition~\eqref{eq:intro-mass-energy-convexity-condition}.
This condition implies genuine nonlinearity: the left-hand side of~\eqref{hyperbolic-eos-3} equals that of~\eqref{eq:intro-mass-energy-convexity-condition} plus \(4c^3(1-c)>0\).
\eei
Moreover, the condition \eqref{eq:intro-mass-energy-convexity-condition} is inherited by every effective pressure \(\eps^{2}p_{\varpi}\), for arbitrary \(\varpi\geq0\), by Lemma~\ref{lem:effective-mass-energy-convexity}; hence the conclusions apply to the corresponding quasi-currents in the two conservative variables of the reduced Euler system, with the normalized momentum parallel to the symmetry orbits held fixed. The physical and effective conservative variables are related by the constant linear change described after~\eqref{eq:section-8-two-principal-variable-pairs}.
\end{proposition}

Necessity of the conditions on $\entropygenerator$ follows by rescaling toward vacuum; sufficiency follows by diagonalizing the Hessian in the acoustic directions. In particle-number variables, the diagonal components form the current with profile $\mathcal D_{\rm N}\entropygenerator$, so kernel positivity applies. In mass-energy variables, the condition~\eqref{eq:intro-mass-energy-convexity-condition} gives non-negative coupling in their characteristic system, propagating the vacuum signs.

The pressure asymptotics imply this constitutive condition near vacuum and at all sufficiently large densities. Indeed, differentiating \(\eps^{2}p'=B(1+\eps^{2}p/\mu)^2\) and using the weighted endpoint bounds gives \(\mu \eps^{2}p''/\eps^{2}p'\to\gamma-1\) at vacuum, whereas \(\mu \eps^{2}p''\to0\) and \(\eps^{2}p'\to \eps^{2}k_\infty^{2}\in(0,1)\) at large density. Thus the condition need only be checked on an intermediate compact density interval.

\begin{proof}

\bse
\noindent{{\bf 1.} \emph{Necessary conditions at the vacuum boundary.}} We choose $\star\in\{{\rm E},{\rm N}\}$ and assume that $F^{0}_{\entropygenerator}$ is convex with respect to~$U_\star$. We first consider the mass-energy variables $U_{\rm E} \, =(r,q)$, with $r>0$.  We denote by $(\mu_\auxeps,v_\auxeps)$  the physical state associated with $\auxeps U_{\rm E}$. Since $\Numb(\mu_\auxeps)\leq\auxeps r$, we have $\mu_\auxeps\to0$. Moreover, the ratio of the two components of $U_{\rm E}$ bounds $v_\auxeps$. Since $\Numb/\mu\to1$ and $h/\mu\to1$ at vacuum, we obtain
\bel{eq:vacuum-rescaled-E-variables}
v_\auxeps\to v_0 \coloneqq\operatorname{arsinh}(q/r), \qquad \frac{\mu_\auxeps}{\auxeps}\to\frac{r}{\cosh v_0}.
\ee
For the particle-number variables,  we set $U_{\rm N} \, =(E,P)$, with $E>|P|$. Then $\mu_\auxeps\leq\auxeps E$ and $|\tanh v_\auxeps|\leq |P|/E<1$. Hence we obtain
\bel{eq:vacuum-rescaled-N-variables}
v_\auxeps\to v_0 \coloneqq\artanh(P/E), \qquad \frac{\mu_\auxeps}{\auxeps}\to\frac{E}{\cosh^2v_0}.
\ee
Both limits are locally uniform in the corresponding open cones. The vacuum normalization of the rest-frame current is
\be
\frac{\alpha_{\entropygenerator}(\mu,v)}{\mu}\to\entropygenerator(v), \qquad \frac{\beta_{\entropygenerator}(\mu,v)}{\mu}\to0.
\ee
After Lorentz reconstruction, the two homogeneous functions in the equation~\eqref{eq:section-8-vacuum-homogeneous-entropies} are therefore the locally uniform limits of $\auxeps^{-1}F^{0}_{\entropygenerator}(\auxeps U_\star)$. Since a locally uniform limit of convex functions is convex, \autoref{prop:section-8-two-boundary-convexities} gives the necessary conditions \eqref{eq:section-8-global-N-trace-condition} and \eqref{eq:section-8-global-E-trace-condition}.
\ese

\medskip
\bse
\noindent{{\bf 2.} \emph{Characteristic components of the Hessian.}} We next prove sufficiency, first for smooth profiles.  We define
\bel{eq:section-8-characteristic-state-directions}
X_\pm \coloneqq\del_\mu\pm\frac{c}{h}\del_v, \qquad r_{\star,\pm} \coloneqq D_{(\mu,v)}U_\star\,X_\pm, \qquad \star\in\{{\rm E},{\rm N}\}.
\ee
These are the two acoustic directions expressed in the conservative variables. If $w\coloneqq(\mu,v)$, $M_\star\coloneqq D_wU_\star$, and $g_\star\coloneqq(D_wF^{0}_{\entropygenerator})M_\star^{-1}$, the chain rule gives
\bel{eq:characteristic-Hessian-change}
\Hess_{U_\star}F^{0}_{\entropygenerator}[M_\star x,M_\star x]=x^{\mathsf T} \Bigl (D_w^2F^{0}_{\entropygenerator}-\sum_{j=1}^2(g_\star)_jD_w^2(U_\star)_j \Bigr)x.
\ee
We apply this identity to $x=(1,\pm c/h)$ and use the current equations, together with their rapidity derivatives, to eliminate all density derivatives. In particular, the mixed derivatives are
\be
\label{eq:characteristic-Hessian-mixed-current-derivatives}
h\del_{\mu v}\alpha_{\vartheta}=\del_v\alpha_{\vartheta}+\del_{vv}\beta_{\vartheta},\qquad h\del_{\mu v}\beta_{\vartheta}=c^2(\del_{vv}\alpha_{\vartheta}+\del_v\beta_{\vartheta}).
\ee
Together with the density derivatives~\eqref{eq:current-second-density-derivatives}, below, these identities express the Hessian through rapidity derivatives and the second derivatives of the coordinate map in the equation~\eqref{eq:characteristic-Hessian-change}. The latter distinguish the conservative Hessian from the Hessian in $(\mu,v)$.

\ese

\bse
For these substitutions, we write $C_v\coloneqq\cosh v$ and $S_v\coloneqq\sinh v$. The second derivatives of the three conservative components are
\be
\label{eq:convexity-conservative-second-derivatives}
\begin{aligned}
D_w^2(\Numb C_v)&=\Numb\begin{pmatrix}-c^2C_v/h^2&S_v/h\\S_v/h&C_v\end{pmatrix},\\
D_w^2(hC_vS_v)&=\begin{pmatrix}\eps^{2}p''C_vS_v&(1+c^2)(C_v^2+S_v^2)\\(1+c^2)(C_v^2+S_v^2)&4hC_vS_v\end{pmatrix},\\
D_w^2(hC_v^2-\eps^{2}p)&=\begin{pmatrix}\eps^{2}p''S_v^2&2(1+c^2)C_vS_v\\2(1+c^2)C_vS_v&2h(C_v^2+S_v^2)\end{pmatrix}.
\end{aligned}
\ee
{Here we use $h'=1+c^2$, $\Numb'=\Numb/h$, and $\Numb''=-\Numb c^2/h^2$. These formulas and the current derivatives give the quadratic forms below for arbitrary $x=(x_\mu,x_v)$, before restriction to the acoustic directions.} We make the change of conservative variables explicit. The fixed-frame density is $F^{0}=C_v\alpha_\vartheta+S_v\beta_\vartheta$. The first-order current equations give $h \alpha_{\vartheta,\mu}=\alpha_\vartheta+\beta_{\vartheta,v}$ and $h\beta_{\vartheta,\mu}=c^2(\beta_\vartheta+\alpha_{\vartheta,v})$. We substitute these identities in $g_\star=(D_wF^{0})M_\star^{-1}$ and obtain
\be
\label{eq:convexity-conservative-current-gradients}
\begin{aligned}
g_{\rm E}&=\Bigl(\frac{\alpha_\vartheta+\beta_{\vartheta,v}-\tanh v\,(\beta_\vartheta+\alpha_{\vartheta,v})}{\Numb},\frac{\beta_\vartheta+\alpha_{\vartheta,v}}{hC_v}\Bigr),\\
g_{\rm N}&=\frac1h\Bigl(C_v(\alpha_\vartheta+\beta_{\vartheta,v})-S_v(\beta_\vartheta+\alpha_{\vartheta,v}),C_v(\beta_\vartheta+\alpha_{\vartheta,v})-S_v(\alpha_\vartheta+\beta_{\vartheta,v})\Bigr).
\end{aligned}
\ee
{These formulas also exhibit the cancellation of the pressure second derivative in the equation~\eqref{eq:characteristic-Hessian-change}. In both formulations, the terms proportional to $\eps^{2}p''$ in $\sum_i(g_\star)_i\del_{\mu\mu}(U_\star)_i$ equal $\eps^{2}p''S_v(\beta_\vartheta+\alpha_{\vartheta,v})/h$, which is exactly the corresponding term in $\del_{\mu\mu}F^{0}$. The remaining entries follow from the equations~\eqref{eq:characteristic-Hessian-mixed-current-derivatives} and~\eqref{eq:current-second-density-derivatives}, below; they give the full quadratic forms in steps~3 and~4.}

\ese

\bse
The flux Jacobian has eigenvectors $r_{\star,\pm}$ with eigenvalues
\be
\lambda_\pm= \frac{\tanh v\pm c}{1\pm c\tanh v}, \qquad \lambda_+-\lambda_-= \frac{2c(1-\tanh^2v)}{1-c^2\tanh^2v}>0.
\ee
The current density--flux compatibility identity implies that \(H_\star A_\star\) is symmetric, where \(H_\star=\Hess_{U_\star}F^{0}_{\entropygenerator}\). Consequently, we have
\bel{eq:acoustic-Hessian-cross-term}
r_{\star,+}^{\mathsf T}H_\star r_{\star,-} \, =0.
\ee
Thus the Hessian is non-negative if and only if its two diagonal characteristic components are non-negative.
\ese

\medskip
\bse
\noindent{{\bf 3.} \emph{Particle-number variables.}} We define the curvature current by
\bel{eq:section-8-particle-curvature-current}
A_{\rm N} \coloneqq\del_v^2\alpha_{\vartheta}-\alpha_{\vartheta}, \qquad B_{\rm N} \coloneqq\del_v^2\beta_{\vartheta}-\beta_{\vartheta}.
\ee
By differentiation under the convolution integral, $(A_{\rm N},B_{\rm N})$ is the rest-frame current defined from $\mathcal D_{\rm N}\entropygenerator= \entropygenerator''-\entropygenerator$. The  full quadratic form is
\be
\label{eq:particle-number-full-pulled-hessian}
\Hess_{U_{\rm N}}F^{0}_{\entropygenerator}[M_{\rm N}x,M_{\rm N}x] =(C_vA_{\rm N}+S_vB_{\rm N})\bigl(c^2x_\mu^2/h^2+x_v^2\bigr) +\frac2h(C_vB_{\rm N}+c^2S_vA_{\rm N})x_\mu x_v.
\ee
{Substituting $x=(1,\pm c/h)$ gives}
\bel{eq:section-8-particle-characteristic-hessian}
\Hess_{U_{\rm N}}(F^{0}_{\entropygenerator})[r_{{\rm N},\pm},r_{{\rm N},\pm}]= \frac{2c(\cosh v\pm c\sinh v)}{h^2}\bigl(c \, A_{\rm N}\pm B_{\rm N} \bigr).
\ee
If $\mathcal D_{\rm N}\entropygenerator\geq0$ and $c\alpha_{\chi}\pm\beta_{\chi}$ are non-negative measures, then
\be
|B_{\rm N}|\leq c \, A_{\rm N}.
\ee
Both terms on the right-hand side of \eqref{eq:section-8-particle-characteristic-hessian} are therefore non-negative. Under the genuine nonlinearity condition \eqref{hyperbolic-eos-3}, the positivity of $c\alpha_{\chi}\pm\beta_{\chi}$ follows from \autoref{theo:section-8-causal-krl-current}. 
\ese

\medskip
\bse
\noindent{{\bf 4.} \emph{Mass-energy variables.}} We define the characteristic curvatures by
\bel{eq:section-8-mass-energy-characteristic-curvatures}
\mathcal Q_{\rm E}^\pm \coloneqq c\bigl(\del_v^2\alpha_{\vartheta}+\del_v\beta_{\vartheta}-\tanh v\,(\del_v\alpha_{\vartheta}+\beta_{\vartheta}) \bigr)\pm(\del_v^2\beta_{\vartheta}-\beta_{\vartheta}).
\ee
The corresponding characteristic Hessian components are
\bel{eq:section-8-mass-energy-characteristic-hessian}
\Hess_{U_{\rm E}}(F^{0}_{\entropygenerator})[r_{{\rm E},\pm},r_{{\rm E},\pm}]= \frac{2c(\cosh v\pm c\sinh v)}{h^2}\mathcal Q_{\rm E}^\pm.
\ee
It remains to prove $\mathcal Q_{\rm E}^\pm\geq0$.

We put $y=k(\mu)$ and write $\mathcal Q_{\rm E}^\pm=cA_{\rm E}\pm B_{\rm E}$, where
\be
A_{\rm E} \coloneqq\del_v^2\alpha_{\vartheta}+\del_v\beta_{\vartheta}-\tanh v\,(\del_v\alpha_{\vartheta}+\beta_{\vartheta}), \qquad B_{\rm E} \coloneqq\del_v^2\beta_{\vartheta}-\beta_{\vartheta}.
\ee
{In these variables, the full quadratic form is}
\be
\label{eq:mass-energy-full-pulled-hessian}
\Hess_{U_{\rm E}}F^{0}_{\entropygenerator}[M_{\rm E}x,M_{\rm E}x] =(C_vA_{\rm E}+S_vB_{\rm E})\bigl(c^2x_\mu^2/h^2+x_v^2\bigr) +\frac2h(C_vB_{\rm E}+c^2S_vA_{\rm E})x_\mu x_v.
\ee
{Restriction to $x=(1,\pm c/h)$ proves the equation~\eqref{eq:section-8-mass-energy-characteristic-hessian}.}
\ese

\bse
For the curvature equations, we put $R_{\rm E}\coloneqq\del_v\alpha_{\vartheta}+\beta_{\vartheta}$, so that $A_{\rm E}=(\del_v-\tanh v)R_{\rm E}$. The first-order current equations give
\be
\label{eq:mass-energy-curvature-calculation}
\del_yR_{\rm E}=(c+c^{-1})R_{\rm E}+c^{-1}B_{\rm E},\qquad \del_yB_{\rm E}=c(\del_v^2-1)R_{\rm E}.
\ee
{Since $c$ depends only on $y$ and $\del_v\tanh v=1-\tanh^2v$, differentiating the first identity with $\del_v-\tanh v$ gives the equation for $A_{\rm E}$ below, while $(\del_v^2-1)R_{\rm E}=\del_vA_{\rm E}+\tanh v\,A_{\rm E}$ gives the equation for $B_{\rm E}$. Thus we obtain}
\bel{eq:mass-energy-curvature-intermediate}
\del_yA_{\rm E}-c^{-1}\del_vB_{\rm E}=(c+c^{-1})A_{\rm E}-c^{-1}\tanh v\,B_{\rm E},\qquad \del_yB_{\rm E}-c\del_vA_{\rm E}=c\tanh v\,A_{\rm E}.
\ee
We set \(A_p(\mu) \coloneqq [h(\mu)c'(\mu)+c(\mu)(1+c(\mu)^2)]/(2c(\mu)^2)\). Since $\del_yc=hc'/c$, we obtain
\bel{eq:section-8-mass-energy-curvature-system}
(\del_y-\del_v)\mathcal Q_{\rm E}^+= A_p\mathcal Q_{\rm E}^++(A_p+\tanh v)\mathcal Q_{\rm E}^-,\qquad (\del_y+\del_v)\mathcal Q_{\rm E}^-=(A_p-\tanh v)\mathcal Q_{\rm E}^++A_p\mathcal Q_{\rm E}^-.
\ee
The constitutive condition \eqref{eq:intro-mass-energy-convexity-condition} is equivalent to $A_p\geq1$. Indeed, \(4c^3(A_p-1)=h \eps^{2}p''+2\eps^{2}p'(1-c)^2\). Hence the two coupling coefficients are non-negative. By contrast, genuine nonlinearity gives $4c^3(A_p-c)=h\eps^{2}p''+2\eps^{2}p'(1-c^2)>0$, and therefore only $A_p>c$. Since $c<1$, this inequality does not ensure the signs of $A_p\pm\tanh v$ at arbitrary rapidity. The characteristic comparison therefore applies under the mass-energy convexity condition~\eqref{eq:intro-mass-energy-convexity-condition}.

\ese

\bse
The vacuum normalization gives
\bel{eq:section-8-mass-energy-curvature-vacuum-trace}
\mathcal Q_{\rm E}^\pm(\mu,v) = \mu c(\mu)\mathcal D_{\rm E}\entropygenerator(v)+o\bigl(\mu c(\mu) \bigr)\qquad\text{as }\mu\to0,
\ee
locally uniformly in~$v$. Indeed, $\del_v^j\alpha_{\vartheta}= \mu\entropygenerator^{(j)}+o(\mu)$ for $j\leq2$, while the flux equation yields
\be
\del_v^j\beta_{\vartheta}(\mu,v) = \frac{h(\mu)}{\Numb(\mu)}\int_0^\mu\frac{c(t)^2\Numb(t)}{h(t)^2}\del_v^{j+1}\alpha_{\vartheta}(t,v)\,dt=\Ocal(\mu^\gamma), \qquad j\leq2.
\ee
Thus the flux contributions are $o(\mu c(\mu))$.

We set $w(\mu) \coloneqq\mu c(\mu)$ and $Z^\pm\coloneqq\mathcal Q_{\rm E}^\pm/w$. The vacuum formula shows that the negative parts of $Z^\pm$ tend to zero as $y\to0$. Moreover, the sums of the coefficients in the normalized system are
\be
2A_p-\frac{w_y}{w}\pm\tanh v= \frac{c^2-\eps^{2}p/\mu}{c}\pm\tanh v.
\ee
They are bounded near vacuum and on every compact positive-density interval. Indeed, $c^2-\eps^{2}p/\mu=\Ocal(\mu^{\gamma-1})$ and $c\simeq\mu^{(\gamma-1)/2}$ near vacuum. We fix a positive density and a compact rapidity interval, and work on their backward characteristic cone. Its rapidity sections lie in one fixed compact interval. On $y\geq\auxeps$, all coefficients are bounded; the off-diagonal coefficients are non-negative and the row sums are bounded above by a constant $C$ independent of $\auxeps$. We choose $\delta_\auxeps$ as the maximum of the negative parts of $Z^+$ and $Z^-$ on the lower section $y=\auxeps$. The locally uniform vacuum limit gives $\delta_\auxeps\to0$. The function $-\delta_\auxeps e^{C(y-\auxeps)}$ is a lower comparison function for both equations, since their row sums are at most $C$. Characteristic integration therefore gives
\be
Z^\pm(y,v)\geq-\delta_\auxeps e^{Cy}.
\ee
Since $\delta_\auxeps\to0$, we obtain $\mathcal Q_{\rm E}^\pm\geq0$. Formula \eqref{eq:section-8-mass-energy-characteristic-hessian} then proves mass-energy convexity.
\ese

\medskip
\bse
\noindent{{\bf 5.} \emph{General profiles and effective pressure laws.}} For particle-number variables, convolution with a non-negative mollifier preserves $\mathcal D_{\rm N}\entropygenerator\geq0$. For mass-energy variables, $\mathcal D_{\rm E}\entropygenerator\geq0$ is equivalent to the convexity of $y\mapsto\entropygenerator(\operatorname{arsinh}y)$. Mollifying this function and pulling it back to the rapidity variable gives smooth profiles converging in $C^2_{\rm loc}$ and preserving the corresponding profile condition. The current equations give
\bel{eq:current-second-density-derivatives}
\alpha_{\vartheta,\mu\mu} \, = \frac{\eps^{2}p'}{h^2}(\alpha_{\vartheta,vv}-\alpha_{\vartheta}), \qquad \beta_{\vartheta,\mu\mu} \, = \frac{\eps^{2}p''}{h}(\beta_{\vartheta}+\alpha_{\vartheta,v})+\frac{\eps^{2}p'}{h^2}(\beta_{\vartheta,vv}-\beta_{\vartheta}).
\ee
These formulas and the mixed-derivative identities give locally uniform convergence of the Hessians on compact non-vacuum state sets, extending the conclusions to all $\entropygenerator\in C^2_{\rm loc}(\Rbb)$.

Preservation of the mass-energy pressure condition follows from Lemma~\ref{lem:effective-mass-energy-convexity}. Its coefficient calculation can also be written as
\bel{eq:effective-E-convexity}
D_\varpi^3 G_{\rm E}[p_\varpi]\bigl(\meff{\varpi}(\mu) \bigr) =(1+\sigma_\varpi)^2G_{\rm E}(\mu)+4c^3\bigl((1+\sigma_\varpi)^2(1-c) -D_\varpi^{3/2}+cD_\varpi \bigr).
\ee
The last parenthesis is the non-negative expression factorized in the proof of that lemma, divided by \(2c>0\), so \(G_{\rm E}\geq0\) implies \(G_{\rm E}[p_\varpi]\geq0\) for every \(\varpi\geq0\).
\ese
\end{proof}

\subsubsection*{Acknowledgments}

This research was initiated when the first author (BLF) was a research fellow at the Princeton Center for Theoretical Science (PCTS), Princeton University, and the second author (PLF) was a visiting research fellow at the Courant Institute for Mathematical Sciences, New York University. During the completion of this project, the authors were partially supported by the ANR project ANR-23-CE40-0010-02 and the ERC project MSCA-101131233.

\phantomsection
\subsubsection*{Note on AI use}
After we had written the first draft, we used ChatGPT 6 to help reorganize the manuscript and simplify some mathematical arguments before making a human-only check of the paper.

\small
\appendix

\section{Current equation in Riemann invariants}
\label{appendix:AB-from-ab}

\subsubsection{Transformation to Riemann invariants} 

We give the derivation of the fixed-frame current equation used in \autoref{section-5}, expressing it in density and rapidity and calculating the coefficients that become singular at vacuum.

We write \(\theta=v\) for the fluid rapidity and take \(w_+=\theta+k(\mu)\), \(w_-=\theta-k(\mu)\), with \(k'=c/h\), and set \(\kappa=\artanh c\). The characteristic speeds are \(\lambda_\pm= \tanh(\theta\pm\kappa)\). For a fixed-frame density-flux pair \((F^{0},F^{1})\), the flux identities are    \(F^{1}_{w_+}= \lambda_+ F^{0}_{w_+}\) and \(F^{1}_{w_-}= \lambda_-F^{0}_{w_-}\). Equality of the mixed flux derivatives gives
\be
(\lambda_+-\lambda_-)F^{0}_{w_+w_-} +(\del_{w_-}\lambda_+)F^{0}_{w_+}- (\del_{w_+}\lambda_-)F^{0}_{w_-}=0.
\ee
Since $\lambda_+-\lambda_-=\sinh(2\kappa)/(\cosh(\theta+\kappa)\cosh(\theta-\kappa))>0$, division yields the coefficients $a,b$ in the equations~\eqref{eq:ab-def}--\eqref{eq:ab-explicit}, with $v=\theta$. Thus the current equation is~\eqref{eq:entropy-eq-riemann}, and the signs of both coefficients are those given in the equation~\eqref{eq:ab-explicit}.

The direct change to the fixed-frame equation of~\autoref{section-5} is
\be
\theta= \frac{w_++w_-}{2}, \qquad {k}(\mu) = \frac{w_+-w_-}{2}, \qquad d_{\rm ac}(\mu) \coloneqq \frac{d\mu}{dk} \, = \frac{\mu+\eps^{2}p(\mu)}{c(\mu)}.
\ee
The coordinate derivatives are
\be
\del_{w_+}= \frac12\del_\theta+ \frac{d_{\rm ac}}{2}\del_\mu, \qquad \del_{w_-}= \frac12\del_\theta-\frac{d_{\rm ac}}{2}\del_\mu,
\ee
Their composition gives the mixed second derivative
\be
F^{0}_{w_+w_-} \, = \frac14F^{0}_{\theta\theta}-\frac{d_{\rm ac}^2}{4}F^{0}_{\mu\mu}-\frac{d_{\rm ac}d_{\rm ac}'}{4}F^{0}_{\mu}.
\ee
Substitution in the current equation in Riemann invariants gives
\be
d_{\rm ac}^2F^{0}_{\mu\mu}-F^{0}_{\theta\theta}+ \bigl(d_{\rm ac}d_{\rm ac}'-2d_{\rm ac}(a-b) \bigr)F^{0}_{\mu}-2(a+b)F^{0}_{\theta}=0
\ee
Division by $d_{\rm ac}^2$ gives the density--rapidity equation
\be
F^{0}_{\mu\mu}-k'(\mu)^2F^{0}_{\theta\theta}+ \Bigl(-\frac{k''(\mu)}{k'(\mu)}-2k'(\mu)(a-b)\Bigr)F^{0}_{\mu}-2k'(\mu)^2(a+b)F^{0}_{\theta}=0.
\ee

\subsubsection{The fixed-frame current equation} 

Eliminating the flux from the  equations for the fixed-frame current density \(F^{0}= F^{0}(\mu,v)\), we obtain
\bel{eq:entropy-mu-v}
\aligned & L_{\mu,v}(F^{0}) =0, \qquad\quad L_{\mu,v}(F^{0}) \coloneqq F^{0}_{\mu\mu}+ \gamma_2(\mu)F^{0}_{vv}+ \gamma_1(\mu,v)F^{0}_{\mu}+ \gamma_3(\mu,v)F^{0}_{v}, \endaligned
\ee
\bse
The coefficient of the second rapidity derivative is
\be
\gamma_2(\mu) =-k'(\mu)^2=-\frac{c(\mu)^2}{(\mu+\eps^{2}p(\mu))^2} \, =-\frac{\eps^{2}p'(\mu)}{(\mu+\eps^{2}p(\mu))^2},
\ee
The coefficients of the first derivatives are
\be
\gamma_1(\mu,v) =-\frac{k''(\mu)}{k'(\mu)}-2k'(\mu)\bigl(a-b \bigr), \qquad \quad \gamma_3(\mu,v) =-2k'(\mu)^2\bigl(a+b \bigr).
\ee
\ese
Writing \(v= \theta\), the coefficients \(a+b\) and \(a-b\) are
\be
\aligned a+b & =-\frac{\bigl(1-\kappa_k(\mu) \bigr)\bigl(1-c(\mu)^2 \bigr)\tanh v}{1-c(\mu)^2\tanh^2 v}, \qquad\quad a-b = \frac{\bigl(1-\kappa_k(\mu) \bigr)\bigl(1+c(\mu)^2\tanh^2 v \bigr)}{\sinh(2\kappa(\mu))\bigl(1-c(\mu)^2\tanh^2 v \bigr)}, \endaligned
\ee
Here we recall the sound-speed rapidity and its derivative with respect to acoustic distance:
\be
\kappa(\mu) = \artanh(c(\mu)), \qquad \quad {\kappa_k}(\mu) = \frac{\mu+\eps^{2}p(\mu)}{c(\mu)(1-c(\mu)^2)}\,c'(\mu).
\ee
Substituting these expressions and using \(h'=1+c^2\), we obtain the coefficients directly in density--rapidity variables:
\be
\label{eq:fixed-frame-expl-coefficients}
\begin{aligned} \gamma_1(\mu,v)&= \frac{2cc'\tanh^2v+2c^2h^{-1}(1-\tanh^2v)}{1-c^2\tanh^2v}, \qquad \quad \gamma_3(\mu,v) = \frac{2c\bigl(c(1-c^2) -hc' \bigr)\tanh v}{h^2(1-c^2\tanh^2v)}. \end{aligned}
\ee
The denominator is positive for every positive density and finite rapidity, since \(0<c<1\). At \(v=0\), these formulas reduce to \(\gamma_1=2c^2/h\) and \(\gamma_3=0\). The rapidity dependence concerns the fixed-frame time component; the two-component current leads instead to the \emph{rapidity-independent} equation used in~\autoref{section-5}.

We set $D\coloneqq\cosh(\theta+\kappa)\cosh(\theta-\kappa)$. The difference and product identities are
\be
\cosh^2(\theta-\kappa)-\cosh^2(\theta+\kappa)=-\sinh(2\theta)\sinh(2\kappa),\qquad 2D=\cosh(2\theta)+\cosh(2\kappa).
\ee
The corresponding sum identity is
\be
\cosh^2(\theta-\kappa)+\cosh^2(\theta+\kappa)=1+\cosh(2\theta)\cosh(2\kappa)
\ee
Adding and subtracting the equations~\eqref{eq:ab-explicit}, and using these three identities, gives
\be
a+b=-\frac{1-\kappa_k}{2}\frac{\sinh(2\theta)}D=-(1-\kappa_k)\frac{\sinh(2\theta)}{\cosh(2\kappa)+\cosh(2\theta)},
\ee
\be
a-b=\frac{1-\kappa_k}{2\sinh(2\kappa)}\frac{1+\cosh(2\kappa)\cosh(2\theta)}D=\frac{1-\kappa_k}{\sinh(2\kappa)}\frac{1+\cosh(2\kappa)\cosh(2\theta)}{\cosh(2\kappa)+\cosh(2\theta)}.
\ee
Substituting $c=\tanh\kappa$ and $v=\theta$ gives the density--rapidity expressions for $a+b$ and $a-b$ above, and hence the equations~\eqref{eq:fixed-frame-expl-coefficients}.

\subsubsection{Vacuum singularity} 

We use the normalization~\eqref{eq:standard-acoustic-distance} and recall \(\lambda=(3-\gamma)/(2(\gamma-1))\), as in the construction of the kernels. The vacuum pressure estimates give \(k(\mu)\sim2c(\mu)/(\gamma-1)\) and \(\kappa_k(\mu)\to(\gamma-1)/2\). In particular, we find
\be
\label{eq:Riemann-vacuum-coefficients}
k(\mu)a(w_+,w_-) \to \frac{\lambda}{2}, \qquad {k}(\mu)b(w_+,w_-) \to -\frac{\lambda}{2} \qquad(\mu\to 0).
\ee
These limits are uniform in \(\theta\): the quotient \(\cosh(\theta-\kappa)/\cosh(\theta+ \kappa)\) lies between \(e^{-2\kappa}\) and \(e^{2\kappa}\). The explicit formulas therefore give \(|a|+|b|\leq C/c(\mu)\) for all \(\mu>0\) and all rapidities. Here we use the global bounds for \(\kappa_k\) and the subluminal gap of the given pressure law; both follow from its asymptotics and continuity on the intervening density interval. 
Since    \(w_+-w_-=2k\), the leading coefficients are    \(\lambda/(w_+-w_-)\) and \(-\lambda/(w_+-w_-)\). They vanish at leading order when \(\gamma=3\), but bounded remainders do not follow from this cancellation alone. Their signs are governed by \(1-\kappa_k\), not by the genuine-nonlinearity factor. This distinction is why the positivity argument in~\autoref{section-7} uses the first-order system with non-negative coupling terms for the current components rather than a sign assumption on \(a\) and \(b\).

\section{Liouville--Green estimates}
\label{appendix-olver}

\subsubsection{The exponential regime}

We prove the Liouville--Green estimates for exponential and oscillatory solutions, including bounds for their derivatives and continuous dependence on parameters. We use a Liouville transformation and a Volterra integral equation~\cite[Chapter~6]{Olver} on half-lines approaching vacuum or large density where the coefficient has no turning point. The resulting estimates are used in Sections~\ref{section-6} and~\ref{section-fourier-continuation}. We take \(I\)  to be a half-line  and take \(Q\in C^2(I)\)  to be positive. We use the error-control  function from the equation~\eqref{eq:appendix-olver-error-control-density}.

\begin{proposition}[Liouville--Green estimate in the exponential regime]
\label{theo:appendix-olver-exponential}
There exists a universal constant \(C>0\) with the following properties.
\bei 

\item \emph{Left endpoint.} Let \(I=(-\infty,z_0]\), assume    \(\int_{-\infty}^{z_0}\sqrt{Q(r)}\,dr=+ \infty\), and set
\bel{eq:appendix-olver-left-variation}
\mathcal V_- (z) \coloneqq \int_{-\infty}^z|\mathcal E_Q(r)|\,dr.
\ee
If \(\mathcal V_- (z_0)<+ \infty\), then the equation
\bel{eq:appendix-olver-left-exponential-equation}
\del_z^2Y-Q(z)Y=0
\ee
admits a fundamental pair \(Y_-^+ \), \(Y_-^-\) satisfying
\bel{eq:appendix-olver-left-exponential-basis}
Y_-^\pm(z) =Q(z)^{-1/4}\exp\!\Bigl(\pm\int_z^{z_0}\sqrt{Q(r)}\,dr\Bigr)\bigl(1+ \auxeps_-^\pm(z) \bigr),
\ee
and the relative errors satisfy
\bel{eq:appendix-olver-left-exponential-error}
\bigl|\auxeps_-^- (z) \bigr|\le C\Bigl(\exp\!\bigl(C\mathcal V_- (z) \bigr) -1\Bigr), \qquad \lim_{z\to-\infty}\auxeps_-^\pm(z) =0.
\ee

\item \emph{Right endpoint.} Let \(I=[z_0,+ \infty)\), assume    \(\int_{z_0}^{+ \infty}\sqrt{Q(r)}\,dr=+ \infty\), and set
\bel{eq:appendix-olver-right-variation}
\mathcal V_+(z) \coloneqq \int_z^{+ \infty}|\mathcal E_Q(r)|\,dr.
\ee
If \(\mathcal V_+(z_0)<+ \infty\), then the equation~\eqref{eq:appendix-olver-left-exponential-equation} admits a fundamental pair \(Y_+^+ \), \(Y_+^-\) satisfying
\bel{eq:appendix-olver-right-exponential-basis}
Y_+^\pm(z) =Q(z)^{-1/4}\exp\!\Bigl(\pm\int_{z_0}^z\sqrt{Q(r)}\,dr\Bigr)\bigl(1+ \auxeps_+^\pm(z) \bigr),
\ee
and the relative errors satisfy
\bel{eq:appendix-olver-right-exponential-error}
\bigl|\auxeps_+^- (z) \bigr|\le C\Bigl(\exp\!\bigl(C\mathcal V_+(z) \bigr) -1\Bigr), \qquad \lim_{z\to+ \infty}\auxeps_+^\pm(z) =0.
\ee

\item For the growing solution, a different bound holds. In either orientation, let \(\tau\) be the phase tending to \(+ \infty\) at the endpoint and write \(v(\tau) = \mathcal V_\pm(z(\tau))\). For sufficiently large \(T\), one may choose the growing solution so that, for \(\tau\geq T\),
\be
|\auxeps^+(\tau)|+|\del_\tau\auxeps^+(\tau)|\leq C \Bigl( v(\tau)+e^{-2(\tau-T)} + \int_T^\tau e^{-2(\tau-r)}v(r)\,dr \Bigr).
\ee
In the same phase variable, the decreasing remainder satisfies \(|\auxeps^-|+|\del_\tau\auxeps^-|\leq C(\exp(Cv) -1)\). Thus both remainders and their phase derivatives tend to zero. The growing remainder need not be bounded by the tail  \(v(\tau) = \mathcal V_\pm(z(\tau))\) alone.
\eei
\end{proposition}

\begin{proof}
In either orientation, we set \(Y=Q^{-1/4}W(\tau)\), with \(\tau\) the outward phase specified in the statement. Differentiation gives
\bel{eq:appendix-olver-exponential-transformation}
W''-W= \mathcal R_QW, \qquad \mathcal R_Q(\tau(z)) = \frac{\mathcal E_Q(z)}{\sqrt{Q(z)}}, \qquad |\mathcal R_Q|\,d\tau=|\mathcal E_Q|\,|dz|.
\ee
\bse
For the decreasing solution \(f=e^{-\tau}(1+e)\), the integral equation is
\be
e(\tau) = \frac12\int_\tau^{+\infty} (1-e^{-2( r-\tau)})\mathcal R_Q(r)(1+e(r))\,dr.
\ee
Its iterated kernels give \(|e(\tau)|\leq e^{v(\tau)/2}-1\). Differentiation under the integral gives
\be
e'(\tau) =-\int_\tau^{+\infty} e^{-2(r-\tau)}\mathcal R_Q(r)(1+e(r))\,dr, \qquad |e'(\tau)|\leq v(\tau)e^{v(\tau)/2}.
\ee
The integrable majorant justifies this differentiation, and the same construction extends the solution over the remaining finite interval and proves the decreasing-solution bounds.
\ese

\bse
We choose \(T\) so that \(|e|\leq1/2\) on \([T,\infty)\). Reduction of order gives the independent solution
\be
g(\tau) \coloneqq2f(\tau)\int_T^\tau f(r)^{-2}\,dr, \qquad W(f,g) =2.
\ee
Dividing by \(e^\tau\), we find
\be
\frac{g(\tau)}{e^\tau} =(1+e(\tau)) \Bigl( 1-e^{-2(\tau-T)} +2\int_T^\tau e^{-2(\tau-r)}\bigl((1+e(r))^{-2}-1 \bigr)\,dr \Bigr).
\ee
Since \(|(1+e)^{-2}-1|\leq C|e|\leq Cv\) on this half-line, this is the growing-solution bound. Differentiating the last formula uses only \(e'\) and the identity    \(\del_\tau\int_T^\tau e^{-2(\tau-r)}k(r)\,dr=k(\tau) -2\int_T^\tau e^{-2(\tau-r)}k(r)\,dr\), with \(k=(1+e)^{-2}-1\). The same right-hand side therefore controls the derivative of \(g/e^\tau-1\). Both bounds tend to zero: we split the integral at a fixed large phase, then use    \(v(r)\to0\)  on the remaining part. Returning to \(z\) proves both endpoint assertions. 
\ese
\end{proof}

\subsubsection{The oscillatory regime} 

We now switch the sign in front of $Q$ and construct oscillatory solutions for the operator $\del_z^2Y+Q(z)Y$, with analogous error bounds.

\begin{proposition}[Liouville--Green estimate in the oscillatory regime]
\label{theo:appendix-olver-oscillatory}
Let \(Q\in C^2([z_0,+ \infty))\) be positive and assume that
\bel{eq:appendix-olver-oscillatory-phase}
\int_{z_0}^{+ \infty}\sqrt{Q(r)}\,dr=+ \infty, \qquad \mathcal V_+(z_0) = \int_{z_0}^{+ \infty}|\mathcal E_Q(r)|\,dr<+ \infty.
\ee
Then the equation
\bel{eq:appendix-olver-oscillatory-equation}
\del_z^2Y+Q(z)Y=0
\ee
admits two complex conjugate solutions \(Y_+^+ \) and \(Y_+^-\) which form a fundamental system and satisfy
\bel{eq:appendix-olver-oscillatory-basis}
Y_+^\pm(z) =Q(z)^{-1/4}\exp\!\Bigl(\pm i\int_{z_0}^z\sqrt{Q(r)}\,dr\Bigr)\bigl(1+ \auxeps_+^\pm(z) \bigr),
\ee
whose relative errors satisfy
\bel{eq:appendix-olver-oscillatory-error}
\bigl|\auxeps_+^\pm(z) \bigr|\le C\Bigl(\exp\!\bigl(C\mathcal V_+(z) \bigr) -1\Bigr), \qquad \lim_{z\to+ \infty}\auxeps_+^\pm(z) =0,
\ee
where \(C>0\) is universal. In the phase variable    \(\tau\coloneqq \int_{z_0}^z\sqrt{Q(r)}\,dr\), the same bound holds for \(|\auxeps_+^\pm|+|\del_\tau\auxeps_+^\pm|\). With the convention \(W(Y_1,Y_2) =Y_1Y_2'-Y_1'Y_2\), one has \(W(Y_+^+,Y_+^-) =-2i\).
\end{proposition}

\begin{proof}
\bse
With    \(\tau(z) = \int_{z_0}^z\sqrt{Q(r)}\,dr\)  and \(Y(z) =Q(z)^{-1/4}W(\tau(z))\), we obtain \(\del_\tau^2W+W= \mathcal R_Q(\tau)W\), with the same function \(\mathcal R_Q\) as in the equation~\eqref{eq:appendix-olver-exponential-transformation}. We put \(v(\tau)\coloneqq\int_\tau^{+\infty}|\mathcal R_Q(r)|\,dr\). For \(W^+=e^{i\tau}(1+e)\), the terminal integral equation is
\be
e(\tau) = \int_\tau^{+\infty} \frac{e^{2i(r-\tau)}-1}{2i}\, \mathcal R_Q(r)(1+e(r))\,dr.
\ee
Since the kernel has modulus at most one, successive substitution gives a convergent series: the term with \(j\) factors of \(\mathcal R_Q\) has modulus at most \(v(\tau)^j/j!\). Consequently \(|e(\tau)|\leq e^{v(\tau)}-1\). Differentiation, justified by the integrable majorant, gives
\be
e'(\tau) =-\int_\tau^{+\infty} e^{2i(r-\tau)} \mathcal R_Q(r)(1+e(r))\,dr, \qquad |e'(\tau)|\leq v(\tau)e^{v(\tau)}.
\ee
A second differentiation verifies the differential equation. In particular, \(W^+=e^{i\tau}(1+o(1))\) and \((W^+)'=ie^{i\tau}(1+o(1))\). Since \(Q\) is real, we take \(W^-= \overline{W^+}\). The Wronskian is constant and its limit is \(W_\tau(W^+,W^-) =-2i.\) Returning to \(z\) preserves this Wronskian, since \(Q^{-1/2}\del_z\tau=1\); the terms obtained by differentiating the amplitude factor \(Q^{-1/4}\) cancel. Thus \(Y^+ \) and \(Y^-\) are independent. Finally, \(v(\tau(z)) = \mathcal V_+(z)\), which proves the error estimate with a universal constant.
\ese
\end{proof}

\subsubsection{Values and derivatives in logarithmic density}

For the connection formulas of~\autoref{section-fourier-continuation}, we convert the phase estimates in~\(\tau\) to the pair \((Y,\del_zY)\) in logarithmic density~\(z\). We denote by \(\sigma\in\{-1,1\}\)  the orientation of the phase, so that $ \del_z\tau= \sigma\sqrt Q. $ We assume that
\be
Y=Q^{-1/4}e^{\omega\tau}(1+\auxeps),
\ee
where \(\omega= \pm1\) in the exponential regime and \(\omega= \pm i\) in the oscillatory regime. Since $ \del_z\auxeps = \sigma\sqrt Q\,\del_\tau\auxeps, $ a direct differentiation gives the identity
\be
\label{eq:olver-original-derivative}
Q^{-1/4}e^{-\omega\tau}\del_zY = \sigma\omega(1+\auxeps)+\sigma\del_\tau\auxeps -\frac{\del_zQ}{4Q^{3/2}}(1+\auxeps).
\ee
For the asymptotic formulas, we assume
\be
\auxeps\to0, \qquad \del_\tau\auxeps\to0, \qquad \frac{\del_zQ}{Q^{3/2}}\to0;
\ee
the identity~\eqref{eq:olver-original-derivative} then gives
\be
Y=Q^{-1/4}e^{\omega\tau}(1+o(1)), \qquad \del_zY= \sigma\omega Q^{1/4}e^{\omega\tau}(1+o(1)).
\ee
The approximation therefore controls \(Y\) and \(\del_zY\) in~\(z\).

To extend these conclusions to compact parameter families, we assume joint continuity of $Q$, $\del_zQ$ and $\del_z^2Q$ on compact intervals, bounds $0<q_*\leq Q\leq q^*<+\infty$, and uniformly vanishing tails of $\int(|\del_zQ|+|\del_z^2Q|)\,dz$. Integrability gives a limit of $\del_zQ$, necessarily zero, and hence
\be
|\del_zQ(z)| \leq\int_z^{+\infty}|\del_z^2Q(r)|\,dr \to 0,
\ee
uniformly in the parameters. Thus $\del_zQ$ is uniformly bounded, and $(\del_zQ)^2\leq C|\del_zQ|$ gives uniformly vanishing tails for $\mathcal E_Q$. We truncate the terminal Volterra series at a finite logarithmic density $L$. Their kernels and iterates are continuous in the parameters on that compact interval. The factorial bounds make the series uniformly convergent, and the truncation error, also for the phase derivative, is bounded by a fixed multiple of the uniform tail $\sup\int_L^{+\infty}|\mathcal E_Q|\,dz\to0$. This proves continuity without requiring a pointwise integrable majorant for the family. Formula~\eqref{eq:olver-original-derivative} gives continuity and uniform asymptotics of the first $z$ derivatives.

Uniform smallness of the decreasing exponential remainder permits one lower integration point $Z$ for the reduction of order~\eqref{eq:connection-continuous-growing-solution}. Its Cauchy data $(0,2/D(Z))$ are continuous in the parameters. The growing-error bound in Proposition~\ref{theo:appendix-olver-exponential} is uniform: we split its convolution at a fixed large phase and use the uniformly small tails on the remainder. The phase values at $Z$ stay bounded, so its separate exponential term also tends uniformly to zero. This applies to the phase derivative as well. For the right-hand phases used at large density, the Wronskians remain $-2i$ and $-2$, since $Q^{-1/2}\del_z\tau=1$; ordinary finite-interval ODE continuation preserves parameter continuity. For the large-density coefficients of Section~\ref{section-fourier-continuation}, the frequency ranges in a compact set separated from the thresholds; Lemma~\ref{lem:large-density-B-condition} gives the coefficient tails required here.

\section{Pressure asymptotics and examples}
\label{appendix-pressure-endpoints}

\subsection{Pressure asymptotics}

\subsubsection{Weighted derivatives of the remainders}

We derive the pressure asymptotics implied by the hypotheses, verify the rescaling of the pressure, and give examples on the full density range. For a $C^2$ function $f$ such that $f'$ and $\mu f''$ are integrable at either endpoint, $(\mu f')'=f'+\mu f''$ is integrable. The endpoint limit of $\mu f'$ therefore exists and must be zero, since a nonzero limit contradicts $f'\in L^1$. Integration from that endpoint gives
\be
\label{eq:weighted-remainder-first-derivative}
|\mu f'(\mu)|\leq
\begin{cases}
\displaystyle\int_0^\mu\bigl(|f'(r)|+|r f''(r)| \bigr)\,dr&\text{near vacuum},\\
\displaystyle\int_\mu^{+\infty}\bigl(|f'(r)|+|r f''(r)| \bigr)\,dr&\text{at large density}.
\end{cases}
\ee
In particular, $\mu f'$ is bounded, and applying this estimate to the pressure remainders also controls products of their derivatives and will justify the quadratic term in the Liouville--Green variation.

\subsubsection{Pressure asymptotics}

For the admissibility conditions of Definition~\ref{def-pression}, we use the local abbreviation $B_{\rm ref}= \mu^2p'/(\mu+\eps^2p)^2$. Lemma~\ref{lem:epsilon-pressure-admissibility}, below, gives the corresponding endpoint estimates. Proposition~\ref{prop:fixed-effective-admissibility} gives admissibility at each fixed parallel momentum; quantitative variation of this family is studied in~\cite{TorusT2-3-asymptotic-2026}.

\begin{proposition}[Pressure limits]
\label{prop:pressure-endpoint-asymptotics}
Every admissible equation of state in the sense of Definition~\ref{def-pression} satisfies
\bel{eq:pressure-vacuum-asymptotics}
p'(\mu)=\kappa_0\mu^{\gamma-1}(1+o(1)), \qquad p(\mu)=\frac{\kappa_0}{\gamma}\mu^\gamma(1+o(1)) \qquad(\mu\to0),
\ee
and its large-density limits, with $k_\infty\in(0,\eps^{-1})$ as in~\eqref{eq:intro-infty-B}, are
\bel{eq:pressure-infinity-asymptotics}
\frac{p(\mu)}{\mu}\to k_\infty^2, \qquad p'(\mu)\to k_\infty^2 \qquad(\mu\to+\infty).
\ee
\end{proposition}
\begin{proof}
We put $y(\mu)\coloneqq p(\mu)/\mu$. Since $0\leq y<\eps^{-2}$, the identity
\bel{eq:pressure-B-identity}
p'(\mu)=B_{\rm ref}(\mu)(1+\eps^2y(\mu))^2
\ee
and the vacuum assumption imply $p'(\mu)\to0$. Since $p(0)=0$, averaging $p'$ on $(0,\mu)$ gives $y(\mu)\to0$. Substitution in the identity~\eqref{eq:pressure-B-identity} yields the first formula in~\eqref{eq:pressure-vacuum-asymptotics}, and integration from zero gives the second one. On the other hand, at large density, direct differentiation gives
\bel{eq:pressure-ratio-endpoint-ode}
\frac{dy}{d\log\mu}=B_{\rm ref}(\mu)(1+\eps^2y)^2-y.
\ee
The limiting right-hand side has, on $[0,\eps^{-2})$, the unique zero $y=k_\infty^2$; its other zero is $\eps^{-4}k_\infty^{-2}>\eps^{-2}$. For the comparison argument, we put $z=\log\mu$ and $F(z,y)\coloneqq B_{\rm ref}(e^z)(1+\eps^2y)^2-y$. To prove that the ratio has a uniform upper bound below  \(\eps^{-2}\), we choose $r_0\in(k_\infty^2,\eps^{-2})$. The limiting vector field is strictly negative on $[r_0,\eps^{-2}]$, so $F(z,y)\leq-c_0<0$ there for all sufficiently large $z$. Since $y<\eps^{-2}$, the solution enters $[0,r_0]$ in finite $z$-time and remains there. Continuity on compact positive-density intervals and $y(\mu)\to0$ at vacuum then give $r_*\coloneqq\sup_{\mu>0}p(\mu)/\mu<\eps^{-2}$. Since $0\leq y\leq r_*$, for every sufficiently small $\auxeps>0$ there exist $z_\auxeps$ and $c_\auxeps>0$ such that
\be
\label{eq:pressure-ratio-trapping}
F(z,y)\geq c_\auxeps\quad\text{for }0\leq y\leq k_\infty^2-\auxeps, \qquad F(z,y)\leq-c_\auxeps\quad\text{for }k_\infty^2+\auxeps\leq y\leq r_*, \qquad z\geq z_\auxeps.
\ee
The vector field points into $[k_\infty^2-\auxeps,k_\infty^2+\auxeps]$ at both endpoints, and a solution outside this interval enters it in finite $z$-time. The interval is then forward invariant, and letting $\auxeps$ tend to zero gives $y(\mu)\to k_\infty^2$. The equation~\eqref{eq:pressure-B-identity} finally gives $p'(\mu)\to k_\infty^2$.
\end{proof}

\subsection{Rescaling and pressure conditions}

\subsubsection{Rescaling of the pressure}
\label{section-pressure-light-speed}

\begin{lemma}[Scaled pressure]
\label{lem:epsilon-pressure-admissibility}
Fix an admissible equation of state $p$ in the sense of Definition~\ref{def-pression}. The pressure law $\mu\mapsto\eps^2p(\mu)$ is admissible with unit inverse light speed, with the same vacuum exponent $\gamma$, vacuum coefficient $\eps^2\kappa_0$, and limiting sound speed $\eps k_\infty$. 
\end{lemma}

\begin{proof}
We put $A(\mu)\coloneqq p(\mu)/\mu$ and recall the reference coefficient $B_{\rm ref}(\mu)=p'(\mu)(1+\eps^2A(\mu))^{-2}$ from~\eqref{eq:intro-vacuum-B} and~\eqref{eq:intro-infty-B}. The coefficient in the current equation is
\be
\label{eq:fixed-parameter-pressure-coefficient}
B(\mu)=\frac{\eps^2\mu^2p'(\mu)}{(\mu+\eps^2p(\mu))^2} =\eps^2B_{\rm ref}(\mu).
\ee
We verify the derivative bounds for the last factor. Near vacuum, conditions~\eqref{eq:intro-vacuum-B}--\eqref{eq:intro-vacuum-r} give $B_{\rm ref}=\Ocal(\mu^{\gamma-1})$, $B_{\rm ref}'\in L^1$, and $\mu B_{\rm ref}''\in L^1$. Since $p'=B_{\rm ref}(1+\eps^2A)^2$ and $0\leq A<\eps^{-2}$, we have $p',A=\Ocal(\mu^{\gamma-1})$ and $A'=(p'-A)/\mu=\Ocal(\mu^{\gamma-2})\in L^1$. Differentiation gives
\be
p''=B_{\rm ref}'(1+\eps^2A)^2+2\eps^2B_{\rm ref}(1+\eps^2A)A'.
\ee
It follows that $p''\in L^1$, $\mu A''=p''-2A'\in L^1$, and $\mu A'$ is bounded near vacuum.

At large density, Proposition~\ref{prop:pressure-endpoint-asymptotics} gives $A\to k_\infty^2$. With $z=\log\mu$, differentiation of $A_z=B_{\rm ref}(e^z)(1+\eps^2A)^2-A$ yields
\be
A_{zz}+\bigl(1-2\eps^2B_{\rm ref}(1+\eps^2A)\bigr)A_z=(1+\eps^2A)^2\del_z\bigl(B_{\rm ref}(e^z)\bigr).
\ee
The coefficient of $A_z$ tends to $(1-\eps^2k_\infty^2)/(1+\eps^2k_\infty^2)>0$, while $\del_z(B_{\rm ref}(e^z))\in L^1(dz)$ by~\eqref{eq:intro-infty-B}--\eqref{eq:intro-infty-r}. We choose a sufficiently large $z_0$ and $d>0$ so that this coefficient is at least $d$ for $z\geq z_0$. Variation of constants gives
\be
|A_z(z)|\leq|A_z(z_0)|e^{-d(z-z_0)} +C\int_{z_0}^ze^{-d(z-r)}\bigl|\del_r(B_{\rm ref}(e^r))\bigr|\,dr.
\ee
Integration gives $A_z\in L^1(dz)\cap L^\infty$, and its differential equation gives $A_{zz}\in L^1(dz)$. In the density variable, these conclusions are
\be
A'\in L^1(d\mu),\qquad \mu A''\in L^1(d\mu),\qquad \sup|\mu A'|<+\infty.
\ee
A smooth function of bounded $A$, with denominator bounded away from zero, has the same two derivative-integrability properties: its second derivative contains $A''$ and $(A')^2$, and $\mu(A')^2\leq\sup|\mu A'|\,|A'|$. Multiplication by a pressure remainder $r$ preserves these properties, since the mixed term satisfies $\mu|r'A'|\leq\sup|\mu A'|\,|r'|$. The factor in the equation~\eqref{eq:fixed-parameter-pressure-coefficient} therefore preserves the weighted derivative bounds in  the assumptions~\eqref{eq:intro-vacuum-r} and~\eqref{eq:intro-infty-r}. The limits $A\to0$ at vacuum and $A\to k_\infty^2$ at infinity give
\be
B(\mu)\sim\eps^2\kappa_0\mu^{\gamma-1}, \qquad B(\mu)\longrightarrow\frac{\eps^2k_\infty^2}{(1+\eps^2k_\infty^2)^2}.
\ee
Scaling preserves the required regularity, $0<\eps^2p'<1$, and positive upper and lower bounds on $\eps^2p/\mu$ at large density. 
\end{proof}

 The endpoint remainders of \(B\) are the same \(r_-\) and \(r_+\) as in~\eqref{eq:intro-vacuum-B} and~\eqref{eq:intro-infty-B}, since \(B=\eps^2B_{\rm ref}\).
\be
\frac{B(\mu)}{\eps^2\kappa_0\mu^{\gamma-1}}=1+r_-(\mu),\qquad \frac{(1+\eps^2k_\infty^2)^2B(\mu)}{\eps^2k_\infty^2}=1+r_+(\mu).
\ee
Each remainder tends to zero at its respective density limit and satisfies the same type of weighted derivative bounds as in the equations~\eqref{eq:intro-vacuum-r} and~\eqref{eq:intro-infty-r}, by the preceding proof. The endpoint estimates for \(B\) use these same remainders. Their constants may depend on $\eps$.

\subsubsection{Pressure signs for positivity and convexity}
\label{section-pressure-sign-properties}

The genuine nonlinearity condition is expressed by the sign of $G_{\rm N}$, while the mass-energy convexity condition uses $G_{\rm E}$. The latter condition is stronger, since, with $c=\eps\sqrt{p'}$,  their difference is
\be
G_{\rm N}-G_{\rm E}=4c^3(1-c)>0.
\ee
These are separate hypotheses for properties of the current families; the construction of the kernels requires only admissibility.

Under the additional restriction \(p'<1\), two sufficient conditions on the pressure are
\be
(\mu+p)p''+2p'(1-p')>0, \qquad (\mu+p)p''+2p'(1-\sqrt{p'})^2\geq0.
\ee
We verify this implication pointwise in $\mu$. We write $h=\mu+\eps^2p$ and define $h_{\rm ref}\coloneqq\mu+p$. For the first sign, we set $A\coloneqq1-\eps^2p'$ and $A_{\rm ref}\coloneqq1-p'$; for the second, we set $A\coloneqq(1-\eps\sqrt{p'})^2$ and $A_{\rm ref}\coloneqq(1-\sqrt{p'})^2$. In both cases $0<h/h_{\rm ref}\leq1$ and $A\geq A_{\rm ref}>0$. The non-strict reference inequality gives $p''\geq-2p'A_{\rm ref}/h_{\rm ref}$, and therefore
\be
hp''+2p'A \geq2p'\bigl(A-(h/h_{\rm ref})A_{\rm ref}\bigr)\geq0.
\ee
The strict reference inequality gives strict positivity in the first case. 

Ordinary pressure convexity, $p''\geq0$, is sufficient for both signs under $0<\eps^2p'<1$, but is not required. Both signs are preserved by the reduction to an effective pressure, and the condition $G_{\rm E}\geq0$ yields convexity for the prescribed profile family. The reference energy is convex independently of this condition, as proved in Proposition~\ref{lem-convexM}.

For the exactly isothermal model $\eps^2p=c_{\rm iso}^2\mu$, where $0<c_{\rm iso}<1$ is the sound speed divided by the speed of light, the expressions are $G_{\rm N}=2c_{\rm iso}^2(1-c_{\rm iso}^2)>0$ and $G_{\rm E}=2c_{\rm iso}^2(1-c_{\rm iso})^2>0$. This model is outside the polytropic-vacuum class. For an admissible pressure, the limits $\mu p''/p'\to\gamma-1$ at vacuum and $\mu p''\to0$, $p'\to k_\infty^2$ at infinity imply both signs near each density limit. Their validity at intermediate densities is an additional hypothesis. The physical polytrope $p=\kappa\mu^\gamma$, $\gamma>1$, satisfies both signs wherever $\eps^2p'<1$, but loses this causal sound-speed bound at large density.

\subsubsection{The vacuum pressure condition}

The vacuum solutions were constructed by the Volterra equation in Section~\ref{section-6-2}. Here we verify additional Liouville--Green integral conditions that can be used to control more precisely the remainder.
The Liouville--Green error-control function of the weight~$Q_-$ defined in~\eqref{eq:def-Qminus-sec9} is
\bel{eq:def-Eminus-Q-sec9}
E_{\rm LG,-} (z,\xi) \coloneqq \frac{4Q_- (z,\xi)\,Q_{-,zz}(z,\xi) -5Q_{-,z}(z,\xi)^2} {16\,Q_- (z,\xi)^{5/2}},
\ee
and the corresponding variation is
\bel{eq:def-Vminus-sec9}
\mathcal V_- (z,\xi) \coloneqq \int_{-\infty}^{z} | E_{\rm LG,-} (r,\xi)|\,dr.
\ee
We imposed the weighted derivative bounds on the pressure remainder in~\eqref{equa-cond-small} to obtain the integrability condition on \(B\) stated in the following lemma, which is needed for the Liouville--Green estimates.

\begin{lemma}[Vacuum coefficient bounds]
\label{prop:olver-vacuum-pressure}

\

\bei

\item \emph{Integral condition on the function \(B\).} Assume that \(B(\mu)\to0\) as \(\mu\to0\). Then, {for each fixed \(\xi\in\Rbb\),} the condition
\bel{eq:vacuum-B-condition-sec9}
\int_0^{\mu_0} \Bigl( |\mu^2B_{\mu\mu}(\mu)+ \mu B_\mu(\mu)| + |\mu B_\mu(\mu)|^2 \Bigr) \frac{d\mu}{\mu} <+ \infty
\ee
for some \(\mu_0>0\) implies that there exists a \emph{threshold} \(z_- (\xi)\leq\log\mu_0\) such that 
\be
Q_- (z,\xi)\geq1/8 \text{ for all } z\leq z_- (\xi) \quad \text{ and } \quad \mathcal V_- (z_- (\xi),\xi)<+ \infty.
\ee

\item \emph{Integral condition on the pressure remainder.} Assume that there exist \(\gamma>1\), \(\kappa_0>0\), and \(\mu_0>0\) such that
\bel{eq:intro-vacuum-B-recalled}
B(\mu) = \eps^{2}\kappa_0\mu^{\gamma-1}\bigl(1+r_- (\mu) \bigr), \qquad 0<\mu\le\mu_0,
\ee
where \(r_-\in C^2((0,\mu_0])\) satisfies
\bel{eq:intro-vacuum-r-recalled}
\int_0^{\mu_0} \Bigl( |r_-'(\mu)|+|\mu r_-''(\mu)| \Bigr) \,d\mu <+ \infty.
\ee
Then condition~\eqref{eq:vacuum-B-condition-sec9} is satisfied.
\eei
\end{lemma}

\begin{proof}
{{\bf 1.} \emph{Reduction to the coefficient $B$.}} With $\mu=e^z$, the definitions~\eqref{eq:def-Qminus-sec9} and~\eqref{eq:def-Eminus-Q-sec9} give
\be
Q_{-,z}=-(1+\xi^2)\mu B_\mu,\qquad Q_{-,zz}=-(1+\xi^2)(\mu^2B_{\mu\mu}+\mu B_\mu).
\ee
For fixed $\xi$, we decrease $\mu_0$ so that $1/8\leq Q_-\leq3/8$ on $z\leq\log\mu_0$, and set $z_-(\xi)\coloneqq\log\mu_0$. Since the denominator of the  error-control function is bounded below there, substitution of the preceding derivatives, followed by $dz=d\mu/\mu$, yields
\be
\mathcal V_-(z_-(\xi),\xi)\leq C_\xi\int_0^{\mu_0}\bigl(|\mu^2B_{\mu\mu}+\mu B_\mu|+|\mu B_\mu|^2\bigr)\,\frac{d\mu}{\mu}.
\ee

{{\bf 2.} \emph{Endpoint estimates.}} We put $r\coloneqq r_-$, so that $B=\eps^2\kappa_0\mu^{\gamma-1}(1+r)$. By  the assumption~\eqref{eq:intro-vacuum-r-recalled}, both $r'$ and $\mu r''$ are integrable at zero. Thus $r$ has a finite endpoint limit. The identity $(\mu r')'=r'+\mu r''$ also gives a limit of $\mu r'$, necessarily zero: a nonzero limit would make $r'$ asymptotic to a nonzero multiple of $1/\mu$, contradicting its integrability. In particular, $r$ and $\mu r'$ are bounded. Differentiation gives
\be
\mu B_\mu=\eps^2\kappa_0\mu^{\gamma-1}\bigl((\gamma-1)(1+r)+\mu r'\bigr),
\ee
\be
\mu^2B_{\mu\mu}+\mu B_\mu=\eps^2\kappa_0\mu^{\gamma-1}\bigl((\gamma-1)^2(1+r)+(2\gamma-1)\mu r'+\mu^2r''\bigr).
\ee
Therefore, using $\mu^{\gamma-1}|\mu r''|\leq\mu_0^{\gamma-1}|\mu r''|$, we obtain
\be
\frac{|\mu^2B_{\mu\mu}+\mu B_\mu|}{\mu}\leq C\mu^{\gamma-2}+C\mu_0^{\gamma-1}|\mu r''|,\qquad \frac{|\mu B_\mu|^2}{\mu}\leq C\mu^{2\gamma-3}.
\ee
Both right-hand sides are integrable at zero because $\gamma>1$ and $\mu r''\in L^1$. This proves the condition~\eqref{eq:vacuum-B-condition-sec9}, including its linear and quadratic terms.
\end{proof}

\subsection{Admissible pressure laws}
\label{section-pressure-examples}

\subsubsection{Pressure laws near vacuum}  

Near vacuum, the admissible class contains asymptotically polytropic laws with higher-order corrections, for instance
\be
p(\mu)=\mu^\gamma+C\mu^{\gamma_1}, \qquad 1<\gamma<\gamma_1,
\ee
where $\gamma_1$ may be arbitrarily close to $\gamma$. It also allows logarithmic corrections such as
\be
p(\mu)=\mu^\gamma+C\frac{\mu^\gamma}{(-\log\mu)^{1+\auxeps}}, \qquad \auxeps>0,
\ee
and oscillatory corrections of the form
\be
p(\mu)=\mu^\gamma\Bigl(1+\frac{\sin((-\log\mu)^\beta)}{(-\log\mu)^{1+\auxeps}}\Bigr), \qquad \gamma>1, \quad 0<\beta<1, \quad \auxeps>0.
\ee
These formulas describe the pressure for sufficiently small $\mu$. They may be extended to a high-density branch satisfying the condition~\eqref{hyperbolic-eos-2} and~\eqref{equa-cond-large}, while preserving hyperbolicity. Indeed, each example has $p''>0$ on a sufficiently small vacuum interval. We extend $p'$ smoothly and monotonically from that interval to a constant $k_\infty^2\in(0,\eps^{-2})$, choosing the interval so that its initial derivative is smaller than this constant. Integration gives $p(0)=0$, $0<p'<\eps^{-2}$, and an affine pressure at large density. Its coefficient $B_{\rm ref}$ has a remainder of order $\mu^{-1}$ at infinity, with the required derivative bounds. Since $p''\geq0$, both conditions~\eqref{hyperbolic-eos-3} and~\eqref{eq:intro-mass-energy-convexity-condition} hold.  The analysis uses the endpoint behavior of $B_{\rm ref}$ and its integrable remainder derivatives, allowing departures from both limiting models.

\subsubsection{A pressure law on the full density range}

For every $\gamma>1$, $\kappa_0>0$, and $k_\infty\in(0,\eps^{-1})$, an explicit admissible law with these vacuum and large-density parameters is
\be
\label{eq:global-pressure-example}
p(\mu)=\int_0^\mu\frac{\kappa_0 k_\infty^2 q^{\gamma-1}}{k_\infty^2+\kappa_0q^{\gamma-1}}\,dq.
\ee
This law also satisfies the mass-energy convexity condition~\eqref{eq:intro-mass-energy-convexity-condition} at every positive density. The classes for the construction and positivity of the kernels and for propagation of convexity are nonempty for every $\gamma>1$, without a density cutoff.

\begin{proof}
\bse
{We set $d\coloneqq\kappa_0/k_\infty^2$. Differentiation gives
\be
\label{eq:global-pressure-example-derivatives}
p'(\mu)=\frac{\kappa_0\mu^{\gamma-1}}{1+d\mu^{\gamma-1}},\qquad p''(\mu)=\frac{\kappa_0(\gamma-1)\mu^{\gamma-2}}{(1+d\mu^{\gamma-1})^2}>0.
\ee
Hence $p\in C^3((0,+\infty))\cap C^1([0,+\infty))$, $p(0)=p'(0)=0$, and $0<p'<k_\infty^2<\eps^{-2}$. Averaging $p'$ gives $0<p(\mu)/\mu<k_\infty^2$ and $p(\mu)/\mu\to k_\infty^2$ as $\mu\to+\infty$. Both constitutive signs follow from $p''>0$ and $0<p'<\eps^{-2}$.

To verify the weighted remainder assumptions, we put $A(\mu)=p(\mu)/\mu$ and integrate by parts to obtain
\be
\label{eq:global-pressure-example-ratio-derivatives}
A'(\mu)=\frac1{\mu^2}\int_0^\mu q p''(q)\,dq,\qquad \mu A''(\mu)=p''(\mu)-2A'(\mu).
\ee
Near zero, these formulas and the equation~\eqref{eq:global-pressure-example-derivatives} give $A=\Ocal(\mu^{\gamma-1})$, $A'=\Ocal(\mu^{\gamma-2})$, and $A''=\Ocal(\mu^{\gamma-3})$. We also have
\be
\label{eq:global-pressure-example-vacuum-remainder}
\frac{B_{\rm ref}(\mu)}{\kappa_0\mu^{\gamma-1}}=\frac1{(1+d\mu^{\gamma-1})(1+\eps^2A(\mu))^2},
\ee
so the vacuum remainder satisfies $r_-=\Ocal(\mu^{\gamma-1})$ and $|r_-'|+|\mu r_-''|\leq C\mu^{\gamma-2}$, which is integrable at zero.

At infinity, $p''=\Ocal(\mu^{-\gamma})$ and $p'''=\Ocal(\mu^{-\gamma-1})$. Formula~\eqref{eq:global-pressure-example-ratio-derivatives} gives
\be
\label{eq:global-pressure-example-infinity-ratio-bound}
|A'(\mu)|+|\mu A''(\mu)|\leq C
\begin{cases}
\mu^{-\gamma},&1<\gamma<2,\\
(1+\log\mu)\mu^{-2},&\gamma=2,\\
\mu^{-2},&\gamma>2,
\end{cases}\qquad \mu\geq1.
\ee
All three bounds are integrable, and differentiating $B_{\rm ref}=p'(1+\eps^2A)^{-2}$ gives
\be
\label{eq:global-pressure-example-B-second-derivative}
B_{\rm ref}''=\frac{p'''}{(1+\eps^2A)^2}-\frac{4\eps^2p''A'}{(1+\eps^2A)^3}-\frac{2\eps^2p'A''}{(1+\eps^2A)^3}+\frac{6\eps^4p'(A')^2}{(1+\eps^2A)^4}.
\ee
Thus $B_{\rm ref}'\in L^1(1,+\infty)$ and $\mu B_{\rm ref}''\in L^1(1,+\infty)$: in addition to $\mu p'''$ and $\mu A''$, the products $\mu p''A'$ and $\mu(A')^2$ are integrable by the bounds above. Finally, $B_{\rm ref}\to k_\infty^2/(1+\eps^2k_\infty^2)^2$, so its large-density normalized remainder tends to zero and has the two required integrable derivatives. This verifies all the hypotheses of~\autoref{def-pression}.}
\ese
\end{proof}


\begin{thebibliography}{10}

\bibitem{AbbresciaSpeck-2025}
\auth{Abbrescia L. and J. Speck,}
Remarkable localized integral identities for 3D compressible Euler flow and the double-null framework,
\emph{Arch. Ration. Mech. Anal.} 249 (2025), Article~10.

\bibitem{CL1-2000} 
\auth{Chen G.-Q. and P.G. LeFloch,}
Compressible Euler equations with general pressure law,
\emph{Arch. Ration. Mech. Anal.} 153 (2000), 221--259.

\bibitem{CL2-2003} 
\auth{Chen G.-Q. and P.G. LeFloch,}
Existence theory for the isentropic Euler equations,
\emph{Arch. Ration. Mech. Anal.} 166 (2003), 81--98.

\bibitem{CP-2010}
\auth{Chen G.-Q. and M. Perepelitsa},
Vanishing viscosity limit of the Navier--Stokes equations to the Euler equations for compressible fluid flow,
\emph{Comm. Pure Appl. Math.} 63 (2010), 1469--1504.

\bibitem{CS-2022} 
\auth{Chen G.-Q. and M.R.I. Schrecker,}
Global entropy solutions and Newtonian limit for the relativistic Euler equations,
\emph{Ann. PDE} 8 (2022), Article~10.

\bibitem{Dafermos-1973-entropy-rate}
\auth{Dafermos C.M.,}
The entropy rate admissibility criterion for solutions of hyperbolic conservation laws,
\emph{J. Differential Equations} 14 (1973), 202--212.

\bibitem{Dafermos-1979-stability}
\auth{Dafermos C.M.,}
The second law of thermodynamics and stability,
\emph{Arch. Rational Mech. Anal.} 70 (1979), 167--179.

\bibitem{Dafermos-book}
\auth{Dafermos C.M.,}
\emph{Hyperbolic conservation laws in continuum physics},
5th ed.,
Grundlehren der mathematischen Wissenschaften, Vol.~325,
Springer, Cham, 2024.

\bibitem{DiPerna-1983a}
\auth{DiPerna R.J.,}
Convergence of the viscosity method for isentropic gas dynamics,
\emph{Comm. Math. Phys.} 91 (1983), 1--30.

\bibitem{DiPerna-1983b}
\auth{DiPerna R.J.,}
Convergence of approximate solutions to conservation laws,
\emph{Arch. Ration. Mech. Anal.} 82 (1983), 27--70.

\bibitem{Frauendiener-2003}
\auth{Frauendiener J.,}
A note on the relativistic Euler equations,
\emph{Class. Quantum Grav.} 20 (2003), L193--L196.

\bibitem{FriedrichsLax-1971}
\auth{Friedrichs K.O. and P.D. Lax,}
Systems of conservation equations with a convex extension,
\emph{Proc. Nat. Acad. Sci. USA} 68 (1971), 1686--1688.

\bibitem{PG-PLF}
\auth{Germain P. and P.G. LeFloch,}
The finite energy method for compressible fluids: the Navier--Stokes--Korteweg model,
\emph{Comm. Pure Appl. Math.} 69 (2016), 3--61.

\bibitem{Godunov-1961}
\auth{Godunov S.K.,}
An interesting class of quasilinear systems,
\emph{Soviet Math. Dokl.} 2 (1961), 947--949.

\bibitem{Hsu-Makino-2001}
\auth{Hsu C.-H., S.-S. Lin, and T. Makino,}
On the relativistic Euler equation,
\emph{Methods Appl. Anal.} 8 (2001), 159--208.

\bibitem{TorusT2-1-2026}
\auth{Le Floch B. and P.G. LeFloch,}
Stability and instability of torus-symmetric Einstein spacetimes with square-integrable connection,
Preprint, 2026.

\bibitem{TorusT2-2-2026}
\auth{Le Floch B. and P.G. LeFloch,} Finite-energy spacetimes with torus symmetry. Cauchy stability of Einstein--Euler and Einstein--Navier--Stokes areal flows, Preprint arXiv:2608.31171, 2026.

\bibitem{TorusT2-3-asymptotic-2026}\auth{Le Floch B. and P.G. LeFloch,} Quasi-currents of perfect compressible fluids and their non-relativistic limits, Preprint, 2026.

\bibitem{TorusT2-4-2026}
\auth{Le Floch B. and P.G. LeFloch,}
The zero-viscosity limit for torus-symmetric Einstein--Navier--Stokes flows,
Preprint, 2026.

\bibitem{LeFloch-curved}
\auth{LeFloch P.G.,}
Hyperbolic balance laws with entropy on a curved spacetime. The weak-strong uniqueness theory,
\emph{J. Hyperbolic Differ. Equ.} 10 (2013), 773--798.

\bibitem{LeFlochShelukhin-2005} 
\auth{LeFloch P.G. and V. Shelukhin,}
Symmetries and global solvability of the isothermal gas dynamics equations,
\emph{Arch. Ration. Mech. Anal.} 175 (2005), 389--430.

\bibitem{LeFlochUkai-2009} 
\auth{LeFloch P.G. and S. Ukai,}
A symmetrization of the relativistic Euler equations in several spatial variables,
\emph{Kinet. Relat. Models} 2 (2009), 275--292.

\bibitem{LeFlochWestdickenberg-2007} 
\auth{LeFloch P.G. and M. Westdickenberg,}
Finite energy solutions to the isentropic Euler equations with geometric effects,
\emph{J. Math. Pures Appl.} 88 (2007), 389--429.

\bibitem{LeFlochYamazaki-2007} 
\auth{LeFloch P.G. and M. Yamazaki,}
Entropy solutions of the Euler equations for isothermal relativistic fluids,
\emph{Intern. J. Dynamical Systems and Differential Equations} 1 (2007), 20--37.

\bibitem{Lions-1-1996} 
\auth{Lions P.-L., B. Perthame, and P.E. Souganidis,}
Existence and stability of entropy solutions for the hyperbolic systems of isentropic gas dynamics in Eulerian and Lagrangian coordinates,
\emph{Comm. Pure Appl. Math.} 49 (1996), 599--638.

\bibitem{Lions-2-1994} 
\auth{Lions P.-L., B. Perthame, and E. Tadmor,}
Kinetic formulation of the isentropic gas dynamics and \(p\)-systems,
\emph{Comm. Math. Phys.} 163 (1994), 415--431.

\bibitem{Makino-2001}
\auth{Makino T.,}
Weak solutions to the compressible Euler equation with an asymptotic \(\gamma\)-law,
\emph{J. Math. Kyoto Univ.} 41 (2001), 557--592.

\bibitem{MakinoMizohataUkai-1992}
\auth{Makino T., K. Mizohata, and S. Ukai,}
The global weak solutions of compressible Euler equation with spherical symmetry,
\emph{Japan J. Indust. Appl. Math.} 9 (1992), 431--449.

\bibitem{MakinoMizohataUkai-1994}
\auth{Makino T., K. Mizohata, and S. Ukai,}
Global weak solutions of the compressible Euler equation with spherical symmetry. II,
\emph{Japan J. Indust. Appl. Math.} 11 (1994), 417--426.

\bibitem{Olver}
\auth{Olver F.W.J.,}
\emph{Asymptotics and special functions},
Academic Press, New York, 1974.

\bibitem{SchreckerSchulz-2020}
\auth{Schrecker M.R.I. and S. Schulz,}
Inviscid limit of the compressible Navier--Stokes equations for asymptotically isothermal gases,
\emph{J. Differential Equations} 269 (2020), 8640--8685.

\bibitem{SchreckerSchulz-2021-corrigendum}
\auth{Schrecker M.R.I. and S. Schulz,} Corrigendum to Inviscid limit of the compressible Navier--Stokes equations for asymptotically isothermal gases, \emph{J. Differential Equations} 287 (2021), 78--87.

\bibitem{Tartar-1979}
\auth{Tartar L.,}
Compensated compactness and applications to partial differential equations,
in \emph{Nonlinear analysis and mechanics: Heriot-Watt Symposium, IV},
Res. Notes in Math., Vol.~39, Pitman, Boston, 1979, pp.~136--212.

\bibitem{Walton-2005}
\auth{Walton R.A.,}
A symmetric hyperbolic structure for isentropic relativistic perfect fluids,
\emph{Houston J. Math.} 31 (2005), 145--160.
 \end{thebibliography}
\end{document}